\documentclass[reqno,11pt]{amsart}
\usepackage{setspace}
\makeatletter

\usepackage[T1]{fontenc}
\usepackage{lmodern}

\usepackage{fullpage}
\usepackage{textcmds}  
\usepackage{amsmath, amssymb, amsfonts, amstext, verbatim, amsthm, mathrsfs, stmaryrd}
\usepackage{mathdots}
\usepackage[all,cmtip,2cell]{xy}

\usepackage{pgf,tikz,pgfplots}
\pgfplotsset{compat=1.15}
\usetikzlibrary{arrows}
\usetikzlibrary{positioning}
\usetikzlibrary{quotes}
\usepackage{tikz-cd}

\usepackage{graphicx}
\graphicspath{ {images/} }

\tikzset{
  symbol/.style={
    draw=none,
    every to/.append style={
      edge node={node [sloped, allow upside down, auto=false]{$#1$}}}
  }
}

\usepackage{enumerate}
\usepackage{enumitem}
\setlist{topsep=1em, itemsep=1em}
\usepackage[colorlinks=true,linkcolor=blue,citecolor=blue,urlcolor=blue,citebordercolor={0 0 1},urlbordercolor={0 0 1},linkbordercolor={0 0 1}]{hyperref} 
\usepackage[nameinlink, capitalize]{cleveref}
\usepackage[shortalphabetic]{amsrefs} 
\usepackage{enotez}
\setenotez{backref=true}

\def\mydefc#1{\expandafter\def\csname c#1\endcsname{\mathcal{#1}}}
\def\mydefallc#1{\ifx#1\mydefallc\else\mydefc#1\expandafter\mydefallc\fi}
\mydefallc ABCDEFGHIJKLMNOPQRSTUVWXYZ\mydefallc

\def\mydefb#1{\expandafter\def\csname b#1\endcsname{\mathbb{#1}}}
\def\mydefallb#1{\ifx#1\mydefallb\else\mydefb#1\expandafter\mydefallb\fi}
\mydefallb ABCDEFGHIJKLMNOPQRSTUVWXYZ\mydefallb

\theoremstyle{plain}
\newtheorem{thm}{Theorem}[section]
\newtheorem{cor}[thm]{Corollary}
\newtheorem{lem}[thm]{Lemma}

\newtheorem{conj}{Conjecture}
\newtheorem{prop}[thm]{Proposition}

\theoremstyle{definition}
\newtheorem{rem}[thm]{Remark}
\newtheorem{defn}[thm]{Definition}

\newtheorem{notn}[thm]{Notation}
\newtheorem{const}[thm]{Construction}
\newtheorem{ex}[thm]{Example}

\newtheorem{warning}[thm]{Warning}

\newtheorem{introthm}{Theorem}

\newtheorem{introconj}{Conjecture}

\def\rm{\mathrm}

\DeclareMathOperator{\rk}{rank}

\DeclareMathOperator{\Aut}{Aut}

\DeclareMathOperator{\DCoh}{D^b_{coh}}
\DeclareMathOperator{\Dqc}{D_{qc}}
\DeclareMathOperator{\Coh}{Coh}
\DeclareMathOperator{\colim}{colim}

\DeclareMathOperator{\Cone}{Cone}

\DeclareMathOperator{\End}{End}
\DeclareMathOperator{\Ext}{Ext}
\DeclareMathOperator{\Fun}{Fun}
\DeclareMathOperator{\GL}{GL}

\DeclareMathOperator{\Hom}{Hom}
\newcommand{\id}{\mathrm{id}}

\DeclareMathOperator{\Map}{Map}

\newcommand{\pt}{\mathrm{pt}}

\DeclareMathOperator{\rank}{rank}

\DeclareMathOperator{\RHom}{RHom}

\DeclareMathOperator{\Span}{Span}
\DeclareMathOperator{\Spec}{Spec}

\DeclareMathOperator{\Stab}{Stab}

\DeclareMathOperator{\Pic}{Pic}

\DeclareMathOperator{\Perf}{Perf}
\DeclareMathOperator{\Ob}{Ob}
\newcommand{\heart}{\heartsuit}

\DeclareMathOperator{\Ind}{Ind}
\DeclareMathOperator{\Pidom}{\Pi}
\DeclareMathOperator{\cofib}{cofib}
\DeclareMathOperator{\fib}{fib}
\DeclareMathOperator{\im}{im}

\DeclareMathOperator{\Filt}{Filt}
\DeclareMathOperator{\Flag}{Flag}
\DeclareMathOperator{\coker}{coker}

\DeclareMathOperator{\Bl}{Bl}

\DeclareMathOperator{\sign}{sign}
\DeclareMathOperator{\gl}{gl}

\DeclareMathOperator{\Grad}{Grad}

\def\ch{v}

\DeclareMathOperator{\logZ}{logZ}
\DeclareMathOperator{\gr}{gr}

\DeclareMathOperator{\dom}{dom}

\DeclareMathOperator{\mscbar}{\mathcal{A}}

\DeclareMathOperator{\Mmscbar}{\mathcal{A}}
\def\d{\mathrm{d}}

\def\db{\mathrm{D}^{\mathrm{b}}}
\def\rmscbar{\mathcal{A}^{\mathbb{R}}}
\def\bf{\mathbf}
\def\cal{\mathcal}
\def\bb{\mathbb}

\def\Astab{\mathcal{A}\Stab}

\usepackage{pbox}
\usepackage[normalem]{ulem}

\makeatletter

\usepackage{babel}
\begin{document}

\title[Augmented stability conditions]{The space of augmented stability conditions}

\author[D. Halpern-Leistner]{Daniel Halpern-Leistner}
\address{Department of Mathematics, Cornell University, Ithaca, NY}
\email{daniel.hl@cornell.edu}

\author[A. Robotis]{Antonios-Alexandros Robotis}
\address{Department of Mathematics, Columbia University, New York, NY}
\email{a.robotis@columbia.edu}

\begin{abstract}
    Given a triangulated category $\cC$, we construct a partial compactification, de\-noted $\Astab(\cC)$, of the quotient of its stability manifold by $\bC$. The purpose of $\Astab(\cC)$ is to shed light on the structure of semiorthogonal decompositions of $\cC$.
    
    A point of $\Astab(\cC)$, called an augmented stability condition on $\cC$, consists of a newly introduced homological structure called a multi-scale decomposition, along with stability conditions on sub\-quotient categories of $\cC$ associated to this multi-scale decomp\-osition. A generic multi-scale decomp\-osition corresponds to a semiorthogonal decomposition along with a configuration of points in $\bC$.
    
    We give a conjectural description of open neighborhoods of certain boundary points, called the \emph{manifold-with-corners conjecture}, and prove it in a special case. We show that this conjecture implies the existence of proper good moduli spaces of Bridgeland semistable objects in $\cC$ when $\cC$ is smooth and proper, and present first examples where the manifold-with-corners conjecture holds.
\end{abstract}

\maketitle

\tableofcontents

\addtocontents{toc}{\protect\setcounter{tocdepth}{1}}

\linespread{1.15}\selectfont

\section{Introduction}

In the seminal paper \cite{Br07}, Bridgeland introduced the notion of a stability condition on a triangulated category $\cC$. We first equip $\cC$ with a surjective homomorphism $v :  \rm{K}_0(\cC) \to \Lambda$ to a finite rank free abelian group $\Lambda$. A stability condition consists of a bounded $t$-structure on $\cC$, which is determined by its heart $\cA \subset \cC$, along with a homomorphism $Z : \Lambda \to \bC$, called the \emph{central charge}, satisfying certain axioms. (We recall the detailed definition and standard terminology in \Cref{S:background}.) The main theorem \cite{Br07}*{Thm. 1.2}, reformulated using the support property of \cite{KS08}, states that the set of stability conditions $\Stab_\Lambda(\cC)$ admits a natural metric topology such that the map $\Stab_\Lambda(\cC)\to \Hom(\Lambda,\bC)$ given by $(Z,\cA)\mapsto Z$ is a local homeomorphism. In particular, $\Stab_\Lambda(\cC)$ has a canonical structure of a complex manifold. It is striking that adding the data of the central charge $Z$ allows one to construct a moduli space parameterizing homological structures on $\cC$.

Because triangulated categories are such a general tool in mathematics, stability manifolds have been studied in algebraic geometry \cites{BMSabelian,BMunreasonable,BridgelandK3,ABLStabilitysurf,Toda_2014}, representation theory \cites{BSSA2,CY2Licata,DKwild,IkedaAnCYN}, and symplectic topology \cites{HaidenCY3,HKK,BridgelandSmith}. In the few examples where the global structure of the stability manifold is known, they have turned out to be geometrically interesting spaces themselves, such as moduli spaces of quadratic differentials \cites{BridgelandSmith,HKK}.

In this paper, we extend the construction of $\Stab_{\Lambda}(\cC)$ to include a different homological structure on $\cC$, a \emph{semiorthogonal decomposition}. A semiorthogonal decomposition, as intro\-duced in \cite{B-KSerre} and written $\cC = \langle \cC_1,\ldots,\cC_n \rangle$, is a collection of full triangulated subcategories $\cC_1,\ldots, \cC_n$ that generate $\cC$ and satisfy the semiorthogonality condition: $i<j$ implies that $\Hom_{\cC}(\cC_j,\cC_i) = 0$. In many senses, the structure of $\cC$ is determined by the simpler categories $\cC_i$. For instance, one has a decomposition of $K$-groups $ \rm{K}_0(\cC) =  \rm{K}_0(\cC_1)\oplus \cdots \oplus  \rm{K}_0(\cC_n)$, and likewise for any additive invariant. Semiorthogonal decompositions are one of the main tools for studying the structure of derived categories of coherent sheaves, and the semiorthogonal factors that appear in $\DCoh(X)$ are expected to carry subtle information about the birational geometry of $X$ \cites{BondalOrlovSOD, KawamataDequivalence, Kuznetsov2010}.
 
The main contributions of this paper are the following:
\begin{enumerate}
    \item We introduce a Hausdorff partial compactification $\Astab(\cC)$ of $\Stab(\cC)/\bC$, called the space of \emph{augmented stability conditions} on $\cC$. It is based on a novel interpretation of a stability condition as a rational curve with infinitely many marked points, labeled by nonzero objects of $\cC$.

    \item We establish a general result on the existence of proper moduli spaces of semistable objects in a smooth and proper dg-category, and more generally a smooth and proper family of dg-categories over a base. The existence result depends on a certain \emph{mass-Hom} inequality, and we show that a conjecture on the structure of the boundary of $\Astab$, the \emph{manifold with corners} conjecture, implies that this inequality always holds.

    \item In the process of defining $\Astab(\cC)$, we introduce a new moduli space of marked genus $0$ curves, the space of \emph{multiscale lines} $\cA_n$, and a new homological structure on $\cC$ called a \emph{multiscale decomposition}, which lets one study mutations of semiorthogonal decompositions as a continuous (as opposed to discrete) process.
\end{enumerate}

\subsection*{Main results about \texorpdfstring{$\Astab(\cC)$}{AStab(C)}}
In order to summarize the main results about the space of augmented stability conditions, we need some terminology. We introduce the notion of a  \emph{generic} augmented stability condition in \Cref{D:terminology}, and an augmented stability condition is \emph{admissible} (\Cref{D:admissible_point}) if, roughly speaking, it can be perturbed to a generic point in any direction. Recall that $\Stab(\cC)$ admits a canonical action of the additive group $\bC$, determined by the property of preserving semistability and its action on the central charge $Z \mapsto e^\alpha Z$ for $\alpha \in \bC$.
\begin{introthm}
    For any stable dg-category $\cC$, $\Astab(\cC)$ has a natural Hausdorff topology such that $\Stab(\cC)/\bC$ embeds as an open subspace (\Cref{T:topology}). Furthermore:
    \begin{enumerate}
        \item To any generic point in $\Astab(\cC)$, one can assign an underlying semiorthogonal decomp\-osition $\cC = \langle \cC_1,\ldots,\cC_n \rangle$, a decomposition $\Lambda = \bigoplus \Lambda_i$ compatible with $ \rm{K}_0(\cC) = \bigoplus  \rm{K}_0(\cC_i)$ under $v :  \rm{K}_0(\cC) \to \Lambda$, and elements $\sigma_i \in \Stab(\cC_i)/\bC$ for $i=1,\ldots,n$; 
        \item Any admissible point in $\Astab(\cC)$ lies in the closure of $\Stab(\cC)/\bC$ (\Cref{P:admissible_points_in_closure}); and 
        \item At a generic admissible point, $\Astab(\cC)$ admits a canonical structure of a real manifold with corners (\Cref{T:genericmanifoldwithcorners}).
    \end{enumerate} 
\end{introthm}

The main open problem about $\Astab(\cC)$ is to show that (3) holds at any admissible point. This is a much more subtle issue, and we will formulate it precisely as the manifold-with-corners conjecture, \Cref{conj:manifold_with_corners_simplified}, after describing augmented stability conditions in more detail. It provides conjectural coordinates in a ``neighborhood of $\infty$'' on the quotient $\Stab(\cC) / \bC$. In examples, it sheds light on the global structure of the stability manifold.

\subsection*{An unexpected connection to moduli spaces}

In algebraic geometry, one can sometimes use a stability condition on the derived category of coherent sheaves $\cC = \DCoh(X)$ on a projective variety $X$ to construct moduli spaces of semistable complexes on $X$. These moduli spaces exhibit interesting wall-crossing behavior as one varies the stability condition, and have been very useful for studying the birational geometry of hyperk\"{a}hler manifolds \cites{BMMMPK3,ProjbiratBM}. Constructions of these moduli spaces currently only work on a case-by-case basis, but our investigation of the boundary of $\Stab(\cC)/\bC$ has led to different approach that works more generally.

We discovered that \Cref{conj:manifold_with_corners_simplified} implies a boundedness result, formulated as \Cref{conj:boundedness} in the body of the paper. The latter states that for any stability condition on a smooth and proper stable dg-category $\cC$ and any object $G \in \cC$, there is a constant $c_G>0$ such that $\dim \Hom(G,E) \leq c_G\cdot m(E)$ for all $E \in \cC$. Surprisingly, this bound is enough to establish the existence of moduli spaces:

\begin{introthm}[=\Cref{T:moduli_spaces}, simplified]
    If \Cref{conj:boundedness} holds for a smooth and proper stable dg-category $\cC$, then for any $\sigma \in \Stab(\cC)$, the stack parameterizing families of semi\-stable $E \in \cC$ with a fixed $\ch(E) \in \Lambda$ is algebraic and admits a proper good moduli space.
\end{introthm}

In fact, the existence of moduli spaces is equivalent to a similar but slightly weaker bound on $\dim \Hom(G,E)$. In order to give the precise formulation in \Cref{T:moduli_spaces}, we introduce in \Cref{D:locally_constant_stability_condition} the notion of a locally constant family of stability conditions on a smooth and proper stable dg-category $\cC$ over a Nagata ring $k$. 

\begin{rem}
    Our definition of locally constant families of stability conditions is inspired by that of \cite{families}, but we feel it is more streamlined for the purpose of establishing existence of moduli spaces. We omit the existence of HN filtrations from the definition of a slicing, resulting in a structure, which we call a sluicing (\Cref{D:sluicing}), that has convenient base change properties. HN filtrations still exist, but are a consequence of the support property.
\end{rem}

\subsection*{Multi-scale decompositions}

To refine the notion of a semiorthogonal decomposition of $\cC$, we introduce a new homological structure called a \emph{multi-scale decomposition}. This is our main conceptual innovation. We denote a multi-scale decomposition $\cC = \langle \cC_\bullet \rangle_{\Sigma}$, where: 
\begin{itemize}
    \item $\Sigma$ is an object that we call a \emph{multi-scale line}, which is a compact connected arithmetic genus $0$ nodal curve equipped with the following data: 1) a smooth marked point $p_\infty$, which gives its dual graph $\Gamma(\Sigma)$ the structure of a rooted tree; 2) a total preorder $\preceq$ on the set of irreducible components of $\Sigma$; and 3) for each irreducible $\Sigma_v \subset \Sigma$, a meromorphic differential $\omega_v$ on $\Sigma_v$ having a pole of order $2$ at the node closest to $p_\infty$, and no other zeros or poles. We call the vertices of $\Gamma(\Sigma)$ that are farthest from $p_\infty$ \emph{terminal}. 

    \item $\cC_\bullet$ denotes a collection of full thick triangulated subcategories $\cC_{\leq v} \subset \cC$, indexed by the terminal vertices of $\Gamma(\Sigma)$. The categories $\cC_{\le v}$ satisfy axioms, which include generating $\cC$ and semiorthogonality conditions.
\end{itemize}

We refer the reader to \Cref{D:multi-scaleline} and \Cref{D:multi-scale_decomposition} for the full set of axioms that these data must satisfy. Instead, we will explain here the simplest non-trivial case:

\begin{ex}\label{E:basic_decomp}
Suppose all components of $\Sigma$ are terminal except for the root component $\Sigma_r \subset \Sigma$, i.e. the one containing $p_\infty$ --- see \Cref{F:2level} for a visualization. One can choose an isomorphism of marked curves with meromorphic differential $(\Sigma_r,p_\infty,\omega_r) \cong (\bP^1,\infty,dz)$, unique up to the action of $\bC$ by trans\-lation on the latter. Under this isomorphism, the set of nodes in $\Sigma_r \setminus \{p_\infty\}$ corresponds to a configuration of distinct complex numbers $n_1,\ldots,n_k \in \bC$. This configuration, taken up to simultaneous translation and scaling by a positive real number, uniquely determines the multi-scale line $\Sigma$ up to a notion of equivalence called \emph{real oriented isomorphism}.

The categorical data consist of thick triangulated subcategories $\cC_{\leq v_i} \subset \cC$, where $v_i$ is the vertex corresponding to the component that connects to $\Sigma_r$ at $n_i$. In this case the axioms state that the $\cC_{\le v_i}$ together generate $\cC$ as a triangulated category, that $\cC_{\le v_i}\subset \cC_{\le v_j}$ whenever $n_j-n_i \in \bR_{>0} \subset \bC$, and that $\Hom_{\cC}(\cC_{\le v_j},\cC_{\le v_i}) = 0$ whenever the imaginary part $\Im(n_j-n_i)>0$. This structure is quite similar to the filtered semiorthogonal decompositions obtained in \cite{quasiconvergence}*{Thm. 2.29}, except now with the additional data of the configuration $(n_1,\ldots,n_k)$.

In the extreme case where $n_i-n_j \in \bR, \forall i,j$, after reindexing so that $n_i$ is monotone increasing, we obtain a filtration $\cC_{\le v_1}\subsetneq \cdots \subsetneq \cC_{\le v_k} = \cC$, and we say $\gr_{v_i}(\cC_\bullet) = \cC_{\leq v_i} / \cC_{\leq v_{i-1}}$. In the \emph{generic} case, where $n_i - n_j \notin \bR$ for any $i \neq j$, after reindexing so that $\Im(n_i)$ is monotone increasing, we obtain a semiorthogonal decomposition $\cC = \langle \cC_{\le v_1},\ldots, \cC_{\le v_k}\rangle$. In this case we say $\gr_{v_i}(\cC_\bullet) = \cC_{\leq v_i}$.
\end{ex}

Returning to a generic multi-scale decomposition $\cC = \langle \cC_\bullet \rangle_\Sigma$, for any terminal vertex $v \in \Gamma(\Sigma)$, we let $\gr_v(\cC_\bullet)$ denote the Drinfeld--Verdier quotient $\cC_{\leq v} / \cC_{<v}$, where $\cC_{<v}$ is the thick triangulated subcategory generated by all $\cC_{\leq u}$ that are contained in $\cC_{\leq v}$. The definition of a multi-scale decomposition in \Cref{D:multi-scale_decomposition} is the minimal structure that generalizes \Cref{E:basic_decomp} and has a recursive nature: given a multi-scale decomposition $\cC = \langle \cC_\bullet \rangle_\Sigma$ and additional multi-scale decomp\-ositions on $\gr_v(\cC_\bullet)$ for each terminal vertex $v \in \Gamma(\Sigma)$, one can extend this to a new multi-scale decomposition $\cC = \langle \cC'_\bullet \rangle_{\Sigma'}$ such that $\langle \cC_\bullet \rangle_\Sigma$ is a \emph{coarsening} of $\langle \cC'_\bullet \rangle_{\Sigma'}$ in the sense of \Cref{D:multi-scaleSODcoarsening}.

\subsection*{Augmented stability conditions}

Given $\sigma \in \Stab(\cC)$ and a $\sigma$-semistable object $E$ of phase $\phi\in \bR$, there is a canonically defined logarithm of $Z(E)$, given by 
\[ 
    \logZ(E) := \log \lvert Z(E)\rvert + i\pi \phi,
\] 
which is holomorphic on its domain of definition. A key idea in this paper is that there is a natural continuous extension of $\logZ(E)$ to all of $\Stab(\cC)$, which we denote by $\ell(E)$. For $\sigma \in \Stab(\cC)$,
\[
    \ell_\sigma(E) := \log m_\sigma(E) + i\pi \phi_\sigma(E),
\]
where $m_\sigma(E)$ is the mass of $E$ with respect to $\sigma$ and $\phi_\sigma(E)$ is the \emph{average phase} of the semistable factors of $E$ (see \Cref{S:background}). Furthermore, given $z\in \bC$, $\ell_{z\cdot \sigma}(E) = \ell_\sigma(E) + z$ so that for any nonzero $E,F\in \cC$, $\ell(E) - \ell(F)$ descends to a continuous function on $\Stab(\cC)/\bC$.

By \Cref{P:logZ_functions}, the stability condition $\sigma$ is uniquely determined by the function $\ell_{\sigma} : \cC \setminus 0 \to \bC = \bP^1 \setminus \infty$, and two stability conditions lie in the same $\bC$-orbit if and only if the corresponding functions differ by an automorphism of $(\bP^1,\infty)$ that preserves $dz$. Thus, we arrive at a different description of the space $\Stab(\cC) / \bC$ as parameterizing smooth genus $0$ curves with meromorphic differential $(\Sigma, p_\infty, \omega)$ along with infinitely many (possibly overlapping) marked points labeled by isomorphism classes of nonzero objects of $\cC$.

\begin{defn} (\Cref{D:generalized_stability_condition})
An \emph{augmented stability condition}, written $\langle \cC_\bullet|\sigma_\bullet\rangle_\Sigma$, is a multi-scale decomposition $\langle \cC_\bullet\rangle_\Sigma$ equipped with marking functions $\ell_v : \gr_v(\cC_\bullet) \setminus 0 \to \Sigma_v \setminus \{\infty\}$ that correspond to an element $\sigma_v \in \Stab(\gr_v(\cC_\bullet)) / \bC$ for each terminal $v \in \Gamma(\Sigma)$.    
\end{defn}

This generalizes the perspective above. When $(\Sigma,p_\infty,\omega) \cong (\bP^1, \infty, dz)$, an augmented stability condition is nothing more than an element of $\Stab(\cC)/\bC$, giving an inclusion $\Stab(\cC)/\bC \hookrightarrow \Astab(\cC)$. More generally, an augmented stability condition $\langle \cC_\bullet | \ell_\bullet \rangle_{\Sigma}$ is completely determined by a multi-scale line $\Sigma$ and the marking function
\[
\ell : \bigsqcup_v (\cC_{\leq v} \setminus \cC_{<v}) \to \Sigma,
\]
which assigns $E \in \cC_{\leq v} \setminus \cC_{<v}$ to $\ell_v(E)$, regarding $E$ as an element in the quotient $\gr_v(\cC_\bullet)$. The disjoint union is over terminal $v \in \Gamma(\Sigma)$. More economically, the augmented stability condition is determined by the marking function only on the subset of stable objects, which are objects in some $\cC_{\leq v}$ whose image in $\gr_v(\cC_\bullet)$ is nonzero and stable (\Cref{D:terminology}).

\begin{figure}
\label{fig:logZconvergence}
    \begin{center}
\scalebox{0.75}{
\begin{tikzpicture}[>={Stealth}, scale=1]

\fill[gray!10] (0,0) rectangle (10,6);

\draw[thick] (0,0) rectangle (10,6);

\draw[blue, thick] (2,4.5) circle (1);
\fill (1.8,4.6) circle (2pt);
\fill (2.2,4.4) circle (2pt);
\node[blue, right=4pt] at (3,4.5) {\(\sim \alpha_1 t\)};  

\draw[->, blue, thick]
  (2,4.5) ++(135:1) -- ++(135:0.8);

\draw[red, thick] (7.5,4.5) circle (1);
\fill (7.3,4.7) circle (2pt);
\fill (7.6,4.3) circle (2pt);
\fill (7.9,4.6) circle (2pt);
\node[red, right=4pt] at (8.5,4.5) {\(\sim \alpha_2 t\)};

\draw[->, red, thick]
  (7.5,4.5) ++(45:1) -- ++(45:0.8);

\draw[green!60!black, thick] (5,1.5) circle (1);
\fill (4.7,1.4) circle (2pt);
\fill (5.1,1.6) circle (2pt);
\fill (4.9,1.2) circle (2pt);
\fill (5.3,1.3) circle (2pt);
\node[green!60!black, right=4pt] at (6,1.5) {\(\sim \alpha_3 t\)};

\draw[->, green!60!black, thick]
  (5,1.5) ++(270:1) -- ++(270:0.4);

\node[anchor=south east] at (10,0) {\(\mathbb{C}\)};

\draw[->, thick, decorate, decoration={snake, amplitude=1mm, segment length=7mm}]
  (10.5,3) -- (13,3);
\node[above=6pt] at (11.75,3) {\(\displaystyle t \to \infty\)};

\begin{scope}[shift={(16.5,3.72)}, scale=1.2]

\shade[ball color=gray!10] (0,0) circle (1.44);

\fill (0,1.44) circle (2pt);
\node[above] at (0,1.44) {\(\displaystyle p_\infty\)};

\shade[ball color=green!20] (0,-2.16) circle (0.72);
\fill (-0.3,-2.3) circle (2pt);
\fill (0.1,-2.1) circle (2pt);
\fill (-.1,-2.5) circle (2pt);
\fill (.3,-2.4) circle (2pt);

\fill[green!40!black] (0,-1.44) circle (2pt);
\node[green!40!black, above = 2pt] at (0.05,-1.44) {\(\alpha_3\)};

\shade[ball color=blue!20] (-1.66,-1.27) circle (0.72);
\fill (-1.9,-1.1) circle (2pt);
\fill (-1.5,-1.3) circle (2pt);

\fill[blue] (-1.15,-0.76) circle (2pt);
\node[blue, above=2pt] at (-1.15,-0.76) {\(\alpha_1\)};

\shade[ball color=red!20] (1.66,-1.27) circle (0.72);
\fill (1.5,-1.1) circle (2pt);
\fill (1.8,-1.5) circle (2pt);
\fill (2.1,-1.2) circle (2pt);

\fill[red] (1.15,-0.76) circle (2pt);
\node[red, above=2pt] at (1.15,-0.76) {\(\alpha_2\)};

\end{scope}

\end{tikzpicture}
}
\end{center}
    \caption{Consider a path $\sigma_t$ in $\Stab(\cC)/\bC$ and $n$ objects $E_1,\ldots, E_n$ such that $v(E_1),\ldots, v(E_n)$ is a basis for $\Lambda_{\bQ}$. The marked points in each cluster on the left represent logarithms of central charges of objects $E_j$, grouped according to their approximate growth rates $\alpha_i t$ for $\alpha_i \in \bC$. As $t\to\infty$, the limit of these configurations is the point in $\cA_n^{\bR}$ depicted on the right side of the diagram. The information of the relative displacement between points within a cluster is preserved in the limit.}
\end{figure}

\subsection*{The topology}

The idea for defining a topology on $\Astab(\cC)$ is to describe how $(\bP^1,\infty,dz)$ degenerates in a space of ``infinitely marked'' multi-scale lines. To make this rigorous, we construct a moduli space $\mscbar_n$ of stable $n$-marked multi-scale lines up to a notion of \emph{complex projective isomorphism}, inspired by \cite{BCGGM}. We prove as \Cref{T:spaceconstruction} that $\mscbar_n$ is a smooth and proper complex variety which is a simple normal crossings compactification of $\bC^n/\bC$. 

The real oriented blowup $\rmscbar_n$ of $\mscbar_n$ along the simple normal crossings boundary divisor obtained as the complement of $\bC^n/\bC \hookrightarrow \cA_n$ has an interpretation as a moduli space of stable $n$-marked multi-scale lines up to a stricter notion of equivalence, called \emph{real oriented isomorphism}. As we will see, the manifold with corners $\rmscbar_n$ plays a role in our theory analogous to the vector space $\Hom(\Lambda,\bC)$.

For any multi-scale decomposition $\cC = \langle \cC_\bullet \rangle_\Sigma$, we let $U(\langle \cC_\bullet \rangle_\Sigma) \subset \Astab(\cC)$ denote the augmented stability conditions whose underlying multi-scale decomposition is a coarsening of $\langle \cC_\bullet \rangle_\Sigma$  --- see \Cref{D:multi-scaleSODcoarsening}. Now consider a collection $E_1,\ldots,E_N \in \bigsqcup_v (\cC_{\leq v} \setminus \cC_{<v})$ such that for every terminal $v \in \Gamma(\Sigma)$, at least one $E_i \in \cC_{\leq v} \setminus \cC_{<v}$. For any $\sigma \in U(\langle \cC_\bullet \rangle_\Sigma)$, the marking function $\ell_\sigma$ for $\sigma$ allows one to define a stable $N$-marked multi-scale line $\ell_\sigma (E_1,\ldots,E_N):=(\Sigma, \ell_\sigma(E_1),\ldots,\ell_\sigma(E_N))$. The assignment $\sigma \mapsto \ell_\sigma(E_1,\ldots,E_N)$ defines a function
\begin{equation} \label{E:ell_function}
    \ell_{(-)}(E_1,\ldots,E_N) : U(\langle \cC_\bullet \rangle_\Sigma) \to \rmscbar_N.
\end{equation}
These locally defined maps are analogous to the forgetful map $\Stab(\cC) \to \Hom(\Lambda,\bC)$ in the original theory of stability conditions.

In \Cref{D:weak_topology}, we introduce the \emph{weak topology} on $\Astab(\cC)$, which is roughly the weakest topology such that the subsets $U(\langle \cC_\bullet \rangle_\Sigma)$ are open and the functions \eqref{E:ell_function} are conti\-nuous. Heuristically, convergence of a net (e.g. a path) in $\Stab(\cC)/\bC$ to a limit-point in $\Astab(\cC)$ is modeled on the convergence of a configuration of points in $\bC^n/\bC$ in $\cA_n^{\bR}$ as depicted in Figure \ref{fig:logZconvergence}. The precise definition requires concepts that we postpone to the body of the paper --- \eqref{E:ell_function} is continuous for any of the marking functions $\ell^t$, rather than just $\ell = \ell^0$, and for $t$-well-placed objects $E_i$, as opposed to $E_i \in \cC_{\leq v}$. In \Cref{T:unique_limit}, we give simplified criteria for when a net in $\Stab(\cC)/\bC$ converges in $\Astab(\cC)$ with its weak topology.

Although in certain examples satisfying strong finiteness properties, such as when $\cC$ is the bounded derived category of finite dimensional representations of an ADE quiver, the weak topology is expected to recover the quotient topology on $\Stab(\cC)/\bC$, this is not true in general.

To remedy this, in \Cref{T:topology} we define a stronger topology on $\Astab(\cC)$ by specifying exactly which nets are convergent. \Cref{D:directeddistance} introduces a directed distance function $\Vec{d}(\sigma,\tau)$, which is defined whenever $\tau$ is a coarsening of $\sigma$, and which generalizes Bridgeland's metric \cite{Br07}*{Prop. 8.1} on $\Stab(\cC)$ ---  see \Cref{R:localmetric}. The topology in \Cref{T:topology} is described as follows:

\begin{defn}[Convergence]\label{D:convergence_intro}
A net $\{\sigma_\alpha\}_{\alpha \in I}$ in $\Astab(\cC)$ converges to $\sigma \in \Astab(\cC)$ precisely if $\lim_{\alpha \to \infty} \Vec{d}(\sigma,\sigma_\alpha) = 0$, it converges in the weak topology, and for all $\sigma$-stable $E \in \cC_{\leq v}$ and $\alpha$ sufficiently large, $E$ is equivalent modulo $\cC_{<v}$ to a $\sigma_\alpha$-stable object.    
\end{defn}

This topology is Hausdorff and recovers the usual topology of $\Stab(\cC)/\bC$ on the interior. The Hausdorffness is a subtle issue. For instance, for a net in $\Stab(\cC)/\bC$ converging to a boundary point $\sigma$, it requires a prescription for constructing all of the data of $\sigma$, including the underlying multi-scale decomposition, from the net alone.


\begin{rem}
    We were led to \Cref{D:convergence_intro} 
    by the following design principle: the topology should be as weak as possible, should be Hausdorff, should recover the usual topology on the interior, and should allow us to prove the manifold-with-corners conjecture at generic boundary points, as we discuss below. It might be that a stronger topology is required to prove the full manifold-with-corners conjecture.
\end{rem}

\subsection*{Manifold-with-corners conjecture}

\begin{introconj}[Manifold-with-corners, simple version]\label{conj:manifold_with_corners_simplified}
Let $\sigma = \langle \cC_\bullet | \sigma_\bullet \rangle_\Sigma$ be an augmented stability condition on a smooth and proper stable dg-category $\cC$ over a field, and suppose that for any pair of terminal vertices $u,v \in \Gamma(\Sigma)$ with $\cC_{\leq u} \subseteq \cC_{\leq v}$, this inclusion functor admits either a right or left adjoint. Then for any collection $E_1,\ldots,E_n \in \cC$ of $\sigma$-stable objects such that $\ch(E_i)$ form a basis for $\Lambda_\bQ$, the log-central-charge map \eqref{E:ell_function} induces a homeomorphism between an open neighborhood of $\sigma$ and an open subset of $\rmscbar_n$.
\end{introconj}

In the body of the paper we split \Cref{conj:manifold_with_corners_simplified} into two conjectures. The first is the boundedness conjecture, \Cref{conj:boundedness}, discussed above.

The second conjecture is a generalization of \Cref{conj:manifold_with_corners_simplified}. We introduce a technical condition in \Cref{D:admissible_point}, admissibility of a point $\sigma \in \Astab(\cC)$. \Cref{conj:manifoldwithcorners} states that this is precisely the condition under which $\ell_{(-)}(E_1,\ldots,E_n)$ is a local homeo\-morphism at $\sigma$, even when $\cC$ is not smooth and proper. \Cref{conj:boundedness} implies that if $\cC$ is smooth and proper, then a point $\langle \cC_\bullet | \sigma_\bullet \rangle_\Sigma \in \Astab(\cC)$ is admissible if and only if for any terminal $u$ and $v$ with $\cC_{\leq u} \subseteq \cC_{\leq v}$, the inclusion admits either a left or right adjoint (\Cref{P:mass_Hom_bound_implies_admissibility}). Therefore, \Cref{conj:boundedness} and \Cref{conj:manifoldwithcorners} together imply \Cref{conj:manifold_with_corners_simplified}.

We establish a partial result towards \Cref{conj:manifoldwithcorners}. A point $\langle \cC_\bullet | \sigma_\bullet \rangle_\Sigma \in \Astab(\cC)$ is \emph{generic} if there is an ordering of the terminal vertices of $\Gamma(\Sigma)$ such that one has a semiorthogonal decomp\-osition $\cC = \langle \cC_{\leq v_1},\ldots,\cC_{\leq v_n} \rangle$ --- see \Cref{D:terminology}.

\begin{introthm}[=\Cref{T:genericmanifoldwithcorners}] The manifold-with-corners conjecture, \Cref{conj:manifoldwithcorners}, holds at every admissible point $\sigma \in \Astab(\cC)$ that is also generic.
\end{introthm}

The key tool in the proof of this theorem is the gluing construction of \cite{CP10}, which we use to construct local charts around generic admissible points. \Cref{conj:manifoldwithcorners} suggests that there is a missing gluing construction that works in a neighborhood of non-generic but admissible boundary points. This would allow us to describe many more stability conditions than arise from the construction of \cite{CP10}, and would give us more insight into the global structure of stability manifolds.

\subsection*{Noncommutative minimal model program}

Beyond understanding the global structure of stability manifolds, our main motivation for defining $\Astab(\cC)$ is to introduce a tool to system\-atically study semiorthogonal decompositions of $\cC = \DCoh(X)$ for a smooth projective variety $X$ over $\bC$. Kontsevich has proposed that there should be canonical semiorthogonal decompositions of $\DCoh(X)$. This has been elaborated as the noncommutative minimal model program \cite{NMMP}, which implies a version of Dubrovin's conjecture and the $D$-equivalence conjecture.

The idea is, roughly, that solutions of the quantum differential equation on $H^*_{\rm{alg}}(X;\bC)$ should give rise to semiorthogonal decompositions of $\DCoh(X)$, and $\Astab(\cC)$ provides a mechanism to explain this. One can use the quantum differential equation to define canonical paths in the space of central charges $Z_\psi(t) \in \Hom(H^\ast_{\rm{alg}}(X),\bC)$, depending on a parameter $\psi$ \cite{Iritani}. The original version of the NMMP conjecture stated that the paths $Z_\psi(t)$ lift to paths in $\Stab(\cC)/\bC$ that are quasi-convergent in the sense of \cite{quasiconvergence}. The present work provides a refined formulation: the paths that lift $Z_\psi(t)$ are expected to converge in $\Astab(\cC)$. As the parameter $\psi$ varies, the limit-point in $\Astab(\cC)$ is expected to stay on the same boundary component. In particular, rather than a canonical semiorthogonal decomposition, one expects in general to get a canonical \emph{multi-scale decomposition} on $\DCoh(X)$, determined up to mutation, along with stability conditions on the associated graded cate\-gories.

\subsection*{Related work and acknowledgments}
Recently, several other works have appeared which consider the problem of constructing (partial) compactifications of spaces of Bridgeland stability conditions. We mention the ones we are aware of here and provide a brief comp\-arison of their approaches to ours.

\begin{enumerate} 
    \item In \cite{ThurstonBDL}, the authors introduce the \emph{Thurston compactification} of $\Stab(\cC)$. One first chooses a class of objects $\cS\subset \Ob(\cC)$ and positive real numbers $q_1,\ldots, q_n\in \bR_{>0}$. The  compactification is the closure of the image of $\prod_{i=1}^n m_{q_i}:\Stab(\cC)/\bC\to (\bP^{\cS})^n$ where $m_{q_i}(\sigma) = [m_{q_i,\sigma}(E):E\in \cS]$, when it is injective; here $\bP^{\cS}$ is the \emph{real} projective space $(\bR^{\cS}\setminus \{0\})/\bR_{>0}$, and $m_{q,\sigma}(E)$ is the $q$-deformed mass of $E$ with respect to $\sigma$, defined by $m^{\log q}_\sigma(E)$ in the notation of \Cref{S:background}.\endnote{This compactification is defined by analogy to Thurston's compactification of Teichm\"uller space $\cT_g$, which can be defined as the moduli space of hyperbolic structures on a fixed closed topological surface $\Sigma$ of genus $g$. There, one can choose $\cS$ to be a suitable set of closed curves on $\Sigma$ and, regarding points of $\cT_g$ as Riemannian metrics $\mu$, one sends $\mu$ to the set of lengths of the curves in $\cS$. Taking the closure of the image in $\bP^{\cS}$ gives the compactification.} 

    In \cite{ThurstonBDL}, the authors prove that for $n \ge 2$, the map $\prod_{i=1}^n m_{q_i}$ is injective with pre-compact image under mild hypotheses on $\cS$. Consequently, in many cases one obtains a continuous injection of $\Stab(\cC)/\bC$ into a compact space. The only difficulty is that it seems to be difficult to prove that $\prod_{i=1}^n m_{q_i}$ is a homeomorphism onto its image in general. 

    The case of the 2-Calabi--Yau categories associated to the $A_2$ and $\hat{A}_1$ quivers was studied in the original paper \cite{ThurstonBDL}, where the authors prove that in both cases $m_1$ is a homeo\-morphism onto its image. On the other hand, the example of $\DCoh(C)$ for $C$ a smooth and projective curve was considered in \cite{ThurstonKKO}. There, it is proven that $m_1$ is a homeomorphism for $g(C)\ge 1$, while in the case of $\bP^1$ it is proven that for $q = 1$, the map $\Stab(\bP^1)/\bC \to \bP^{\cS}$ is \emph{not} injective for any choice of $\cS$ \cite{ThurstonKKO}*{Prop. 6.3}. 
    
    In \cite{DeopurkarThurstonk3}, the Thurston compactification of the space of stability conditions on analytic K3 surfaces $S$ with $\Pic(S) = 0$ is studied, where $m_1$ is proven to be a homeo\-morphism. Furthermore, the Thurston compactification of $\Stab(S)/\bb{C}$ is shown to be homeomorphic to a closed 2-dimensional ball.

    \item In \cite{Bolognesecompactification}, a different enlargement $\widehat{\Stab}(\cC)$ of $\Stab(\cC)$ was considered; it is constructed as the metric completion of the metric on $\Stab(\cC)$ obtained by pulling back the metric on the finite dimensional vector space $\Hom(\Lambda,\bC)$ induced by a real inner product.\endnote{Implicitly, we are assuming that the central charges factor through a finite rank lattice $\Lambda$ as discussed at the beginning of the introduction. This imposes finite dimensionality on the space of central charges, which is used crucially here to get a metric completion which is independent of choices. This is equivalent to the fact that all norms on a finite dimensional vector space are equivalent!} In the presence of some extra hypotheses on the structure of the forgetful map $\Stab(\cC) \to \Hom(\Lambda,\bC)$, it is proven \cite{Bolognesecompactification}*{Thm. 1.1} that there is an embedding
        \[
            \widehat{\Stab}(\cC) \hookrightarrow \rm{GStab}(\cC)
        \]
    where 
        \[
            \rm{GStab}(\cC) = \{(\cK,\sigma): \cK\subset \cC \text{ a thick subcategory, and } \sigma \in \Stab(\cC/\cK)\}
        \]
    is called the set of \emph{generalized stability conditions}. This gives a modular interpretation of the points added in the completion process under suitable hypotheses.\endnote{The hypotheses of \cite{Bolognesecompactification}*{Thm. 1.1} are that the natural projection $\pi:\Stab(\cC) \to \Hom( \rm{K}_0(\cC),\bC)$ given by $(Z,\cP)\mapsto Z$ is a covering map onto the open dense subset of $\Hom( \rm{K}_0(\cC),\bC)$ given by the complement of a locally finite collection of smooth submanifolds.}

    \item In \cite{BPPWmassless}, two enlargements of the space of Bridgeland stability conditions are constructed. The points of the first space are called \emph{lax stability conditions} and consist of the data of a central charge and a slicing as in \cite{Br07}; the caveat is that the mass of semistable objects may be zero. Consequently, associated to a lax stability condition $\sigma$ is a thick subcategory $\cN$ of massless objects. A lax stability condition induces a stability condition on $\cC/\cN$ by \cite{BPPWmassless}*{Prop. 4.19}.\endnote{This is reminiscent of the generalized stability conditions introduced in \cite{Bolognesecompactification} and a partial comparison is made in \cite{BPPWmassless}*{\S 11.2}.} The space of \emph{quotient stability conditions} on $\cC$, $\Stab^Q(\cC)$, is defined to be the quotient of the space of lax stability conditions, $\Stab^L(\cC)$, by the equivalence relation that identifies lax stability conditions with the same subcategory $\cN$ and that induce the same stability condition on $\cC/\cN$.

    \item In \cite{BMSmulti-scale}, the authors introduce a notion of \emph{multi-scale stability conditions} on a triangulated category. Their construction aims to give a smooth (orbifold) compact\-ification of the double quotient $\bC\backslash\Stab(\cC)/{\Aut(\cC)}$. Roughly, a multi-scale stability condition consists of \vspace{-2mm}
    \begin{enumerate} 
        \item a multi-scale heart $\cA_\bullet = (\cA_i)$, which is a nested collection $\cA_L\subset \cdots \subset \cA_1\subset \cA_0$ of abelian categories satisfying some conditions; and \vspace{-2mm}
        \item a multi-scale central charge, which is a collection $Z_\bullet = (Z_i)_{i=0}^L$ of non-zero homo\-morphism $Z_i: \rm{K}_0(\cA_i) \to \bC$, where $Z_i$ factors through the kernel of $Z_{i-1}$. \vspace{-2mm}
    \end{enumerate}
    
    They are able to topologize the collection of multi-scale stability conditions in the case that $\cC$ is 3-Calabi--Yau category associated to the $A_n$ quiver. In this case, they obtain a complete description of the resulting compactification \cite{BMSmulti-scale}*{Thm. 1.1}, proving that it is a blowup of $\overline{\cM}_{0,n+2}$. One of the difficulties of this approach is that the results depend on finiteness conditions on the hearts of the t-structures encountered in $\Stab(\cC)$ --- see \cite{BMSmulti-scale}*{p. 4}. For example, the derived category of finite dimensional representations of the Kronecker quiver (equivalently $\DCoh(\bP^1)$) presents some technical challenges as discussed in \emph{loc. cit}. 
\end{enumerate}

Our partial compactification incorporates some features of the other const\-ructions detailed above with certain key differences:

\begin{itemize}
    \item The boundary points of $\Astab(\cC)$ that parameterize filtrations $\cC_1\subset \cdots \subset \cC_n = \cC$ along with stability conditions on the associated graded categories $\cC_i/\cC_{i-1}$ are reminiscent of the boundary points in (2), (3), and (4). However, none of the other compactifications has boundary points encoding semiorthogonal decompositions of $\cC$, which is the main goal of $\Astab(\cC)$.\endnote{In fact, the failure of the map $m_1:\Stab(\bP^1)/\bC\to \bP^{\cS}$ to be injective \cite{ThurstonKKO}*{Prop. 6.3} is directly related to the existence of the Beilinson decompositions $\DCoh(\bP^1) = \langle \cO(n),\cO(n+1)\rangle$. Indeed, their result proves non-injectivity for any triangulated category with a strong full exceptional collection.
    
    However, in the case of a two-term semiorthogonal decomposition $\cC = \langle \cC_1,\cC_2\rangle$, the space $\Stab^L(\cC)$ of \cite{BPPWmassless} adds points where $\cC_1$ is regarded as massless. In \cite{BPPWmassless}*{\S 12.9}, the authors study the case of $\DCoh(\bP^1)$ and describe the points of (a canonical subspace of) $\Stab^Q(\cC)/\bC$. The result is similar to our description of $\Astab(\bP^1)$ in \Cref{S:firstexamples} with two major differences: (1) where we add a boundary stratum canonically identified with a half-open interval in $\bR$ for each semiorthogonal decomposition $\DCoh(\bP^1) = \langle \cO(n),\cO(n+1)\rangle$, their space adds a point. Consequently, $\Astab(\bP^1)$ contains a copy of $\bR$ corresponding to the ``admissible'' boundary, while the boundary of the other space contains a copy of $\bZ$. (2) There is a unique non-admissible boundary point in $\Astab(\bP^1)$ which roughly corresponds to the torsion sheaves becoming massless relative to the locally free sheaves. This point is not present in the space considered in \cite{BPPWmassless}*{\S 10.8} as mentioned on p. 57 \emph{ibid}.}

    \item Our construction and main results hold in full generality for a triangulated category $\cC$ equipped with a choice of surjection $ \rm{K}_0(\cC) \twoheadrightarrow\Lambda$, whereas the constructions above, with the exception of \cite{BPPWmassless}, require special properties of $\cC$.
\end{itemize}

\subsubsection*{Acknowledgements}

We are grateful to Severin Barmeier, Arend Bayer, Tom Bridgeland, Hannah Dell, And\-res Fernandez Herrero, Jeffrey Jiang, Kimoi Kemboi, Maxim Kontsevich, Antony Licata, Emanuele Macr\`{i}, Alex Perry, Ian Selvaggi, Yan Soibelman, Yukinobu Toda, and Xiaolei Zhao for many useful conversations on the topics in this paper. We also thank David Ploog for useful comments on the first version of this paper.

The authors appreciate the support of the NSF CAREER grant DMS-1945478, and NSF FRG grant DMS-2052936, and the support and hospitality of the Simons Laufer Mathematical Sciences Institute during the preparation of this paper. The first author was also supported by an Alfred P. Sloan Foundation Research Fellowship (FG-2022-18834), and a Simons Foundation Fellowship.

The second author would additionally like to thank Maria Teresa for her love, patience, and encouragement during the preparation of this manuscript. Additionally, he thanks the first author for his unwavering support and guidance during this project and the duration of his doctoral studies.


\addtocontents{toc}{\protect\setcounter{tocdepth}{2}}

\section{Preliminaries}

\subsection{Stability conditions and associated functions} \label{S:background}

\begin{defn}
\label{D:stabilityconditions}
For the reader's convenience, we recall the definitions of our main objects of study:
\begin{itemize}
    \item A \emph{slicing} $\cP$ on a stable dg-category $\cC$ is a collection of full additive subcategories $\cP(\phi) \subset \cC$ for $\phi \in \bR$, called \emph{semistable objects} of phase $\phi$, satisfying two conditions: 1) If $E$ and $F$ are semistable of phases $\phi_E>\phi_F$, then $\Hom(E,F) = 0$; and 2) Every object $E$ admits a Harder--Narasimhan (HN) filtration $0 = E_0 \to E_{1} \to \cdots \to E_n = E$, meaning that $\gr_i(E_\bullet) := \cofib(E_{i-1} \to E_{i})$ is semistable of some phase $\phi_i$, and $\phi_1>\cdots>\phi_n$.

    \item A \emph{pre-stability condition} on $\cC$ is a pair $\sigma = (Z,\cP)$, where $\cP$ is a slicing, and $Z :  \rm{K}_0(\cC) \to \bC$ is a group homomorphism, called the \emph{central charge}, such that $Z(E) \in \bR_{>0} \cdot e^{i \pi \phi}$ for any $E \in \cP(\phi)$.

    \item Given a finitely generated free abelian group $\Lambda$, equipped with a norm on $\Lambda \otimes \bR$ and a homomorphism $\ch :  \rm{K}_0(\cC) \to \Lambda$, we say that $\sigma$ satisfies the \emph{support property} if $Z :  \rm{K}_0(\cC) \to \bC$ factors through a homomorphism $\Lambda \to \bC$, which we also denote $Z$, and there is a constant $c>0$ such that $Z(E) \geq c \lVert v(E)\rVert$ for all semistable $E$. A \emph{stability condition} (with respect to $\ch :  \rm{K}_0(\cC) \to \Lambda$) is a pre-stability condition satisfying the support property.\endnote{Note that the choice of norm on $\Lambda\otimes \bR$ is irrelevant, since all norms on a finite dimensional vector space are equivalent.}
\end{itemize}
Given a non-zero semistable object $E$, its \emph{mass} with respect to a pre-stability condition $\sigma = (Z,\cP)$ is $m(E) := \lvert Z(E)\rvert$. Given an arbitrary non-zero object $E$ of $\cC$, its mass is 
\[
    m(E) = \sum_{i} m(\gr_i(E_\bullet))
\]
where $\{\gr_i(E_\bullet)\}$ are the Harder--Narasimhan factors defined above. For more on stability conditions, see \cites{Bayer_short,Br07,KS08}.
\end{defn}


If $\cC$ is a stable dg-category with a pre-stability condition $\sigma$, then any $E \in \cC$ defines a finitely supported measure on $\bR$, which we call the \emph{mass measure},
\begin{equation}
\label{E:massmeasure}
\d m_{\sigma,E} := \sum_\theta |Z_\sigma(\gr^\sigma_\theta(E))| \delta_\theta
\end{equation}
where $\delta_\theta$ is the indicator measure of $\{\theta\}$. The \emph{polynomial mass} function on $\cC$, introduced in \cite{DynamicsDHKK}, is defined by
\[
    m_\sigma^t(E) := \int e^{t \theta} \d m_{\sigma,E}(\theta) = \sum_{\theta} e^{t\theta} |Z_\sigma(\gr^\sigma_{\theta}(E))|.
\]
In \cite{massmeasures} we also refer to this as the \emph{deformed mass} of $E$. When the context is clear, we will drop $\sigma$ from the notation, and we will drop $t$ when $t=0$, i.e., $m(E) = m^0(E)$. We will use this same notational convention with the functions $\phi_\sigma^t$ and $\ell_\sigma^t$ introduced below. We refer the reader to \cite{massmeasures} for a more systematic study of the functions $m_\sigma^t$ and $\phi_\sigma^t$, as well as the mass measures $\d m_{\sigma,E}$. A key property of $m_\sigma^t$ is the following:

\begin{thm}
[\cite{massmeasures}*{Thm. D}]
    For any exact triangle $E \to F \to G$ in $\cC$ and $t,a\in \bf{R}$ one has
    \begin{equation} \label{E:triangle_inequality}
        0 \leq m^t_\sigma(E)+m^t_\sigma(G) - m^{t}_\sigma(F) \leq (1+e^{|t|}) \left(m^t_\sigma(G^{>a}) + m^t_\sigma(E^{\leq a}) \right).
    \end{equation}
\end{thm}

The first inequality, the ``triangle inequality,'' appears as \cite{ikeda_2021}*{Prop. 3.3}, but the proof there is incomplete. The reader is referred to \cite{massmeasures} for a complete proof. The second inequality in \eqref{E:triangle_inequality}, which is \cite{massmeasures}*{Cor. 5.4}, says that when the mass of $E$ is concentrated in phase $>a$ and the mass of $G$ is concentrated in phase $\leq a$, then $m^t_\sigma$ is approximately additive on the exact triangle $E \to F \to G$.

In fact, we will need the following stronger version of \eqref{E:triangle_inequality}, which applies to filtrations of arbitrary length, as opposed to just length two.

\begin{thm} 
[\cite{massmeasures}*{Thm. 6.1}]
\label{T:filtration_inequality}
Let $a_1<\cdots<a_{n-1}$, and let $\epsilon \in (0,1)$. Then there is a constant $C_{n,\epsilon,t}>0$, depending only on $n$, $\epsilon$, and $t$, such that for any $E \in \cC$ with a finite descending filtration $E = E_n \to E_{n-1} \to \cdots \to E_1 \to E_0 = 0$ with associated graded objects $F_i := \fib(E_i \to E_{i-1})$,
\[
    m^t(E) \leq \sum_{i=1}^n m^t(F_i) \leq m^t(E) + C_{n,\epsilon, t} \left( \sum_{j=1}^{n-1} m^t(F_{j+1}^{\leq a_j}) + m^t(F_j^{>a_j-\epsilon}) \right).
\]
\end{thm}

Informally, this says that if the mass of each $F_i$ is concentrated in the phase interval $(a_{i-1},a_i-\epsilon]$, then $m^t(E) \approx m^t(F_1)+\cdots+m^t(F_n)$.

\subsubsection{Smoothed phase and log-central charge}

Using the deformed mass function, we intro\-duce the \emph{smoothed phase} functions for $t \in (-\infty,\infty)$ as
\begin{equation}
\label{E:smoothedphase}
    \phi_\sigma^t(E):=\left\{ \begin{array}{ll} \frac{1}{m(E)} \sum_\theta \theta |Z(\gr_\theta(E))|, & \text{if } t=0 \\ \frac{1}{t} \log\left( \frac{m^t(E)}{m(E)} \right), & \text{if } t \neq 0. \end{array} \right. 
\end{equation}
Note that $\phi(E)$ is just the expectation value of the random variable $\theta$ with respect to the finite probability measure $\d m_E(\theta) / m(E)$. Several useful properties of $\phi_\sigma^t$ are established in \cite{massmeasures}*{Prop. 2.5}. For example, if $E$ and $\sigma$ are fixed, then $\phi^t(E)$ is an analytic function of $t$ such that 
\[
    \lim_{t\to \pm\infty} \phi^t(E) = \phi^\pm(E),
\]
where $\phi^+(E) := \max \{\theta : \gr^\sigma_\theta(E) \neq 0\}$ and $\phi^-(E) := \min \{\theta : \gr^\sigma_\theta(E) \neq 0\}$ are the max and min phase functions introduced in \cite{Br07}. In this sense, $\phi^t(E)$ is a smooth interpolation between $\phi^+(E)$ and $\phi^-(E)$. It is also shown that $\phi^t(E)$ is strictly monotone increasing in $t$ if $E$ is not semistable, and it is constant in $t$ if $E$ is semistable. 

For any $t\in \bR$, the \emph{log central charge} functions on $\cC$ are
\begin{equation}
\label{E:logcentralchargefunction}
    \ell_\sigma^t(E):= \log(m_\sigma(E)) + i \pi \phi^t_\sigma(E) \in \bC.
\end{equation}
Note that $\ell^t_\sigma(E) = \logZ_\sigma(E)$ for any $t$ when $E$ is semistable, so $\ell^t_\sigma$ is an extension to all of $\Stab(\cC)$ of the natural log central charge function defined on the subset of $\Stab(\cC)$ where $E$ is semistable. 


\begin{prop} \label{P:logZ_functions}
A pre-stability condition $\sigma$ such that the categories $\cP(\phi)$ are finite length for all $\phi \in \bR$ is uniquely determined by the function $\ell_\sigma : \cC \setminus 0 \to \bC$.
\end{prop}

\begin{proof}
    See the proof of \cite{massmeasures}*{Thm 3.6}. 
\end{proof}

When $\sigma$ satisfies the support property as in \Cref{D:stabilityconditions}, the categories $\cP(\phi)$ are auto\-matically finite length. Let $\mathfrak{c}$ denote the set of isomorphism classes of non-zero objects of $\cC$.\footnote{Here, we tacitly assume that the category $\cC$ is essentially small, so that $\mathfrak{c}$ is indeed a set.} \Cref{P:logZ_functions} says that $\ell$ embeds $\Stab(\cC)$ into the set of functions $\Fun(\mathfrak{c},\bC) = \bC^{\mathfrak{c}}$. The map $\ell$ is $\bC$-equivariant, where $\bC$ acts on $\bC^{\mathfrak{c}}$ by componentwise translation. Furthermore, if one equips $\bC^{\mathfrak{c}}$ with the product topology, then one can show that $\ell$ is continuous --- see \cite{massmeasures}*{Thm. 3.6}. In particular, the average phase function 
\[
    \phi_\sigma(E) = \frac{1}{m_\sigma(E)} \int_{\bR} \theta \:\d m_{\sigma,E}(\theta)
\]
depends continuously on $\sigma$. So, we have a continuous injection $\Stab(\cC)/\bC\hookrightarrow \bC^{\mathfrak{c}}/\bC$. In general, Bridgeland's topology on $\Stab(\cC)/\bC$ is finer than the subspace topology induced by the embedding, so we refer to the latter topology as the \emph{weak topology}. 

The reader is invited to consult \cite{massmeasures} for a discussion of embeddings of stability conditions into spaces of measures. 

\subsubsection{\texorpdfstring{$\bC$}{C}-invariant functions}
\label{S:C_invariant_functions}
The functions introduced above are not invariant for the action of $\bC$ on $\Stab(\cC)$. For any $z \in \bC$, we have the identities
\begin{gather*}
m^t_{z \cdot \sigma}(E) = |e^{(1-\frac{it}{\pi}) z}| m^t_\sigma(E), \\
\phi^t_{z \cdot \sigma}(E) = \phi^t_\sigma(E) + \tfrac{1}{\pi}\Im(z), \\
\ell^t_{z \cdot \sigma}(E) = \ell^t_\sigma(E) + z.
\end{gather*}
This implies that for any nonzero $E,F \in \cC$, the following functions are well-defined on $\Stab(\cC)/\bC$:
\begin{gather*}
    m^t_\sigma(E/F) :=  m^t_\sigma(E) / |e^{(1-\frac{it}{\pi}) \ell_\sigma(F)}|, \\
    \phi^t_\sigma(E/F) := \phi^t_\sigma(E) - \phi_\sigma(F), \\
    \ell^t_\sigma(E/F) := \ell^t_\sigma(E) - \ell_\sigma(F).
\end{gather*}
Finally, for any $t \in \bR$ we consider the function on $\Stab(\cC)/\bC$
\begin{equation} \label{E:cosh_expression}
\begin{array}{rl} 
    c^t_\sigma(E) & = \frac{1}{m(E)} \int \cosh t(\theta-\phi(E)) \d m_E \\
        &= \frac{1}{2}\left(m^t_\sigma(E/E) + m^{-t}_\sigma (E/E) \right).
\end{array}
\end{equation}
The quantity $c^t_\sigma(E)$ measures how concentrated the probability measure $\d m_E/m(E)$ is around its mean. Since $\cosh(\theta) \geq 1$ for all $\theta \in \bf{R}$, $1 \leq c^t_\sigma(E)$ for any $E \in \cC$, with equality if and only if $E$ is $\sigma$-semistable.

\subsection{Boundedness of functors}

Let $\cC$ and $\cD$ be stable dg-categories, equipped with pre\-stability conditions $\sigma$ and $\tau$, respectively.
\begin{defn}[Boundedness conditions] \label{D:boundedness}
We say that an exact functor $\Phi : \cC \to \cD$ is \emph{$t$-bounded} if there is a constant $n>0$ such that $\Phi(\cC^{(0,1]}) \subset \cD^{(-n,n]}$. We say that $\Phi$ is \emph{mass-bounded} if there is a constant $C>0$ such that $m_{\tau}(\Phi(E)) \leq C m_{\sigma}(E)$ for all $E \in \cC$. If both of these conditions hold, we simply say $\Phi$ is \emph{bounded}.
\end{defn}

We observe that if $\sigma'$ and $\tau'$ are two pre-stability conditions such that $d(\sigma,\sigma')<\infty$ and $d(\tau,\tau')<\infty$, then $\Phi$ is ($t$-/mass-)bounded with respect to $\sigma$ and $\tau$ if and only if it is so for $\sigma'$ and $\tau'$. In particular, if $\sigma$ and $\tau$ satisfy the support property, these notions of boundedness only depend on the connected component of $\sigma$ and $\tau$ in $\Stab(\cC)$ and $\Stab(\cD)$, respectively. Another useful observation is that the ($t$-/mass-)bounded functors form a thick stable subcategory of $\Fun^{\rm{ex}}(\cC,\cD)$, the stable category of exact functors $\cC \to \cD$.\endnote{If $\Phi_1 \to \Phi_2 \to \Phi_3$ is an exact triangle of functors, then for any $E \in \cC$, the triangle inequality for mass implies that for any $E \in \cC$, $m_\tau(\Phi_2(E)) \leq m_\tau(\Phi_1(E)) + m_\tau(\Phi_3(E))$. Therefore, if $\Phi_1$ and $\Phi_2$ are mass-bounded, then so is $\Phi_2$, because $m_\tau(\Phi_2(E)) \leq (C_1 + C_2) m_\sigma(E)$. The same claim for $t$-boundedness follows from the long-exact homology sequence for $\Phi_1(E) \to \Phi_2(E) \to \Phi_3(E)$ for any $E \in \cC^{(0,1]}$. This shows that ($t$-/mass-)bounded functors are closed under extensions. They are closed under homological shifts because $m_\tau$ is invariant under homological shifts and $\Phi$ is exact.

Next if $\Phi_1$ is a direct summand of $\Phi_2$, then for any $E \in \cC$, $\Phi_1(E)$ is a direct summand of $\Phi_2(E)$. If there is a constant $C>0$ such that $m_\tau(\Phi_2(E)) \leq C m_\sigma(E)$, then the same bound holds for $m_\tau(\Phi_1(E))$, because $m_\tau(\Phi_1(E)) \leq m_\tau(\Phi_2(E))$. Likewise, any homological bounds on $\Phi_2(E)$ also apply to $\Phi_1(E)$.}

\begin{lem}\label{L:simplify_boundedness}
    To verify $t$- or mass-boundedness, it suffices to verify respectively that 
    \[
        |\phi^\pm_\tau(\Phi(E))-\phi_\sigma(E)| \leq n \quad \text{or} \quad m_\tau(\Phi(E)) \leq C m_\sigma(E)
    \]
    for all \emph{stable} $E \in \cC^{(0,1]}$.
\end{lem}
\begin{proof}
    The claim about $t$-boundedness follows from the long exact homology sequence for extensions and the fact that every $E \in \cC^{(0,1]}$ can be constructed by a finite sequence of extensions from stable objects. For the mass-boundedness claim, let $G_i$ be the associated graded pieces of the Jordan-H\"{o}lder filtration of some fixed $E \in \cC$.\endnote{By this, we mean the refinement of the Harder--Narasimhan filtration obtained by choosing a filtration of every Harder--Narasimhan subquotient with stable subquotients.} Then $\Phi(E)$ admits a filtration with graded pieces $\Phi(G_i)$, so the triangle inequality gives 
    \[ 
        m_\tau(\Phi(E)) \leq \sum_i m_\tau(\Phi(G_i)) \leq C \sum_i m_\sigma(G_i).
    \]
    On the other hand, $m_\sigma(E) = \sum_i m_\sigma(G_i)$.
\end{proof}

We can sometimes deduce mass-boundedness from the following apparently simpler con\-dition.

\begin{defn}[mass-Hom bound] \label{D:mass_bound}
Let $\cC$ be a stable dg-category over a field $k$, equipped with a stability condition $\sigma$. We say that $\sigma$ has a \emph{mass-Hom bound} if for every $E \in \cC$, there is a constant $C_E$ such that
\begin{equation} \label{E:mass_bound}
\dim_k \Hom(E,F) \leq C_E\cdot m_\sigma(F), \qquad \forall \:F \in \cC.    
\end{equation}
\end{defn}

Inequalities of this type have already been considered in the literature, see for example \cites{pseudoanosov,ikeda_2021,DynamicsDHKK}.

\begin{lem}\label{L:simplify_mass_hom_bound}
The following variants of \Cref{D:mass_bound} are equivalent to a stability condition having a mass-Hom bound:\vspace{-2mm}
\begin{enumerate}
    \item It suffices to verify \eqref{E:mass_bound} for $E$ in a classical generating set.\endnote{By this we mean a set of objects that generates $\cC$ under the formation of shifts, extensions, and direct summands.}\vspace{-2mm}
    \item It suffices to verify \eqref{E:mass_bound} for stable $F$.\vspace{-2mm}
\end{enumerate}
\end{lem}

\begin{proof}
(1) follows from the observation that the objects $E \in \cC$ that admit a bound of the form \eqref{E:mass_bound} span a thick stable subcategory of $\cC$. The proof of (2) is analogous to that of \Cref{L:simplify_boundedness}.
\end{proof}

We will see in \Cref{P:manifold_with_corners_implies_boundedness} that \Cref{conj:manifold_with_corners_simplified} implies the following:

\begin{conj}[Boundedness conjecture]\label{conj:boundedness}
If $\cC$ is smooth and proper over a field $k$ then any stability condition admits a mass-Hom bound.
\end{conj}

\begin{ex}[Algebraic stability conditions]
    A stability condition $\sigma = (Z,\cP)$ on $\cC$ is called \emph{algebraic} \cite{ikeda_2021}*{p. 138} if the underlying heart $\cA = \cP(0,1]$ is of finite length with finitely many simple objects, $S_1,\ldots, S_k$. By \cite{ikeda_2021}*{p. 151}, there is a stability condition $\sigma_0$ obtained from $\sigma$ by deformation of $Z$ such that for any object $F$ one has $m_{\sigma_0}(F) = \sum_{i=1}^k \lvert n_i\rvert$ where $[F] = \sum_{i=1}^k n_i \cdot [S_i]$ in $ \rm{K}_0(\cC)$ with $n_i \in \bZ$. By \Cref{C:metric_mass_hom_bound} we may assume $\sigma_0 = \sigma$. Then, one verifies that $\dim_k \Hom(G,F) \le U\cdot m_\sigma(F)$ and we are done.\endnote{More precisely, we prove the claim inductively by taking an object $F$ in $\cA$ (the general case of $F\in \cA[n]$ being identical) and considering its Jordan-H\"older filtration. Suppose the result has been proven for objects $F$ with length at most $\ell-1$ Jordan-H\"older filtration and consider an object $F$ with length $\ell$ Jordan-H\"older filtration. We have a short exact sequence $0\to E\to F \to Q\to 0$ where $E$ is a subobject with length $\ell-1$ Jordan-H\"older filtration and $Q = S_j^{\oplus d}$ for some $d\ge 1$. Next, we have a pair of maps, exact in the middle:
    \[
        \Hom(G,E)\xrightarrow{\alpha} \Hom(G,F)\xrightarrow{\beta} \Hom(G,S_j^{\oplus d})
    \]
    and it follows that $\dim \Hom(G,F) = \dim \im(\alpha) + \dim \im(\beta) \le \dim \Hom(G,F) + \dim\Hom(G,S_j^{\oplus d})$ and the result now follows by induction.}
    

    Examples of algebraic stability conditions include stability conditions with heart the abelian category of finite modules over a finite dimensional algebra of finite global dimension. As a consequence, any smooth and proper variety $X$ admitting a tilting bundle, such as a rational surface \cite{Tiltingrationalsurfaces}, has a connected component of $\Stab(X)$ consisting of stability conditions satisfying \eqref{E:mass_bound}.
\end{ex}

\begin{ex}[Curves]
    Let $X$ denote a smooth genus $g$ projective curve over a field with ample line bundle $\cL$ and let $\sigma\in \Stab(X)$ denote slope stability as in \cite{Br07}*{Ex. 5.4}. By \cite{Orlovgeneration}, $G = \cO_X \oplus \cL$ is a strong, hence classical, generator of $\DCoh(X)$. One can use Le Potier's bounds \cite{LePotier}*{Lem. 7.1.2} to prove that for any $\cF$ semistable vector bundle one has\endnote{First, given a semistable locally free sheaf $F$ on $X$ of slope $\mu$, we have $\dim_k \Hom(\cO_X,F) = h^0(X,F) +h^1(X,F) = h^0(X,F) +h^0(X,F^\vee \otimes \omega_X)$ using Serre Duality. Applying Le Potier's bounds, one has 
    \[
    \dim_k \Hom(\cO_X,F) \le \deg F + \deg(F^\vee\otimes \omega_X) + 2r = 2gr
    \]
    where $g$ is the genus of $X$ and $r$ the rank of $F$. Since the central charge of the Bridgeland stability condition $\sigma$ corresponding to slope stability on $X$ is given by $Z(F) = -\deg(F) + i\rk(F)$, we have $r\le m_\sigma(F)$ and $\dim \Hom(\cO_X,F)\le 2gm_\sigma(F)$ for all locally free $F$. Now, 
    \begin{align*}
    \dim \Hom(\cL^{k-1} \oplus \cL^k,\cF) &= \dim \Hom(\cO_X,F') + \dim \Hom(\cO_X,F'')\\
    &\le 2g(m_\sigma(F') + m_\sigma(F''))
    \end{align*}
    where $F' = F \otimes \cL^{-k+1}$ and $F'' = F\otimes \cL^{-k}$.
    Write $d = \deg F$. One computes $\deg(F') = d+r(1-k)\deg(\cL)$ and $\deg(F'') = d - kr\deg(\cL)$.
    Using this, we compute 
    \begin{align*}
        m_\sigma(F') + m_\sigma(F'') & = \sqrt{(d+r(1-k)\deg(\cL))^2 + n^2} + \sqrt{(d-nk\deg(\cL))^2+ n^2}\\
        & \le 2m_\sigma(F) + 2\sqrt{\lvert 2drk\deg(\cL)\rvert} + 2rk\deg \cL\\
        & \le 2m_\sigma(F) + 4\sqrt{k\deg \cL}\cdot m_\sigma(F) + 2k\deg \cL\cdot  m_\sigma(F)\\
        & \le C_Gm_\sigma(F)
    \end{align*}
    where $C_G$ is a constant depending only on the choice of classical generator $G$.}
    \[
        \dim_k \Hom(G,\cF) \le 4g \left(1+\sqrt{\deg(\cL)} \right)^2 m_\sigma(F).
    \]
    Furthermore, for any torsion sheaf $\cT$ one has $\dim_k(G,\cT) = 2\cdot m_\sigma(\cT)$. This implies the result for slope stability and since it is known that $\Stab(X)$ is connected in this case, the result follows for all stability conditions on curves. 
\end{ex}

\begin{ex}[Gluing]
    Consider a semiorthogonal decomposition $\cC = \langle \cC_1,\ldots, \cC_n\rangle$ of a smooth and proper stable dg-category $\cC$ equipped with stability conditions $\sigma_j \in \Stab(\cC_j)$ satisfying \eqref{E:mass_bound} for all $1\le j \le n$. For $t\in \bR$, define $(Z_{j,t},\cP_{j,t}) = \sigma_{j,t} := \exp(\sqrt{-1} jt) \cdot \sigma_j$. By \cite{quasiconvergence}*{Lem. 3.5}, there exists $t_0\in \bR$ such that for all $t\ge t_0$ one can glue $(\sigma_{1,t},\ldots, \sigma_{n,t})$ to obtain stability conditions $\sigma_t = (Z_t,\cP_t)$ on $\cC$ with the property that $\cP_t(\phi) = \bigoplus_{j=1}^n \cP_{j,t}(\phi)$ for all $\phi \in \bR$. 

    Let $t\ge t_0$ be given. We will show that $\sigma_t$ satisfies \eqref{E:mass_bound}. By induction, we may assume $n=2$. Denote by $i:\cC_1\to \cC$ the inclusion functor and by $j\dashv i$ its left adjoint. Let $X_1\in \cP_{1,t}(\phi)$ be given. Given a classical generator $G$ for $\cC$,\endnote{Toen and Vaqui\'{e} show that any smooth and proper category over a commutative ring is homotopically finitely presented, and in particular has a classical generator.} $j(G)$ is a classical generator of $\cC_1$ and by hypothesis, $\dim_k \Hom(G,X_1) = \dim_k \Hom(j(G),X_1) \le C_1(t)\cdot m_{\sigma_{1,t}}(X_1)$. Mutating, we obtain a similar inequality for any $X_2\in \cP_{2,t}(\phi)$ for some $C_2(t) > 0$ depending only on $G$ and $t$. So, for any $X = X_1\oplus X_2\in \cP(\phi)$ we have $\dim_k \Hom(G,X) \le \max\{C_1(t),C_2(t)\} m_{\sigma,t}(X)$. 

    As a very special case, any smooth and proper variety $X$ whose derived category admits a full exceptional collection has a connected component of $\Stab(X)$ whose elements satisfy the mass-Hom bound.
\end{ex}

\begin{ex}[Induction]
    Consider an exact functor $\Phi:\cC\to \cD$ of $k$-linear trian\-gulated categories. In some circumstances it is possible to induce a stability condition $\sigma$ on $\cC$ from a given $\sigma' \in \Stab(\cD)$; see \cite{MMSinducing}*{Thm. 2.14}. Suppose $\sigma'$ satisfies \eqref{E:mass_bound}. If $\Phi$ is faithful then for any $X,Y\in \cC$ with $Y$ semistable with respect to $\sigma$, one has 
    \[
        \dim_k \Hom(X,Y) \le \dim_k \Hom(\Phi(X),\Phi(Y)) \le C_{\Phi(X)}m_{\sigma'}(\Phi(Y)).
    \]
    However, $m_{\sigma'}(\Phi(Y)) = m_\sigma(Y)$ and thus setting $C_{\Phi(X)} =: C_X$, it follows that $\sigma$ satisfies \eqref{E:mass_bound}.
\end{ex}


\begin{prop}\label{P:mass_hom_implies_bounded_functors}
Let $\cC$ be a smooth, proper, and idempotent complete stable dg-category over a field $k$. A pre-stability condition on $\cC$ has a mass-Hom bound if and only if for any $E \in \cC$, the functor $\RHom(E,-) : \cC \to \DCoh(k)$ is mass-bounded for the standard stability condition on $\DCoh(k)$. Furthermore in this case:
\begin{enumerate}
    \item Any exact $k$-linear functor $\Phi : \cC \to \cD$ to an idempotent complete stable $k$-linear dg-category $\cD$ with a pre-stability condition is bounded; and
    \item There is a constant $C>0$ such that $\dim \left(\bigoplus_i \Ext^i(E,F)\right) \leq C m_\sigma(E) m_\sigma(F)$ for all $E,F \in \cC$.
\end{enumerate}
\end{prop}
\begin{proof}
    The functor $\RHom(E,-)$ is mass-bounded if and only if there is some constant $c_E$ such that $\dim H^\ast(\RHom(E,F)) \leq c_E m_\sigma(F)$, because the mass of any $C^\bullet \in \DCoh(k)$ is proportional $\sum_i \dim H^i(C^\bullet)$. By \Cref{L:simplify_boundedness}, it suffices to check this for all $F \in \cC^{(0,1]}$. The functor $\RHom(E,-)$ has bounded $t$-amplitude by \cite{quasiconvergence}*{Prop. 3.1}, so that there exists some $N>0$ such that $\dim H^\ast(\RHom(E,F)) = \sum_{i=-N}^N \dim \Hom(E,F[i])$ for all $F \in \cC^{(0,1]}$. This finite sum is bounded above by a constant times $m_\sigma(F)$ if and only if the same is true for each term, which is the statement of the mass-Hom bound \eqref{E:mass_bound}.
    
    Now let $\Phi : \cC \to \cD$ be a functor as in (1), and let $\tau$ be a pre-stability condition on $\cD$. The fact that $\Phi$ has bounded $t$-amplitude is proved in \cite{quasiconvergence}*{Prop. 3.1}, and a similar proof gives mass-boundedness as well: $\cC$ admits a classical generator $G$, and because $\cC$ is smooth and proper, the identity functor $\cC \to \cC$ lies in the smallest thick stable subcategory of $\Fun^{\rm{ex}}(\cC,\cC)$ containing the functor $\RHom(G,-) \otimes_k G$, and therefore $\Phi \cong \Phi \circ \id$ lies in the smallest thick stable subcategory of $\Fun(\cC,\cD)$ generated by $\RHom(G,-) \otimes_k \Phi(G)$. Choose a $C>0$ such that $\dim_k H^\ast(\RHom(G,E)) \leq C m_\sigma(E)$ for any $E \in \cC$. Because $\RHom(G,E)$ is isomorphic to its homology as a $k$-module, it follows that $m_\tau(\RHom(G,E) \otimes_k \Phi(G)) = m_\tau(\Phi(G)^{\oplus n}) = n m_\tau(\Phi(G))$ for some $n \leq C m_\sigma(E)$. It follows that the functor $\RHom(G,-) \otimes_k \Phi(G)$ is mass-bounded, with constant $C m_\tau(\Phi(G))$. So the claim holds because bounded functors form a thick stable subcategory of $\Fun^{\mathrm{ex}}(\cC,\cD)$.

    The claim (2) follows from the same argument. If $\cB \subset \Fun^{\mathrm{ex}}(\cC,\cC)$ is the full subcategory of functors $\Phi$ for which $\exists C>0$ such that $\dim (\Ext^\ast(\Phi(E),F)) \leq C m_\sigma(E) m_\sigma(F)$ for all $E,F \in \cC$, then $\cB$ is closed under shifts, cones, and retracts. $\cB$ contains $\RHom(G,-)\otimes G$, because $\dim \Ext^\ast(\RHom(G,E) \otimes G,F) = (\dim \Ext^\ast(G,E)) \cdot (\dim \Ext^\ast(G,F)) \leq c_G^2 m_\sigma(E) m_\sigma(F)$. It follows that $\id_{\cC} \in \cB$, which proves (2).
\end{proof}

\begin{cor}\label{C:metric_mass_hom_bound}
    Suppose $\cC$ is a stable dg-category over a field, $\sigma \in \Stab(\cC)$, and $\tau$ is a pre-stability condition on $\cC$.
    \begin{enumerate}
        \item If $d(\sigma,\tau)<\infty$, then $\tau$ has a mass-Hom bound if and only if $\sigma$ does.
        \item If $\cC$ is smooth, proper, and idempotent complete, and $\sigma$ and $\tau$ have a mass-Hom bound, then $d(\sigma,\tau)<\infty$.
    \end{enumerate}
\end{cor}
\begin{proof}
    The first claim follows from the fact that if $d(\sigma,\tau)<\infty$, then for some $r>0$, $m_\sigma(E) \leq r m_\tau(E)$ and $m_\tau(E) \leq r m_\sigma(E)$ for all $E \in \cC$. The second claim follows from \Cref{P:mass_hom_implies_bounded_functors} applied to the identity functor $(\cC,\sigma) \to (\cC,\tau)$ and $(\cC,\tau) \to (\cC,\sigma)$, which implies the existence of a constant $K>0$ such that $\phi_\sigma^+(E) \leq \phi_\tau^+(E)+K$, $\phi_\sigma^-(E) \geq \phi_\tau^-(E)-K$, and $e^{-K} m_\tau(E) \leq m_\sigma(E) \leq e^K m_\tau(E)$ for all $E \in \cC$. 
\end{proof}

\subsection{Sluicings and base change}

The question of base change of $t$-structures and slicings on derived categories of coherent sheaves has been studied in \cites{Polishchuk,AbramovichPolishchuk,families,halpernleistner2022structure}. We give a self-contained discussion here, extending these ideas to the more general context of an essentially small, idempotent complete stable dg-category $\cC$ over a commutative ring $k$. We believe that we have simplified some of the earlier treatments, especially with the following:

\begin{defn}[Sluicing]\label{D:sluicing}
    A \emph{sluicing} of $\cC$ of width $w$, with $0<w\leq 1$, consists of a collection of $t$-structures $(\cC_{>\varphi},\cC_{\leq \varphi})$ for $\varphi \in \bR$ such that:\vspace{-2mm}
    \begin{enumerate}
        \item $\forall a \in \bR$, $\cC_{>a} = \bigcup_{\varphi > a} \cC_{>\varphi}$. \vspace{-2mm}
        \item $\forall 0<b-a<w$, $\exists$ a bounded noetherian $t$-structure $(\cC_{\geq 0},\cC_{<0})$ with $\cC_{>a} \subseteq \cC_{\geq 0} \subseteq \cC_{>b-1}$; and \vspace{-2mm}
        \item Any $E \in \cC$ lies in $\cC_{>a} \cap \cC_{\leq b}$ for some $a<b$.
    \end{enumerate}
\end{defn}

\begin{rem} A sluicing is more restrictive than a slicing in that for intervals $I$ of width $< w$, $\cP(I) \subseteq \cC$ is always contained in the heart of a bounded \emph{noetherian} $t$-structure, but it is more flexible in that objects are not required to have HN filtrations.\end{rem}

\begin{figure}[h!]
    \centering
    \includegraphics[width=0.5\linewidth]{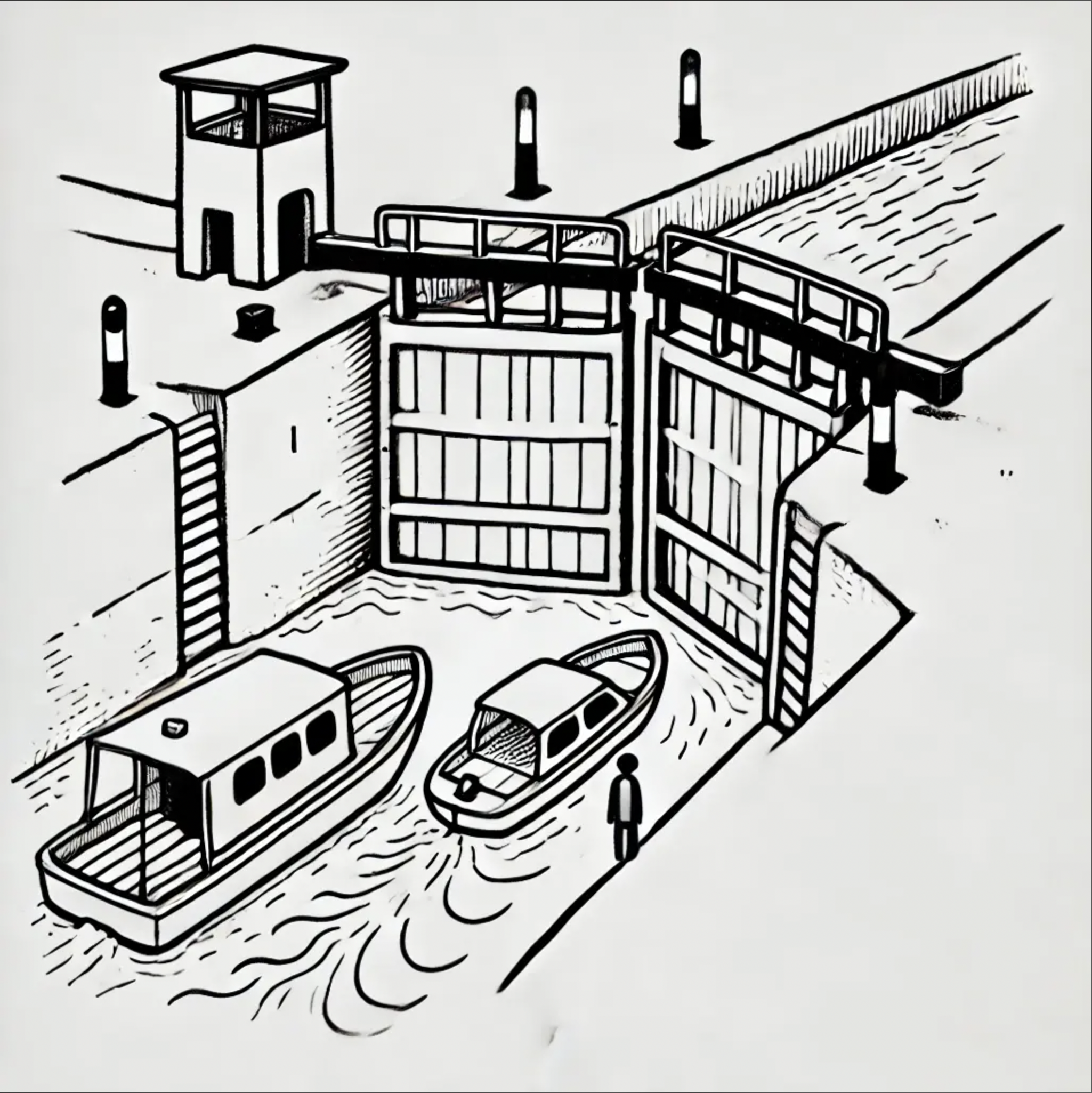}
    \caption{A sluice, rendered with ChatGPT.}
    \label{Fig:sluice}
\end{figure}

Given a sluicing on $\cC$, for any (possibly infinite) half-open interval $I = (a,b]$ we will write $\cP(I) = \cC_{>a} \cap \cC_{\leq b}$, and we will let $\tau_I = \tau_{>a} \circ \tau_{\leq b} \cong \tau_{\leq b} \circ \tau_{>a} : \cC \to \cC$ denote the projection functor with respect to the $t$-structures. The categories $\cP(I)$ determine the sluicing. We also define $\phi^-(E) = \sup\{a | E \in \cC_{>a}\}$ and $\phi^+(E) = \inf\{b | E \in \cC_{\leq b}\}$ --- note that condition (1) guarantees that $E \in \cC_{\leq \phi^+(E)}$, but the analogous claim for $\phi^-(E)$ does not hold. 

\begin{prop}\label{P:sluicing_vs_stability_condition}
    Fix a homomorphism $\ch :  \rm{K}_0(\cC) \to \Lambda$. Giving a stability condition on $\cC$ is equivalent to giving a sluicing $\cP$ and a homomorphism $Z : \Lambda \to \bC$ such that $\exists w \in (0,1)$ and $C>0$ such that:
    \begin{itemize}
    \item For any interval $I$ of width $w$ and any $0 \neq E \in \cP(I)$, $Z(E)$ lies in the cone $\bR_{>0} e^{i \pi I}$ and $|Z(E)| > C \lVert \ch(E) \rVert$.
    \end{itemize}
    An object $E$ is semistable of phase $\phi$ if $\phi = \phi^+(E) = \phi^-(E)$, i.e., $E \in \cC_{\leq \phi}$ and $\tau_{\leq \psi}(E) \cong 0$ for all $\psi<\phi$. 
\end{prop}

\begin{proof}
To show that a stability condition gives a sluicing, use Bridgeland's theorem to deform the stability condition to one for which $Z : \Lambda \to \bC$ factors through $\bQ[i]$. Then a dense set of the $t$-structures $(\cP(>\psi),\cP(\leq \psi))$ are noetherian, and the fact that this slicing is close to the original with respect to the metric on slicings shows that these noetherian $t$-structures satisfy \Cref{D:sluicing}(2) --- see \cite{Polishchuk}*{\S 1.2, Ex.}.

Let $I=(a-b,a+b]$ for $a \in \bR$ and $0<b<1/2$. For any $F \in \cP(I)$, one has $|Z(\ch(F))| \leq \Re(e^{-i\pi a} Z(\ch(F))) / \cos(\pi b)$. Because the right-hand side is additive, applying this to the associated graded pieces of the HN filtration implies that
    \[
    |Z(\ch(F))| \leq m(F) \leq \frac{\Re(e^{-i\pi a} Z(\ch(F)))}{\cos(\pi b)} \leq \frac{|Z(\ch(F))|}{\cos(\pi b)}.
    \]
The support property for the stability condition implies that $\exists C>0$ such that $|Z(\ch(F))| > C \lVert \ch(F) \rVert$ for all $F \in \cP(I)$.

Conversely, assuming the condition in the lemma, we must show that HN filtrations exist.\endnote{The condition that $Z(E) \in \bR_{>0} e^{i\pi \phi}$ for a semistable object of phase $\phi$ follows from the condition in the lemma, because $E \in \cP((\phi-w,\phi])$ and $E \in \cP((\psi, \psi +w])$ for any $\psi < \phi$. The support property for the resulting pre-stability conditions follows from the inequality $m(E) \geq Z(E) > C \lVert \ch(E) \rVert$ in the statement of the lemma.} For $E \in \cP(I)$, one can subdivide $I$ and concatenate HN filtrations of the projection of $E$ onto each of these subintervals to obtain an HN filtration of $E$. It therefore suffices to show the existence of an HN filtration when $I = (a-w/2,a+w/2]$ for some $a \in \bR$. The function $\Re(e^{-i\pi a}Z(\ch(-)))$ is additive on short exact sequences in the quasi-abelian category $\cP(I)$ and bounded below by $C \cos(\pi w/2) \inf\{\lVert v \rVert: 0 \neq v \in \Lambda\}$. It follows that $\cP(I)$ is finite length in that any ascending or descending sequence of strict monomorphisms must stabilize.

Now consider a sequence $\phi_1<\phi_2<\cdots$ converging to $\phi^+(E)$. This gives a descending chain of strict monomorphisms $\cdots \hookrightarrow \tau_{>\phi_2}(E) \hookrightarrow \tau_{>\phi_1}(E) \hookrightarrow E$, which must stabilize at some index $n$ because $\cP(I)$ is artinian. The resulting object $E_1 := \tau_{>\phi_n}(E) \hookrightarrow E$ is semistable of phase $\phi^+(E)$ by construction, and it is nonzero if $E$ is nonzero. Repeating this process for the quotient $\cofib(E_1 \to E) \in \cP((a-w/2,\phi_n])$ and iterating gives an ascending chain of strict monomorphisms $E_1 \hookrightarrow E_2 \hookrightarrow E$ whose quotients are semistable with deceasing phases. This chain must stabilize because $\cP(I)$ is noetherian, resulting in an HN filtration of $E$.
\end{proof}

For any $k$-algebra $R$ we let $\cC_R$ denote the compact objects in the tensor product of $k$-linear presentable stable dg-categories $\Dqc(R) \otimes_{\Dqc(k)} \Ind(\cC)$. The latter is compactly generated, so there is a canonical isomorphism with the ind-completion $\Ind(\cC_R) \cong \Dqc(R) \otimes_{\Dqc(k)} \Ind(\cC)$.

Given a sluicing on $\cC$ and any $\varphi \in \bR$, we let $\Ind(\cC_R)_{>\varphi} \subset \Ind(\cC_R)$ denote the smallest full subcategory that is closed under small colimits and extensions and contains $R \otimes_k E$ for all $E \in \cP(\varphi,\infty)$. By \cite{HigherAlgebra}*{Prop.~1.4.4.11}, $\Ind(\cC_R)_{>\varphi}$ is presentable, and is the connective part of an accessible $t$-structure $(\Ind(\cC_R)_{>\varphi},\Ind(\cC_R)_{\leq \varphi})$ on $\cC_R$, where accessible means it satisfies the equivalent conditions of \cite{HigherAlgebra}*{Prop.~1.4.4.13}. The fact that $\Ind(\cC_R)_{>\varphi}$ is generated by objects that are compact in $\Ind(\cC_R)$ guarantees that the truncation functors commute with filtered colimits.\endnote{Consider a filtered system $\{E_\alpha \in \Ind(\cC_R)\}_{\alpha \in I}$. The fact that the inclusion functor $\Ind(\cC_R)_{>\varphi} \hookrightarrow \Ind(\cC_R)$ is a left adjoint implies that this subcategory is closed under all colimits. Now consider the diagram
\[
\colim_{\alpha} \tau_{>\varphi}(E_\alpha) \to E = \colim_\alpha E_\alpha \to \colim_\alpha \tau_{\leq \varphi}(E_\alpha).
\]
The first object lies in $\Ind(\cC_R)_{>\varphi}$, so if the third object lies in $\Ind(\cC_R)_{\leq \varphi}$, then this triangle must agree with the truncation sequence $\tau_{>\varphi}(E) \to E \to \tau_{\leq \varphi}(E)$, i.e., the truncation functors commute with filtered colimits. This reduces the claim to showing that the subcategory $\Ind(\cC_R)_{\leq \varphi} \subset \Ind(\cC_R)$ is closed under filtered colimits. $F \in \Ind(\cC_R)$ lies in $\Ind(\cC_R)_{\leq \varphi}$ if and only if $\Hom(R \otimes_k E, F) = 0$ for all $E \in \cP(\varphi,\infty)$. This condition on $F$ is closed under filtered colimits because each of the objects $R\otimes_k E$ is compact.}

The pullback functor $R \otimes_k(-) : \Ind(\cC) \to \Ind(\cC_R)$ is always right exact and is left exact if $R$ is flat over $k$, and its right adjoint $\Ind(\cC_R) \to \Ind(\cC)$ is always exact.\endnote{The definition of the $t$-structures imply that pullback maps $\Ind(\cC)_{>\varphi}$ to $\Ind(\cC_R)_{>\varphi}$, which is the meaning of right exactness. This implies by adjunction that the pushforward functor is left exact. To show that pushforward is right-exact, note that $\Ind(\cC_R)_{> \varphi}$ is generated under filtered colimits and finite colimits by objects of the form $R \otimes_k E$ for $E \in \cP((\varphi,\infty))$, so it suffices to show that $R \otimes_k E \in \Ind(\cC)_{>\varphi}$ for any $E \in \cP(\varphi,\infty)$. This holds because as a $k$-module, $R$ admits a resolution by free $k$-modules in cohomological degree $\leq 0$. Now, because pushforward is exact and conservative, we know that an object $E \in \Ind(\cC_R)$ lies in the $>\varphi$ or $\leq \varphi$ subcategories if and only if its pushforward to $\Ind(\cC)_R$ does. If $R$ is flat over $k$, it is a filtered colimit of projective $k$-modules, and $M \otimes_k (-)$ is left $t$-exact for any projective $k$-module $M$, pullback maps $\Ind(\cC)_{\leq \varphi}$ to $\Ind(\cC_R)_{\leq \varphi}$ by the previous observation.} For any (possibly infinite) interval $I = (a,b]$, we define $\overline{\cP}_R(I) := \Ind(\cC_R)_{>a} \cap \Ind(\cC_R)_{\leq b}$, and define $\cP_R(I) = \overline{\cP}_R(I) \cap \cC_R$. If $b-a<1$, then $\overline{\cP}_R(I)$ is a quasi-abelian category (see \cite{Br07}*{Lem. 4.3}), where strict short exact sequences are exact triangles in $\Ind(\cC_R)$ in which all three objects lie in $\overline{\cP}_R(I)$.

\begin{ex}\label{EX:base_change_perf}
    If $k$ is a field and $\cC = \Perf(X)$ for some smooth $k$-scheme $X$, then $\Ind(\cC_R) \cong \Dqc(X_R)$ and  $\cC_R \cong \Perf(X_R)$. If one starts with the standard $t$-structure on $\Perf(X) \cong \DCoh(X)$, then the construction above gives the standard $t$-structure on $\Dqc(X_R)$ for any $k$-algebra $R$.
\end{ex}

Below it will be convenient to introduce the following notion:

\begin{defn}[Ind-noetherian $t$-structure]
    We call an accessible $t$-structure on a com\-pactly generated stable dg-category $\cA$ \emph{ind-noetherian} if its truncation functors preserve filtered colimits and preserve the subcategory of compact objects $\cA^\omega \subset \cA$, and the induced $t$-structure on $\cA^\omega$ is bounded and noetherian.
\end{defn}

For any ind-noetherian $t$-structure one has $\cA^\heart \cong \Ind((\cA^\omega)^\heart)$,\endnote{The fully faithful inclusion $(\cA^\omega)^\heart \subset \cA^\heart$ induces a morphism $\Ind((\cA^\omega)^\heart) \to \cA^\heart$, which is fully faithful because the objects in $(\cA^\omega)^\heart$ are compact in $\cA^\heart$. It is essentially surjective because every $E \in \cA$ is a filtered colimit of objects in $\cA^\omega$, and the truncation functors preserve filtered colimits because the $t$-structure is accessible.} so $\cA^\heart$ is a locally noetherian abelian category, meaning it has a small generating family of noetherian objects. Also, the category of compact objects in $\cA^\heart$ (also called objects of finite presentation) and the category of noetherian objects in $\cA^\heart$ both coincide with $\cA^\omega \cap \cA^\heart$.\endnote{In a locally noetherian abelian category, compact objects and noetherian objects coincide. This agrees with $\cA^\omega \cap \cA^\heart$ because this abelian subcategory consists of compact objects, generates $\cA^\heart$, and is closed under retracts, so it must agree with the category of compact objects in $\cA^\heart$.}
\begin{lem}\label{L:ind-noetherian_condition}
    An accessible $t$-structure on a compactly generated stable dg-category $\cA$ is ind-noetherian if and only if $\cA$ is generated by a subset $\cG \subset \cA^\heart \cap \cA^\omega$ such that for any $G \in \cG$ and any injective morphism $E \to G$ in $\cA^\heart$, $E \in \cA^\omega$ as well. In this case the condition holds with $\cG = \cA^\heart \cap \cA^\omega$.
\end{lem}
\begin{proof}
    If the $t$-structure is ind-noetherian, then it is generated by $\cA^\heart \cap \cA^\omega$ because every object in $\cA^\omega$ is an iterated extension of its homology objects. Now if $E \subset F$ in $\cA^\heart$ and $F \in \cA^\omega$, then $E$ is a compact object in the abelian category $\cA^\heart$, because $\cA^\heart$ is locally noetherian. Writing $E = \colim_\alpha E_\alpha$ as a filtered colimit with $E_\alpha \in \cA^\heart \cap \cA^\omega$, it follows that $E$ is a retract of some $E_\alpha$, hence $E \in \cA^\omega$.

    Conversely, assume the conditions of the lemma hold, and let $\bar{\cG} \subset \cA^\heart$ be the closure of $\cG$ under taking subobjects and quotients. Note that the hypotheses on $\cG$ imply $\bar{\cG} \subset \cA^\omega$. The full subcategory $\{E \in \cA | \bigoplus_i H_i(E) \in \bar{\cG}\}$ is a thick triangulated subcategory of compact objects that generates $\cA$. It follows that this is precisely the category $\cA^\omega$. The induced $t$-structure on $\cA^\omega$ must be noetherian because for any ascending chain $E_1 \subseteq E_2 \subseteq \cdots E$ in $(\cA^\omega)^\heart$, the hypotheses of the lemma imply that $\colim E_i \subseteq E$ is also compact, hence $E_i = \colim E_i$ for some $i$, i.e., the chain stabilizes.
\end{proof}

The $t$-structures on $\Ind(\cC_R)$ induced by a sluicing on $\cC$ do not necessarily preserve the category of compact objects $\cC_R \subset \Ind(\cC_R)$, as can be seen from \Cref{EX:base_change_perf} in the case where $X = \Spec(k)$ and $R$ is not regular. Nevertheless we have the following:

\begin{prop}\label{P:base_change_sluicing}
    Suppose that $k \to R$ is a composition of ring homomorphisms of the following types: 1) a polynomial algebra, 2) a localization, or 3) $R$ is perfect as a $k$-module. Then the $t$-structures $(\Ind(\cC_R)_{>\varphi},\Ind(\cC_R)_{\leq \varphi})$ preserve the subcategory $\cC_R$ of compact objects, and define a sluicing of $\cC_R$ of the same width as that on $\cC$.
\end{prop}

\begin{ex}\label{ex:sluicing_finite_type}
    If $k$ is a regular noetherian ring, then any essentially finite type $k$-algebra $k \to R$ satisfies the conditions of \Cref{P:base_change_sluicing}. Indeed such a morphism can be factored $k \to k[x_1,\ldots,x_n] \twoheadrightarrow S \to R$, where $S \to R$ is a localization. $S$ is a perfect $k[x_1,\ldots,x_n]$-module because $k[x_1,\ldots,x_n]$ is regular.
\end{ex}

\begin{lem} \label{L:sluicing_base_change_boundedness}
    Given a sluicing of $\cC$, if $k \to R$ is of finite Tor-amplitude, then any $E \in \cC_R$ lies in $\overline{\cP}_R(I)$ for some bounded interval $I \subset \bR$, i.e., $\cC_R = \bigcup_{I\text{ bounded}} \cP_R(I)$.
\end{lem}
\begin{proof}
    $E \in \Ind(\cC_R)$ lies in $\cP_R(I)$ if and only if its image in $\Ind(\cC)$ lies in $\overline{\cP}_k(I)$. Therefore for any $F \in \cP(I) \subset \cC$, the pullback $R \otimes_k F \in \cC_R$ lies in $\cP_R(J)$ for some finite $J$ if and only if the object $R \otimes_k F \in \overline{\cP}_k(J)$ for some finite $J$. Because $k \to R$ is of finite Tor-amplitude, $R$ has a finite resolution by flat $k$-modules. It therefore suffices to show that $M \otimes_k F \in \overline{\cP}_k(I)$ for any flat $k$-module $M$. This holds because any flat $k$-module $M$ is a filtered colimit of free modules.

    This combined with \Cref{D:sluicing}(3) shows that any object of the form $R \otimes_k F \in \cC_R$ for $F \in \cC$ lies in $\overline{\cP}_R(I)$ for a finite interval $I$. The latter condition is closed under shifts, extensions, and retracts. It therefore holds for any $E \in \cC_R$, because the objects $R \otimes_k F$ with $F \in \cC$ classically generate $\cC_R$.
\end{proof}

Next let $f^\ast : \cA \to \cB$ be a localization of $k$-linear compactly generated stable dg-categories, i.e., $f^\ast$ admits a right adjoint $f_\ast$ that commutes with filtered colimits and such that the counit $f^\ast f_\ast \to \id_\cB$ is an isomorphism of functors. 

\begin{lem}\label{L:localization}
    Suppose that $\cA$ admits an ind-noetherian $t$-structure such that $f_\ast \circ f^\ast$ is left $t$-exact. Then the $t$-structure induced on $\cB$ by defining $\cB_{>0}$ to be generated by $f^\ast(\cA_{>0})$ under colimits is also ind-noetherian, and $f^\ast : \cA^\omega \to \cB^\omega$ is essentially surjective.
\end{lem}
\begin{proof}
    Using the left $t$-exactness of $f_\ast \circ f^\ast$, one can show that $f^\ast$ is $t$-exact.\endnote{$f^\ast(\cA_{>0}) \subset \cB_{>0}$ by definition. $f^\ast(\cA_{\leq 0}) \subset \cB_{\leq 0}$, because $\cB_{>0}$ is generated under colimits by $f^\ast(E)$ for $E \in \cA_{>0}$, and for any $E \in \cA_{>0}$ and $F \in \cA_{\leq 0}$, $\Hom_\cB(f^\ast(E),f^\ast(F)) \cong \Hom(E,f_\ast(f^\ast(F))) \cong 0$ because $f_\ast(f^\ast(F)) \in \cA_{\leq 0}$.} Given $E \in \cB^\omega$, write $f_\ast(E) \cong \colim_\alpha F_\alpha$ as a filtered colimit of objects $F_\alpha \in \cA^\omega$. The isomorphism $f^\ast(f_\ast(E)) \cong E$ implies that $E$ is a retract of $f^\ast(F_\alpha)$ for some $\alpha$. Therefore $E \in \cB_{(m,n]}$ for some $m<n$, and exactness of $f^\ast$ implies that $H_i(E) \in \cB^{\omega}$ for all $i$.

    Next consider $E \in (\cB^{\omega})^\heart$, $F \in \cA^\omega_{>0}$, and a morphism $f^\ast(F) \to E$. Write $H_0(f_\ast(E)) = \bigcup_\alpha E_\alpha$ as a filtered union with $E_\alpha \in (\cA^\omega)^\heart$.\endnote{This is possible because $\cA^\heart$ is a locally noetherian category, where $\cA^\heart \cap \cA^\omega$ is both the category of compact objects and the category of noetherian objects.} Because $E$ is compact, the filtered union $E \cong \bigcup_\alpha f^\ast(E_\alpha)$ must stabilize, and $f^\ast(F) \to E$ factors through $f^\ast(E_\alpha)$ for some $\alpha$ because $F$ is compact. It follows that there is a morphism $F \to E_\alpha$ with $E_\alpha \in (\cA^\omega)^\heart$ that restricts to the given morphism $f^\ast(F) \to E$. Because any $E \in \cB^\omega$ can be constructed as an iterated extension of objects in $(\cB^\omega)^\heart$, it follows that $f^\ast : \cA^\omega \to \cB^\omega$ is essentially surjective.

    Finally, for any ascending chain $E_1 \subseteq E_2 \subset \cdots \subset E$ in $(\cB^\omega)^\heart$, choose a compact subobject $\tilde{E} \subseteq H_0(f_\ast(E))$ that becomes isomorphic to $E$ upon restriction, as in the previous paragraph. Then define $\tilde{E}_i := H_0(f_\ast(E_i)) \times_{H_0(f_\ast(E))} \tilde{E}$. These are subobjects of $\tilde{E}$ by construction, and they restrict to $E_i$ by the exactness of $f^\ast$. The ascending chain $\tilde{E}_1 \subseteq \tilde{E}_2 \subseteq \cdots$ must stabilize, so $(\cB^\omega)^\heart$ is noetherian.
\end{proof}

\begin{lem}[Hilbert basis theorem]\label{L:hilbert_basis}
    Suppose $R=k[t]$ and we are given a bounded noetherian $t$-structure $\cC = (\cC_{\geq 0},\cC_{<0})$. Then the induced $t$-structure on $\Ind(\cC_R)$ is ind-noetherian.
\end{lem}
\begin{proof}
    For an object $G \in \Ind(\cC)$, we denote $G_R := R \otimes_k G$. The objects $E_R$ for $E \in \cC^\heart$ generate $\Ind(\cC_R)$, so it suffices by \Cref{L:ind-noetherian_condition} to show that for any $E \in \cC^\heart$ and any subobject $F \subset E_R$ in $\Ind(\cC_R)^\heart$, we have $F \in \cC_R$. The following is a modification of the argument in \cite{ArtinZhang}*{B5.2}:

    Pushing forward to $\Ind(\cC)$, we have $E_R \cong E \otimes_k k[t] \cong \bigoplus_{i \geq 0} E \cdot t^i$. Let
    \[
    (E_R)_{\leq n} = \bigoplus_{i=0}^n E \cdot t^n \subset E_R \quad \text{and} \quad F_{\leq n} = E_{\leq n} \cap F
    \]
    in $\Ind(\cC)^\heart$. Note that $F = \colim F_{\leq n}$ because the same is true for $E_R$. Letting $F_n := F_{\leq n}/ F_{\leq n-1}$, we have injective morphisms $F_n \hookrightarrow (E_R)_{\leq n}/(E_R)_{\leq n-1} \cong E \cdot t^n$. Furthermore, because multiplication by $t$ maps $F_{\leq n} \to F_{\leq n+1}$, we get an ascending chain $F_0 \hookrightarrow F_1 \hookrightarrow \cdots \hookrightarrow E$. Because $E$ is noetherian, this chain stabilizes to some subobject $\bar{F} \subset E$.

    Let $Q := \coker(\bar{F} \to E)$, and consider the composition of morphisms $F \to E_R \to Q_R$. For any $n$, the composition $F_{\leq n} \to (E_R)_{\leq n} \to (Q_R)_{\leq n} \to Q \cdot t^n$ vanishes, so the image of $F_{\leq n}$ in $Q_R$ lies in $(Q_R)_{\leq n-1}$. But the preimage of $(Q_R)_{\leq n-1} \subset (Q_R)_{\leq n}$ in $F_{\leq n}$ is $F_{\leq n-1}$ by definition, so the image of $F_{\leq n} \to (Q_R)_{\leq n}$ is the same as the image of $F_{\leq n-1} \to (Q_R)_{\leq n-1} \hookrightarrow (Q_R)_{\leq n}$. By induction this implies that  $F_{\leq n} \to Q_R$ is the zero morphism for all $n$, and hence $F \to Q_R$ is the zero morphism. Therefore, $F \hookrightarrow E_R$ factors through $\bar{F}_R \hookrightarrow E_R$. Now $\bar{F} \in \cC$ and $F \subset \bar{F}_R$, so it suffices to replace $E$ with $\bar{F}$. Therefore we may assume for the remainder of the proof that the injection $F_n \hookrightarrow E$ is an isomorphism for $n \gg 0$. In other words, we assume one has $t^n \cdot E_R \hookrightarrow F \hookrightarrow E_R$ for a sufficiently large $n$.

    Under this new assumption, the surjection $E_R \twoheadrightarrow A := E_R /F$ factors through a surjection $E_R / (t^n \cdot E_R) \twoheadrightarrow A$. In particular the image of $A$ in $\Ind(\cC)$ actually lies in $\cC$. Now let $K_i := \ker(t^i : A \to A)$ for $i=0,\ldots,n$ and let $G_i = K_i/K_{i-1}$ for $i\geq 1$. $K_n = Q$, so this gives a finite filtration of $A$ in $\Ind(\cC_R)^\heart$. The image of each $G_i$ under the forgetful functor $\Ind(\cC_R) \to \Ind(\cC)$ lies $\cC$ because it is a subquotient of $A$ in $\Ind(\cC)^\heart$. On the other hand, because $t$ acts trivially on $G_i$, we have $G_i = \cofib( t : (G_i)_R \to (G_i)_R) \in \cC_R$. It follows that $A \in \cC_R$, and hence $F = \fib(E_R \to A) \in \cC_R$.
\end{proof}

\begin{proof}[Proof of \Cref{P:base_change_sluicing}]
    \Cref{D:sluicing}(3) follows from \Cref{L:sluicing_base_change_boundedness}. Also, the contain\-ment $(\cC_R)_{>\phi_1} \subseteq (\cC_R)_{>\phi_2}$ for $\phi_1>\phi_2$ is immediate from the definition. Once we have shown that the truncation functors $\tau_I$ on $\Ind(\cC_R)$ preserve $\cC_R$, condition (1) is equivalent to the claim that $\tau_{>a}(E) \cong \colim_{\varphi>a} \tau_{>\varphi}(E)$ for every $E \in \Ind(\cC_R)$, because this filtered colimit must stabilize if $E$ is compact. The identity $\tau_{>a}(E) \cong \colim_{\varphi>a} \tau_{>\varphi}(E)$ holds automatically, though, because it holds in $\Ind(\cC)$ and the forgetful functor $\Ind(\cC_R) \to \Ind(\cC)$ is $t$-exact, conservative, and commutes with filtered colimits. It therefore suffices to show that the truncation functors $\tau_I$ preserve $\cC_R$, and verify the condition \Cref{D:sluicing}(2).

    \smallskip
    \noindent{\textit{Case (1): $R = k[t]$}}
    \smallskip

    For any interval $I$ of width $w<1$, $\cP(I) \subset \cC^\heart$ for some bounded noetherian $t$-structure on $\cC$. Then for any $E \in \cP(I)$, any strict subobject $F \hookrightarrow R \otimes_k E$ in $\overline{\cP}_R(I)$ is also a subobject with respect to the $t$-structure on $\Ind(\cC_R)$ induced by $\cC^\heart$. \Cref{L:hilbert_basis} and \Cref{L:ind-noetherian_condition} then imply that $F \in \cC_R$.
    
    Now consider an arbitrary $E \in \cC_R$ --- we may suppose $E \in \overline{\cP}_R((a,b])$ for some $a<b$ --- and choose an $\epsilon < w$. The object $\tau_{\leq a+\epsilon}(E) \in \overline{\cP}_R((a,a+\epsilon])$ is compact, because $\tau_{\leq a+\epsilon}$ is left-adjoint to the inclusion functor, which preserves filtered colimits. Let $F$ denote the image of $\tau_{\leq a+\epsilon}(E)$ via the forgetful functor $\Ind(\cC_R) \to \Ind(\cC)$. Then the counit of adjunction gives a strict epimorphism $R \otimes_k F \to \tau_{\leq a+\epsilon}(E)$ in the quasi-abelian category $\overline{\cP}_R(I)$.\endnote{After pushing forward to $\Ind(\cC)$ it is split, which shows that $\fib(R \otimes_k F \to \tau_{\leq a+\epsilon}(E)) \in \overline{\cP}_R(I)$.} Because $\tau_I$ preserves $\cC_R$, we may write $F = \colim_\alpha F_\alpha$ as a filtered colimit with $F_\alpha \in \cP(I)$. 
    Let $K_\alpha = \ker(R \otimes_k F_\alpha \to \tau_{\leq a+\epsilon}(E)) \in \overline{\cP}_R(I)$, and let $Q_\alpha = \coker(K_\alpha \to R \otimes_k F)$. Then because the formation of cokernels and kernels commutes with filtered colimits in $\overline{\cP}_R(I)$ and $R \otimes_k R \to \tau_{\leq a+\epsilon}(E)$ is a strict epimorphism\endnote{I am using that if $K = \ker(R \otimes_k F \to \tau_{\leq a+\epsilon}(E))$, then because the original morphism was a strict epimorphism $\tau_{\leq a+\epsilon}(E) \cong \coker(K \to R \otimes_k F)$.}, we have $\tau_{\leq a+\epsilon}(E) \cong \colim_\alpha Q_\alpha$. Because $\tau_{\leq a+\epsilon}(E)$ is compact, there exists some $\alpha$ such that $\tau_{\leq a+\epsilon}(E)$, is a retract of $Q_\alpha$. $K_\alpha$ is a strict subobject of $R \otimes_k F_\alpha$ and hence also lies in $\cC_R$ as observed above, and therefore $Q_\alpha = \cofib(K_\alpha \to R \otimes_k F_\alpha) \in \cC_R$ as well. It follows that $\tau_{\leq a+\epsilon}(E) \in \cC_R$, and hence $\tau_{>a+\epsilon}(E) \in \cC_R$.

    Iterating this argument shows that $\tau_{\leq a + n\epsilon}(E)$ and $\tau_{>a+n\epsilon}(E)$ lie in $\cC_R$ for all $n \in \bZ$. Because $\epsilon$ can be chosen arbitrarily close to $0$, this shows that $\tau_{\leq \varphi}(E)$ and $\tau_{>\varphi}(E)$ lie in $\cC_R$ for all $\varphi \in \bR$. This verifies that the $t$-structures on $\Ind(\cC_R)$ preserve $\cC_R$, and \Cref{L:hilbert_basis} provides the necessary noetherian $t$-structures $(\cC_R)_{>a} \subset (\cC_R)_{>0} \subset (\cC_R)_{>b}$ to verify \Cref{D:sluicing}(2) using the corresponding $t$-structures on $\cC$.

    \smallskip
    \noindent{\textit{Case (2): $k \to R$ is a localization}}
    \smallskip

    Because $k \to R$ is flat, $f_\ast(f^\ast(-)) \cong R \otimes_k (-)$ is exact, so \Cref{L:localization} implies that any noetherian $t$-structure on $\cC$ induces an ind-noetherian $t$-structure on $\Ind(\cC_R)$. The proof of Case (1) now applies verbatim, using the fact that $k \to R$ is flat.

    \smallskip
    \noindent{\textit{Case (3): $R$ a finite perfect $k$-algebra}}
    \smallskip

    The object $R \otimes_k R$ generates $\Dqc(R)$, because any perfect complex with full support generates $\Dqc(R)$ by the Hopkins--Neeman theorem \cite{NeemanHopkins}.\endnote{This is a result of the Hopkins-Neeman theorem, because localizing subcategories of $\Dqc(R)$ are in bijection with specialization-closed subsets of $\Spec(R)$, so the only localizing subcategory that contains an object supported everywhere is $\Dqc(R)$ itself.} This implies that $R$ itself lies in the thick triang\-ulated closure of $R \otimes_k R$. Therefore, for any $F \in \Ind(\cC_R)$, $F$ lies in the thick triangulated closure of $R \otimes_k F \cong (R \otimes_k R) \otimes_R F$. This implies that $F \in \Ind(\cC_R)$ is compact if and only if its pushforward to $\Ind(\cC)$ is compact. This implies that the $t$-structures on $\Ind(\cC_R)$ induced by a sluicing preserve $\cC_R$, because the pushforward $\Ind(\cC_R) \to \Ind(\cC)$ is exact and the $t$-structures on $\Ind(\cC)$ preserve $\cC$.

    Similarly, the noetherian $t$-structures on $\cC$ in \Cref{D:sluicing}(2) induce $t$-structures on $\cC_R$ that must also be noetherian, because the pushforward $\cC_R \to \cC$ preserves ascending chains. These noetherian $t$-structures verify \Cref{D:sluicing}(2) for the induced sluicing on $\cC_R$.
\end{proof}

\begin{cor} \label{R:base_change_sluicing_fields}
    If $k$ is a regular noetherian ring, then given a sluicing $\cP$ on $\cC$ any field $K$ over $k$, the induced $t$-structures $(\Ind(\cC_K)_{>\varphi},\Ind(\cC_K)_{\leq \varphi})$ preserve $\cC_R$ and restrict to a bounded $t$-structure on $\cC_R$. In particular, $\phi^\pm(E)$ is defined and finite for any $E \in \cC_K$, and $\phi^\pm(E) = \phi^\pm(E_L)$ for any field extension $K \subset L$.
\end{cor}
\begin{proof}
    Write $K = \bigcup_\alpha R_\alpha$ as a filtered union with each $R_\alpha$ a finitely generated $k$-algebra. By \Cref{ex:sluicing_finite_type} and \Cref{P:base_change_sluicing}, a sluicing on $\cC$ induces a sluicing on each $\cC_{R_\alpha}$. $\cM$ is locally of finite presentation, so for any $E \in \cC_K$ there is some $\alpha$ and $E_\alpha \in \cC_{R_\alpha}$ such that $E \cong E_\alpha \otimes_{R_\alpha} K$. $R_\alpha \to K$ is flat because it is injective, so we have $\tau_I(E) \cong K \otimes_{R_\alpha} \tau_I(E_\alpha) \in \cC_K$ for any interval $I$. The fact that $\phi^\pm(E)$ are preserved under a field extension $K \subset L$ follows from the fact that $L$ is faithfully flat over $K$.
\end{proof}

\subsection{Locally constant stability conditions}

In this subsection we consider a smooth, proper, and idempotent complete dg-category $\cC$ over a ring $k$. Let $\cM$ denote the functor from simplicial $k$-algebras to $\infty$-groupoids taking
\[
\cM : R \mapsto \cC_R^{\cong} := (\cC \otimes_k \Perf(R))^{\cong},
\]
where $(-)^{\cong}$ denotes the largest $\infty$-subcategory in which all morphisms are isomorphisms, and $\otimes_k$ denotes the symmetric monoidal structure on $k$-linear stable idempotent complete dg-categories.\endnote{Actually, the moduli functor is defined by $R \mapsto \Map(\cC^{\rm{op}},\Perf(R))$, where $\Map$ is maps of $k$-linear dg-categories. This is equivalent to $R \mapsto \cC \otimes_k \Perf(R)$ because $\cC$ is smooth and proper.} It is shown in \cite{ToenVaquie}*{Thm. 3.6} that $\cM$ is an algebraic higher derived stack locally of finite presentation over $k$. More precisely, it is a filtered union of algebraic derived $n$-stacks of finite presentation over $k$, where $n$ grows arbitrarily large.

The set of points $|\cM|$ can be defined as the disjoint union of $\pi_0(\cM(K))$, i.e., the set of isomorphism classes in $\cC_K$, over all fields $K$ essentially of finite type over $k$, modulo the equivalence relation generated by $[E] \sim [E \otimes_K L]$ for any extension of fields $K \subset L$ and $E \in \cC_K$. A point in $|\cM|$ is \emph{finite type} if it has a representative defined over a field $K$ that is finitely generated as a $k$-algebra. A subset of $|\cM|$ is \emph{bounded} if it is contained in the image of $|\Spec(R)| \to |\cM|$ for some morphism $\Spec(R) \to \cM$ from an affine $k$-scheme of finite type. A subset $U \subset |\cM|$ is \emph{open} if for any such morphism, the preimage of $U$ in $|\Spec(R)|$ is open. Any open $U \subset |\cM|$ is the set of points of a unique open substack $\cU \subset \cM$.

Let us fix a finitely generated free abelian group $\Lambda$ with a norm $\lVert-\rVert$. We say that a map $\ch : |\cM| \to \Lambda$ is \emph{locally constant} if for any morphism $\Spec(R) \to \cM$ from a connected affine scheme, the composition $|\Spec(R)| \to |\cM| \to \Lambda$ is constant. We say that $\ch$ is \emph{additive} if the following diagram of sets commutes
\[
\xymatrix{ | \cM \times_{\Spec(k)} \cM | \ar[r]^-{\oplus} \ar[d]^{\ch \times \ch} & |\cM| \ar[d]^{\ch} \\
\Lambda \times \Lambda \ar[r]^-{+} & \Lambda },
\]
where the morphism $\oplus$ maps $(E_1,E_2) \in \cC_R^2 \mapsto E_1 \oplus E_2 \in \cC_R$. Given a locally constant $\ch$, we let $\cM_u \subset \cM$ be the open and closed substack on which $\ch(E)=u \in \Lambda$.

\begin{lem}
    It suffices to check additivity of a locally constant $\ch$ at finite type residue fields $k \to \kappa$, meaning $\ch(E \oplus F) = \ch(E)+\ch(F)$, $\forall E,F \in \cC_\kappa$. If $v$ is additive, then $E \mapsto \ch(E)$ defines a group homomorphism $ \rm{K}_0(\cC_K) \to \Lambda$ for every field $K$ over $k$.
\end{lem}
\begin{proof}
    Because $\ch$ is locally constant, it suffices check additivity for any field $K$ of finite type over $k$, which must be a finite extension of a finite type residue field $\kappa$.
    
    Let $f_\ast : \Ind(\cC_K) \to \Ind(\cC_\kappa)$ be the pushforward, which commutes with filtered colimits and maps $\cC_K$ to $\cC_{\kappa}$ because $K$ is finite over $\kappa$. We claim that for any $E \in \cC_K$,
    \begin{equation}\label{E:finite_pushforward}
        v(f_\ast(E)) = \dim_\kappa(K) \cdot v(E).
    \end{equation}
    Because $v$ is a function on $|\cM|$, it suffices to prove this after base change to an algebraic closure $\bar{\kappa}$ of $\kappa$. If $f' : X := \Spec(K \otimes_\kappa \bar{\kappa}) \to \Spec(\bar{\kappa})$, then by the base change formula\endnote{I'm appealing to a more general notion of the base change formula that is proved in the same way as in derived algebraic geometry. Namely, there is a canonical base change homomorphism, and the fact that it is an isomorphism is equivalent to the projection formula after pushing forward to the original base. Then the projection formula is proved using the fact that $R$ generates $\Dqc(R)$ under filtered colimits and $f_\ast$ commutes with filtered colimits.} implies that $f_\ast(E) \otimes_\kappa \bar{\kappa}$ is the pushforward $f'_\ast(E|_X)$. $X$ can have non-reduced structure if $\kappa$ is not perfect, but it is always the case that $\cO_X$ has a filtration of length $\dim_\kappa(K)$ whose associated graded is a direct sum of copies of skyscraper sheaves $\cO_{p}$ at various $p \in X(\bar{\kappa})$. It follows that $f'_\ast(E|_X)$ has a filtration of length $\dim_\kappa(K)$ whose associated graded is a direct sum of copies of $p^\ast(E|_X)$ for various sections $p : \bar{\kappa} \to X$. Again because $\ch$ is a function on $|\cM|$, one has $\ch(p^\ast(E|_X)) = \ch(E)$ for any such section, which proves the formula \eqref{E:finite_pushforward}.
    
    Using \eqref{E:finite_pushforward} along with the facts that $f_\ast$ is exact and $\Lambda$ is torsion free, the equation $0=\ch(E)+\ch(F)-\ch(E \oplus F)$ for $E,F \in \cC_K$ follows from the equation $0 = \ch(f_\ast(E))+\ch(f_\ast(F)) - \ch(f_\ast(F)\oplus f_\ast(F))$, which is the hypothesis of the lemma.

    Now assume $\ch$ is additive. Then for any field $K$ over $k$ and any exact triangle $E \to F \to G$, the Rees construction\endnote{Given an extension like this in $\cC_K$, the corresponding $K[t]$-module object in $\cC_K$ is $E \oplus \bigoplus_{n \geq 1} F$, which has a natural grading with $E$ in degree $0$. The $t$ action is homogeneous of degree $1$. The map $t \cdot : E \to F$ is the map from the triangle, and the map $t \cdot F \to F$ is the identity in all higher degrees. The fiber at $t=1$ is the colimit of the diagram $E \to F \to F \to \cdots$, which is $F$, and the fiber at $t=0$ is $E \oplus G$.} gives a $K[t]$-point whose fiber at $t=1$ is $F$ and whose fiber at $t=0$ is $E\oplus G$, and thus $\ch(F) = \ch(E\oplus G) = \ch(E)+\ch(G)$. This implies that $\ch$ descends to a map $ \rm{K}_0(\cC) \to \Lambda$, which is a group homomorphism because $\ch$ is additive.
\end{proof}

\begin{defn}\label{D:locally_constant_stability_condition}
A \emph{locally constant stability condition} on a smooth, proper, and idempotent complete\endnote{The smooth and proper condition is not really necessary to formulate this definition, but it is not clear that it is the correct one for categories that are not smooth and proper.} stable dg-category $\cC$ over a Nagata\endnote{The class of Nagata rings is rather general, and includes many of the rings encountered in algebraic geometry. For example, any (quasi-)excellent ring is Nagata \cite{stacks-project}*{\href{https://stacks.math.columbia.edu/tag/07QV}{Tag 07QV}}, and the class of excellent rings includes finite type $k$-algebras, among many other examples \cite{stacks-project}*{\href{https://stacks.math.columbia.edu/tag/07QW}{Tag 07QW}}.} ring $k$ consists of
\begin{enumerate}
    \item a sluicing $\cP_\kappa$ on $\cC_\kappa$ for every residue field $\kappa$ of $k$,
    \item an additive locally constant map $\ch : |\cM| \to \Lambda$, and
    \item a homomorphism $Z : \Lambda \to \bC$.
\end{enumerate}

By \Cref{P:base_change_sluicing} and \Cref{ex:sluicing_finite_type}, the data of (1) induce a sluicing on $\cC_K$ for any field $K$ essentially of finite type over $k$.\endnote{Because $k \to K$ factors through a residue field $k \to \kappa$, and any residue field is essentially finite type over $k$ because $k$ is noetherian.} We require these data to satisfy the following conditions:
\begin{enumerate}[label=\alph*)]

    \item If $R$ is a discrete valuation ring over $k$ whose fraction field $\kappa$ is a residue field of $k$, and if $\lambda = R/\mathfrak{m}_R$ is the residue field of $R$, then there is a sluicing of width $1$ on $\cC_R$ that induces the same sluicings on $\cC_\kappa$ and $\cC_{\lambda}$ as those induced by the sluicings in (1).\endnote{This is the same as saying $R$ is a discrete valuation ring that lies between $k$ and one of its residue fields, $k \to R \to \kappa$. $\lambda = R/\mathfrak{m}_R$ is not necessarily a residue field of $k$ --- it is a residue field of a blowup of $k$ at some prime ideal.}
    
    \item There is a $w \in (0,1)$ and $C>0$ such that $\forall u \in \Lambda$, if there is a finite type point $x \in |\cM_u|$ with $\phi^+(x)-\phi^-(x) \leq w$, then $Z(u) \in \bR_{>0} e^{i\pi (\phi^-(x),\phi^+(x)]}$ and $|Z(u)| > C \lVert u \rVert$.
\end{enumerate}
\end{defn}

The intent of \Cref{D:locally_constant_stability_condition} is to identify the minimal set of axioms that allow for the construction of proper moduli spaces of semistable objects in \Cref{T:moduli_spaces} below. The idea of considering a stability condition over every residue field of $k$, and imposing a compatibility condition on the slicings for discrete valuation ring over $k$ comes from \cite{families}. Our compatibility condition (a) is slightly different than the notion of HN structures over curves used in \cite{families}. As discussed above, condition (a) is automatic if all of the sluicings on $\cC_\kappa$ for residue fields $\kappa$ are induced from a single sluicing on $\cC$ itself, but it appears to be too strong to require a sluicing on $\cC$ when $\dim(k)>1$ or $k$ is not regular.

\begin{ex}
    If $k$ is a field, then a stability condition on $\cC$ defines a locally constant stability condition as long as it satisfies the following condition: the homomorphism $v :  \rm{K}_0(\cC) \to \Lambda$ factors through the quotient of $ \rm{K}_0(\cC)$ by the subgroup generated by
    \[
    [\kappa(s):k] [E_t] - [\kappa(t):k] [E_s],
    \]
    for all triples $(T,t,s)$, where $T$ is a connected affine $k$-scheme of finite type, $s,t \in T$ are points whose residue fields $\kappa(t),\kappa(s)$ are finite extensions of $k$, and $E \in \cC_T$. Here $[E_t] \in  \rm{K}_0(\cC)$ denotes the image of the fiber $E_t \in \cC_{\kappa(t)}$ under the pushforward functor $\cC_{\kappa(t)} \to \cC$, and likewise for $[E_s]$. This condition holds if $v$ is numerical, in the sense that it factors through the quotient of $ \rm{K}_0(\cC)$ by the kernel of the Euler pairing $(E,F) = \chi(\RHom(E,F))$.
\end{ex}

For any interval $I=(a,b]$, we let $\cM^{I} \subset \cM$ denote the substack whose $T$-points are those $E_T$ whose restriction to any point of $T$ lie in $\cP_K((a,b])\subset \cC_K$. This defines a substack by the following:

\begin{lem} \label{L:phase_mass_functions}
    A locally constant stability condition on $\cC$ induces a stability condition on $\cC_K$ for any field $K$ over $k$. The HN filtration of an object $E \in \cC_K$ is preserved under base change to an extension field of $K$, hence $\phi^+(E)$, $\phi^-(E)$, and $m(E)$ are well-defined functions on $\lvert \cM\rvert$.
\end{lem}
\begin{proof}
    For any field extension $K/\kappa$ of a residue field $\kappa$ of $k$, the sluicing on $\cC_\kappa$ induces $t$-structures $((\cC_K)_{>\varphi},(\cC_K)_{\leq \varphi})$ on $\cC_K$ by \Cref{R:base_change_sluicing_fields}, but it is not clear \emph{a priori} that they satisfy the condition \Cref{D:sluicing}(2) on a sluicing. However, if $K/\kappa$ is essentially finite type, \Cref{P:base_change_sluicing} and \Cref{P:sluicing_vs_stability_condition} imply that these $t$-structures and $Z$ determine a stability condition $\cC_K$.\endnote{The condition that $Z(E) \subset \bR_{>0} \cdot e^{i\pi \phi(E)}$ and the support property then follow from \Cref{D:locally_constant_stability_condition}(b), which implies the condition on $Z$ in \Cref{P:sluicing_vs_stability_condition}.} In general, every object of $\cC_K$ is the base change of some $E \in \cC_L$ for a subfield $L \subset K$ essentially of finite type over $\kappa$.\endnote{This follows from the fact the moduli functor $\cM$ is locally finitely presented, because any $K$ is a filtered union of subfields essentially finite type over $k$.} Because $L \subset K$ is faithfully flat, the pullback functor $\cC_L \to \cC_K$ is exact and preserves semistable objects,\endnote{We have defined an object to be semistable of phase $\phi$ if $E \in \cC_{\leq \phi}$ and $\tau_{>\varphi}(E)=0$ for all $\varphi<\phi$. This condition is preserved for any conservative exact functor $\cC \to \cB$.} so the base change of the HN filtration of $E$ is an HN filtration of $E_K$. The function $Z$ is preserved by base change by definition, and it satisfies the support property with a uniform constant over all essentially finite type fields over $k$, so we have verified directly that we have a stability condition on $\cC_K$.
\end{proof}

\begin{lem} \label{L:phase_semicontinuity}
Given a locally constant stability condition on $\cC$, suppose $x_0$ is a specialization of a point $x \in |\cM|$. Then $\phi^+(x) \leq \phi^+(x_0)$ and $\phi^-(x) \geq \phi^-(x_0)$.
\end{lem}

\begin{proof}
Because the stack $\cM$ is locally of finite presentation over $k$ and $k$ is Nagata, one can find a discrete valuation ring $R$ essentially of finite type over $k$, with residue field $\kappa$ and fraction field $K$, and an $\cE \in \cC_R$ such that $\cE_\kappa$ represents $x_0$ and $\cE_K$ represents $x$.\endnote{By definition of specialization, one can find an affine $k$-scheme $T$ and points $t,t_0$ realizing $x$ and $x_0$. For instance, you can take the local ring of a smooth cover of $\cM$ at a point representing $x_0$. One can write $T = \Spec(\colim_\alpha R_\alpha)$ for a filtered colimit of finite type $k$-algebras $R_\alpha$. Because $M$ is locally of finite presentation, the given $T$-point is the restriction of an $R_\alpha$-point for some $\alpha$. You can then consider the local ring at the point $t_0$ representing $x_0$, and dominate this by a local map from an essentially finite type discrete valuation ring because $k$ is Nagata.} Either $\Spec(R) \to \Spec(k)$ maps to a single point, or it factors through a discrete valuation ring whose fraction field is a residue field of $k$. In both cases, \Cref{D:locally_constant_stability_condition}(a) and \Cref{P:base_change_sluicing} give a sluicing on $\cC_R$ of width $1$ that induces the given sluicings on $\cC_K$ and $\cC_\kappa$. To show that $\phi^-(\cE_\kappa) \leq \phi^-(\cE_K)$, we will suppose that $b := \phi^-(\cE_\kappa) > a > \phi^-(\cE_K)$ with $b-a<1$ and derive a contradiction.\endnote{Actually, this argument works for a sluicing of any width $0<w\leq 1$. We instead would choose $a$ and $b$ with $b-a<w$.}

We can choose a noetherian $t$-structure on $\cC_R$ such that $\cP((b,\infty)) \subseteq (\cC_R)_{> 0} \subseteq \cP((a,\infty))$ by condition (2) of \Cref{D:sluicing}. The restriction functor $\cC_R \to \cC_K$ is $t$-exact, so if we can show that $\cE_{\kappa} \in (\cC_{\kappa})_{> 0}$ implies $\cE \in (\cC_R)_{> 0}$, this would contradict the assumption that $\cE_K$ has a HN factor of phase $<a$. Likewise, to show that $\phi^+(\cE_\kappa) \geq \phi^+(\cE_K)$, it suffices to show that for a noetherian $t$-structure on $\cC_R$, $\cE_\kappa \in (\cC_{\kappa})_{\leq 0}$ implies $\cE \in (\cC_R)_{\leq 0}$.

In fact, something stronger is true: for any $i$ such that $H^i(\cE_\kappa)=0$, $H^i(\cE)=0$ as well. Because the pushforward functor $\iota : \cC_\kappa \to \cC_R$ is $t$-exact and conservative, this is equivalent to showing that
\[
H^i(\iota(\cE_\kappa)) \cong H^i(\cofib(\cE \xrightarrow{\pi \cdot \id_{\cE}} \cE)) = 0
\]
implies $H^i(\cE) = 0$ in $\cC_R$. Letting $\cF := \cofib(\cE \xrightarrow{\pi \cdot \id_{\cE}} \cE)$, we have a long exact sequence in $\cC_R^\heart$,
\[
\cdots \to H^{i-1}(\cF) \to H^i(\cE) \xrightarrow{\pi} H^i(\cE) \to H^i(\cF) \to \cdots.
\]
Therefore $H^i(\cF)=0$ implies that $\pi \cdot \id$ is surjective on $H^i(\cE)$, which by \Cref{L:nakayama} below implies that $H^i(\cE) = 0$.
\end{proof}

\begin{lem}[generalized Nakayama]\label{L:nakayama}
    Let $R$ be a discrete valuation ring with maximal ideal $(\pi)$, and let $\cA$ be an $R$-linear abelian category. Let $E \in \cA$ be a noetherian object such that $\End(E)$ is finitely generated as an $R$-module. If $\pi \cdot \id_E : E \to E$ is surjective, then $E=0$.
\end{lem}
    
\begin{proof}
    Let $K_n := \ker(\pi^n \cdot \id_E : E \to E)$. Because $\pi$ is surjective, $\pi^n$ is as well, and we have a short exact sequence for any $n$, $0 \to K_n \to E \to E \to 0$. Pulling back along the inclusion $K_1 \subset E$ into the quotient gives a short exact sequence $0 \to K_n \to K_{n+1} \to K_1 \to 0$. It follows that $K_1 = 0$, or else the ascending chain $K_1 \subset K_2 \subset \cdots$ in $E$ would not stabilize. Therefore, $\pi \cdot \id_E : E \to E$ is an isomorphism. This implies that on the finitely generated $R$-module $\End(E)$, multiplication by $\pi$, which is equivalent to composition with $\pi \cdot \id_E$, is an isomorphism. Then $\End(E)=0$ by the usual Nakayama lemma, and hence $E=0$.
\end{proof}

A stack is $\cX$ is \emph{$\Theta$-complete} (also known as $\Theta$-reductive) if for any discrete valuation ring $R$, any morphism $\Theta_R \setminus 0 \to \cX$ extends uniquely to $\Theta_R$, where $\Theta_R := \Spec(R[t])/\bG_m$ and $\bG_m$ acts with weight $-1$ on $t$.\endnote{Here $0$ denotes the unique closed point, which is $\{\pi=t=0\}$ for a uniformizer $\pi$ of $R$.} Similarly, $\cX$ is $S$-complete if the analogous codimension-$2$ filling condition holds with $\Spec(R[s,t]/(st-\pi)) /\bG_m$ in place of $\Theta_R$, where $s$ has weight $1$, $t$ has weight $-1$, and $\pi$ is a uniformizer of $R$.\endnote{One can consider these conditions for any stack, not necessarily algebraic. For the moment, we are not assuming that our moduli stack $\cM^{(\varphi,\varphi+1]}$ of objects is algebraic --- we prove this under additional hypotheses in the next section.}

\begin{lem}\label{L:Theta_reductivity_S_completeness}
For any $\varphi \in \bR$, the substack $\cM^{(\varphi,\varphi+1]} \subset \cM$ is $\Theta$-complete and $S$-complete with respect to discrete valuation rings essentially of finite type over $k$, as is the substack $\cM_v^{\rm{ss}} \subset \cM^{(\varphi,\varphi+1]}$ of semistable points of class $v \in \Lambda$.
\end{lem}
\begin{proof}
    \smallskip
    \noindent\textit{$\cM^{(\varphi,\varphi+1]}$ is $\Theta$-complete:}
    \smallskip
    
    Let $R$ be a discrete valuation ring essentially of finite type over $k$ with residue field $\kappa = R/\mathfrak{m}_R$, and let $\cP$ denote the sluicing of width $1$ on $\cC_R$ given by \Cref{D:locally_constant_stability_condition}(a).
    
    Morphisms $\Theta_R \to \cM$ correspond to objects in $\cC_{\Theta_R} \subset \Ind(\cC)_{\Theta_R}$, which are diagrams $\cdots E_{n+1} \to E_n \to \cdots$ in $\Ind(\cC_R)$ satisfying certain conditions:
    \begin{enumerate}[label=(\roman*)]
        \item Each $E_n \in \cC_R$, because the functor taking the diagram $E_\bullet \mapsto E_n$ admits a co\-continuous right adjoint, taking $F \in \cC_R$ to $\cdots\to 0 \to F \to 0 \to \cdots$ with the only non-zero term in degree $n$;
        \item The associated graded object $\bigoplus_n \cofib(E_{n+1}\to E_n)$ is in $\cC_R$, which implies that $E_{n+1} \to E_n$ is an isomorphism for all but finitely many $n$; and
        \item The fact that $\Ind(\cC)_{\Theta_R}$ is classically generated by diagrams of the form $ \cdots \to 0 \to E_n \to E_{n} \to E_n \to \cdots$ implies that for any object in $\cC_{\Theta_R}$, $E_i \cong 0$ for $i \gg 0$.
    \end{enumerate}
    By \Cref{L:phase_semicontinuity}, the condition that $E_\bullet$ defines a morphism $\Theta_R \to \cM^{(\varphi,\varphi+1]}$ amounts to the condition that $\bigoplus_n \cofib(E_{n+1}\otimes_R \kappa \to E_n \otimes_R \kappa)$ lies in the heart $\cP((\varphi,\varphi+1]) \subset \cC_R$. Therefore, $\Theta_R$-points of $\cM^{(\varphi,\varphi+1]}$ correspond to $\bZ$-weighted filtered objects in $\cP((\varphi,\varphi+1])$ that remain filtrations after tensoring with $\kappa$, and the analogous description applies to maps $\Theta_K \to \cA$, where $K$ is the fraction field of $R$.
    
    One can now check $\Theta$-reductivity directly. Consider a morphism $\Theta_R \setminus 0 \to \cM^{(\varphi,\varphi+1]}$, corresp\-onding to an object $E\in \cC_R^\heart$ and a finite weighted filtration $\cdots \hookrightarrow E_{n+1} \hookrightarrow E_n \hookrightarrow \cdots$ of $E\otimes K \in \cC_K$. One obtains a filtration $\cdots\hookrightarrow E \cap E_{n+1} \hookrightarrow E \cap E_{n} \hookrightarrow \cdots$ of $E$, where the intersection is taken in the heart of the $t$-structure on $\Ind(\cC_R)$ corresponding to $\cP((\varphi,\varphi+1])$. It follows from this definition that $E \cap E_i \cong E$ for $i\ll0$ and $E \cap E_i \cong 0$ for $i\gg 0$, and $E \cap E_i \in \cC_R$ because the sluicing has width $1$, which implies both objects lie in the heart of some ind-noetherian $t$-structure on $\Ind(\cC_R)$.\endnote{\Cref{D:sluicing}(1) implies that $E \in \cP((a,\varphi+1])$ and $E_i \in \cP((b,\varphi+1])$. The interval $I=(\min(a,b),\varphi_1]$ has width $<1$, so the fact that the sluicing has width $1$ means $\cP(I)$ is contained in the heart of a noetherian $t$-structure.} The argument of \cite{ModuliSpaces}*{Lem. 7.17} then shows that this filtration of $E$ defines a morphism $\Theta_R \to \cM^{(\varphi,\varphi+1]}$ extending the given one on $\Theta_R \setminus 0$.
    
    \smallskip
    \noindent\textit{$\cM^{(\varphi,\varphi+1]}$ is $S$-complete:}
    \smallskip
    
    We regard $R[s,t]/(st-\pi)$ as a commutative algebra object in the category of graded objects $\Ind(\cC_R)^{\bZ}$. We let $\overline{ST}_R$ be the relative spectrum of this algebra over $(B\bG_m)_R$. As in the case of $\Theta$-reductivity, maps $\overline{ST}_R \to \cM$ correspond to $R[s,t]/(st-\pi)$-module objects $\bigoplus_n E_n \in \Ind(\cC_R)^{\bZ}$ such that $\forall n\in \bZ, E_n \in \cC_R$, $F_n := \cofib(s : E_{n-1} \to E_n) \cong 0$ for $n\gg 0$, and $G_n := \cofib(t : E_{n+1} \to E_{n}) \cong 0$ for $n \ll 0$.
    
    The only subset of $\overline{ST}_R$ that is stable under generalization and contains the closed point $\{s=t=0\}$ is $|\overline{ST}_R|$ itself, so because $|\cM^{(\varphi,\varphi+1]}|$ is stable under generalization (\Cref{L:phase_semicontinuity}) and the substack $\cM^{(\varphi,\varphi+1]}$ is defined by a pointwise condition, a morphism $\overline{ST}_R \to \cM$ factors through $\cM^{(\varphi,\varphi+1]}$ if and only if
    \[
        \cofib(G_{n-1} \xrightarrow{s}  G_n) \in \cP((\varphi,\varphi+1]) \subset \cC_R
    \]
    for all $n \in \bZ$. In this case all $E_n$, $F_n$ and $G_n$ also lie in $\cP((\varphi,\varphi+1])$, and the maps $s : G_{n-1} \to G_n$ and $t : F_{n+1} \to F_n$ are injective. Therefore, a morphism $\overline{ST}_R \to \cA$ can be identified with an $R[s,t]/(st-\pi)$-module object in the abelian category of graded objects in $\Ind(\cC_R)^{(\varphi,\varphi+1]}$ satisfying the conditions above.
    
    As in the proof of \cite{ModuliSpaces}*{Lem. 7.16}, we can now verify $S$-completeness directly. Consider a morphism $\overline{ST}_R \setminus 0 \to \cM^{(\varphi,\varphi+1]}$, corresponding to a pair of objects $E_1,E_2 \in \cP((\varphi,\varphi+1]) \subset \cC_R$ and an isomorphism $F \cong E_1 \otimes K \cong E_2 \otimes K$, where $K$ is the fraction field of $R$. We have a graded $R[s,t]/(st-\pi)$ module object $\bigoplus_n (E_1 \cap (\pi^{-n} \cdot E_2)) t^n \subset \bigoplus_n F t^n$, where the intersection is taken in $\Ind(\cC_R)_{(\varphi,\varphi+1]}$. Here we are identifying $E_1,E_2$ as subobjects of $F$, and we are using the fact that multiplication by $\pi$ is an automorphism of $F$. The hypothesis that the sluicing on $\cC_R$ has width $1$ implies that $E_1$ and $(E_1 \cap (\pi^{-n} \cdot E_2)$ both lie in the heart of some ind-noetherian $t$-structure on $\Ind(\cC_R)$, and therefore $(E_1 \cap (\pi^{-n} \cdot E_2) \subseteq E_1$ lies in $\cC_R$ for all $n$. One can check that this object defines a morphism $\overline{ST}_R \to \cM^{(\varphi,\varphi+1]}$ extending the original morphism.
    
    \smallskip
    \noindent\textit{$\cM_v^{\rm{ss}}$ is $\Theta$-complete and $S$-complete:}
    \smallskip
    
    As discussed in \cite{halpernleistner2022structure}*{\S6.4}, the open substack $\cM_v^{\rm{ ss}} \subset \cM^{(\varphi,\varphi+1]}$ consists of points $E \in \cC_K^{(\varphi,\varphi+1]}$ of class $v \in \Lambda$ such that for any $\bZ$-weighted filtration $\cdots \to E_{w+1} \to E_w \to \cdots$ of $E$ one has
    \begin{equation}\label{E:linear_function}
    \ell(E_\bullet) := \sum_{w \in \bZ} w \Im \left(\overline{Z(v)} Z(\gr_w(E_\bullet)) \right) \leq 0.
    \end{equation}
    In other words, $\cM_{v}^{\rm{ss}}$ is the $\Theta$-semistable locus with respect to the linear function on the component fan of $\cM^{(\varphi,\varphi+1]}_v$ defined by $\ell$. Whenever one has a stability condition of this kind on a $\Theta$-complete and $S$-complete stack, the semistable locus is also $\Theta$-complete and $S$-complete \cite{halpernleistner2022structure}*{Thm. 5.5.8}. (See also \cite{ModuliSpaces}*{\S7.3}.)
\end{proof}

\begin{cor}\label{C:mass_semicontinuity}
    Given a locally constant stability condition on $\cC$, suppose $x_0$ is a special\-ization of a point $x \in |\cM|$. If $x \in {(\varphi,\varphi+1]}$, then $m(x) \leq m(x_0)$.
\end{cor}
\begin{proof}
    As in the proof of \Cref{L:phase_semicontinuity}, it suffices to prove this when there is a morphism $\Spec(R) \to \cM$ for a discrete valuation ring essentially of finite type over $k$ whose generic point $\Spec(K)$ maps to $x$ and special point $\Spec(\kappa)$ maps to $x_0$. Because $\cM^{(\varphi,\varphi+1]}$ is $\Theta$-complete, the Harder--Narasimhan filtration of $E_K$ extends to a filtration of $E$ as an $R$-point of $\cM^{(\varphi,\varphi+1]}$. In particular, it gives a filtration $E_n \subseteq \cdots \subseteq E_1 = E_\kappa$ in $\cP_\kappa((\varphi,\varphi+1])$ that is convex in the sense that $\Im(Z(\gr_j(E_\bullet)) \overline{Z(\gr_i(E_\bullet))}) > 0$ for all $i<j$. For any such convex filtration, \cite{Bayer_short}*{Prop. 3.3} implies that $m(E_\kappa) = m(E_1) \geq \sum_j |Z(\gr_j(E_\bullet))| = m(E_K)$ --- using the fact that for two real concave functions $f,g$ on $[a,b]$ with $f(a)=g(a)$, $f(b)=g(b)$, and $f(x)\leq g(x)$ for all $x \in [a,b]$, the arc length of $f$ is $\leq$ that of $g$ over $[a,b]$.
\end{proof}

\subsection{Existence of moduli spaces}

Our main result of this section is the following:

\begin{thm} \label{T:moduli_spaces}
    Let $(\cP,v,Z)$ be a locally constant stability condition on a smooth, proper, and idempotent complete stable dg-category $\cC$ over an excellent ring $k$ of characteristic $0$. Let $G \in \cC$ be a classical generator such that\endnote{Note that by \Cref{L:phase_mass_functions}, if this condition holds for every finite type residue field, then it holds for any field of finite type over $k$. Here are some examples to clarify the residue field of finite type condition: if $k$ is a discrete valuation ring with maximal ideal $\mathfrak{m}$, then both $k/\mathfrak{m}$ and the field of fractions $\mathrm{Frac}(k)$ are residue fields of finite type over $k$. However if $k = \bC[\![x,y]\!]$ and $\mathfrak{m} = (x,y)$, then $\mathrm{Frac}(k)$ is a residue field that is not finite type, because it is not finitely generated as a $k$-algebra. On the other $\mathrm{Frac}(k / (f))$ for any irreducible element $f \in k$ is a residue field that is finitely generated as a $k$ algebra because $\mathrm{Frac}(k/(f))$ can be obtained from $k/(f)$ by inverting a single element --- any element that set-theoretically cuts out the point $\mathfrak{m}$.} 
    \begin{equation} \label{E:bound_on_generator}
        \sup_{\substack{\kappa \text{ a residue field of}\\ \text{finite type over } k}} \max \left( m(G_\kappa), \phi^+(G_\kappa) - \phi^-(G_\kappa) \right) < \infty.
    \end{equation}
    Then the following are equivalent:
\begin{enumerate}
    \item There is a function $f : \bR_{>0} \to \bZ_{\geq 0}$ such that for any field $K$ of finite type over $k$ and any nonzero $E \in \cC_K$, $\dim \Hom(G_K,E) \leq f(m(E))$. 
    
    \item For any $C>0$, the set of finite type points $E \in |\cM|$ such that $m(E)<C$ and $|\phi^\pm(E)|<C$ is bounded. 
    
    \item For any $v \in \Lambda$, the subset of $|\cM|$ parameterizing $\sigma$-semistable objects $E \in \cC_K^\heart$ of numerical class $v$ is open. The corresponding open substack $\cM^{\rm {ss}}_{v} \subset \cM$ has affine diagonal (and in particular is a derived $1$-stack), and the underlying classical stack of $\cM^{\rm{ss}}_{v}$ admits a proper good moduli space.
\end{enumerate}
\end{thm}


\begin{rem}
    The condition \eqref{E:bound_on_generator} is automatic if $G$ is a direct sum of objects $G_i$ for which there is an interval $I$ of width $<1$ such that $(G_i)_\kappa \in \cP_\kappa(I)$ for every residue field of finite type over $k$, because then one has $m(E) \leq r |Z(E)|$ for some constant $r$ depending on the width of $I$. This holds if the sluicings $\cP_\kappa$ are induced by a single sluicing on $\cC$, because one can replace a classical generator $G$ with the direct sum of its projections $\tau_I(G)$ onto various small intervals. For instance \eqref{E:bound_on_generator} is automatic if $k$ is a field or discrete valuation ring.
\end{rem}

\begin{rem}
    As the proof will show, in conditions $(1)$ and $(2)$ it is equivalent to quantify over all points in $|\cM|$, instead of just the finite type points. Also, we can take the function in $(1)$ to be monotone increasing and superadditive, in the sense that $f(x+y) \geq f(x)+f(y)$.
\end{rem}

There has been recent progress in the construction of moduli spaces for Bridgeland semi\-stable objects \cite{ModuliSpaces}*{\S7}. However, the obstruction to proving the existence of moduli spaces in general has been verifying the ``generic flatness'' property for the heart of a stability condition (see \cites{ AbramovichPolishchuk,ArtinZhang}).\endnote{Generic flatness means that for any object $E \in \cC_R^\heart$, if $E_\kappa \in \cC_R^\heart$ for some residue field $\kappa$ of $R$, then there is an $f \in R$ such that $f \notin \ker(R \to \kappa)$ and for any other prime $\mathfrak{p} \subset R$ with $f \notin \mathfrak{p}$, $E_{R/ \mathfrak{p}} \in \cC_{R/\mathfrak{p}}^\heart$.} Our main contribution here is that the boundedness hypotheses of \Cref{T:moduli_spaces} imply this generic flatness property. The following is a very strong form of the generic flatness property:

\begin{prop}\label{P:generic_flatness}
In the context of \Cref{T:moduli_spaces}, if condition (2)  holds, then for any $C \in \bR$, the subsets $\{E \in |\cM| \text{ s.t. } \phi^+_\sigma(E) \leq C\}$ and $\{E \in |\cM| \text{ s.t. } \phi^-_\sigma(E) \geq C\}$ are open.
\end{prop}

We will need some lemmas to prove this.

\begin{lem} \label{L:HN_openness}
Let $\cE_n \to \cE_{n-1} \to \cdots \to \cE_0$ be a diagram in $\cC_T$ for an affine $k$-scheme $T$. If $t \in T$ specializes to $t_0$ and $(\cE_\bullet)_{t_0}$ is an HN filtration in $\cC_{t_0}$, then $(\cE_\bullet)_{t}$ is an HN filtration in $\cC_{t}$ with the same phases.
\end{lem}

\begin{proof}
Because the formation of associated graded objects commutes with base change, it suffices to show that if $\cE_{t_0}$ is semistable of phase $\phi$, then so is $\cE_t$. This is an immediate consequence of \Cref{L:phase_semicontinuity}.
\end{proof}

\begin{lem} \label{L:boundedness}
    Under the hypotheses of \Cref{T:moduli_spaces}, if $T$ is an affine $k$-scheme of finite type and $\cE \in \cC_T$, then the values of $|\phi^+(\cE_t)|$, $|\phi^-(\cE_t)|$, $m(\cE_t)$, and the length of the HN filtration of $\cE_t$ are uniformly bounded above over all finite type points $t \in |T|^{\rm{ft}}$. In addition, only finitely many $\phi \in \bR$ arise as the phase of an HN factor of $\cE_t$ for some $t \in |T|^{\rm{ft}}$.
\end{lem}
\begin{proof}
    The stable dg-category $\cC_T$ is classically generated by $G_T$. Therefore the object $\cE \in \cC_T$ can be realized as a retract of an object $\cE'$ constructed from an explicit sequence of extensions starting with $G_T$ and its shifts. Every fiber $\cE_t$ is therefore also a retract of $\cE'_t$, which is constructed via the same sequence of extensions. The bounds on $\phi^-(G_t)$ and $\phi^+(G_t)$ in \eqref{E:bound_on_generator} imply that there is a uniform bound $-N < \phi^-(\cE_t) < \phi^+(\cE_t) < N$ for all $t \in |T|^{\rm{ft}}$, where the latter denotes the set of finite type points of $T$.

    Using the same generation argument, the triangle inequality for masses implies that $m(\cE_t) \leq m(\cE'_t) \leq M \cdot m(G_t) \leq M K$ for every $t \in |T|^{\rm{ft}}$, 
    where $M$ is the number of copies of $G$ that appear in the construction of $\cE'$ and $K$ is the supremum appearing in \eqref{E:bound_on_generator}.
    
    The support property \Cref{D:locally_constant_stability_condition}(b) then implies that there are only finitely many classes $v \in \Lambda$ that can appear as HN factors of $\cE_t$ for $t \in |T|^{\rm{ft}}$. Combined with the uniform phase bound, this implies that there are only finitely many $\phi \in \bR$ that arise as the phases of HN factors of some $\cE_t$, and hence the length of the HN filtration of $\cE_t$ is uniformly bounded above for all $t\in |T|^{\rm{ft}}$.
\end{proof}

\begin{proof}[Proof of \Cref{P:generic_flatness}]
We consider a finite type affine $k$-scheme $T$ and a map $T \to \cM$, corresp\-onding to an object $\cE \in \cC_T$. Our goal is to show that the set of $t \in T$ such that $\phi^+(\cE_t)\leq C$ is open, and likewise for the set of $t$ with $\phi^-(\cE_t)\geq C$.

Let $\phi_{\max} := \max \{\phi^+(\cE_t) : t \in T \text{ closed point}\}$ and analogously $\phi_{\min} := \min\{ \phi^-(\cE_t) : t \in T \text{ closed point} \}$ --- they exist by \Cref{L:boundedness}. It follows from \Cref{L:phase_semicontinuity} that $\phi_{\max}$ is also the maximum of $\phi^+(\cE_t)$ over all points in $T$, not just closed points, and likewise for $\phi^-_{\min}$. Let 
\[
S_+ := \{t \in T : \phi^+(\cE_t) = \phi_{\max}\} \subset |T|,
\]
and define $S_-$ analogously using $\phi^-$ and $\phi_{\min}$. Let $S_\pm^{\rm{ft}} \subset S_\pm$ be the subset of finite type points. By \Cref{L:boundedness}, only finitely many phases appear among HN factors of $\cE_t$ for $t \in |T|^{\rm{ft}}$, and $\phi_{\max}$ and $\phi_{\min}$ must be among them. Therefore, it suffices by induction to show that $S_+,S_- \subset |T|$ are closed. We will give the argument for $S_+$, as the argument for $S_-$ is completely analogous.

Our argument is a modification of the construction of a $\Theta$-stratification in \cite{halpernleistner2022structure}*{Thm. 2.2.2}.

Consider the stack $\Filt^{\rm{un}}(\cM)$ of unweighted filtered points in $\cM$. This is a union of stacks $\Filt^{\rm{un}}(\cM)_n$ parameterizing filtrations of length $n$. $\Filt^{\rm{un}}(\cM)_n$ is the moduli stack of objects in the smooth and proper $k$-linear dg-category of diagrams in $\cC$ of the form $E_n \to \cdots \to E_1$, which we regard as a filtration of $E_1$. Therefore $\Filt^{\rm{un}}(\cM)$ is an algebraic higher derived stack, locally of finite presentation over $k$.

The map taking such a diagram to $E_1$ defines a forgetful morphism $\Filt^{\rm{un}}(\cM) \to \cM$. We define the flag stack $\Flag(\cE) := T \times_{\cM} \Filt^{\rm{un}}(\cM)$; points in $|\Flag(\cE)|$ classify a point in $T$ and a filtration of $\cE_t$. Let $|\Flag(\cE)|^{\rm{HN}} \subset |\Flag(\cE)|$ denote the subset of points whose underlying filtration is an HN filtration, and let $|\Flag(\cE)|^{\rm{ftHN}} \subset |\Flag(\cE)|^{\rm{HN}}$ denote the subset of finite type points.

The associated graded morphism $\gr : \Filt^{\rm{un}}(\cM)_n \to \cM^n$ taking 
\[
    (E_1\to \ldots \to E_n) \mapsto \left( E_1, \cofib(E_1 \to E_2),\ldots, \cofib(E_{n-1} \to E_n) \right)
\]
is $\infty$-quasi-compact, and a point $E_\bullet \in \Filt^{\rm{un}}(\cM)_n$ is a HN filtration if and only if $\gr(E_\bullet)$ is a tuple of semistable objects with decreasing phases. It follows from condition (2) of \Cref{T:moduli_spaces} for finite type points, combined with the observations above on the bounded\-ness of $m_\sigma$ and the length of HN filtrations of $\cE_t$ for $t \in |T|^{\rm{ft}}$, that the image of $|\Flag(\cE)|^{\rm{ftHN}}$ under $\gr$ is bounded, and hence $|\Flag(\cE)|^{\rm{ftHN}}$ is bounded.

We can therefore choose a smooth morphism $U \to \Flag(\cE)$ from a finite type affine scheme $U$ whose image contains $|\Flag(\cE)|^{\rm{ftHN}}$. Let $F$ denote the composition $F : U \to \Flag(\cE) \to T$. Let $Y \subset U$ denote the reduced closure of the set of generic points of $U$ whose image in $\Filt^{\rm{un}}(\cM)$ is an HN filtration containing a factor of phase $\phi_{\max}$.

Any point $t \in F(Y)$ is a specialization of $F(\eta)$ for some generic point $\eta$ of $U$ with $\phi^+(\cE_{F(\eta)}) = \phi_{\max}$, so \Cref{L:phase_semicontinuity} implies that $F(Y) \subset S_+$.

On the other hand, for any $t \in |T|^{\rm{ft}}$ with $\phi^+(\cE_t) = \phi_{\max}$, there is a point in $u \in U$ with $t = F(u)$ and whose image in $\Filt^{\rm{un}}(\cM)$ is an HN filtration of $\cE_t$. By \Cref{L:HN_openness}, the generic point of any component of $U$ containing $u$ also corresponds to an HN filtration with maximal phase $\phi_{\rm{max}}$, and thus $u \in Y$. We therefore have $S_+^{\rm{ft}} \subset F(Y)$.

Now consider $t \in T \setminus F(Y)$. The set $T \setminus F(Y)$ is constructible by Chevalley's theorem, because $F$ is of finite presentation. Therefore $t$ specializes to a finite type point $t_0 \in T \setminus F(Y)$. Because $S_+^{\rm{ft}} \subset F(Y)$, \Cref{L:phase_semicontinuity} implies that $\phi^+(\cE_t) \leq \phi^+(\cE_{t_0}) < \phi_{\max}$, so $t \notin S_+$. Thus, ultimately, $S_+=F(Y)$ is a constructible set, and \Cref{L:phase_semicontinuity} then implies that $S_+$ is closed by \cite{stacks-project}*{\href{https://stacks.math.columbia.edu/tag/0060}{Tag 0060}}.
\end{proof}

\begin{prop}\label{P:existence}
    In the context of \Cref{T:moduli_spaces}, if condition (2) holds, then:
    \begin{enumerate}
    \item $\cM^{(\varphi,\varphi+1]} \subset \cM$ is an open substack with affine diagonal, algebraic and locally of finite presentation over $k$;
    \item $\cM^{(\varphi,\varphi+1]}$ admits a well-ordered $\Theta$-stratification by Harder--Narasimhan filtrations, and the subset $\{x | m(x) \leq C\}$ is an open and quasi-compact union of strata for any $C>0$.
    \item For any $v \in \Lambda$, $\cM_v^{\rm{ss}} = \{x | \ch(x)=v \text{ and } m(x) \leq |Z(v)|\}  \subset \cM^{(\varphi,\varphi+1]}$ is open and admits a proper good moduli space;
    \end{enumerate}
\end{prop}

\begin{proof}
\Cref{P:generic_flatness} implies that the set of semistable points in $|\cM|$ is open, because these are precisely the points where $\phi^+(E) \leq C$ and $\phi^-(E) \geq C$ for some $C$, and $\cM^{(\varphi,\varphi+1]}$ is open because it is the union of the substacks of points satisfying $\varphi+\epsilon\leq \phi^-(E) \leq \phi^+(E) \leq \varphi+1$ for $0<\epsilon \ll 1$, by \Cref{D:sluicing}(1). In addition, the condition $v(E) = v$ is open, because the function $E \mapsto v(E) \in \Lambda$ is locally constant on $|\cM|$. Therefore, $\cM^{\rm{ss}}_v$ and $\cM^{(\varphi,\varphi+1]}$ are open substacks of $\cM$. $\cM^{(\varphi,\varphi+1]}$ and $\cM_v^{\rm{ss}}$ are derived $1$-stacks, because an $E \in \cC_T$ corresponding to a $T$-point of $\cM^{(\varphi,\varphi+1]}$ has vanishing negative self-Ext groups after arbitrary base change.

By hypothesis, the set of finite type points in $\cM_v^{\rm {ss}}$ is bounded. This implies that $\cM_v^{\rm {ss}}$ is quasi-compact, because given a smooth morphism $f : U \to \cM_v^{\rm {ss}}$ whose image contains all of the finite type points, the complement $|\cM_v^{\rm {ss}}| \setminus f(|U|)$ is a locally constructible subset that does not contain any finite type points, and is thus empty.

Finally, we check the conditions on the moduli functor $\cM_v^{\rm{ss}}$ that imply the existence of a proper good moduli space:

\smallskip
\noindent\textit{$\cM^{(\varphi,\varphi+1]}$ has affine diagonal:}
\smallskip

Although it is formulated in the context of a moduli functor for objects in an abelian category, the proof of \cite{ModuliSpaces}*{Thm. 7.20} carries over to our context verbatim.

\smallskip
\noindent\textit{$\cM^{(\varphi,\varphi+1]}$ admits a $\Theta$-stratification by Harder--Narasimhan filtrations:} 
\smallskip

One can use the function $\ell$ in \eqref{E:linear_function} and a norm on graded points to define a numerical invariant in the sense of \cite{halpernleistner2022structure}*{Def. 4.1.1} that determines a $\Theta$-stratification of the open substack of torsion-free objects $\cM^{(\phi,\phi+1)}$ \cite{halpernleistner2022structure}*{Thm. 6.5.3}. However, we are not aware of a real-valued numerical invariant that does the same for the full stack $\cM^{(\varphi,\varphi+1]}$.

Instead, we choose an arbitrary set of weights $w_0<\cdots<w_p$ for every possible choice of $z_0,\ldots,z_p \in \bH \cup \bR_{<0}$ with increasing phases, and we let $\overline{\cS} \subset \Filt(\cM)$ be the open and closed substack of $\bZ$-weighted filtrations with $Z(\gr_{w_i}(E_\bullet)) = z_i$ for all $i$ and $\gr_w(E_\bullet)=0$ for all other $w$. We equip $\overline{\cS}$ with the locally constant function $\mu : \overline{\cS} \to \bR$ taking $\mu(s) = \sum_w |Z(\gr_w(E_\bullet))|-|Z(v)|$. The set of Harder--Narasimhan filtrations $\cS \subset \overline{\cS}$ is characterized as the points that maximize $\mu$ in their fiber for the morphism $\rm{ev}_1 : \overline{\cS} \to \cM$ that forgets the filtration. With this setup \cite{halpernleistner2022structure}*{Thm. 2.2.2} gives criteria for $\cS$ to be the HN filtrations of a $\Theta$-stratification. (See \cite{halpernleistner2022structure}*{Rem. 2.1.3}.)
\begin{itemize}
    \item The \textit{HN property} amounts to the observation that $\rm{ev}_1 : \cS \to \cM^{(\varphi,\varphi+1]}$ is universally bijective, because of the existence and uniqueness of Harder--Narasimhan filtrations.\endnote{There is a slight difference in terminology here, because $\cS$ also contains one-step Harder--Narasimhan filtrations, which corresponds to the stratum of semistable points in $\cM^{(\varphi,\varphi+1]}$. In \cite{halpernleistner2022structure}, the semistable locus was not regarded as a stratum in the stratification, but there is no harm in doing so.}

    \item The \emph{HN specialization property} is automatic because $\cM^{(\varphi,\varphi+1]}$ is $\Theta$-complete \cite{halpernleistner2022structure}*{Rem 2.2.5}.
    
    \item The \emph{open strata} property is the fact that $\cS$ is closed under generization. In fact, $\cS \subset \Filt(\cM^{(\varphi,\varphi+1]})$ is open, because the condition that $\gr_i(E_\bullet)$ is semistable of phase $\phi_i$ is an open condition on $\Filt(\cM^{(\varphi,\varphi+1]})$.
    
    \item The \emph{HN consistency} condition amounts to the fact that given an HN filtration $E_n \subset \cdots \subset E_0$, one has $m(E_0) = m(\bigoplus_j \gr_j(E_\bullet))$.
 
    \item The \emph{strong HN boundedness} property states that the preimage under $\rm{ev}_1 : \cS \to \cM^{(\varphi,\varphi+1]}$ of any bounded set of finite type points is also bounded. \Cref{L:boundedness} implies an upper bound on the mass and length of HN filtration on any bounded subset $X$ of finite type points of $\cM^{(\varphi,\varphi+1]}$. The support property \Cref{D:locally_constant_stability_condition}(b) then implies that only finitely many classes $v \in \Lambda$ appear as HN factors of points in $X$.\endnote{We have already seen that there is a constant $N>0$ such that $m(x) \leq N$ for any $x\in X$. This implies an upper bound $|Z(E)| \leq m(x) \leq N$ for any $E$ appearing as the associated graded piece of the HN filtration of some $x \in X$. On the other hand, \Cref{D:locally_constant_stability_condition}(b) implies that $\lVert \ch(E) \rVert < |Z(E)|/C$ for any such semistable $E$.} By the HN consistency condition, the preimage $\rm{ev}_1^{-1}(X) \subset \cS$ is contained in the preimage of a bounded substack of $\Grad(\cM^{(\varphi,\varphi+1]})$ under the quasi-compact morphism $\gr : \Filt(\cM^{(\varphi,\varphi+1]}) \to \Grad(\cM^{(\varphi,\varphi+1]})$.
    
\end{itemize}

Note that for any HN filtration of length $>1$, the function $\ell$ of \eqref{E:linear_function} satisfies $\ell(E_\bullet)>0$, so $\ell$ is compatible with the $\Theta$-stratification of $\cM^{(\varphi,\varphi+1]}$.

\smallskip
\noindent\textit{$\cM^{(a,b]}$ satisfies the valuative criterion for universal closedness for any $a<b$:}\endnote{This is the existence part of the valuative criterion for properness.}
\smallskip

\Cref{D:locally_constant_stability_condition}(a) and \Cref{P:base_change_sluicing} provides a sluicing $\cP_R$ of width $1$ on $\cC_R$.\endnote{Either $k \to R$ factors through some discrete valuation ring of the kind in \Cref{D:locally_constant_stability_condition}(a), or it factors through a residue field of $k$. Either way \Cref{P:base_change_sluicing} induces a sluicing on $\cC_R$.} We choose a constant $0<w<1$ and first prove the claim for $I = (a,b]$ with $b-a\leq w$.

Let $\cT \subset \cP_R(I)$ be the kernel of the restriction functor $\cP_R(I) \to \cP_K(I)$. We claim that the inclusion $\cT \hookrightarrow \cP_R(I)$ admits a right adjoint. The quasi-abelian category $\cP_R(I)$ is contained in the heart of a noetherian $t$-structure, so for any $G \in \cP_R(I)$, any chain of strict monomorphisms $F_1 \hookrightarrow F_2 \hookrightarrow \cdots \hookrightarrow G$ must stabilize.\endnote{Strict monomorphisms are also monomorphisms in the containing abelian category.} Zorn's lemma then implies that there exists a maximal strict monomorphism $T \to G$ with $T \in \cT$, meaning for any strict monomorphism $T'' \to G$ with $T'' \in \cT$, if $T \to G$ factors through a morphism $T \to T''$, then this must be an isomorphism $T \cong T''$.

We claim that for any $T' \in \cT$, any morphism $T' \to G$ factors uniquely through $T \to G$. In other words, $T \to G$ is unique, and the inclusion functor $\cT \hookrightarrow \cP_R(I)$ has a right adjoint taking $G \mapsto T$. To show this, consider $F := \cofib(T \to G) \in \cP_R(I)$. We must show that the composition $f : T' \to F$ is zero, and because $f$ factors through $\ker(\coker(f))$, it suffices to show that $\ker(\coker(f)) \cong 0$. $\ker(\coker(f)) \to F$ is a strict monomorphism, and pulling back gives a commutative diagram with exact rows 
\[
\xymatrix{
    T \ar[r] \ar[d] & T'' \ar[r] \ar[d] & \ker(\coker(f)) \ar[d] \\
    T \ar[r] & G \ar[r] & F 
}.\]
Furthermore, the fact that restriction $\cP_R(I) \to \cP_K(I)$ is exact implies that $\ker(\coker(f)) \in \cT$ and hence $T'' \in \cT$. Now $T'' \to G$ is a strict monomorphism, so maximality of $T$ implies that $T \to T''$ is an isomorphism, and therefore $\ker(\coker(f)) \cong 0$, proving the claim.

Now consider $E \in \cP_K(I)$ for $I$ as above. \Cref{D:sluicing}(2) provides a noetherian $t$-structure on $\cC_R$, and $E \mapsto E \otimes_R K$ is automatically exact,\endnote{This is because $E \otimes_R K \in \Ind(\cC_R)$ is the colimit of the diagram $E \to E \to \cdots$, where all of the maps are multiplication by a uniformizer in $R$.} so \Cref{L:localization} implies one can extend $E$ to an $F \in \cC_R$. Letting $G := \tau_I(F) \in \cP_R(I)$, one still has $G \otimes K \cong E$.\endnote{The definition of a locally constant stability condition provides that $(-)\otimes K : \cC_R \to \cC_K$ is exact with respect to the sluicings.} Let $T \to G$ be the maximal strict monomorphism in $\cP_R(I)$ with $T \in \cT$, and let $\tilde{E} = \cofib(T \to G) \in \cP_R(I)$. Then $\tilde{E}$ admits no non-zero morphisms $T \to \tilde{E}$ with $T \in \cT$. In particular, if $\pi \in R$ is a uniformizer, then we must have vanishing
\[
    0 \cong \ker(\pi: \tilde{E} \to \tilde{E}) \cong \tau_{>a}(\fib(\pi : \tilde{E} \to \tilde{E})).
\]
It follows that if $\kappa = R/(\pi)$, then $\tilde{E} \otimes_R \kappa \cong \cofib(\pi : \tilde{E} \to \tilde{E}) \in \cP_R((a,a+1])$. So $\tilde{E}$ determines a morphism $\Spec(R) \to \cM^{(a,a+1]}$.

We now use the fact, established above, that $\cM^{(a,a+1]}$ has a $\Theta$-stratification by Harder--Narasim\-han filtrations. The semistable reduction theorem \cite{ModuliSpaces}*{Thm. 6.3} implies that, after possibly replacing $R$ with a totally ramified extension, there is another morphism $\Spec(R) \to \cM^{(a,a+1]}$, corresponding to $F \in \cC_R$, such that $F \otimes K \cong \tilde{E} \otimes K \cong E$ in $\cC_K$ and the HN filtration of $E$ extends to a filtration of $F$ over $R$ that also restricts to an HN filtration of $F_\kappa$. In particular, $\phi^+(F_{\kappa}) \leq b$, so $F$ defines an $R$-point of $\cM^{(a,b]}$ extending $E$.

Finally, suppose $E \in \cP_K((a,b])$ for an arbitrary $a<b$, and let $I_1 = (a,b-w]$ and $I_2=(b-w,b]$. Suppose that for $i=1,2$ we have $E_i \in \cP_R(I_i)$ such that $K \otimes E_i \cong \tau_{I_i}(E)$ and $E_i|_{\kappa} \in \cP_\kappa(I_i)$. Then $E$ is the fiber of a morphism $E_1\otimes K \to E_2[1] \otimes K$ in $\cC_K$, which corresponds to a morphism in $\Ind(\cC_R)$
\[
E_1 \to E_2[1]\otimes K \cong \colim(E_2[1] \xrightarrow{\pi} E_2[1] \xrightarrow{\pi} \cdots).
\]
Because $E_1$ is compact, this factors through a morphism $E_1 \to E_2[1]$, and $\tilde{E}:= \fib(E_1 \to E_2[1])$ has the property that $E \cong K \otimes_R \tilde{E}$ and $\kappa \otimes \tilde{E} \in \cP_\kappa(I)$. By induction, it follows that for any interval $I$ and $E \in \cP_K(I)$, there is a $\tilde{E} \in \cC_R$ such that $E \cong \tilde{E} \otimes_R K$ and $\tilde{E} \otimes_R \kappa \in \cP_\kappa(I)$.

\smallskip
\noindent\textit{$\cM_v^{\rm {ss}}$ admits a proper good moduli space:}
\smallskip

This follows from \cite{ModuliSpaces}*{Thm. C}, whose hypotheses we have verified above. $\Theta$-reductivity and $S$-completeness is shown in \Cref{L:Theta_reductivity_S_completeness}. The function $\ell$ of \eqref{E:linear_function} does not explicitly come from a class in $H^2(\cM^{(\varphi,\varphi+1]};\bR)$, but the proof of \cite{ModuliSpaces}*{Thm. C} applies more generally to any linear function on the component fan.

\end{proof}

\begin{proof}[Proof of \Cref{T:moduli_spaces}]

\noindent{\textit{Proof that $(1) \Rightarrow (2)$:}}
\smallskip

Consider the map of higher derived stacks $\pi : \cM \to \Perf$ that acts on $R$-points by
\[
(E\in \cC_R) \mapsto \RHom_R (G_R, E) \in \Perf(R).
\]
In the course of proving the algebraicity of $\cM$, \cite{ToenVaquie}*{Prop 3.13} shows that $\pi$ is of finite presentation, and in particular $\infty$-quasi-compact. Therefore, the preimage of any bounded subset of $|\Perf|$ is bounded in $|\cM|$.

Now consider a finite type point in $\cM$, corresponding to a finite type field $K$ over $k$ and an object $E \in \cC_K$. The bounds $\phi^+_\sigma(E) < C$ and $\phi^-_\sigma(E)>-C$ imply the existence of an $N$, independent of $E$, such that $H^i(\RHom_K(G_K,E))$ vanishes for $i \notin [-N,N]$. It follows from the Hom-bound in $(1)$ and the bound $m_\sigma(E) \leq C$ that $\pi(E)$ lies in the substack of $\Perf$ of points whose homology is bounded in the range $[-N,N]$ and whose homology dimension is bounded above by $f(C)$. This is an $\infty$-quasi-compact substack of $\Perf$ \cite{ToenVaquie}*{Prop. 3.20}, which verifies the condition $(2)$.

\smallskip
\noindent\textit{Proof that $(2) \Rightarrow (3)$:}
\smallskip

This is part of \Cref{P:existence}.

\medskip
\noindent{\textit{Proof that $(3) \Rightarrow (2)$, at all points instead of just finite type points:}}
\medskip

For any $v \in \Lambda$, $\cM^{\rm {ss}}_v$ must contain a finite type point if it is nonempty. The support property \Cref{D:locally_constant_stability_condition}(b) then guarantees that for any $C$ there are finitely many classes $v \in \Lambda$ such that $|Z(v)|<C$ and $\cM_v^{\rm{ss}} \neq \emptyset$. It follows that the set of semistable points of $|\cM|$ with mass $<C$ and phase in any finite interval $(a,b]$ is bounded. Also, because $|Z(v)|$ has a positive minimum on this set, it follows that there is an $N$ such that the HN filtration of any object $E$ with $m_\sigma(E)<C$ has length $<N$. The subset of $|\Filt^{\rm{un}}(\cM)|$ parameterizing HN filtrations of mass $<C$ with length $n$ and phases in $(a,b]$ is contained in the preimage under $\gr : \Filt(\cM)_n \to \cM^n$ of the subset of $|\cM^n|$ parameterizing $n$-tuples of semistable objects with increasing phase in $(a,b]$ and mass $<C$. It follows that this subset of $|\Filt^{\rm{un}}(\cM)|$ is bounded, and therefore its image under the forgetful map $\Filt^{\rm{un}}(\cM) \to \cM$ is bounded.

\medskip
\noindent{\textit{Proof that $(2) \Rightarrow (1)$, with both conditions at all points, not just finite type points:}}
\medskip

Consider the function
\[
f(x) := \sup \left\{\dim \Hom_K(G_K,E) \left| \begin{array}{c} K \text{ a field over }k,\\ E \in \cM(K) \text{ s.t. } m_\sigma(E) \leq x \end{array} \right. \right\}.
\]
We claim that there is an $a<b$ such that $f(x)$ is unchanged if we only take the supremum over $E \in \cP_K((a,b])$. Let $S : \cC \to \cC$ be the Serre functor, i.e., $\RHom_k(E,F) \cong \RHom_k(F,S(E))^\ast$ for all $E,F$. Because the formation of $\RHom_k(E,F) \in \Perf(k)$ commutes with base change, we have $S(E)_K \cong S_{K}(E_K)$ for any field $K$ over $k$ and any $E \in \cC$, where $S_K$ denotes the Serre functor on $\cC_{K}$. Because $G$ generates $\cC$, the supremum \eqref{E:bound_on_generator} is also finite if $G$ is replaced with any $E \in \cC$. In particular, there is an $a<b$ such that $G_K \in (\cC_K)_{> a}$ and $S_K(G_K) \in (\cC_K)_{\leq b}$ for all fields $K$ over $k$.\endnote{The supremum is only stated for residue fields of $k$, but any homomorphism $k \to K$ to a field $K$ factors uniquely through a residue field $k \to \kappa$, and the pullback functor $\cC_{\kappa} \to \cC_K$ is exact, so the supremum is unchanged if one considers all fields over $k$.} The first bound implies that $\Hom_K(G_K,E) = \Hom_K(G_K,\tau_{>a}(E))$ for any $K$ and $E \in \cC_K$, and the second bound implies that 
\[
    \Hom_K(G_K,E) \cong \Hom_K(E,S(G_K))^\ast \cong \Hom_K(\tau_{\leq b}(E),S(G_K))^\ast \cong \Hom_K(G_K,\tau_{\leq b}(E)).
\]
By condition $(2)$, the set of $E$ appearing in the supremum is now bounded, so there is a finite type affine $k$-scheme $T$ and an $\cE \in \cC_T$ such that
\[
f(x) = \sup\{\dim \Hom_{\kappa(t)}(G_{\kappa(t)}, \cE_t) | t \in T \text{ finite type and }m_\sigma(\cE_t)\leq x\},
\]
where $\kappa(t)$ is the residue field of $t \in T$. If $\cH := \RHom_T(G_T,\cE) \in \Perf(T)$, then 
\[
    \Hom_{\kappa(t)} (G_{\kappa(t)}, \cE_t) \cong H^0(\cH_t)
\]
for all $t$. The semicontinuity theorem then implies that the supremum in the second description of $f(x)$ is achieved at some point of $T$, and in particular $f(x)$ is finite. This $f(x)$ satisfies the condition $(1)$ by construction.
\end{proof}

\section{Main definitions}
\label{S:maindefinitions}

In this section we introduce two structures on a stable dg-category: multi-scale decomp\-ositions and augmented stability conditions.

\subsection{Multi-scale lines}

Recall that a rooted tree has a partial order ``$\subseteq$'' on vertices in which $u \subseteq v$ if $v$ lies on the unique simple path connecting $u$ to the root.\endnote{This is non-standard notation for this partial order, but it is meant to distinguish it from two additional order relations on vertices to be defined below. In \Cref{D:multi-scale_decomposition}, we will associate a category $\cC_{\leq v}$ to every vertex of a rooted graph, in such a way that $\cC_{\leq u} \subseteq \cC_{\leq v}$ when $u\subseteq v$ in this sense, justifying the notation.} The least upper bound $u \vee v$ is the first vertex where the simple paths from $u$ and $v$ to the root vertex meet.

\begin{defn}
\label{D:multi-scaleline}
A \emph{multi-scale line} is a tuple $(\Sigma,p_\infty,\preceq,\omega_\bullet)$, where
\begin{enumerate}
    \item $\Sigma$ is a connected nodal complex genus $0$ curve\endnote{That is, every component of $\Sigma$ is isomorphic to $\bP^1$, and the dual graph of $\Sigma$ is a tree.} with marked smooth point $p_\infty$. We let $\Gamma(\Sigma) = (V(\Sigma),E(\Sigma))$ denote the dual graph of $\Sigma$, and let $v_0 \in V(\Sigma)$ denote the vertex corresponding to the component containing $p_\infty$, which we regard as the root.\medskip
    
    \item $\preceq$ is a total preorder on the set $V(\Sigma)$ of irreducible components of $\Sigma$ such that:\medskip
    \begin{enumerate}
        \item $\forall v \in V(\Sigma) \setminus \{v_0\}$, $\exists ! w\in V(\Sigma)$ that is adjacent to $v$ with $v\prec w$. The edge from $v$ to $w$ is called the edge \emph{ascending from} $v$; and \medskip
        \item $v \sim w$, meaning $v \preceq w$ and $w \preceq v$, for any two vertices $v,w \in V(\Sigma)$ that are minimal with respect to $\subseteq$. We call such vertices \emph{terminal}. \medskip
    \end{enumerate}
    
    \item The root vertex of $\Gamma(\Sigma)$ has valence zero or at least two, and all other non-terminal vertices have valence at least three.\medskip
    
    \item For every $v \in V(\Sigma)$, $\omega_v$ is a meromorphic differential on $\Sigma_v$ with a pole of order $2$ at the node corresponding to the ascending edge from $v$ when $v \neq v_0$ or the marked point $p_\infty$ when $v=v_0$, and no other zeros or poles.
\end{enumerate}

A \emph{complex projective (resp. real oriented) isomorphism} of multi-scale lines is an iso\-morphism of nodal curves $f : \Sigma \to \Sigma'$ that preserves the respective preorderings and marked points, and such that $f^\ast(\omega'_v) = c_v \omega_v$ for some constant $c_v \in \bC^\ast$ (resp. $\bR_{>0}$), where $c_v = c_w$ whenever $v \sim w$ and $c_v=1$ if $v$ is minimal.\endnote{It follows from condition (3) of \Cref{D:multi-scaleline} that the $\omega_v$ are of the form $\lambda \cdot dz$ for any chosen affine coordinate $z:\Sigma_v\setminus \{n_v\}\to \bC$, where $n_v$ is the ascending node of $v$ and $\lambda \in\bC^*$. A complex projective isomorphism of multi-scale lines is an isomorphism of nodal curves that does not preserve the data of any $\omega_v$ for $v$ non-minimal, but informally preserves $\omega_v/\omega_{v'}$ for any $v\sim v'$.}
\end{defn}

Condition (2) implies that $u \subseteq v \Rightarrow u \preceq v$, that \emph{every} edge of $\Gamma$ is the ascending edge for one of its vertices, and that $v_0$ is the unique maximum with respect to the preorder $\preceq$.\endnote{This claim is verified as follows. As $V(\Gamma)$ equipped with $\preceq$ is a poset, one can define $\rm{height}(v)$ to be the maximum length of chains $v\supset v_1\supset \cdots$. We prove $u\subseteq v\Rightarrow u \preceq v$ by inducting on the height of $v$. For height $0$, $v$ is minimal with respect to $\subseteq$ and the claim follows from (2)(b). Suppose the claim has been proven for height $<k$. Let $w$ of height $k$ be given. Let $v\subset w$ be given such that $v$ and $w$ are connected by an edge. $v$ is of height $\le k-1$ so $u\subseteq v\Rightarrow u\preceq v$. Since $\Gamma$ is a tree, $w$ is attached to $v$ via the edge ascending from $v$, i.e. $v\prec w$. Consequently, $u\preceq w$ for all $u\subseteq w$.

Since $u\subseteq v_0$ for all $u \in V(\Gamma)$, it now follows that $v_0$ is a maximum element with respect to $\preceq$. Lastly, let $e\in E(\Gamma)$ be given connecting vertices $v$ and $w$. Then $v\subseteq w$ or $w\subseteq v$. If $v\subseteq w$, $v\preceq w$ and so $e$ is the ascending edge from $v$ to $w$. 
} Also, minimal vertices form an equivalence class under $\sim$.

\begin{defn}[Normalized periods]
\label{D:normalizedperiods}
Consider a multi-scale line $(\Sigma,p_\infty,\preceq,\omega_\bullet)$ with dual graph $\Gamma$. Given $v,v' \in V(\Gamma)$ with neither $v\subseteq v'$ nor $v\supseteq v'$, we define the \emph{normalized period} $\mathfrak{p}(v,v')$ as follows: If $n, n' \in \Sigma_{v\vee v'}$ denote the nodes from which $v$ and $v'$ descend, respectively, then
\[
\mathfrak{p}(v,v') := \frac{\int_{n}^{n'} \omega_{v\vee v'}}{\lvert \int_{n}^{n'} \omega_{v\vee v'} \rvert},
\]
where the integral is along any path $\gamma : [0,1] \to \Sigma_{v\vee v'}$ avoiding the ascending node (or $p_\infty$ when $v\vee v'$ is the root) such that $\gamma(0) = n$ and $\gamma(1) = n'$.
\end{defn}

Note that $\mathfrak{p}(v,v')$ only depends on the real oriented isomorphism class of $\Sigma$, and that $\mathfrak{p}(u,v)$ is always defined for distinct $u\sim v$. In addition to the partial order ``$\subseteq$'' and the total preorder ``$\preceq$,'' the normalized period function allows one to define a parameterized family of partial orders on the terminal components of $\Sigma$.

\begin{defn}[Directional ordering]
Let $\Sigma$ be a multi-scale line, $\zeta \in S^1$, $t \in [0,\infty]$, and $v,w \in V(\Sigma)$. If $t<\infty$, we say $v \leq_{\zeta, t} w$ when either 1) $v=w$, or 2) $v \nsubseteq w$ and $w \nsubseteq v$ and 
\[
    \Re \left((1 + \frac{it}{\pi}) \frac{1}{\zeta} \mathfrak{p}(v,w)\right) > 0 \quad \text{and} \quad \Re \left((1 - \frac{it}{\pi}) \frac{1}{\zeta} \mathfrak{p}(v,w)\right) > 0.
\]
For $t=\infty$, condition (2) is replaced with $\frac{1}{\zeta} \mathfrak{p}(v,w)=1$.
\end{defn}

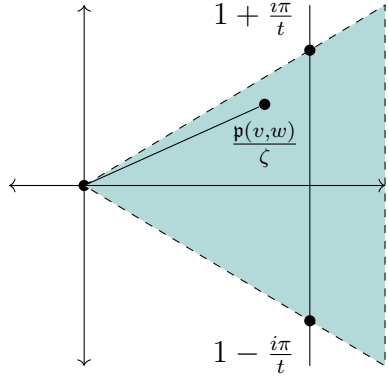
\begin{figure}[ht]
    \centering

\begin{tikzpicture}[xscale=.5,yscale=.3]
		\node [] (0) at (0, 8) {};
		\node [] (1) at (0, -8) {};
		\node [] (2) at (-2, 0) {};
		\node [] (3) at (8, 0) {};
		\node [] (4) at (6, 8) {};
		\node [] (5) at (6, -8) {};
          \node [circle,fill,inner sep=1.5pt] (7) at (0, 0) {};
		\node [] (8) at (8, 8) {};
		\node [] (10) at (8, -8) {};
		\draw [style=dashed, fill=teal!30] (7.center) -- (8.center) -- (10.center) -- cycle;		
         
         \node [circle,fill,inner sep=1.5pt, label=north west:$1+\frac{i\pi}{t}$] (6) at (6, 6) {};
          \node [circle,fill,inner sep=1.5pt, label=south west:$1-\frac{i\pi}{t}$] (9) at (6, -6) {};
		\node [circle,fill,inner sep=1.5pt, label=south:$\frac{\mathfrak{p}(v,w)}{\zeta}$] (11) at (4.8,3.6) {};
            
            \draw [<->] (0.center) to (1.center);
		\draw [<->](2.center) to (3.center);
		\draw [] (4.center) to (5.center);
		
		\draw (7.center) to (11);
\end{tikzpicture}
    
    \caption{This illustrates the partial order on terminal vertices of a multi-scale line. We say $v \leq_{\zeta,t} w$ if the point $\frac{\mathfrak{p}(v,w)}{\zeta}$ on the unit circle lies in the interior of the shaded cone shown.}
    \label{fig:partial_order_condition}
\end{figure}

\begin{lem}
\label{L:partialorder}
For any $\zeta,t$, $\leq_{\zeta,t}$ is a partial order on $V(\Sigma)$. Furthermore, $v \leq_{\zeta,t} w$ implies $v \leq_{\zeta,s} w$ for any $s < t$.
\end{lem}

\begin{proof}
The condition $v \leq_{\zeta,t} w$ is described graphically in \Cref{fig:partial_order_condition}. The graph shows that $v \leq_{\zeta,t} w$ implies $v \leq_{\zeta,s} w$ for $s<t$, and also that $v \leq_{\zeta,\infty} w$ is equivalent to $v \leq_{\zeta,t} w$ holding simultaneously for all finite $t$. To prove the lemma, then, we can assume $t$ is finite. We will show that if $v$ is a vertex such that no two of $\{u,v,w\}$ are related by $\subseteq$, then $u <_{\zeta,t} v$ and $v<_{\zeta,t}w$ implies $u \leq_{\zeta,t} w$, leaving the verification of reflexivity and anti-symmetry to the reader.\endnote{Reflexivity is immediate from the definition. Antisymmetry is verified as follows: If $v\le_{\zeta,t}w$ and $w\le_{\zeta,t} v$ then either $v = w$ or $v\not\subseteq w$ and $w\not\subseteq v$. Consider the nodes $n_v$ and $n_w$ corresponding to $v$ and $w$ on $\Sigma_{v\vee w}$. By \Cref{fig:partial_order_tree}, one can see that for $v\le_{\zeta,t}w$ and $w\le_{\zeta,t} v$ to hold, one must have $n_v = n_w$. However, this implies $v = v\vee w = w$.}

If one considers the subtree of $\Gamma(\Sigma)$ spanned by $u$, $w$, and the root node $r$, then the four possible places that the path from $r$ to $v$ splits from this subtree are represented by the vertices $v_1,\ldots,v_4$ in \Cref{fig:partial_order_tree}, so we can consider each case in turn. For $v_1$, $u \vee w = v_1 \vee w$ and $\mathfrak{p}(u,w) = \mathfrak{p}(v_1,w)$, which implies the claim. Likewise, the equality $\mathfrak{p}(u,w) = \mathfrak{p}(u,v_3)$ implies the claim for $v_3$. For $v_4$, one has $\mathfrak{p}(u,v_4) = - \mathfrak{p}(v_4,w)$, so this case is ruled out by the hypotheses that $u <_{\zeta,t} v_4$ and $v_4<_{\zeta,t}w$.

Finally, in the case $v_2$, the paths from $r$ to each of $u$,$v_2$, and $w$ all meet at the same vertex $u \vee w$. The three nodes at which the component $\Sigma_{u\vee w}$ attaches to the components leading to $u,v_2,$ and $w$ correspond to three complex numbers $x,y,z \in \bC$ respectively, defined up to simultaneous translation. The condition $u <_{\zeta,t} v_2$ is equivalent to the condition that the vector $\frac{1}{\zeta} \overrightarrow{xy}$ lies in the shaded cone of \Cref{fig:partial_order_condition}, and likewise for $v_2 <_{\zeta,t} w$ and $u \leq_{\zeta,t} w$. The claim follows from the convexity of this cone, and the fact that $\overrightarrow{xz}=\overrightarrow{xy}+\overrightarrow{yz}$.
\end{proof}

\begin{figure}[h]
    \centering
    \begin{tikzpicture}
		\node [circle,fill,inner sep=1.5pt,label=below:$u$] (0) at (-12, -1) {};
		\node [circle,fill,inner sep=1.5pt,label=left:$u \vee w$] (1) at (-9, 2) {};
		\node [circle,fill,inner sep=1.5pt,label=below:$w$] (2) at (-6, -1) {};
		\node [circle,fill,inner sep=1.5pt,] (3) at (-8, 3) {};
            \node [circle,fill,inner sep=1.5pt,label=above:$r$] (r) at (-8, 4) {};
		\node [circle,fill,inner sep=1.5pt,] (4) at (-11, 0) {};
		\node [circle,fill,inner sep=1.5pt,label=below:$v_1$] (5) at (-10, -1) {};
		\node [circle,fill,inner sep=1.5pt,label=below:$v_3$] (6) at (-8, -1) {};
		\node [circle,fill,inner sep=1.5pt,label=below:$v_4$] (7) at (-4, -1) {};
		\node [circle,fill,inner sep=1.5pt,] (8) at (-7, 0) {};
		\node [circle,fill,inner sep=1.5pt,label=below:$v_2$] (9) at (-9, -1) {};
		\draw (0.center) to (4.center);
            \draw (r.center) to (3.center);
		\draw (4.center) to (5.center);
		\draw (4.center) to (1.center);
		\draw (1.center) to (3.center);
		\draw (3.center) to (7.center);
		\draw (1.center) to (2.center);
		\draw (6.center) to (8.center);
		\draw (1.center) to (9.center);
    \end{tikzpicture}

    \caption{Given distinct terminal components $u$ and $w$ of a multi-scale line, this diagram shows all of the possible ways that the path from the root $r$ to a third terminal component $v$ can meet the sub-tree spanned by $u$, $w$, and $r$.}
    \label{fig:partial_order_tree}
\end{figure}
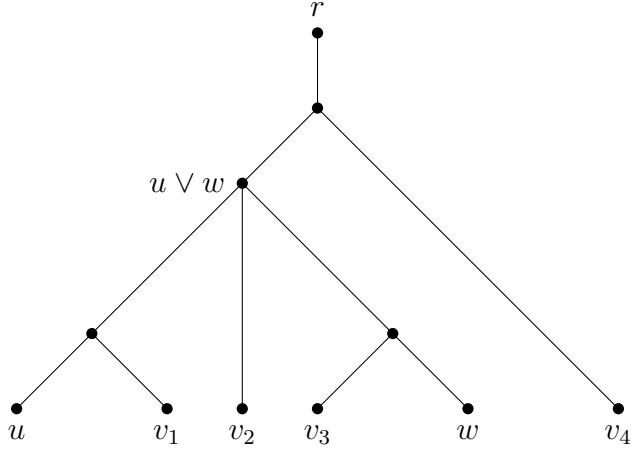

\begin{rem}[$\GL_2(\bR)$-action]
\label{R:GL2R}
Specifying a multi-scale line $\Sigma$ is equivalent to specifying the pre-ordered rooted tree $\Gamma(\Sigma)$ along with $(x_1,\ldots,x_{n_v}) \in \rm{Conf}_{n_v}(\bC)$ for each non-terminal $v \in V(\Sigma)$, where $n_v$ is the number of descending edges from $v$. Two collections of configurations are related by a real-oriented isomorphism if $\tilde{x}_i = a_v x_i + b_v$ for constants $b_v \in \bC$  and $a_v \in \bR_{>0}$, with $a_v$ only depending on the class of $v$ in $V(\Sigma)/{\sim}$. Given $g \in \GL_2(\bR)$, we let $g \cdot \Sigma$ be the multi-scale line with the same dual graph $\Gamma(\Sigma)$, and specified by the configurations $(gx_1,\ldots,gx_{n_v}) \in \rm{Conf}_{n_v}(\bC) = \rm{Conf}_{n_v}(\bR^2)$. The real-oriented isomorphism class of $g\cdot \Sigma$ only depends on the real-oriented isomorphism class of $\Sigma$.
\end{rem}

\subsection{Multi-scale decompositions}

\begin{defn}\label{D:multi-scale_decomposition}
A \emph{multi-scale decomposition} of a stable dg-category $\cC$ consists of a multi-scale line $(\Sigma,p_\infty,\preceq,\omega_\bullet)$ and a nonzero thick stable subcategory $\cC_{\leq v}$ for every terminal component $v$ of $\Sigma$ such that: 
\begin{enumerate}
    \item $\cC_{\leq v} \cap \cC_{\leq w}$ is generated as a triangulated category by the sub\-categories $\cC_{\leq u}$ for all $u$ such that $u \leq_{1,\infty} v$ and $u \leq_{1,\infty} w$.
    
    \item If $w \leq_{i,0} v$, i.e. $\mathfrak{p}(w,v)\in \mathbb{H}$, then $\Hom(\cC_{\leq v}, \cC_{\leq w}) = 0$ in the quotient category $\cC / (\cC_{\leq v} \cap \cC_{\leq w})$.
    \item For any $v \in V(\Sigma)$, the triangulated categories $\cC_{\leq v} := \Span\{\cC_{\leq u} | u\in V(\Sigma)_{\rm{term}}, u \subseteq v\}$ and $\cC_{<v}:=\Span\{\cC_{\leq u} |u\in V(\Sigma)_{\rm{term}}, u <_{1,\infty} v\}$ are thick.
    \item $\gr_v(\cC_\bullet) := \cC_{\leq v} / \cC_{<v} \neq 0$ for all $v \in V(\Sigma)$, and $\cC_{\leq v}=\cC$ for the root vertex $v \in V(\Sigma)$.  
\end{enumerate}
We will typically denote such a multi-scale decomposition by $\cC = \langle \cC_\bullet \rangle_\Sigma$, leaving the additional data $p_\infty, \preceq, \omega_\bullet$ implicit unless they are needed.
\end{defn}

Condition (1) implies that $\cC_{\leq v} \subset \cC_{\leq w}$ whenever $v \leq_{1,\infty} w$. On the other hand, if $\Im(\mathfrak{p}(v,w)) \neq 0$ and $v \vee w$ is the least upper bound in the dual graph of $\Sigma$, then the terminal vertices $u$ appearing in (1) are precisely the terminal vertices that do \emph{not} lie below $v \vee w$ and such that $u \leq_{1,\infty} v \vee w$.\endnote{First note that $u$ can not lie below $v \vee w$. If it did, then in the configurations of points in $\bC$ corresponding to the descending edges from $v \vee w$, $u$ would correspond to a point in $\bC$ that differs by a real number from both of the points corresponding to $v$ and $w$, but this is impossible if $\Im(\mathfrak{p}(v,w)) \neq 0$. 

Next, the definition of $\leq_{1,\infty}$ implies that if $u$ does not lie below $v \vee w$, then $u \leq_{1,\infty} v$ if and only if $u \leq_{1,\infty} v \vee w$ if and only if $u \leq_{1,\infty} w$.}

The key property of the categories $\cC_{\leq v}$ for non-terminal vertices is that if $u \subseteq v$ in $\Gamma(\Sigma)$, then one has
\begin{equation} \label{E:vertex_category_containment}
    \cC_{< v} \subset \cC_{< u} \subset \cC_{\leq u} \subset \cC_{\leq v}.
\end{equation}
This follows immediately from the definition, because for $w\in V(\Sigma)_{\rm{term}}$ that does not lie below $v$, one has $w \leq_{1,\infty} u$ if and only if $w \leq_{1,\infty} v$.

\begin{ex}
If $\Im(\mathfrak{p}(u,v)) \neq 0$ for all terminal components $u$ and $v$, then $\leq_{i,0}$ is a total ordering on terminal components. Let us index the terminal components $v_1,\ldots,v_n$ such that $v_i <_{i,0} v_j$ for all $i<j$. Then the subcategories $\cC_{\leq v_i}$ have only the zero object in common, and \Cref{D:multi-scale_decomposition} is equivalent to giving a semiorthogonal decomposition $\cC = \langle \cC_{\leq v_1},\ldots,\cC_{\leq v_n} \rangle,$ along with the data of the multi-scale line $\Sigma$.
\end{ex}

\begin{ex}
If $\Im(\mathfrak{p}(u,v))=0$ for all terminal components $u$ and $v$, then $\leq_{1,0}$ is a total ordering on terminal components. In this case \Cref{D:multi-scale_decomposition}(2) is vacuous, and the multi-scale decomposition is simply a filtration by thick stable subcategories $0\subsetneq \cC_{\leq v_1} \subsetneq \cC_{\leq v_2} \subsetneq \cdots \subsetneq \cC_{\leq v_n} = \cC$, along with the data of the multi-scale line $\Sigma$.
\end{ex}

The notion of multi-scale decomposition is recursive via the following:

\begin{const}[Descendent decomposition]\label{const:descendent}
For any component $\Sigma_v \subset \Sigma$, we can obtain a new multi-scale line $\Sigma_{\subseteq v} = \bigcup_{u \subseteq v} \Sigma_u \subset \Sigma$, where the new marked point $p'_\infty \in \Sigma_{\subseteq v}$ is the point that previously connected this sub-curve to the rest of $\Sigma$. The dual graph of $\Sigma_{\subseteq v}$ is the subgraph lying below $v \in V(\Sigma)$, and $v$ is the root vertex. One then has a canonical multi-scale decomposition of $\gr_v(\cC_\bullet)$ indexed by the components of $\Sigma_{\subseteq v}$, where for every terminal component $u$ of $\Sigma_{\subseteq v}$,
\[
	\gr_v(\cC_\bullet)_{\leq u} := \cC_{\leq u} / \cC_{<v}
\]
is the essential image in $\gr_v(\cC_\bullet)$ of $\cC_{\leq u}$, which is well-defined by \eqref{E:vertex_category_containment}. For notational simplicity, we write $\cC_{\le u}^v = \gr_v(\cC_\bullet)_{\le u}$. One can show that $\gr_v(\cC_\bullet) = \langle \cC_\bullet^v \rangle_{\Sigma_{\subseteq v}}$ is a multi-scale decomposition of $\gr_v(\cC_\bullet)$.\endnote{By \eqref{E:vertex_category_containment}, if $w\subseteq v$ then $\cC_{<v}\subset \cC_{\le w}$ so it makes sense to consider the quotient $\cC_{\le w}/\cC_{<v}$. By \cite{Verdierquotient}*{Ch. II, Prop. 2.3.1}, if $w\subseteq v$, then $\cC_{\le w}^v$ is a thick subcategory of $\gr_v(\cC_\bullet)$ if and only if $\cC_{\le w}$ is thick in $\cC$, which implies the result for $w$ terminal and also gives condition (3). 

Condition (1) follows from the fact that the map taking a subcategory of $\cC_{\leq v} / \cC_{<v}$ to its preimage in $\cC_{\leq v}$ establishes a bijection between full triangulated subcategories of $\cC_{\leq v} / \cC_{<v}$ and full triangulated subcategories of $\cC_{\leq v}$ that contain $\cC_{<v}$, and this bijection is compatible with intersection of subcategories. The first claim can be seem from the universal property of a quotient dg-category, and the second is immediate --- see \cite{Verdierquotient}*{Ch. II, Prop. 2.3.1}.

Next, we prove (2). There is a canonical equivalence $\cC/\cC_{\le u}\cap \cC_{\le w} \to \gr_v(\cC_\bullet)/\cC_{\le u}^v\cap \cC_{\le w}^v$ by \cite{Verdierquotient}*{Ch. II, Prop. 2.3.1}. Under this equivalence, the essential images of $\cC_{\le u}$ and $\cC_{\le w}$ in $\cC/\cC_{\le u} \cap \cC_{\le w}$ are taken to $\cC_{\le u}^v$ and $\cC_{\le w}^v$, respectively. In particular, (2) for $\langle \cC_\bullet\rangle_{\Sigma}$ implies (2) for the descendent decomposition.

Condition (4) follows from the fact that for any $w\subseteq v$ the natural functor $\cC_{\le w}/\cC_{<w} \to \cC_{\le w}^v/\cC_{\le w}^v$ is an equivalence.} 
\end{const}

\begin{const}[Large scale semiorthogonal decomposition]
\label{const:SODfrommulti-scale}
Consider a multi-scale decomp\-osition $\cC = \langle \cC_\bullet\rangle_\Sigma$ and let $u_1,\ldots, u_n \in V(\Sigma)$ denote the components lying immediately below the root vertex. We order these components uniquely so that for all $i<j$ one has $\mathfrak{p}(u_i,u_j) \in \bH\cup \bR_{>0}$. There is then a unique set of indices $i_0=0<i_1<i_2<\cdots<i_k = n$ such that $\mathfrak{p}(u_i,u_j)=1$ for all $i_{\ast-1} < i < j \leq i_{\ast}$, and $\Im(\mathfrak{p}(u_i,u_j))>0$ for any other $i<j$. 

It follows directly from \Cref{D:multi-scale_decomposition} that $\cC_{\leq u_i} \subset \cC_{\leq u_j}$ whenever $i_{\ast-1} < i < j \leq i_{\ast}$, that $\Hom(\cC_{\leq u_j}, \cC_{\leq u_i}) = 0$ for any other $i<j$, and that the subcategories $\cC_{\leq u_{i_a}}$ for $a=1,\ldots,k$ generate $\cC$ as a triangulated category. Thus we have a semiorthogonal decomposition $\cC = \langle \cC_{\leq u_{i_1}},\ldots,\cC_{\leq u_{i_k}}\rangle$, which we refer to as the \emph{large scale semiorthogonal decomposition}.
\end{const}

\begin{ex}
\label{ex:twolevelmulti-scaledecomp}
If $\Sigma$ has exactly two levels, a multi-scale decomposition $\cC = \langle \cC_\bullet\rangle_\Sigma$ is a semi\-orthogonal decomposition $\cC = \langle \cC_{\le u_{i_1}},\ldots, \cC_{\le u_{i_k}}\rangle$ as in \Cref{const:SODfrommulti-scale} such that each $\cC_{\le u_{i_p}}$ is filtered by the categories $\cC_{\le u_a}$ for $i_{p-1}<a\le i_p$. Given such an $a$, $\gr_{u_a}(\cC_\bullet) = \cC_{\le u_a}/\cC_{<u_a}$. So, a multi-scale decomposition of this type is a semiorthogonal decomposition of $\cC$ along with a filtration of each semiorthogonal factor. See Figure \ref{F:2level} for a visualization:
\end{ex}

\begin{figure}[h]
\begin{center}
	\begin{tikzpicture}
	\draw[black] (1,0) arc (0:-180:1cm and 0.5cm);
	\draw[black,dashed] (1,0) arc (0:180:1cm and 0.5cm);
	\draw (0,0) circle (1cm);
	\filldraw[black] (0,1) circle  (0.05); 
	\shade[ball color=gray!10] (0,0) circle (1cm);

    \draw[black] (.5,-1.5) arc (0:-180:.5cm and 0.25cm);
	\draw[black,dashed] (.5,-1.5) arc (0:180:.5cm and 0.25cm);
	\draw (0,-1.5) circle (.5cm);
	\shade[ball color=gray!10] (0,-1.5) circle (.5cm);

    \draw (1.05,-1.05) circle (.5cm);
    
    \draw[black] (1.55,-1.05) arc (0:-180:.5cm and .25 cm); 
    \draw[black,dashed] (1.55,-1.05) arc (0:180:.5cm and .25 cm);
    \shade[ball color=gray!10] (1.05,-1.05) circle (.5cm);

    \draw (-1.05,-1.05) circle (.5cm);
    
    \draw[black] (-.55,-1.05) arc (0:-180:.5cm and 0.25cm);
    \draw[black,dashed] (-.55,-1.05) arc (0:180:.5cm and 0.25cm);
    \shade[ball color=gray!10] (-1.05,-1.05) circle (.5cm);


    \draw node at (1,1) {$\Sigma_{\rm{root}}$};
    \draw node at (.65, -1.65) {\tiny $\iddots$};
    \draw node at (0,1.3) {$p_\infty$};
    \draw node at (-1.5,-1.6) {$\Sigma_{v_1}$};
    \draw node at (-.5,-2.1) {$\Sigma_{v_2}$};
    \draw node at (1.8,-1.6) {$\Sigma_{v_k}$};
    \draw node at (-.45,-.7) {\tiny{\textcolor{red}{$n_1$}}};
    \draw node at (0,-.8) {\tiny{\textcolor{red}{$n_2$}}};
    \draw node at (.45,-.7) {\tiny{\textcolor{red}{$n_k$}}};

    \draw node at (5,0) {$\Gamma(\Sigma) \:=$ }; 
    \filldraw[black] (6.75,.5) circle (0.05);
    \filldraw[black] (6,-.5) circle (0.05);
    \filldraw[black] (6.5,-.5) circle (0.05);
    \filldraw[black] (7.5,-.5) circle (0.05);
    \draw (6.75,.5) -- (6,-.5);
    \draw (6.75,.5) -- (6.5,-.5);
    \draw (6.75,.5) -- (7.5,-.5);
    \draw node at (7,-.5) {$\cdots$}; 
    \draw node at (6.75,.7) {\tiny{root}};
    \draw node at (6,-.7) {\footnotesize{$v_1$}};
    \draw node at (6.5,-.7) {\footnotesize{$v_2$}};
    \draw node at (7.5,-.7) {\footnotesize{$v_k$}};

    \filldraw[red] (-.70,-.70) circle (0.05);
    \filldraw[red] (.70,-.70) circle (0.05);
    \filldraw[red] (0,-1) circle  (0.05);

    \draw[fill=gray!42] (-2,-7) rectangle (2,-3);
    \draw node at (1.4,-3.3) {$\Sigma_{\text{root}}$};

    \filldraw[red] (-1,-6) circle (0.05);
    \filldraw[red] (-1.5,-5) circle (0.05);
    \filldraw[red] (-.5,-5) circle (0.05);
    \filldraw[red] (.5,-5) circle (0.05);
    \filldraw[red] (-1,-4) circle (0.05);
    \filldraw[red] (0,-4) circle (0.05);

    \draw node at (-.8,-5.8) {\tiny{\textcolor{red}{$n_1$}}}; 
    \draw node at (-1.3,-4.8) {\tiny{\textcolor{red}{$n_2$}}}; 
    \draw node at (-.3,-4.8) {\tiny{\textcolor{red}{$n_3$}}};  
    \draw node at (.7,-4.8) {\tiny{\textcolor{red}{$n_4$}}}; 
    \draw node at (-.8,-3.8) {\tiny{\textcolor{red}{$n_5$}}};   
    \draw node at (.2,-3.8) {\tiny{\textcolor{red}{$n_6$}}};  

    \draw node at (6.75,-3.8) {$\cC_{\le {v_5}} \subsetneq \cC_{\le v_6}$};
    \draw node at (7,-4.8) {$\cC_{\le {v_2}} \subsetneq \cC_{\le v_3} \subsetneq \cC_{\le v_4}$};
    \draw node at (6.75,-5.8) {$\cC_{\le {v_1}}$};
    \draw node at (6.75,-6.8) {$\cC = \langle \cC_{\le {v_1}},\cC_{\le v_4},\cC_{\le v_6}\rangle$.};
	\end{tikzpicture}
\end{center}
\caption{A picture of a two-level multi-scale line, its dual tree, the configuration of node points on $\Sigma_{\rm{root}}\setminus \{p_\infty\}$, and the resulting filtered semiorthogonal decomposition. For simplicity, we have taken $k=6$.}
\label{F:2level}
\end{figure}
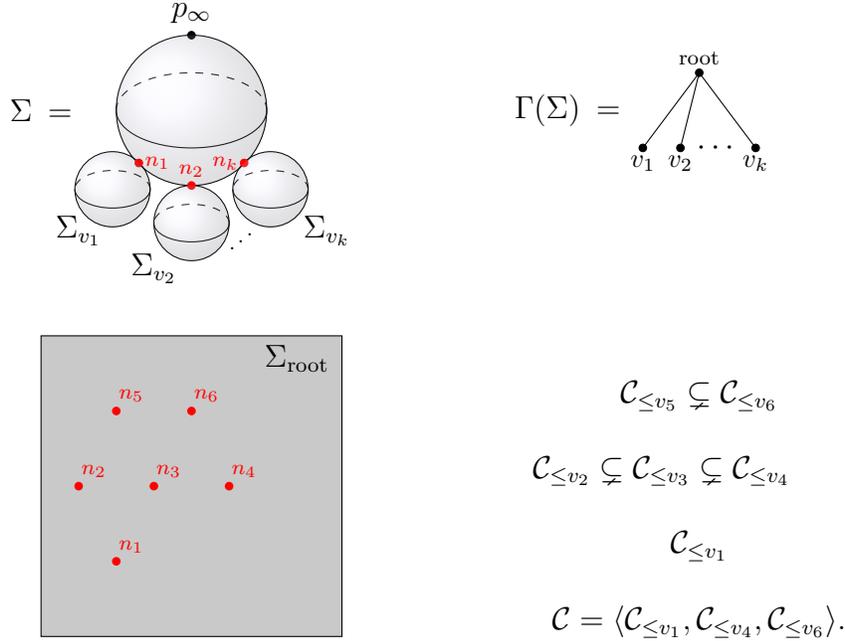

\begin{defn} 
\label{D:lexicographicordering}
    For any multi-scale line $\Sigma$ and all $0<\epsilon\ll 1$, the partial order $\leq_{e^{i\epsilon},0}$ is a total ordering on the terminal components of $\Sigma$ and is independent of $\epsilon$. We will refer to this as the \emph{lexicographic ordering}, and denote it $\leq_{\rm{lex}}$, because $u \leq_{\rm{lex}} v$ if either 1) $u \leq_{1,0} v$ or 2) $u$ and $v$ are incomparable with respect to $\leq_{1,0}$ and satisfy $u \leq_{i,\infty} v$. From this description, it is evident that $\le_{\rm{lex}}$ refines the partial ordering $\le_{1,t}$ for all $t\in [0,\infty]$.
\end{defn}

\begin{prop} 
\label{L:lex_filtration}
A multi-scale decomposition $\cC = \langle \cC_\bullet \rangle_\Sigma$ is uniquely determined by the multi-scale line $\Sigma$ and the following subcategories indexed by $v \in V(\Sigma)_{\rm{term}}$,
\[
    F_v(\cC_\bullet) := \Span \{ \cC_{\leq u} \:|\: u \in V(\Sigma)_{\rm{term}}, u \leq_{\rm{lex}} v \}.
\]
\end{prop}

\begin{proof}
We prove this by induction on the cardinality of $V(\Sigma)/{\sim}$. If there is only one level, then there is nothing to prove. Otherwise, let $u_1,\ldots,u_n \in V(\Sigma)$ be the nodes lying immediately below the root node, index the terminal nodes of $\Sigma$ so that $v_1 \leq_{\rm{lex}} v_2 \leq_{\rm{lex}} \cdots \leq_{\rm{lex}} v_N$, and let $\cF_j := F_{v_j}(\cC_\bullet)$.

First, we show that the semiorthogonal decomposition of \Cref{const:SODfrommulti-scale} is determined by the filtration $\cF_1 \subset \cdots \subset \cF_N = \cC$ and $\Sigma$. Indeed, one can show by induction that for every $s$ one can obtain the semiorthogonal decomposition $\cF_{s+1} = \langle \cF_{s+1} \cap \cC_{\leq u_{i_1}},\ldots, \cF_{s+1} \cap \cC_{\leq u_{i_k}} \rangle$ from the corresponding semi\-orthogonal decomposition of $\cF_s$, because all but one of the factors remains unchang\-ed, and each term of a semiorthogonal decomposition is uniquely determined by the remaining terms and the ordering of terms.

Next, for any $j = 1,\ldots,n$, let $S_j \subset \{1,\ldots, N\}$ be the set of $r$ for which $v_r \subset u_j$, and let $p$ be the unique index such that $i_{p-1}<j\leq i_p$. Because $\leq_{\rm{lex}}$ refines $\leq_{1,\infty}$, for each $r \in S_j$ the projection of $\cF_r$ onto the semiorthogonal factor $\cC_{\leq u_{i_p}}$ above contains $\cC_{<u_j}$ and is contained in $\cC_{\leq u_j}$. This projection therefore descends to a subcategory $\cF^{u_j}_r \subseteq \gr_{u_j}(\cC_\bullet).$ In fact, these categories $\cF^{u_j}_r$ are precisely the subcategories $F_v(\cC_\bullet^{u_j})$ of the descendent decomposition $\gr_{u_j}(\cC_\bullet) = \langle \cC_{\bullet}^{u_j} \rangle_{\Sigma_{\subseteq u_j}}$, because the $\leq_{\rm{lex}}$ ordering on $\Sigma_{\subseteq u_j}$ agrees with that on $\Sigma$. In particular, $\cC_{\leq u_j} = \cF_s$ where $s = \max(S_j)$.

We have shown that the subcategories $\cC_{\leq u_j}$ can be reconstructed uniquely for all $j$, and that the analogous filtration $\cF^{u_j}_r$ can be obtained on $\gr_{u_j}(\cC_\bullet)$. The inductive hypothesis implies that the induced multi-scale decomposition $\gr_{u_j}(\cC_\bullet) = \langle \cC_{\bullet}^{u_j} \rangle_{\Sigma_{\subseteq u_j}}$ is uniquely determined. From this one can uniquely recover the subcategories $\cC_{\leq v}$ for any terminal $v \in V(\Sigma)$ as the preimage of $\cC^{u_j}_{\leq v}$ under the quotient map $\cC_{\leq u_j} \to \gr_{u_j}(\cC_\bullet)$, where $j$ is the unique index such that $v \subseteq u_j$.
\end{proof}


In general, unlike a semiorthogonal decomp\-osition, a multi-scale decomposition does not give functorial filtrations of objects, but we do have the following:

\begin{lem}[Scale filtration] \label{L:multi-scale_decomp_filtration}
Let $E \in \cC$ be non-zero and consider a multi-scale decomp\-osition $\langle \cC_\bullet \rangle_\Sigma$. Then there is a unique set of terminal components $s_1,\ldots,s_m$ of $\Sigma$ such that:
\begin{enumerate}
    \item There exists a descending filtration $E_\bullet$ of $E$ such that $\gr_i(E_\bullet) \in \cC_{\leq s_i}$ and $\Pi_i(\gr_i(E_\bullet)) \in \gr_{s_i}(\cC_\bullet)$ is non-zero; and
    \item $s_a \leq_{i,0} s_b$ for all $a<b$.
\end{enumerate}
We refer to $E_\bullet$ as a \emph{scale filtration} of $E$. Given another scale filtration $F_\bullet$ of $E$, there is a third scale filtration $G_\bullet$ of $E$ and morphisms $E_\bullet \leftarrow G_\bullet \to F_\bullet$ that induce isomorphisms $\Pi_i(\gr_i(E_\bullet)) \leftarrow \Pi_i(\gr_i(G_\bullet)) \to \Pi_i(\gr_i(F_\bullet))$ for all $i$.
\end{lem}

\begin{proof}
By \Cref{const:descendent}, for any component $v$ of $\Sigma$, we have a descendent decomposition $\gr_v(\cC_\bullet) = \langle \cC_\bullet^v \rangle_{\Sigma_{\subseteq v}}$. We will prove the claim for this induced multi-scale decomposition by induction on the level of $v$. The base case is when $v$ is terminal, in which case the claims of the lemma are tautological.

For the inductive step, it suffices to consider the case where $v$ is the root component, and assume the claim is known for the components $u_1,\ldots,u_n$ lying immediately below $v$. We will use the notation and conventions of \Cref{const:SODfrommulti-scale} and consider the corresponding semi\-orthogonal decomposition $\cC = \langle \cC_{\leq u_{i_1}},\ldots, \cC_{\leq u_{i_k}} \rangle.$

There is a unique descending filtration of $E$ such that $E_a := \gr_a (E) \in \cC_{\leq u_{i_a}}$ for $a=1,\ldots,k$. For each $a$ such that $E_a \neq 0$, let $j_a$ be the minimal $i_{a-1}<j \leq i_a$ such that $E_a \in \cC_{\leq u_j}$. Letting $u_a := u_{j_a}$, then removing the steps in the filtration where $E_a=0$ and reindexing so that $E_a \neq 0$ for all $a$, we have identified a unique (up to isomorphism) descending filtration of $E$ with $E_a := \gr_a(E) \in \cC_{\leq u_a}$ for $a=1,\ldots,k$, $\Im(\mathfrak{p}(u_a,u_b))>0$ for $a<b$, and such that the image $E'_a \in \gr_{u_a}(\cC_\bullet)$ of each $E_a$ is non-zero.

Now by the inductive hypothesis, each $E'_a$ admits a filtration in $\gr_{u_a}(\cC_\bullet)$ satisfying the conditions of the lemma. Each of these filtrations can be lifted to a filtration of $E_a$ satisfying the conditions of the lemma.\endnote{Assertions about lifting filtrations can be verified using the following: \vspace{2mm}

\emph{Claim:} Let $\pi:\cC\to \cD$ be a Verdier quotient of triangulated categories by some thick subcategory $\cT$ of $\cC$. Given a morphism $f \in \Hom_{\cD}(E,F)$, there exists an object $E'$ of $\cC$ and a morphism $\beta:E'\to E$ in $\cC$ which is an isomorphism in $\cD$ such that there is a morphism $h\in \Hom_{\cC}(E',F)$ with $\beta \circ f = h$ in $\cD$. \vspace{4mm}

The claim implies that one can lift filtrations from $\cD$, where by this we mean: given a descending filtration $F_\bullet$ of $F \in \cD$, there exists a filtration $F_\bullet'$ of $F\in \cC$ such that $\pi(F_\bullet') = F_\bullet$, regarded as diagrams in $\cC$ and $\cD$, respectively. Since $\pi$ exact, this also implies that $\pi(\gr_i(F_\bullet')) \cong \gr_i(F_\bullet)$ for all indices $i$.

\begin{proof}[Proof of Claim]
    $f\in \Hom_{\cD}(E,F)$ corresponds to a roof diagram $E \xleftarrow{\beta} E' \xrightarrow{h} F$ of morphisms in $\cC$, where $\Cone(\beta) \in \cT$. In particular, $\beta \in \Hom_{\cC}(E,E')$ is sent to an isomorphism under $\pi$ and one can verify using the composition law for roof diagrams that in $\cD$ one has $\beta \circ f = h$. 
\end{proof}
} Note that any automorphism of $E_a$ descends canonically to an automorphism of $E'_a$, so $\gr_i(E_a) \in \gr_{s_i}(\cC_\bullet)$ is still determined up to canonical isomorphism. Finally, the filtration of each $E_a$ can be combined with the filtration of $E$ with $\gr_a(E) = E_a$ constructed in the previous paragraph, and the resulting filtration satisfies the conditions of the lemma.

The last assertion about the morphism of filtrations $E_\bullet \leftarrow G_\bullet \to F_\bullet$ inducing isomorphisms on graded pieces can be proven by induction, noting that the projection of a scale filtration $E_\bullet$ of $E$ to $\gr_{u_a}(\cC_\bullet)$ (in the above notation) gives a scale filtration of the image of $E$.\endnote{Consider the large scale semiorthogonal decomposition of $\cC$ as in \Cref{const:SODfrommulti-scale} and as in the body of the proof let $E_a \in \cC_{\le u_a}$ be one of the associated graded factors. Then, applying the projection functor $\cC_{\le u_a} \to \gr_{u_a}(\cC_\bullet)$ gives a pair of scale filtrations of $E_a$ in $\gr_{u_a}(\cC_\bullet)$, denoted $E_\bullet'$ and $F_\bullet'$. We may assume inductively that these are refined by a scale filtration $E_\bullet' \leftarrow G_\bullet'\to F_\bullet'$ as in the statement. Up to replacing $G_\bullet'$ by an equivalent filtration $G_\bullet''$ in $\gr_{u_a}(\cC_\bullet)$, we can lift this to a diagram of filtrations $E_\bullet \leftarrow G_\bullet \to F_\bullet$ in $\cC_{\le u_a} \subseteq \cC$. We can concatenate the lifted filtrations for each $a$ to complete the proof of the claim}
\end{proof}

\begin{warning}
In general, the scale filtration is not compatible with direct sums. For instance, consider a terminal component $v$, $E \in \cC_{\leq v}$ with $\Pi_v(E)\neq 0$, and any object $F$ for which the $s_i$ appearing in \Cref{L:multi-scale_decomp_filtration} satisfy $s_i <_{1,\infty} v$. Then $E \oplus F \in \cC_{\leq v}$ and $\Pi_v(E\oplus F) \cong \Pi_v(E) \neq 0$, so the filtration of $E \oplus F$ is trivial. 
\end{warning}

\begin{defn}
\label{D:twellplaced}
    Let $t\in [0,\infty]$ be given. If the terminal components $s_1,\ldots, s_m$ of $\Sigma$ associated to the scale filtration of $E$ have a maximal element $s_i$ with respect to the partial order $\le_{1,t}$, then we say $E$ is $t$-\emph{well-placed} with \emph{dominant} vertex $s_i$, written $\dom(E) = s_i$. The \emph{dominant projection} of $E$ is the image of $\gr_i(E_\bullet)$ in $\gr_{s_i}(\cC_\bullet)$ and is written $\Pi_{\rm{dom}}(E)$ or simply $\Pi(E)$ for brevity.
\end{defn}

\begin{rem}
\label{R:inftywellplaced}
    Note that an object $E$ is $\infty$-well-placed with respect to a multi-scale decomp\-osition $\langle \cC_\bullet\rangle_\Sigma$ if and only if $E\in \cC_{\le v}$ for some $v\in V(\Sigma)_{\rm{term}}$ and has nonzero image in $\gr_v(\cC_\bullet)$. This $v$ is the dominant vertex in the scale filtration of $E$.\endnote{Suppose $E\in \cC_{\le v}$ and has nonzero projection to $\gr_v(\cC_\bullet)$. Then, we have a scale filtration given by $E_1 = E$ with only one factor $E \in \cC_{\le v}$. Conversely, if $E$ is $\infty$-well-placed, the components $s_1,\ldots, s_m$ appearing in its scale filtration have a maximum $s_i$ with respect to $\le_{1,\infty}$. However, $s_j\le_{i,0} s_i$ (or $s_i\le_{i,0}s_j$) and $s_j\le_{1,\infty}s_i$ implies that $s_i = s_j$. Consequently, there is only one terminal component $s = s_i$ appearing in the scale filtration, and it follows that $E\in \cC_{\le s}$.} 
\end{rem}

\subsection{Coarsening multi-scale decompositions}

\begin{defn}\label{D:specialization_totally_preordered_rooted_tree}
A \emph{contraction} of totally preordered rooted trees $f : (\Gamma,\preceq,v_0) \twoheadrightarrow (\Gamma',\preceq,v_0')$ is a surjection $f:V(\Gamma) \to V(\Gamma')$ such that\endnote{This is a slight abuse of notation. We use the same symbol $f$ to denote both the contraction, and the underlying surjection on vertex sets.}
\begin{enumerate} 
    \item $v\preceq w$ implies $f(v)\preceq f(w)$;
    \item for adjacent vertices $v,w \in V(\Gamma)$, $f(v)$ and $f(w)$ are either equal or adjacent.
    \item for any $v' \in V(\Gamma')$, $f^{-1}(v')$ spans a connected subgraph of $\Gamma$.
\end{enumerate}
Condition (1) implies that $f(v_0) = v_0'$, since $u'\preceq f(v_0)$ for all $u'$ and this condition uniquely characterizes the root. 
\end{defn}

\begin{lem}\label{L:specializations}
For any totally preordered rooted tree $(\Gamma,\preceq,v_0)$ and any order preserving surjection of totally ordered sets
\[
    \beta: V(\Gamma) /{\sim} \twoheadrightarrow [n] := \{0<\cdots<n\},
\]
there is a unique $(\Gamma',\preceq,v_0')$ which is a contraction $f:(\Gamma,\preceq,v_0)\twoheadrightarrow (\Gamma',\preceq,v_0')$ such that $V(\Gamma')/{\sim}$ is isomorphic to $[n]$ under $V(\Gamma)/{\sim}$.\endnote{By this, we mean that there is an order-preserving bijection $V(\Gamma')/{\sim} \to [n]$ whose composition with the canonical map $V(\Gamma)/{\sim} \to V(\Gamma')/{\sim}$ induced by $f$ agrees with the given map $V(\Gamma)/{\sim} \to V(\Gamma')/{\sim}$. Such a map is unique if it exists.} The contraction $f$ is uniquely determined by $\beta$ and the induced map $V(\Gamma)_{\rm{term}}\to V(\Gamma')_{\rm{term}}$ and any two such contractions differ by a unique automorphism of $(\Gamma',\preceq',v_0')$.
\end{lem}

\begin{proof}
    Write $g$ for the composite $V(\Gamma)\to V(\Gamma)/{\sim} \twoheadrightarrow [n]$. We construct $\Gamma'$ as follows: for all $k\in [n]$, $g^{-1}(k)$ spans a disjoint union of trees $\{T_k^i\}$. For each $0\le k \le n$, the level $k$ vertex set of $\Gamma'$ is $\{T_k^i\}$. Two elements of $V(\Gamma')$ are connected by an edge if they are connected by an edge in $\Gamma$. The tree spanned by $g^{-1}(n)$ contains $v_0$ and defines the root vertex $v_0'$ of $\Gamma'$. The reader may verify that this defines a totally preordered rooted tree $(\Gamma',\preceq',v_0')$, which is unique up to isomorphism with the claimed properties.\endnote{$\Gamma'$ is a tree: any pair of elements of $V(\Gamma')$ can be joined by a unique simple path. This follows from the corresponding fact for $\Gamma$. The total preorder on $V(\Gamma')$ is specified by saying $v\preceq w$ if $v$ is on level $k$ and $w$ is on level $l$ with $k\le l$. Condition (2)(a) of \Cref{D:multi-scaleline} follows from the fact that each of the $T_k^i$ is a totally preordered rooted tree itself. The root vertex is the unique maximum for $\preceq$ in $\Gamma$. This root vertex is connected to a unique vertex via an ascending edge, and this implies the corresponding property for the tree $T_k^i$. \Cref{D:multi-scaleline}(2)(b) follows from the fact that all terminal vertices of $\Gamma$ are contained in $g^{-1}(0)$. This reasoning also gives condition (3).
    
    The most important part here is the uniqueness: the condition that there is a contraction $f:(\Gamma,\preceq,v_0)\twoheadrightarrow(\Gamma',\preceq',v_0')$ such that $V(\Gamma')/{\sim'}$ is isomorphic to $[n]$ under $V(\Gamma)/{\sim}$ forces each $T_k^i$ to be contracted to a vertex. If $T_k^i$ and $T_k^j$ are sent to the same vertex $v'$ under a putative $f$ then the preimage $f^{-1}(v')$ would not be connected (as required by \Cref{D:specialization_totally_preordered_rooted_tree}(3)) unless $T_k^i = T_k^j$. Consequently, the level $k$ vertices of $\Gamma'$ must be in bijective correspondence with $\{T_k^i\}$. This also uniquely determines the total preorder and the root. The definition of the edges is forced by \Cref{D:specialization_totally_preordered_rooted_tree}(2).}

    There is a canonical map $f:(\Gamma,\preceq,v_0)\twoheadrightarrow (\Gamma',\preceq',v_0')$ sending each $v\in V(\Gamma)$ to the sub-tree of $\Gamma$ to which it belongs. One can verify that this $f$ is a contraction and such that $V(\Gamma')/{\sim'}$ is isomorphic to $[n]$ under $V(\Gamma)/{\sim}$. Finally, it follows from \Cref{D:specialization_totally_preordered_rooted_tree} that contractions of totally preordered rooted trees preserve joins. Now, since any $v\in V(\Gamma)$ is the join of some elements $u,w\in V(\Gamma)_{\rm{term}}$ it follows that $f(v) = f(u\vee w) = f(u)\vee f(w)$ is uniquely determined by the map $V(\Gamma)_{\rm{term}}\to V(\Gamma')_{\rm{term}}$.\endnote{The join $u\vee w$ of two vertices in a totally preordered rooted tree can be characterized as the first common vertex in the unique simple paths connecting $u$ and $w$ to $v_0$. A path here can be regarded as a sequence of vertices $u=u_0,u_1,\ldots, u_k = v_0$ and $w = w_0,w_1,\ldots, w_l = v_0$. Thus, $u\vee w$ is the first vertex which is common to both sequences; say $u\vee w = u_s$ and $u\vee w = w_t$. Since $\Gamma'$ is also a tree, these simple paths are sent to simple paths. If $x=f(u_{s'}) = f(w_{t'})$ for $s'<s$ and $t'<t$ is the meet $f(u)\vee f(w)$, then the preimage of $x$ contains both $u_{s'}$ and $w_{t'}$. For this preimage to be connected, as required by \Cref{D:specialization_totally_preordered_rooted_tree}(3), it must be that $u_{s'}\vee w_{t'} = u\vee w$ is in the preimage of $x$. Consequently, $x = f(u\vee w)$. }

    Finally, consider another contraction $g:(\Gamma,\preceq,v_0)\twoheadrightarrow (\Gamma',\preceq',v_0')$ such that $V(\Gamma')/{\sim}$ is isomorphic to $[n]$ under $V(\Gamma)/{\sim}$ via $g$. $g$ must map each of the trees $T_k^i$ to a single level $k$ vertex of $\Gamma'$ and in particular induces a bijection from $\{T_k^i\}$ to itself. Therefore, there is a unique $\alpha \in \Aut(V(\Gamma'))$ such that $\alpha \circ f = g$, where $f$ denotes the canonical map from above. One can verify that $\alpha$ is an isomorphism as claimed.\endnote{By its definition, $\alpha$ preserves the $\preceq$ and the root vertex since it preserves the level values in $0,\ldots, n$. 
    
    Consider sub-trees $T$ and $T'$ of $\Gamma$ corresponding to elements of $V(\Gamma').$ If $T$ and $T'$ are connected in $\Gamma'$ then they correspond to subtrees sent to different values under $h$. So, $g(T) \ne g(T')$ but yet $T$ and $T'$ contain adjacent vertices. \Cref{D:specialization_totally_preordered_rooted_tree}(2) now implies that $g(T)$ and $g(T')$ are adjacent. Conversely, suppose that $g(T)$ and $g(T')$ are adjacent in $\Gamma'$. Suppose also that $h(T) = l<n$ and $h(T') \ge l+1$. If $T$ and $T'$ are not adjacent, let $T''$ be a subtree of $\Gamma$ with $h(T'') \ge l+1$ such that $T''$ and $T$ are adjacent. Then in $\Gamma'$, $T$ is connected to $T''$ and $T'$ which, by uniqueness of ascending edges, implies that $T'' = T'$. 

    Finally, $\alpha$ is a graph isomorphism because $g(T)$ and $g(T')$ are adjacent if and only if $T$ and $T'$ are by the previous paragraph.}
\end{proof}

Identifying $V(\Gamma)/{\sim}$ with $[\ell]$ for some $\ell\geq 0$, \Cref{L:specializations} says that totally preordered rooted trees $\Gamma'$ which are contractions of $\Gamma$ correspond bijectively to order-preserving surjections $[\ell] \twoheadrightarrow [n]$ for $n \in [0,\ell]$. The additional information of a map $V(\Gamma)_{\rm{term}} \to V(\Gamma')_{\rm{term}}$ then determines the contraction.

\begin{ex}
\label{ex:contractinglevels}
    Consider the unique isomorphism $V(\Gamma)/{\sim} \cong [\ell]$ of totally ordered sets taking $v_0$ to $\ell$. Given $k\in \{1,\ldots,\ell\}$, the \emph{contraction of the $k^{\text{th}}$ level} is the one induced by the unique order preserving surjection $[\ell]\twoheadrightarrow[\ell-1]$ such that $k,k-1\mapsto k-1$, i.e. the degeneracy map $\sigma_{k-1}^\ell:[\ell]\to [\ell-1]$ defined by
    \[
        \sigma_{k-1}^\ell(p) = 
        \begin{cases}
            p&p\le k-1\\
            p-1 & p \ge k.
        \end{cases}
    \]
    As explained in the proof of \Cref{L:specializations}, the graph $\Gamma'$ obtained by contracting the $k^{\rm{th}}$ level of $\Gamma$ can be obtained from $\Gamma$ be declaring a vertex $v_k$ on level $k$ equivalent to a vertex $v_{k-1}$ if they are joined by an edge, and then removing the edge. The level structure of the resulting graph is encoded by the surjection $[\ell]\twoheadrightarrow [\ell-1]$. In particular, since $V(\Gamma')$ is the quotient of $V(\Gamma)$ by a relation, there is a canonical map $V(\Gamma)_{\rm{term}}\to V(\Gamma')_{\rm{term}}$ and thus a canonical contraction induced by the data of $[\ell]\twoheadrightarrow [\ell-1]$ and $\Gamma$.

    Since every order preserving surjection $[\ell]\twoheadrightarrow[\ell-k]$ for $0\le k \le \ell$ factorizes uniquely as $\sigma_{j_1}^{\ell-k+1}\circ \cdots \circ \sigma_{j_k}^\ell$ for $0\le j_1<\cdots<j_k\le \ell-1$ by \cite{Maclanecategories}*{Lem. p. 173}, every contraction of $V(\Gamma)$ is determined by a choice of levels $1\le i_1<\cdots<i_k\le \ell$ to be contracted. The corresponding contraction is given by $\sigma_{i_1-1}^{\ell-k+1}\circ \cdots \circ \sigma_{i_k - 1}^\ell$ and the canonical contraction constructed here is called \emph{the} \emph{contraction of levels} $i_1<\cdots<i_k$.
\end{ex}

The definition of a contraction $f : \Gamma(\Sigma) \twoheadrightarrow \Gamma(\Sigma')$ implies that for any $v \in V(\Sigma')$ the subset $f^{-1}(v) \subset V(\Sigma)$ has a unique maximal element with respect to $\subseteq$. It is convenient to introduce the function
\[
f^\dagger : V(\Sigma') \to V(\Sigma) \text{ given by } v \mapsto \max (f^{-1}(v)).
\]

\begin{defn}[Coarsening]
\label{D:multi-scaleSODcoarsening}
    A \emph{coarsening} of $\cC = \langle \cC_\bullet \rangle_{\Sigma_1}$, denoted $\langle \cC_\bullet \rangle_{\Sigma_1} \rightsquigarrow \langle \cB_\bullet \rangle_{\Sigma_2}$, is a multi-scale decomposition $\cC = \langle \cB_\bullet \rangle_{\Sigma_2}$ and a contraction $f : \Gamma(\Sigma_1) \twoheadrightarrow \Gamma(\Sigma_2)$ satisfying:
    \begin{enumerate}[label=(\alph*)]
        \item For terminal vertices, $u \leq_{1,0} v \Rightarrow f(u) \leq_{1,0} f(v)$ and $u \leq_{i,0} v \Rightarrow f(u) \leq_{i,0} f(v)$, and 
        \item $\cB_{\leq v} \subseteq \cC_{\leq f^\dagger(v)}$ for all $v\in V(\Sigma_2)$, and the composite functor $\cB_{\leq v} \to \gr_{f^\dagger(v)}(\cC_\bullet)$ factors through an equivalence of quotient categories $\gr_v(\cB_\bullet) \cong \gr_{f^\dagger(v)}(\cC_\bullet)$. 
    \end{enumerate}
\end{defn}

\begin{lem}
\label{L:coarseningconsequences}
    Given a coarsening $\langle \cC_\bullet\rangle_{\Sigma_1}\rightsquigarrow \langle \cB_\bullet\rangle_{\Sigma_2}$, one has $\Span\{\cC_{<f^\dagger(v)},\cB_{\le v}\} = \cC_{\le f^\dagger(v)}$ for any $v\in V(\Sigma_2)_{\rm{term}}$.
\end{lem}

\begin{proof}
    By \Cref{D:multi-scaleSODcoarsening}(b), since $\cB_{\le v} \to \gr_{f^\dagger(v)}(\cC_\bullet)$ is essentially surjective, the essential image of $\Span\{\cC_{<f^\dagger(v)},\cB_{\le v}\}$ under $\cC_{\le f^\dagger(v)}\to \gr_{f^\dagger(v)}(\cC_\bullet)$ is $\gr_{f^\dagger(v)}(\cC_\bullet)$. The result now follows from the correspondence between strict triangulated subcategories of $\cC_{\le f^\dagger(v)}$ containing $\cC_{<f^\dagger(v)}$ and strict triangulated subcategories of $\gr_{f^\dagger(v)}(\cC_\bullet)$ \cite{Verdierquotient}*{Ch. II, Prop. 2.3.1}. 
\end{proof}

\begin{lem}
If $f : \langle\cD_\bullet\rangle_{\Sigma_1} \rightsquigarrow \langle \cC_\bullet \rangle_{\Sigma_2}$ and $g : \langle\cC_\bullet\rangle_{\Sigma_2} \rightsquigarrow \langle \cB_\bullet \rangle_{\Sigma_3}$ are coarsenings, then the composition of contractions $\Gamma(\Sigma_1) \twoheadrightarrow \Gamma(\Sigma_2) \twoheadrightarrow \Gamma(\Sigma_3)$ defines a coarsening $\langle\cD_\bullet\rangle_{\Sigma_1} \rightsquigarrow \langle \cB_\bullet \rangle_{\Sigma_3}$.
\end{lem}

\begin{proof}
\Cref{D:multi-scaleSODcoarsening}(a) for $g \circ f$ follows directly from the corresponding conditions for $f$ and $g$. Consider $v\in V(\Sigma_3)$ and note that $f^\dagger(g^\dagger(v)) = (g\circ f)^\dagger(v)$. Applying \Cref{D:multi-scaleSODcoarsening}(b) to $f$ and $g$ gives a diagram
\[
    \begin{tikzcd}
    \cB_{\le v}\arrow[r,hook]\arrow[d,two heads]&\cC_{\le f^\dagger(v)}\arrow[r,hook]\arrow[d,two heads]& \cD_{\le f^\dagger g^\dagger(v)}\arrow[d,two heads]\\
    \gr_v(\cB_\bullet)\arrow[r,"\sim"]& \gr_{f^\dagger(v)}(\cC_\bullet)\arrow[r,"\sim"] & \gr_{f^\dagger g^\dagger(v)}(\cD_\bullet)
    \end{tikzcd}
\]
whose outer square implies \Cref{D:multi-scaleSODcoarsening}(b).
\end{proof}

\begin{ex}
Let us study coarsenings in the case where $\Sigma_1$ has exactly two terminal nodes $u_1,u_2$, and let $\alpha_1 := \mathfrak{p}(u_1,u_2)$. The trivial multi-scale decomp\-osition $\cC = \langle \cC \rangle_{\bP^1}$ is automatically a coarsening, so suppose that $\Sigma_2$ also has two terminal nodes, in which case $f : \Gamma(\Sigma_1) \to \Gamma(\Sigma_2)$ is an isomorphism of preordered rooted trees. We can therefore regard $\Sigma_2$ as the same curve as $\Sigma_1$, but with a different choice of differential, encoded by the normalized period $\alpha_2 = \mathfrak{p}(u_1,u_2)$ computed in $\Sigma_2$. We consider the following cases:

\begin{itemize}

\item $\Im(\alpha_1) > 0$: $\cC = \langle \cC_\bullet \rangle_{\Sigma_1}$ is equivalent to specifying a semiorthogonal decomposition $\cC = \langle \cC_{\leq u_1}, \cC_{\leq u_2} \rangle$. \Cref{D:multi-scaleSODcoarsening}(a) implies that $\Im(\alpha_2)>0$ as well, so we have a second semiorthogonal decomposition $\cC = \langle \cB_{\leq u_1}, \cB_{\leq u_2} \rangle$ such that $\cB_{\leq u_1} \subseteq \cC_{\leq u_1}$ and $\cB_{\leq u_2} \subseteq \cC_{\leq u_2}$. It follows that $\cB_{\leq u_\bullet} = \cC_{\leq u_\bullet}$.

\item $\alpha_1 = 1$: $\cC = \langle \cC_\bullet\rangle_{\Sigma_1}$ is equivalent to specifying a filtration $0 \subsetneq \cC_{\leq u_1} \subsetneq \cC_{\leq u_2} = \cC$. \Cref{D:multi-scaleSODcoarsening}(a) implies that $\Re(\alpha_2)>0$. \Cref{L:coarseningconsequences} implies that $\cB_{\leq u_1}=\cC_{\leq u_1}$, because $\cC_{<u_1}=0$. In addition, one of the following holds:

\begin{itemize}
 \item $\alpha_2=1$: We have $\cB_{\leq u_2}=\cC$, so the decompositions are identical.

\item $\Im(\alpha_2)>0$: The existence of a semiorthogonal decomposition $\cC = \langle \cB_{\leq u_1}, \cB_{\leq u_2} \rangle$ implies that $\cB_{\leq u_2}$ is the right orthogonal complement of $\cC_{\leq u_1}$.

\item $\Im(\alpha_2)<0$: We instead have $\cC = \langle \cB_{\leq u_2}, \cB_{\leq u_1} \rangle$, so $\cB_{\leq u_2}$ is the \emph{left} orthogonal complement of $\cC_{\leq u_1}$. In this and the last case $\cB_{<u_2} = 0$, so \Cref{D:multi-scaleSODcoarsening}(b) implies that the canonical morphism $\cB_{\leq u_2} \to \gr_{u_2}(\cC_\bullet)$ is an equivalence. Com\-posing one equivalence with the inverse of the other gives the usual mutation equivalence between the left and right orthogonal complement of the admissible subcategory $\cC_{\leq u_1} \subset \cC$.
\end{itemize}

\end{itemize}

In all cases, the categories $\cB_{\leq u_\bullet}$ are uniquely determined by the value of $\alpha_2$. When $\Im(\alpha_1)>0$, any value of $\alpha_2$ with $\Im(\alpha_2)>0$ underlies a coarsening. On the other hand, if $\alpha_1=1$, the existence of a coarsening for different values of $\alpha_2$ encounters a categorical obstruction, depending on whether the inclusion $\cC_{\leq u_1} \subset \cC$ admits a left or right adjoint.
\end{ex}

\begin{defn}[Signature]
    \label{D:signature} Given a multi-scale line $\Sigma$, the collection of values of the functions $\sign \Im(\mathfrak{p}(u,v)),\sign\Re(\mathfrak{p}(u,v)) \in \{0,\pm 1\}$ where $u,v$ range over all distinct pairs in $V(\Sigma)_{\rm{term}}$ is called the \emph{signature} of $\Sigma$. 
\end{defn}

The signature is an invariant of $\Sigma$ up to real-oriented isomorphism.

\begin{lem}
\label{L:coarseninguniquelydetermined}
The categories $\cB_{\le v}$ of a coarsening $\langle \cC_\bullet \rangle_{\Sigma_1} \rightsquigarrow \langle \cB_\bullet \rangle_{\Sigma_2}$ are determined by $\langle \cC_\bullet \rangle_{\Sigma_1}$, the contraction $f : \Gamma(\Sigma_1) \twoheadrightarrow \Gamma(\Sigma_2)$, and 
the signature of $\Sigma_2$. Conversely, there is at most one contraction $\Gamma(\Sigma_1) \twoheadrightarrow \Gamma(\Sigma_2)$ realizing $\langle \cB_\bullet \rangle_{\Sigma_2}$ as a coarsening of $\langle \cC_\bullet \rangle_{\Sigma_1}$. 
\end{lem}

\begin{proof}
    Index the terminal components of $\Sigma_2$ so that $v_1 \le_{\rm{lex}} v_2 \le_{\rm{lex}} \cdots \le_{\rm{lex}} v_n$ as in \Cref{D:lexicographicordering}. We consider two sequences of subcategories of $\cC$ 
    \begin{gather*}
    \cF_j := F_{v_j}(\cB_\bullet) = \Span \{ \cB_{\leq v_i} | i \leq j \}, \\
    \cG_j := \Span \{ \cC_{\leq f^\dagger(v_i)} | i \leq j\}.
\end{gather*}
\Cref{D:multi-scaleSODcoarsening}(b) implies that we have a diagram of inclusions:
\[
\begin{tikzcd}[column sep=.7em, row sep=.7em]
\cF_1 \arrow[r,symbol=\subseteq] \arrow[d,symbol=\subseteq]  & \cF_2 \arrow[r,symbol=\subseteq] \arrow[d,symbol=\subseteq] & \cdots \arrow[r,symbol=\subseteq] & \cF_n \arrow[d,symbol={=}] \arrow[r,symbol={=}] & \cC \\
\cG_1 \arrow[r,symbol=\subseteq] & \cG_2 
 \arrow[r,symbol=\subseteq] & \cdots \arrow[r,symbol=\subseteq] & \cG_n &
 \end{tikzcd}
\]
For any $u \in V(\Sigma_1)_{\rm{term}}$ such that $u \nsubseteq f^\dagger(v_j)$  (in particular $f(u) \neq v_j$) and $u \leq_{1,\infty} f^\dagger(v_j)$, it follows from \Cref{D:multi-scaleSODcoarsening}(a) that $f(u) <_{1,0} v_j$, which implies that $f(u) <_{\rm{lex}} v_j$. So, for some $i<j$, $f(u) = v_i$ and hence $\cC_{\leq u} \subset \cG_i$.\endnote{By definition, $\cC_{\le f^\dagger(v_i)}$ is generated by the $\cC_{\le u}$ where $u\in V(\Sigma_1)$ is a terminal component such that $u\le_{1,\infty} f^\dagger(v_i)$ or $u\subseteq f^\dagger(v_i)$. In particular, since $f(u) = v_i$ it follows that $u\subset f^\dagger(v_i)$, and so $\cC_{\le u}\subset \cC_{\le f^\dagger(v_i)}\subset \cG_i$.} We conclude that $\cC_{<f^\dagger(v_j)} \subseteq \cG_{j-1}$ for all $j$, using the convention that $\cG_0 = 0$. \Cref{L:coarseningconsequences} now implies that each functor $\cF_j \to \cG_j/\cG_{j-1}$ is essentially surjective. An inductive argument then implies that $\cF_j = \cG_j$ for all $j=1,\ldots,n$. We have thus recovered the lexicographic filtration of $\cC$ associated to $\langle \cB_\bullet \rangle_{\Sigma_2}$ from the multi-scale decomposition $\cC = \langle \cC_\bullet \rangle_{\Sigma_1}$, the total order $\leq_{\rm{lex}}$ on the terminal components of $\Sigma_2$, and $f$. Thus, \Cref{L:lex_filtration} implies that the categories $\cB_{\leq v}$ are uniquely determined.\medskip

We now prove the converse. Consider a pair of contractions $f,g:\Gamma(\Sigma_1)\twoheadrightarrow \Gamma(\Sigma_2)$ realizing $\langle \cB_\bullet\rangle_{\Sigma_2}$ as a coarsening of $\langle \cC_\bullet\rangle_{\Sigma_1}$. By \Cref{L:specializations}, there exists $\alpha \in \Aut(\Gamma(\Sigma_2))$ such that $\alpha \circ f= g$, which is determined by its action on $V(\Sigma_2)_{\rm{term}}$.\endnote{An automorphism of a totally preordered rooted tree is in particular a contraction. Consequently, by \Cref{L:specializations} it follows that it is uniquely determined by what it does to the terminal vertices.} So, consider $v\in \Gamma(\Sigma_2)_{\rm{term}}$ and suppose $v\ne \alpha(v)=:u$. Since $u\ne v$ and both are terminal, it follows that $f^\dagger(v) \not\subseteq f^\dagger(u)$ and $f^\dagger(u)\not\subseteq f^\dagger(v)$.\endnote{If $f^\dagger(v) \subseteq f^\dagger(u)$ then it follows that $v\subseteq u$ but since $u$ is terminal it follows that $v=u$.} By \Cref{D:multi-scaleSODcoarsening}(b), $\Span\{\cB_{\le v},\cB_{\le u}\}\subset \cC_{\le f^\dagger(v)}$ and $\cC_{\le f^\dagger(u)}$.\endnote{The identity $\alpha \circ f = g$ implies that $g^\dagger(v) = f^\dagger(u)$ and $g^\dagger(u) = f^\dagger(v)$. So, applying \Cref{D:multi-scaleSODcoarsening}(b) to $f$ and $g$ gives the claim.} Suppose $f^\dagger(v)$ and $f^\dagger(u)$ are not compar\-able for $\le_{1,\infty}$ and without loss of generality say $f^\dagger(v)\le_{i,0}f^\dagger(u)$. Write $w = u \vee v$ and note that $f^\dagger(w)  \supseteq f^\dagger(u),f^\dagger(v)$. So, for $\star = u,v$, we have a commutative diagram of functors:
\begin{equation*}
    \begin{tikzcd}
        \cB_{\le v}\arrow[r,hook]\arrow[d,hook]&\cC_{\le f^\dagger(\star)}\arrow[d,hook]\arrow[dr]\\
        \cB_{\le w}\arrow[r,hook]&\cC_{\le f^\dagger(w)}\arrow[r, two heads]&\gr_{f^\dagger(w)}(\cC_\bullet).
    \end{tikzcd}
\end{equation*}
Now, $\cB_{\le v}$ has nonzero image in $\gr_w(\cB_\bullet) \xrightarrow{\sim} \gr_{f^\dagger(w)}(\cC_\bullet)$ by \Cref{D:multi-scale_decomposition}(4) and $v\subseteq w$.\endnote{We have inclusions $\cB_{<w}\subset \cB_{<v}\subset \cB_{\le v}\subset \cB_{\le w}$ and since $\gr_v(\cB_\bullet) \ne 0$, it follows that the image of $\cB_{\le v}$ in $\gr_w(\cB_\bullet)$, i.e. $\cB_{\le v}/\cB_{<w}$ is nonzero.} So, the image of $\cB_{\le v}$ in $\gr_{f^\dagger(w)}(\cC_\bullet)$ is in the image of both $\cC_{\le f^\dagger(v)}$ and $\cC_{\le f^\dagger(u)}$. Thus, for $x = f^\dagger(u)\vee f^\dagger(v)\subseteq f^\dagger(w)$, we see that $\cC_{\le f^\dagger(v)}^x \cap \cC_{\le f^\dagger(u)}^x \subset \gr_{x}(\cC_\bullet)$, is non-zero, containing the image of $\cB_{\le v}$. This contradicts semiorthogonality of $\cC_{\le f^\dagger(v)}^x$ and $\cC_{\le f^\dagger(u)}^x$ in the descendent decomposition of $\gr_{x}(\cC_\bullet)$.\endnote{The only thing to verify is that the image of $\cB_{\le v}$ in $\gr_x(\cC_\bullet)$ is nonzero. We have inclusions $\cB_{\le v}\subset \cC_{\le f^\dagger(v)} \subset \cC_{\le x}$ since $f^\dagger(v)\subset x$. Also, we have $\cC_{<x}\subset \cC_{<f^\dagger(v)}\subset \cC_{\le f^\dagger(v)}\subset \cC_{\le x}$. Since $\cB_{\le v}\to \gr_{f^\dagger(v)}(\cC_\bullet)$ is essentially surjective, we cannot have $\cB_{\le v}\subset \cC_{<x}$. Thus, the image of $\cB_{\le v}$ in $\gr_x(\cC_\bullet)$ is nonzero.} 

So, we may suppose that $f^\dagger(v)\le_{1,\infty} f^\dagger(u)$. Then, we have inclusions $\cB_{\le u} \hookrightarrow \cC_{\le f^\dagger(v)}\hookrightarrow \cC_{<f^\dagger(u)} \hookrightarrow \cC_{\le f^\dagger(u)} \twoheadrightarrow \gr_{f^\dagger(u)}(\cC_\bullet)$. However, the composite $\cB_{\le u} \to \gr_{f^\dagger(u)}(\cC_\bullet)$ is essentially surjective by \Cref{D:multi-scaleSODcoarsening}(b), but $\cB_{\le u}$ is contained in $\cC_{<f^\dagger(u)}$. This implies $\gr_{f^\dagger(u)}(\cC_\bullet) = 0$, contradicting \Cref{D:multi-scale_decomposition}(4). In particular, $\alpha = \id$ and $f = g$.
\end{proof}

\begin{lem}
\label{L:0wpcoarsening}
If $E \in \cC$ is $0$-well-placed with respect to a multi-scale decomposition $\langle \cC_\bullet \rangle_\Sigma$, then it is $0$-well-placed with respect to any coars\-ening $f:\langle \cC_\bullet \rangle_\Sigma\rightsquigarrow \langle \cB_\bullet\rangle_{\Sigma_2}$ as well. Furthermore, $\dom_{\langle \cB_\bullet\rangle}(E) = f(\dom_{\langle \cC_\bullet\rangle}(E))$
\end{lem}

\begin{proof}
    Consider a nonzero object $E$ of $\cC$, and let $E_\bullet = [E_1 \to \cdots \to E_n]$ be a scale filtration with respect to $\langle \cB_\bullet \rangle_{\Sigma_2}$, with associated terminal vertices $s_1,\ldots,s_n \in V(\Sigma_2)_{\rm{term}}$. One can then obtain a scale filtration of $E$ with respect to $\langle \cC_\bullet \rangle_{\Sigma_1}$ in two steps as follows:
\begin{itemize}
    \item[(1)] For any $a =1,\ldots,n$, delete the $a^{\rm{th}}$ step in the filtration if there is a $b \neq a$ such that $f^\dagger(s_a) \leq_{1,\infty} f^\dagger(s_b)$. After renumbering, this gives a filtration $E'_\bullet$ with $\gr_i(E_\bullet') \in \cB_{\leq s_i}$ and $f^\dagger(s_a) <_{i,0} f^\dagger(s_b)$ for all $a<b$.
    \item[(2)] Concatenate $E'_\bullet$ with the preimage of the scale filtration of each $\Pi_i(\gr_i(E'_\bullet)) \in \gr_{f^\dagger(s_i)}(\cC_\bullet)\simeq \gr_{s_i}(\cB_\bullet)$ with respect to the descendent multi-scale decomposition of $\gr_{f^\dagger(s_i)}(\cC_\bullet)$ --- see \Cref{const:descendent}.
\end{itemize}
If $E$ is $0$-well-placed with respect to $\langle \cC_\bullet\rangle_\Sigma$, then the terminal vertices associated to this scale filtration have a maximal element with respect to $\le_{1,0}$. This implies that $\{f^\dagger(s_a):a=1,\ldots,n\}$ has a maximal element $f^\dagger(s_*)$ with respect to $\le_{1,0}$ and \Cref{D:multi-scaleSODcoarsening}(a) now implies that $s_1,\ldots, s_n$ has maximum $f(f^\dagger(s_*)) = s_*$ with respect to $\le_{1,0}$.\endnote{Note that strictly speaking, we have shown that the subset of $s_1,\ldots, s_n$ that were not deleted in Step (1) above has a maximal element with respect to $\le_{1,0}$, call it $s_{\rm{max}}$. However, if $s_a$ was deleted because there exists $b>a$ such that $f^\dagger(s_a)\le_{1,\infty} f^\dagger(s_b)$, then it follows that $s_a\le_{1,0} s_b\le_{1,0}s_{\rm{max}}$.}
\end{proof}

\begin{lem}
\label{L:recovercategories}
    Given a coarsening $\langle \cC_\bullet\rangle_{\Sigma_1}\rightsquigarrow \langle \cB_\bullet\rangle_{\Sigma_2}$ with underlying contraction $f:\Gamma(\Sigma_1)\twoheadrightarrow\Gamma(\Sigma_2)$, we can recover the categories $\cC_{\le v}$ from $\langle \cB_\bullet\rangle_{\Sigma_2}$, $\Sigma_1$, and the descendent decompositions $\gr_{f^\dagger(v)}(\cC_\bullet) = \langle \cC_{\bullet}^{f^\dagger(v)}\rangle_{\Sigma_{\subseteq f^\dagger(v)}}$ for all $v\in V(\Sigma_2)_{\rm{term}}$. 
\end{lem}

\begin{proof}
    First, we record the following identities: 
    \begin{equation}
    \label{E:categories}
        \begin{split}
            \cC_{<f^\dagger(v)} &= \Span \{\cC_{\leq f^\dagger(u)} : u \text{ terminal}, f^\dagger(u) <_{1,\infty} f^\dagger(v) \}\\
            \cC_{\leq f^\dagger(v)} &= \Span \{\cC_{\leq    f^\dagger(u)} : \text{terminal }u \subseteq v \}.
        \end{split}
    \end{equation}
    The first equality comes from the observation that if $w \in V(\Sigma_1)_{\rm{term}}$ and $w \nsubseteq f^\dagger(v)$, then $w \subseteq f^\dagger(u)$ for a unique terminal $u \nsubseteq v$, and $w <_{1,\infty} f^\dagger(v)$ if and only if $f^\dagger(u) <_{1,\infty} f^\dagger(v)$.\endnote{$w$ maps to a terminal vertex $u$ in $V(\Sigma_2)$ under the contraction. In particular, $w \subset f^\dagger(u)$ and this determines $u$. If $u\subseteq v$ then $w\subseteq f^\dagger(u)\subseteq f^\dagger(v)$, so $u\not\subseteq v$. The fact that $w<_{1,\infty} f^\dagger(v)$ if and only if $f^\dagger(u)<_{1,\infty} f^\dagger(v)$ is simply because $w\subseteq f^\dagger(u)$.} The second equality follows from the fact that any $w \in V(\Sigma_1)_{\rm{term}}$ lies below $f^\dagger(x)$ for a unique terminal component $x \subseteq v$, and $\cC_{\leq f^\dagger(x)}$ already contains $\cC_{<f^\dagger(v)}$.

    Next, consider $\{f^\dagger(v):v\in V(\Sigma_2)_{\rm{term}}\}$. By \eqref{E:categories}, for $v\in \Gamma(\Sigma_2)_{\rm{term}}$ we have $\cC_{<f^\dagger(v)} = \Span\{\cC_{\le f^\dagger(u)}:u\in \Gamma(\Sigma_2)_{\rm{term}}, f^\dagger(u)<_{1,\infty} f^\dagger(v)\}$. By \Cref{L:partialorder}, $\le_{1,\infty}$ is a partial order on $\{f^\dagger(v):v\in \Gamma(\Sigma_2)_{\rm{term}}\}$. Consider a maximal chain $f^\dagger(v_1)\le_{1,\infty} \cdots \le_{1,\infty} f^\dagger(v_k)$ for $\le_{1,\infty}$ with all $v_i \in V(\Sigma_2)_{\rm{term}}$. Now, $\cC_{<f^\dagger(v_1)} = 0$ so that $\cC_{\le f^\dagger(v_1)} = \gr_{f^\dagger(v_1)}(\cC_\bullet)$. For each terminal $w\subseteq f^\dagger(v_1)$ we recover $\cC_{\le w}$ from $\gr_{f^\dagger(v_1)}(\cC_\bullet) = \langle \cC_{\bullet}^{f^\dagger(v_1)}\rangle_{\Sigma_{\subseteq f^\dagger(v_1)}}$. Suppose that $\cC_{\le f^\dagger(v_{i-1})}$ and all $\cC_{\le w}$ for terminal $w\subseteq f^\dagger(v_{i-1})$ have been determined. Then $\cC_{<f^\dagger(v_i)} = \cC_{\le f^\dagger(v_{i-1})}$ and we recover $\cC_{\le f^\dagger(v_i)}$ as the preimage of $\gr_{f^\dagger(v_i)}(\cC_\bullet)$ under $\cC\twoheadrightarrow \cC/\cC_{\le f^\dagger(v_{i-1})}$. For any terminal $w\subset f^\dagger(v_i)$, we recover $\cC_{\le w}$ as the preimage of $\cC_{\le w}/\cC_{<f^\dagger(v_i)}$, which we know from the descendent decomposition at $f^\dagger(v_i)$. 
\end{proof}

\subsection{Admissible multi-scale decompositions}
Given a triangulated category $\cC$, a full triangulated subcategory $\cA$ is called \emph{admissible} if the inclusion functor $\cA\hookrightarrow \cC$ admits both left and right adjoints. This is equivalent to $\cA$ fitting into two semiorthogonal decompositions $\cC = \langle \cA,{}^\perp\cA\rangle$ and $\cC = \langle \cA^{\perp}, \cA\rangle$ --- see \cite{B-KSerre}.

\begin{defn}[Admissible multi-scale decomposition] \label{D:admissible_decomposition}
A multi-scale decomposition $\cC = \langle \cC_\bullet \rangle_\Sigma$ is \emph{admissible} if for any pair of vertices $u_1,u_2 \in V(\Sigma)$ that lie immediately below the same vertex $v$, i.e. are \emph{subjacent}, such that $u_1 \leq_{1,\infty} u_2$, the inclusion functor $\cC_{\leq u_1} \hookrightarrow \cC_{\leq u_2}$ is admissible.
\end{defn}

We will frequently use the following basic fact.

\begin{lem}
\label{L:admissiblecorrespondence}
    Given a triangulated category $\cC$ and an admissible subcategory $\cA$, there is a bijection between strict admissible filtrations $\cA\subset \cB \subset \cC$ and strict admissible subcategories of $\cC/\cA$ given by $[\cA \subset \cB\subset \cC]\mapsto \cB/\cA$ and $[\cB'\subset \cC/\cA]\mapsto [\cA\subset \cB\subset \cC]$ where $\cB$ is the essential preimage of $\cB'$ under $\cC\to \cC/\cA$.
\end{lem}

\begin{proof}
    Omitted.\endnote{In what follows, subcategories are assumed to be full and strict. The correspondence between triangulated sub\-categories of $\cC$ containing $\cA$ and subcategories of $\cC/\cA$ follows from \cite{Verdierquotient}*{V, Ch. II, Prop. 2.3.1}. We just need to show that $\cB\supset \cA$ is admissible if and only if its essential image in $\cC/\cA$ is.
    
    First, consider an admissible subcategory $\cB \hookrightarrow \cC$. Then, $\cA\hookrightarrow \cB$ is an admissible inclusion. Consequently, we can form a semiorthogonal decomposition $\cC = \langle \cA,{}^\perp\cA \cap \cB,{}^\perp\cB\rangle$, and $\cB/\cA \hookrightarrow \cC/\cA$ is identified with the admissible inclusion ${}^\perp \cA \cap \cB \hookrightarrow \langle {}^\perp \cA \cap \cB,{}^\perp \cB\rangle$. Thus, $\cB/\cA$ is an admissible subcategory of $\cC/\cA$. Similar reasoning shows that if $\cB/\cA$ is an admissible subcategory of $\cC/\cA$, identifying ${}^\perp \cA$ with $\cC/\cA$, we have an admissible inclusion $\cB/\cA \hookrightarrow {}^\perp \cA$ which now implies that $\cB/\cA$ lifts to an admissible subcategory in $\cC$.}
\end{proof}

\begin{lem}
\label{L:admissiblevertexinclusion}
    If $\cC = \langle \cC_\bullet\rangle_\Sigma$ is an admissible multi-scale decomposition, then for any $v\in V(\Sigma)$
    \begin{enumerate}  
        \item the inclusion $\cC_{<v}\hookrightarrow\cC_{\le v}$ is admissible; and 
        \item for any $w\subseteq v$ the inclusion $\cC_{\le w}\hookrightarrow \cC_{\le v}$ is admissible.
    \end{enumerate}
\end{lem}

\begin{proof}
    We begin with (1). There are two cases: suppose first that $v$ is subjacent to $x\in V(\Sigma)$ and there exists $u$ subjacent to $x$ such that $u \le_{1,\infty} v$. Choosing the maximal such $u$ with respect to $\le_{1,\infty}$, it follows that $\cC_{<w} = \cC_{\le u}$ and the claim follows from \Cref{D:admissible_decomposition}. 
    
    If there is no such $u$, then $\cC_{<v} = \cC_{<x}$. Induction on the maximal length of a chain from $v$ to the root allows us to assume that $\cC_{<x}\hookrightarrow \cC_{\le x}$ is admissible. Next, using \Cref{D:admissible_decomposition} one can show that $\cC_{\le v}^x$ is part of an admissible filtration of $\gr_x(\cC_\bullet)$ and so lifting gives an admissible inclusion $\cC_{<v} = \cC_{<x} \hookrightarrow \cC_{\le v}$.\endnote{Consider the large scale semiorthogonal decomposition of $\gr_x(\cC_\bullet)$. There is a $w$ subjacent to $x$ such that $v\le_{1,\infty} w$ and $\cC_{\le w}^x \hookrightarrow \gr_x(\cC_\bullet)$ is an admissible inclusion. So, $\cC_{\le w} \hookrightarrow \cC_{\le x}$ is admissible. Then, by \Cref{D:admissible_decomposition}, $\cC_{\le v} \hookrightarrow \cC_{\le w}$ is admissible. Consequently, the inclusion $\cC_{\le v}\hookrightarrow \cC_{\le x}$ is admissible and it follows that $\cC_{<x} \hookrightarrow \cC_{\le v}\hookrightarrow \cC_{\le x}$ is an admissible filtration, as needed.}

    For (2), it suffices to prove the result when $w$ is subjacent to $v$. The large scale semi\-orthogonal decomposition of $\gr_v(\cC_\bullet)$ implies that there exists a $u$ subjacent to $v$ such that $w\le_{1,\infty} u$ (possibly equal) and $\cC_{\le u}^v \subset \gr_v(\cC_\bullet)$ is admissible. Since $\cC_{<v} \subset \cC_{\le v}$ is admissible, it follows from \Cref{L:admissiblevertexinclusion} that $\cC_{\le u}\subset \cC_{\le v}$ is admissible and by \Cref{D:admissible_decomposition} it follows that $\cC_{\le w}\subset \cC_{\le u}$ is. So, $\cC_{\le w}\subseteq \cC_{\le v}$ is an admissible inclusion.
\end{proof}

\begin{cor}
\label{C:admissibledesc}
    If $\cC = \langle \cC_\bullet\rangle_\Sigma$ is an admissible multi-scale decomposition, then for any $v\in V(\Sigma)$ the descendent decomposition $\gr_v(\cC_\bullet) = \langle \cC_{\bullet}^v\rangle_{\Sigma_{\subseteq v}}$ is also admissible.
\end{cor}

\begin{proof}
    This follows from \Cref{L:admissiblecorrespondence} and \Cref{L:admissiblevertexinclusion}.\endnote{
    Suppose that $x\subseteq v$ is given and consider vertices $u_1$ and $u_2$ subjacent to $x$ with $u_1\le_{1,\infty} u_2$. We have a chain of admissible inclusions: $\cC_{\le u_1}\subseteq \cC_{\le u_2}\subseteq \cC_{\le x} \subseteq \cC_{\le v}$ using \Cref{D:admissible_decomposition} and \Cref{L:admissiblevertexinclusion}(2). Next, $\cC_{<v}$ is an admissible subcategory of $\cC_{\le v}$ contained in $\cC_{\le u_1}$ by \eqref{E:vertex_category_containment}. Consequently, $\cC_{<v}\subseteq \cC_{\le u_1}\subseteq \cC_{\le u_2}\subseteq \cC_{\le x} \subseteq \cC_{\le v}$ is an admissible filtration and by \Cref{L:admissiblecorrespondence} it follows that $\cC_{\le u_1}^v\hookrightarrow\cC_{\le u_2}^v$ is an admissible inclusion in $\gr_v(\cC_\bullet)$.}
\end{proof}

\begin{cor}
\label{C:admissibleMSD}
    Consider a multi-scale line $\Sigma$ and a collection of categories $\cC_\bullet = \{\cC_{\le v}\subseteq \cC: v\in V(\Sigma)_{\rm{term}}\}$ as in \Cref{D:multi-scale_decomposition}. If the pair $(\Sigma,\cC_\bullet)$ satisfies \Cref{D:multi-scale_decomposition}(1),(2),(4) and the condition of \Cref{D:admissible_decomposition}, then \Cref{D:multi-scale_decomposition}(3) follows. That is, the pair $(\Sigma,\cC_\bullet)$ defines an admissible multi-scale decomposition of $\cC$. 
\end{cor}

\begin{proof}
    If $w$ is subjacent to the root, consider the large scale semiorthogonal decomposition $\cC = \langle \cC_{\le u_{i_1}},\ldots, \cC_{\le u_{i_k}}\rangle$ as in \Cref{const:SODfrommulti-scale}. If $w = u_{i_j}$ for some $j$ then $\cC_{\le w} = \cC_{\le u_{i_j}}$ is thick, being a semiorthogonal factor. If not, then by \Cref{D:admissible_decomposition}, $\cC_{\le w}$ is a thick subcategory of some $\cC_{\le u_{i_j}}$ and thus thick in $\cC$. 

    Suppose we know the result for $v$ and let $w$ be a vertex subjacent to it. Consider the category $\gr_v(\cC_\bullet)$ and the pair $(\Sigma_{\subseteq v},\cC_\bullet^v=\{\cC_{\le u}^v\})$ where $u$ ranges over terminal vertices of $\Sigma_{\subseteq v}$. One can verify directly that $(\Sigma_{\subseteq v}, \cC_\bullet^v)$ satisfies \Cref{D:multi-scale_decomposition}(1),(2),(4); $(\Sigma_{\subseteq v}, \cC_\bullet^v)$ satisfies the property of \Cref{D:admissible_decomposition} by the same argument as the one used to prove \Cref{C:admissibleMSD}. The base case implies that $\cC_{\le w}^v$ is a thick subcategory of $\gr_v(\cC_\bullet)$ and by \cite{Verdierquotient}*{Ch. II, Prop. 2.3.1} it follows that $\cC_{\le w}$ is thick in $\cC$.
\end{proof}

\subsubsection{Constructing coarsenings} Next, we discuss 
constructing a coarsening $\langle \cC_\bullet\rangle_{\Sigma_1} \rightsquigarrow \langle \cB_\bullet\rangle_{\Sigma_2}$ from a multi-scale decomposition $\cC = \langle \cC_\bullet\rangle_{\Sigma_1}$ and a contraction $\Gamma(\Sigma_1)\twoheadrightarrow \Gamma(\Sigma_2)$.

\begin{lem}
\label{L:admissiblecontractions}
    If $\langle \cC_\bullet \rangle_{\Sigma_1}$ is admissible, $\Sigma_2$ is a multi-scale line, and $f : \Gamma(\Sigma_1) \twoheadrightarrow \Gamma(\Sigma_2)$ is a contraction satisfying \Cref{D:multi-scaleSODcoarsening}(a), then $f$ underlies a coarsening $\langle \cC_\bullet \rangle_{\Sigma_1} \rightsquigarrow \langle \cB_\bullet \rangle_{\Sigma_2}$ such that $\langle \cB_\bullet\rangle_{\Sigma_2}$ is also admissible. 
\end{lem}

\begin{proof}
    First, we prove the lemma in the case where $f$ is bijective and $\leq_{i,0}$ is a total order on $V(\Sigma_2)_{\rm{term}}$. Let $u_1,\ldots, u_n \in V(\Sigma_1)$ denote the vertices subjacent to the root, ordered as in \Cref{const:SODfrommulti-scale}. To construct such a coarsening $f:\langle \cC_\bullet\rangle_{\Sigma_1} \rightsquigarrow \langle \cB_\bullet\rangle_{\Sigma_2}$, it suffices to give \vspace{-2mm}
    \begin{enumerate} 
        \item categories $\cB_{\le u_i}$ for $i=1,\ldots,n$ forming a semiorthogonal decomposition of $\cC$, ordered using $\le_{i,0}$ on $\Sigma_2$, such that $\cB_{\le u_i} \hookrightarrow \cC_{\le u_i}$ fits into a semiorthogonal decomposition $\langle \cC_{<u_i},\cB_{\le u_i}\rangle = \cC_{\le u_i}$; and \vspace{-2mm}
        \item for each $i=1,\ldots, n$, a coarsening of $\gr_{u_i}(\cC_\bullet) = \langle \cC_\bullet^{u_i}\rangle_{(\Sigma_1)_{\subseteq u_i}}$ compatible with the restriction of $f$ to a bijective contraction $\Gamma((\Sigma_1)_{\subseteq u_i}) \twoheadrightarrow \Gamma((\Sigma_2)_{\subseteq f(u_i)})$.\endnote{We will identify $u\in V(\Sigma_1)$ with its image under $f$. For each $u_i$, we have a multi-scale decomposition $\langle \cB_\bullet^{u_i}\rangle_{(\Sigma_2)_{{\subseteq u_i}}}$ which is a coarsening of the descendent decomposition of $\gr_{u_i}(\cC_\bullet)$. First, for each $v\in V(\Sigma_2)_{\rm{term}}$ with $v\subset u_i$ define $\cB_{\le v}$ as $\cB_{\le v}^{u_i}\subset \cB_{\le u_i}$. It is a quick exercise to verify that defined this way, $\langle \cB_\bullet\rangle_{\Sigma_2}$ defines a (generic) admissible multi-scale decomposition of $\cC$. We check that $f:\langle \cC_\bullet\rangle_{\Sigma_1}\rightsquigarrow\langle \cB_\bullet\rangle_{\Sigma_2}$ is a coarsening as claimed. 
        
        First, note that since $\langle \cC_{<u_i},\cB_{\le u_i}\rangle = \cC_{\le u_i}$, the composite $\cB_{\le u_i}\hookrightarrow \cC_{\le u_i}\twoheadrightarrow \gr_{u_i}(\cC_\bullet)$ is an equivalence for $i=1,\ldots, n$. For any $v\subseteq u_i$, we have a commutative square 
        \[
        \begin{tikzcd}[ampersand replacement = \&]
            \cB_{\le v}^u\arrow[r,"\sim"] \& \cC_{\le v}^u/\cC_{<v}^u\\
            \cB_{\le v}\arrow[r] \arrow[u,"\sim",sloped]\& \cC_{\le v}/\cC_{<v}\arrow[u,"\sim",swap,sloped]
        \end{tikzcd}
        \]
        which implies that $\cB_{\le v}\to \cC_{\le v}/\cC_{<v}$ is an equivalence.}
    \end{enumerate}
    
    By an inductive argument, it suffices to construct the categories $\cB_{\leq u_j}$ in (1).\endnote{The base case is where $\Sigma_1 \cong \bP^1$, in which case there is nothing to prove. The claim (2) is the inductive hypothesis.} Just as in \Cref{const:SODfrommulti-scale}, we have indices $0=i_0<i_1<\cdots<i_k=n$ and a semiorthogonal decomposition $\cC = \langle \cC_{\leq u_{i_1}}, \ldots \cC_{\leq u_{i_k}} \rangle$, along with admissible inclusions $\cC_{\leq u_i} \subseteq \cC_{\leq u_j}$ whenever $i_{\ast-1}<i<j\leq i_\ast$.

    For any $i_{\ast-1}<j\leq i_\ast$, we inductively define a full subcategory $\cB_{\leq u_j} \subset \cC_{\leq u_j}$ consisting of objects that are right orthogonal to $\cB_{\leq u_a}$ for all $i_{\ast-1}<a<j$ with $f(u_j) \leq_{i,0} f(u_a)$ and left orthogonal to $\cB_{\leq u_a}$ for all $i_{\ast-1}<a<j$ with $f(u_a) \leq_{i,0} f(u_j)$. This gives a semiorthogonal decomposition of $\cC_{\leq u_{i_\ast}}$ whose terms are $\cB_{\leq u_j}$ for $i_{\ast-1}<j\leq i_\ast$, where the categories are ordered according to the total order $\leq_{i,0}$ on $\Sigma_2$. Finally, note that each inclusion $\cB_{\le u_j}\hookrightarrow\cC_{\le u_j}$ fits into a semiorthogonal decomposition $\cC_{\le u_j} = \langle \cC_{<u_j},\cB_{\le u_j}\rangle$ for all $j$.

    Now we consider a general contraction $f:\Gamma(\Sigma_1) \to \Gamma(\Sigma_2)$ satisfying \Cref{D:multi-scaleSODcoarsening}(a). By perturbing the configuration of nodes on each level slightly, we can construct a bijective contraction $g : \Gamma(\Sigma_1) \to \Gamma(\Sigma_1')$ with the property that $\leq_{i,0}$ is a total ordering on $\Sigma_1'$, and such that if either $f(u) \leq_{i,0} f(v)$ or $f(u) \leq_{1,0} f(v)$, then the same relation holds between $g(u)$ and $g(v)$. We are now in the situation above, so we have a coarsening $\langle \cC_\bullet \rangle_{\Sigma_1} \rightsquigarrow \langle \cD_\bullet \rangle_{\Sigma_1'}$. For all $u \in V(\Sigma_2)_{\rm{term}}$, we let
    \[
        \cB_{\leq u} := \Span \{\cD_{\leq v} : v \in V(\Sigma_1')_{\rm{term}} \text{ and } f(v) \leq_{1,\infty} u\}.
    \]
    The reader can verify, using the coarsening $\langle \cC_\bullet\rangle_{\Sigma_1}\rightsquigarrow\langle \cD_\bullet\rangle_{\Sigma_1'}$, that $\langle \cB_\bullet\rangle_{\Sigma_2}$ is an admissible multi-scale decomposition such that $f:\langle \cC_\bullet\rangle_{\Sigma_1}\rightsquigarrow \langle \cB_\bullet\rangle_{\Sigma_2}$ is a coarsening.\endnote{First, we check that $\langle \cB_\bullet\rangle_{\Sigma_2}$ is an admissible multi-scale decomposition. First of all, the subcategories $\cB_{\le u}$ are thick and triangulated since they are generated by factors of a semiorthogonal decomposition of $\cC$. \Cref{D:multi-scale_decomposition}(1) is by 
    \begin{align*}
        \cB_{\le u} \cap \cB_{\le w} &= \Span\{\cD_{\le v}|f(v)\le_{1,\infty} u\} \cap \Span\{\cD_{\le v}|f(v)\le_{1,\infty} w\}\\
        &= \Span\{\cD_{\le v}|f(v)\le_{1,\infty} u \text{ and } w\}\\
        & = \Span \{\cD_{\le v}|f(v)\le_{1,\infty} x \text{ and } x\le_{1,\infty} u,w\}\\
        & = \Span\{\cB_{\le x}| x\le_{1,\infty} u,w\}.
    \end{align*}
    where $v \in V(\Sigma_1')_{\rm{term}}$ above. For \Cref{D:multi-scale_decomposition}(2), we have just seen that
    \[ 
        \cB_{\le u}\cap \cB_{\le w} = \Span\{\cD_{\le v}|f(v)\le_{1,\infty} u \text{ and } w\}.
    \]
    Now, consider $u,w\in \Gamma(\Sigma_2)_{\rm{term}}$ such that $u<_{i,0}w$ and $v,v'\in \Gamma(\Sigma_1')_{\rm{term}}$ such that $f(v)\le_{1,\infty} u$ and $f(v')\le_{1,\infty} w$. It is not possible for $f(v)\vee f(v') \subsetneq u\vee w$ given these conditions. Indeed, if this were the case, then we would have $f(v) \le_{1,\infty} u$ and $w$, which is impossible.
    
    If $f(v)\vee f(v') = u\vee w$, then $f(v)<_{i,0} f(v')$ and thus $v<_{i,0} v'$ in $\Sigma_1'$ (here, we are suppressing $g$ from the notation). It follows that $\Hom_{\cC}(\cD_{\le v'},\cD_{\le v}) = 0$. On the other hand, if $u\vee w \subsetneq (f(v) \vee f(v'))\vee (u\vee w)$ then $f(v)$ or $f(v')<_{1,\infty} w$ and $u$ so that $\cD_{\le v}$ or $\cD_{\le v'}$ is in $\cB_{\le u}\cap \cB_{\le w}$, and so $\Hom(\cB_{\le u},\cB_{\le w}) = 0$ in $\cB/(\cB_{\le u}\cap \cB_{\le w})$.

    For \Cref{D:multi-scale_decomposition}(3), for any $w\in V(\Sigma_2)$, $\cB_{\le w}$ and $\cB_{<w}$ are of the form $\Span\{\cD_{\le v}|v\in A\}$ for $A \subseteq \Gamma(\Sigma_1')_{\rm{term}}$. Since the $\cD_{\le v}$ are factors in a semiorthogonal decomposition, this implies thickness of $\cB_{\le w}$ and $\cB_{<w}$. For \Cref{D:multi-scale_decomposition}(4), the first part is because $\gr_w(\cB_\bullet) = \cB_{\le w}/\cB_{<w} = \Span \{\cD_{\le v}|f(v)\subseteq w\}\ne 0$. The second part is immediate from the fact that the $\cD_{\le v}$ for $v\in V(\Sigma_1')_{\rm{term}}$ form a semiorthogonal decomposition of $\cC$.
    
    For \Cref{D:admissible_decomposition}, simply note that given $x\in V(\Sigma_2)$ and $u,w$ subjacent to it with $u\le_{1,\infty} w$, we have $\cB_{\le u} = \Span\{\cD_{\le v}|v\in A\}$ and $\cB_{\le w} = \Span\{\cD_{\le v}|v\in B\}$ with $A\subset B \subset \Gamma(\Sigma_1')_{\rm{term}}$. Hence, $\cB_{\le u}\hookrightarrow \cB_{\le w}$ is admissible. 

    Next, we check that $f:\Gamma(\Sigma_1)\twoheadrightarrow \Gamma(\Sigma_2)$ defines a coarsening $\langle \cC_\bullet\rangle_{\Sigma_1}\rightsquigarrow \langle \cB_\bullet\rangle_{\Sigma_2}$. \Cref{D:multi-scaleSODcoarsening}(a) is by hypothesis, so it remains to check (b). First, note that for any $u\in V(\Sigma_2)$, 
    \begin{align*}
        \cB_{\le u} & = \Span\{\cD_{\le v}|f(v)<_{1,\infty} u\text{ and } f(v)\subseteq u\}\\
        \cB_{<u} & = \Span\{\cD_{\le v}|f(v)<_{1,\infty} u\}
    \end{align*}
    where here and afterwards $v\in V(\Sigma_1')_{\rm{term}}$. Suppose $f(v)<_{1,\infty} u$. Then $f^\dagger f(v)<_{1,\infty} f^\dagger(u)$. Since $v\subset f^\dagger f(v)$, we have $\cD_{\le v}\subset \cD_{\le f^\dagger f(v)} \subset \cC_{\le f^\dagger f(v)}\subset \cC_{\le f^\dagger(u)}$. This implies that $\cB_{<u}\subset \cC_{<f^\dagger(u)}$. If $f(v)\subseteq u$, then $f^\dagger f(v)\subseteq f^\dagger(u)$ implies that $\cD_{\le v}\subset \cD_{\le f^\dagger f(v)} \subset \cC_{\le f^\dagger f(v)}\subset \cC_{\le f^\dagger(u)}$. So, we have inclusions $\cB_{\le u} \hookrightarrow \cC_{\le f^\dagger(u)}$ and $\cB_{<u}\hookrightarrow \cC_{<f^\dagger(u)}$ which give an induced map $\gr_u(\cB_\bullet) \to \gr_{f^\dagger(u)}(\cC_\bullet)$. It remains to check that this induced map is an equivalence. 
    
    First, note that by the description of the categories $\cB_{\le u}$ and $\cB_{<u}$, we have that the composite 
    \[ 
        \Span\{\cD_{\le v}|f(v)\subseteq u\}\hookrightarrow \cB_{\le u} \twoheadrightarrow \gr_u(\cB_\bullet)
    \] 
    is an equivalence. Since $g:\langle \cC_\bullet\rangle_{\Sigma_1}\rightsquigarrow \langle \cD_\bullet\rangle_{\Sigma_1'}$ is a coarsening, the composite 
    \[ 
        \cD_{\le f^\dagger(u)}\hookrightarrow \cC_{\le f^\dagger(u)} \twoheadrightarrow \gr_{f^\dagger(u)}(\cC_\bullet)
    \]
    induces an equivalence $\gr_{f^\dagger(u)}(\cD_\bullet)\to \gr_{f^\dagger(u)}(\cC_\bullet)$. However, the identities
    \begin{align*}
        \cD_{\le f^\dagger(u)} & = \Span \{\cD_{\le v}|v<_{1,\infty} f^\dagger(u) \text{ or } v\subseteq f^\dagger(u)\}\\
        \cD_{<f^\dagger(u)} & = \Span \{\cD_{\le v}|v<_{1,\infty}f^\dagger(u)\}
    \end{align*}
    imply that the composite $\Span\{\cD_{\le v}|f(v) = u\} \hookrightarrow \cC_{\le f^\dagger(u)} \twoheadrightarrow \gr_{f^\dagger(u)}(\cC_\bullet)$ is an equivalence. Thus, the induced map $\gr_u(\cB_\bullet)\to \gr_{f^\dagger(u)}(\cC_\bullet)$ is an equivalence.
    }
\end{proof}

\begin{lem}
    If $\cC$ is smooth and proper, and there is a coarsening $\langle \cC_\bullet \rangle_{\Sigma_1} \rightsquigarrow \langle \cB_\bullet \rangle_{\Sigma_2}$ such that $f : \Gamma(\Sigma_1) \to \Gamma(\Sigma_2)$ is bijective and $\leq_{i,0}$ is a total order on $V(\Sigma_2)_{\rm{term}}$, then $\langle \cC_\bullet \rangle_{\Sigma_1}$ is admissible.    
\end{lem}

\begin{proof}
    Since $f$ is a bijection, we suppress it from the notation, writing $v$ in lieu of $f(v)$. Consider the large scale semiorthogonal decomposition (\Cref{const:SODfrommulti-scale}) of $\langle \cC_\bullet\rangle_{\Sigma_1}$. It follows that $\cC_{<u_1} = 0$ so by \Cref{L:coarseningconsequences}, $\cC_{\le u_1} = \Span\{\cC_{<u_1},\cB_{\le u_1}\} = \cB_{\le u_1}$. Next, for all $j\le i_1$ one has $\cC_{< u_j} = \cC_{\le u_{j-1}}$ and by induction $\cC_{\le u_{i_p}} = \Span\{\cB_{\le u_1},\ldots, \cB_{\le u_{i_p}}\}$ for all $1\le p \le i_1$. 
    
    Since $\le_{i,0}$ is a total order on $V(\Sigma_2)_{\rm{term}}$, up to possible reindexing, $\Hom(\cB_{\le u_j},\cB_{\le u_i}) = 0$ for all $1\le i<j \le i_1$ and so $\cC_{\le u_p} = \langle \cB_{\le u_{i_1}},\ldots, \cB_{\le u_{i_p}}\rangle$ for each $1\le p \le i_1$. Similarly, for all $i_{*-1}<p\le i_*$, we have $\cC_{\le u_p} = \langle \cB_{\le u_{j}}:i_{*-1}<j\le p\rangle$, which implies the claimed admissibility. Let $v,w \in V(\Sigma_1)$ be given which are subjacent to the root. We have proven:
    \begin{enumerate}
        \item $v\le_{1,\infty} w$ implies $\cC_{\le v}\hookrightarrow \cC_{\le w}$ is admissible; and
        \item $\cC_{<v}\hookrightarrow \cC_{\le v}$ is admissible.
    \end{enumerate}
    Next, consider a vertex $v$ such that $\cC_{<v}\hookrightarrow \cC_{\le v}$ is admissible. Using the fact that $\langle \cC_\bullet^{u}\rangle_{\Sigma_{1,\subseteq u}}\rightsquigarrow \langle \cB_\bullet^u\rangle_{\Sigma_{2,\subseteq u}}$ is a coarsening of multi-scale decompositions of $\gr_u(\cC_\bullet)$,\endnote{\Cref{D:multi-scaleSODcoarsening}(a) is immediate from the corresponding condition of $\langle \cC_\bullet\rangle_{\Sigma_1}\rightsquigarrow \langle \cB_\bullet\rangle_{\Sigma_2}$. For \Cref{D:multi-scaleSODcoarsening}(b), consider $w\in V(\Sigma_2)$ such that $w\subseteq v$. Since $\cB_{\le w}\subset \cC_{\le w}$ it follows that $\cB_{\le w}^v\subset \cC_{\le w}^v$. Finally, the equivalence $\cB_{\le w}^v/\cB_{<w}^v \simeq \gr_w(\cB_\bullet)$ and the fact that $\cB_{\le w} \to \gr_{f^\dagger(w)}(\cC_\bullet)$ induces an equivalence $\gr_w(\cB_\bullet)\simeq \gr_{f^\dagger(w)}(\cC_\bullet)$ is enough to imply \Cref{D:multi-scaleSODcoarsening}(b).} the above reasoning implies that for all $x,y$ subjacent to $v$ with $x\le_{1,\infty} y$, $\cC_{\le x}^v\hookrightarrow \cC_{\le y}^v$ is admissible. Using the fact that $\cC_{<v}\hookrightarrow \cC_{\le v}$ is admissible, \Cref{L:admissiblecorrespondence} implies that $\cC_{\le x}\hookrightarrow \cC_{\le y}$ is admissible. This implies that $\cC_{<x}\hookrightarrow \cC_{\le x}$ is admissible for all $x$ subjacent to $v$ and so proceeding by induction, the result follows.
\end{proof}

\subsection{Augmented stability conditions}
\label{S:augmentedstabilityconditions}

Let us fix a finitely generated free abelian group $\Lambda$ and a homomorphism $\ch :  \rm{K}_0(\cC) \to \Lambda$ that is surjective after tensoring with $\bQ$. Given a multi-scale decomposition $\cC = \langle \cC_\bullet \rangle_\Sigma$, for any $v \in V(\Sigma)$ we define
\[
\Lambda_{\leq v} := \ch(\cC_{\leq v})^{\rm{sat}}, \quad \Lambda_{<v} := \left(\sum_{u <_{1,\infty} v} \Lambda_{\leq u} \right)^{\rm{sat}}, \text{ and} \quad \gr_v(\Lambda) := \Lambda_{\leq v} / \Lambda_{<v},
\]
where $G^{\rm{sat}} := \{x \in \Lambda | m x \in G \text{ for some }m\in \bZ\}$ denotes the saturation of a subgroup $G$ of $\Lambda$. By the exact sequence $ \rm{K}_0(\cC_{<v}) \to  \rm{K}_0(\cC_{\leq v}) \to  \rm{K}_0(\gr_v(\cC_\bullet)) \to 0$ given by \cite{Kellerdgcat}*{Thm. 5.2}, $\ch$ descends uniquely to a homomorphism $\ch :  \rm{K}_0(\gr_v(\cC_\bullet)) \to \gr_v(\Lambda)$ that is surjective after tensoring with $\bQ$.\endnote{Consider the diagram 
\[
    \begin{tikzcd}[ampersand replacement =\&]
         \rm{K}_0(\cC_{<v})\arrow[r]\arrow[d] \&  \rm{K}_0(\cC_{\le v})\arrow[r]\arrow[d] \&  \rm{K}_0(\gr_v(\cC_\bullet))\arrow[r]\arrow[d]\& 0 \\
        \Lambda_{<v}\arrow[r] \& \Lambda_{\le v}\arrow[r]\& \gr_v(\Lambda)\arrow[r]\&0.
    \end{tikzcd}
\]
After tensoring with $\bQ$, the composite $ \rm{K}_0(\cC_{\le v})_{\bQ} \to (\Lambda_{\le v})_{\bQ} \to \gr_v(\Lambda)_{\bQ}$ is surjective. In particular, $ \rm{K}_0(\gr_v(\cC_\bullet))_{\bQ}\to \gr_v(\Lambda)_{\bQ}$ is also surjective.}

\begin{defn}[Augmented stability conditions] \label{D:generalized_stability_condition}
An \emph{augmented stability condition} on $\cC$, denoted $\sigma = \langle \cC_\bullet | \ell_\bullet \rangle_\Sigma$, consists of a multi-scale decomposition $\cC = \langle \cC_\bullet \rangle_\Sigma$ and for each terminal component $v$ of $\Sigma$ a function $\ell_v : \gr_v(\cC_\bullet) \setminus 0 \to \Sigma_v\setminus \{p_{\rm{ascending}}\}$. We require that after choosing an isomorphism $(\Sigma_v\setminus \{p_{\rm{ascending}}\}, \omega_v) \cong (\bC,dz)$ each $\ell_v$ determines an element $\sigma_v \in \Stab_{\gr_v(\Lambda)}(\gr_v(\cC_\bullet))$ by applying \Cref{P:logZ_functions}, and that there exists a direct sum decomposition $\Lambda \otimes \bQ = \bigoplus_{v \in V(\Sigma)_{\rm{term}}} M_{v}$ such that $\Lambda_{\leq v} \otimes \bQ = \sum_{u \leq_{1,\infty} v} M_u$ for all terminal $v$.

We let $\Astab(\cC)$ denote the set of equivalence classes of augmented stability conditions on $\cC$, where two augmented stability conditions are equivalent if there is a real oriented isomorphism of underlying multi-scale lines that is compatible with the subcategories $\cC_{\leq v}$ and the marking functions $\ell_v$.
\end{defn}

Note that any two isomorphisms $(\Sigma_v\setminus \{p_{\rm{ascending}}\},\omega_v) \xrightarrow{\sim}(\bC,dz)$ differ by a translation. So, for each $v\in V(\Sigma)_{\rm{term}}$, $\ell_v$ canonically determines an element of $\Stab_{\gr_v(\Lambda)}(\gr_v(\cC_\bullet))/\bC$ by \Cref{P:logZ_functions}. For any $0$-well-placed objects $E,F \in \cC$ with $\dom(E)=\dom(F)=v$, we define functions
\begin{equation}
\label{E:functionsastab}
    \begin{split}
        m_{\sigma_\bullet}^t(F/E) &:= m^t_{\sigma_v}(\Pidom(F)/\Pidom(E)) \\
        \ell_{\sigma_\bullet}^t(F/E) &:= \ell_{\sigma_v}^t(\Pidom(F)/\Pidom(E)) \\
        \phi_{\sigma_\bullet}^t(F/E) &:= \phi_{\sigma_v}^t(\Pidom(F)/\Pidom(E))
    \end{split}
\end{equation}
As in the case of log central charge, we can regard $E \mapsto \ell^t_{\sigma_v}(\Pidom(E))$ as a function 
\[
    \ell_{\sigma_\bullet}^t : \{0\text{-well-placed } E \in \cC \} \to \Sigma
\]
whose image lies in the smooth loci of the terminal components.

\begin{ex}
    An element of $\Astab(\cC)$ with underlying curve $\bP^1$ is specified by a function $\ell : \cC\setminus 0 \to \bC$ up to translation, i.e. an element of $\Stab(\cC)/\bC$. Consequently, the set of such points is in canonical bijection with $\Stab(\cC)/\bC$.
\end{ex}

\begin{ex}
    Suppose next that $\Sigma$ is a multi-scale line with exactly two levels. By \Cref{ex:twolevelmulti-scaledecomp}, an augmented stability condition $\langle \cC_\bullet | \ell_\bullet\rangle_\Sigma$ consists of a semiorthogonal de\-composition of $\cC$ labeled by a strictly increasing sequence of imaginary numbers, a filtration of each semiorthogonal factor labeled by a strictly increasing sequence of real numbers, and a stability condition on each associated graded category for these filtrations. Except for the labels, these are the structures obtained as limits of quasi-convergent paths in \cite{quasiconvergence}*{\S 2}. One obtains equivalent augmented stability conditions by simultaneously scaling all of the labels, simultaneously translating the real or imaginary labels, or translating one of the stability conditions by an element of $\bC$.
\end{ex} 

Finally, let us introduce some useful terminology regarding augmented stability conditions:
\begin{defn}\label{D:terminology} Let $\cC$ be a stable dg-category.
    \begin{itemize}
        \item Given $\sigma = \langle \cC_\bullet | \sigma_\bullet \rangle_\Sigma \in \Astab(\cC)$, we say that an object $E \in \cC$ is $\sigma$-\emph{stable} (respectively \emph{semistable}) if it is $\infty$-well-placed, and its image in $\gr_v(\cC_\bullet)$ is $\sigma_v$-stable (respectively $\sigma_v$-semistable), where $v = \dom(E)$. $\cC_{\le v}^{\rm{ss}} \subseteq \cC_{\le v}$ denotes the full subcategory of $\sigma$-semistable objects with dominant vertex $v$. \smallskip

        \item Given $\sigma,\tau \in \Astab(\cC)$, we say that $\tau$ is a \emph{coarsening} of $\sigma$, and write $\sigma \rightsquigarrow \tau$, if the same holds for their underlying multi-scale decompositions.\smallskip

        \item We say that $\langle \cC_\bullet | \sigma_\bullet \rangle_\Sigma$ is \emph{generic} if $\Sigma$ has the property that $\Im(\mathfrak{p}(u,v)) \neq 0$ for any pair of distinct $u,v \in \Gamma(\Sigma)_{\rm{term}}$, i.e., $\leq_{i,0}$ is a total order on $\Gamma(\Sigma)_{\rm{term}}$. \smallskip

        \item Let $\sigma = \langle \cC_\bullet|\ell_\bullet\rangle_{\Sigma}$ denote an augmented stability condition. Given $v\in V(\Sigma)$, we can form the descendent decomp\-osition $\gr_v(\cC_\bullet) = \langle \cC_\bullet^v\rangle_{\Sigma_{\subseteq v}}$ (\Cref{const:descendent}). We construct an augmented stability condition $\gr_v(\sigma) = \langle \cC_\bullet^v|\ell_\bullet^v\rangle_{\Sigma_{\subseteq v}}$ with \vspace{-2mm}
        \begin{enumerate}[label=(\roman*)] 
            \item multi-scale decomposition $\gr_v(\cC_\bullet) = \langle \cC_\bullet^v\rangle_{\Sigma_{\subseteq v}}$;\vspace{-2mm}
            \item homomorphism $\gr_v(\cC_\bullet)\to \gr_v(\Lambda)$; and \vspace{-2mm}
            \item for each $w \in V(\Sigma_{\subseteq v})_{\rm{term}}$, $\ell_w^v:\cC_{\le w}^v/\cC_{<w}^v\to \Sigma_w\setminus \{p_{\rm{ascending}}\}$ is defined using the equivalence $\cC_{\le w}^v/\cC_{<w}^v \simeq \gr_w(\cC_\bullet)$ and $\ell_w$.\endnote{First of all, it follows from \Cref{D:generalized_stability_condition} that for $\{M_w:w\in V(\Sigma)_{\rm{term}}\}$ as in \emph{loc. cit.} we have 
                \[
                    \Lambda_{\le v} = \sum_{w\le_{1,\infty} v \text{ or } w\subseteq v} M_w
                \]
            for \emph{any} $v\in V(\Sigma)$. Indeed, by \Cref{D:multi-scale_decomposition}, it follows that $\cC_{\le v} = \Span \{\cC_{\le w}:w\subseteq v\}$ and thus that $ \rm{K}_0(\cC_{\le v}) = \sum_{w\subseteq v}  \rm{K}_0(\cC_{\le w}).$ Therefore, $v(\cC_{\le v})_{\bQ} = \sum_{w\le_{1,\infty} v\text{ or }w\subseteq v} M_w$. Now, we only need to check the condition about the direct sum decomposition $\gr_v(\Lambda) \otimes \bQ = \bigoplus_{w\in V(\Sigma_{\subseteq v})_{\rm{term}}} M_w$. However, this decomposition follows from $\Lambda_{\le v} \otimes \bQ = \sum_{u\le_{1,\infty} v} M_u$ for all $v\in V(\Sigma)$. Indeed, since $\cC_{<v} =\cC_{\le w}$ for some $w \in V(\Sigma)$, one has that $\gr_v(\Lambda)\otimes \bQ = (\Lambda_{\le v}\otimes \bQ)/(\Lambda_{<v}\otimes \bQ) = \sum_{u \in V(\Sigma_{\subseteq v})_{\rm{term}}} M_u$. Similarly, for $w\in V(\Sigma_{\subseteq v})_{\rm{term}},$ writing $\Lambda_{\le w}^v$ for the image of $\Lambda_{\le w}$ in $\gr_v(\Lambda)$, one has $\Lambda_{\le w}^v \otimes \bQ = \sum_{u \in V(\Sigma_{\subseteq v})_{\rm{term}}, u\le_{1,\infty} w} M_u$.}
        \end{enumerate}
    \end{itemize}
\end{defn}

\section{Spaces of marked multi-scale lines}

In this section, we construct a moduli space $\mscbar_n$ of $n$-marked multi-scale lines, which is a smooth compactific\-ation of $\bC^n/\bC$. Later, we consider its real oriented blowup along the simple normal crossings boundary divisor $\partial:= \mscbar_n \setminus (\bC^n/\bC)$ to obtain a manifold with corners, denoted $\rmscbar_n$. The definition of the topology of $\Astab(\cC)$ makes crucial use of $\cA_n^{\bR}$.

\begin{defn}
\label{D:stablenmarkedmulti-scaleline}
An $n$-marked multi-scale line is the data of $(\Sigma,p_\infty,\preceq,\omega_\bullet,p_1,\ldots,p_n)$, where $(\Sigma,p_\infty,\preceq, \omega_\bullet)$ is a multi-scale line and the $p_i$ are smooth marked points lying on terminal components of $\Sigma$. An $n$-marked multi-scale line is called \emph{stable} if each terminal component contains at least one of the $p_i$.\endnote{Note that a stable multi-scale line has trivial complex projective automorphism group.}
\end{defn}

Associated to an $n$-marked multi-scale line is an $n$-marked totally preordered rooted tree $(\Gamma,\preceq,v_0,h)$, where $(\Gamma,\preceq,v_0)$ is the totally preordered rooted tree of $(\Sigma,p_\infty,\preceq,\omega_\bullet)$ and the marking function $h:\{1,\ldots,n\}\to V(\Sigma)_{\rm{term}}$ is defined by $h(i) = v$ if $p_i \in \Sigma_v$. 

Associated to a totally preordered rooted tree $(\Gamma,\preceq,v_0)$ is a unique isomorphism of totally ordered sets $\lambda:V(\Gamma)/{\sim} \to [\ell]$ which is called the \emph{level function} of $(\Gamma,\preceq,v_0)$. Note that $\lambda^{-1}(\ell) = \{v_0\}$ and that $\lambda^{-1}(0) = V(\Gamma)_{\rm{term}}$.\endnote{This follows from the fact that $u\subseteq v\Rightarrow u\preceq v$. If $u\subset v$ then $u\prec v$ and thus the root is sent to the maximal element of $[\ell]$ with respect to the usual ordering, i.e. $\ell$. Similarly, all of the terminal vertices are mapped to $0$; note that this reflects the assumption that all terminal vertices are in the same equivalence class with respect to $\sim$.} 

As a matter of terminology, a \emph{dual tree} is one arising from a stable $n$-marked multi-scale line: it is a totally preordered rooted tree $(\Gamma,\preceq,v_0)$ with the additional data of a totally ordered set of half edges such that each vertex has the necessary number edges or half edges attached. To simplify notation, we will often refer to a dual tree as simply $\Gamma$.

\subsection{The moduli space of marked multi-scale lines}

The set $\Mmscbar_n$ consists of complex projective isomorphism classes of $n$-marked stable multi-scale lines $(\Sigma,p_\infty,\preceq,\omega_\bullet,p_\bullet)$. We often denote a point of $\Mmscbar_n$ by $\Sigma$, unless the other data are explicitly used.

\begin{defn}[Period functions]
For any $i,j \in \{1,\ldots,n\}$, we define functions $\Pi_{ij} : \Mmscbar_n \to \bP^1$ as follows: If $\Sigma$ is a stable $n$-marked multi-scale line, then $\Pi_{ij}(\Sigma) = \int_{p_i}^{p_j} \omega_v$ if $p_i$ and $p_j$ are both contained in the same terminal component $\Sigma_v$, and $\Pi_{ij}=\infty$ otherwise. The integral is taken over any path from $p_i$ to $p_j$ in the smooth locus of $\Sigma_v$.
\end{defn}

$\Mmscbar_n^\circ \subset \Mmscbar_n$ denotes the subset of isomorphism classes of multi-scale lines with irreducible underlying curve. $\Mmscbar_n^{\circ}$ can be identified with the configuration space of $n$ points in $\bC$ up to simultaneous translation, $\bC^n/\bC$. Indeed, $(\Sigma,p_\infty,\omega,p_\bullet) \in \Mmscbar_n^\circ$ is isomorphic to $(\bP^1,\infty,\omega,p_\bullet)$, where $p_i\ne \infty$ for all $1\le i \le n$ and $\omega$ has an order $2$ pole at $\infty$ and no other zeros or poles. $\omega$ determines a coordinate $z:\bP^1\setminus \{\infty\}\to \bC$ up to translation such that $dz = \omega$. Consequently, $(\Sigma,p_\infty,\omega,p_\bullet)$ is equivalent to $(\bP^1,\omega,dz)$ together with $(z(p_1),\ldots, z(p_n))\in \bC^n/\bC$. Subject to this identification, $\Pi_{ij}$ restricts to the coordinate on $\bC^n/\bC$, given by $\Pi_{ij}(z) = z_j - z_i$. 

Given a stable $n$-marked multi-scale line $\Sigma$, for each irreducible component $\Sigma_v$ write $N_v$ for the set of descending nodes on $\Sigma_v$ when $\Sigma_v$ is nonterminal or the set of marked points in $\Sigma_v$ when it is terminal. Choose $n_v \in N_v$, for each $v$.

\begin{lem}
\label{L:determinedbyintegrals}
    Suppose $\Sigma$ is a stable $n$-marked multi-scale line with dual level tree $\Gamma(\Sigma)$. $\Sigma$ is determined up to isomorphism by $\Gamma(\Sigma)$ and 
    \begin{equation}
    \label{E:integrals}
        \left\{\int_{n_v}^{n} \omega_v\:\bigg|\: n\in N_v \setminus \{n_v\}\right\}_{v\in \Gamma(\Sigma)},
    \end{equation}
    where $\int_{n_v}^n\omega_v$ is the integral along any path in the smooth locus of $\Sigma_v$ connecting $n_v$ and $n$. It is uniquely determined up to complex projective isomorphism by $\Gamma(\Sigma)$ and
    \begin{equation}
    \label{E:integrals2}
    \left\{ \left[ 
        \int_{n_v}^{n} \omega_v\:\bigg|\: 
        \begin{array}{c} 
            v\in \lambda^{-1}(m)  \\ 
            n \in N_v\setminus \{n_v\}
        \end{array} 
    \right]_{1\le m \le \ell},\left(\int_{n_v}^n\omega_v\:\bigg|\: 
    \begin{array}{c}
    v\in \lambda^{-1}(0) \\ 
    n\in N_v\setminus \{n_v\}
    \end{array}
    \right)
    \right\}
    \end{equation}
    where \eqref{E:integrals2} is in $\prod_{m=1}^{\ell}\bP^{\nu_m} \times \bA^{\nu_0}$, for $\nu_m = \sum_{v\in \lambda^{-1}(m)} |N_v\setminus \{n_v\}|$.
\end{lem}

\begin{proof}
    If $\Sigma$ is irreducible, then $(\Sigma,\omega,p_\bullet) \in \Mmscbar_n^\circ$ is uniquely determined by $\{\int_{p_1}^{p_j}\omega\}_{j=2}^n$ since this is equivalent to the statement that $\{\Pi_{1j}\}_{j=2}^n$ give global coordinates on $\bC^n/\bC$. For a general $\Sigma$, an isomorphism $\alpha:(\Sigma,p_\infty,\preceq,\omega_\bullet,p_\bullet)\to (\Sigma',p_\infty',\preceq',\omega_\bullet',p_\bullet')$ of stable $n$-marked multi-scale lines with dual tree $\Gamma$ is equivalent to the data of $\{\alpha_v:v\in V(\Gamma)\}$, where writing $n_+$ for the ascending node of $\Sigma_v$ and $n_+'$ for that of $\Sigma_v'$ (resp. $p_\infty$ and $p_\infty'$ for $v = v_0$), $\alpha_v$ is an isomorphism of irreducible multi-scale lines $(\Sigma_v,n_+,\omega_v,N_v) \to (\Sigma_v',n_+',\omega_v',N_v')$. Thus, the isomorphism class of $(\Sigma,p_\infty,\preceq,\omega_\bullet,p_\infty)$ is determined by \eqref{E:integrals}.

    When $\alpha$ is a complex projective isomorphism, we have again $\{\alpha_v:v\in V(\Gamma)\}$ except that $\alpha_v^*(\omega_v') = c_v\omega_v$, where $c_v\in \bC^*$ is a constant depending only on the level of $v$, and equal to $1$ on level $0$. The second claim follows.
\end{proof}

For any stable $n$-marked multi-scale line $\Sigma$, let 
\begin{equation}
\label{E:Iijdefn}
    I_{ij}:= \int_{n_i}^{n_j}\omega_{v}
\end{equation}
where $v:= h(i)\vee h(j) \in V(\Sigma)$ and $n_i,n_j\in \Sigma_v$ denote either 
\begin{enumerate}
    \item[(1)] the marked points $p_i$ and $p_j$ if $h(i)=h(j)$; or
    \item[(2)] the nodes in $\Sigma_{v}$ that connect to $\Sigma_{h(i)}$ and $\Sigma_{h(j)}$, respectively. 
\end{enumerate} 
The integral is taken along any path connecting $n_i$ and $n_j$ in the smooth locus of $\Sigma_v$. For any given dual level tree $\Gamma$, we introduce the following subset of $\Mmscbar_n$,
\[
    U_\Gamma := \left\{ \Sigma \in \mscbar_n \left| \begin{array}{c} \exists \text{ contraction } \Gamma \twoheadrightarrow \Gamma(\Sigma), \text{ and} \\\Pi_{ij}(\Sigma) \neq 0 \text{ if } h(i)\neq h(j) \text{ in } \Gamma \end{array} \right. \right\}.
\]
Now consider a stable $n$-marked multi-scale line $\Sigma$ whose isomorphism class lies in $U_\Gamma$, and let $f : \Gamma \twoheadrightarrow \Gamma(\Sigma)$ be the associated contraction. Suppose $\Gamma/{\sim} = [\ell]$ and for each $1\le m \le \ell$ choose a vertex $v_m$ on level $m$. Also, choose $i_m,j_m \in \{1,\ldots, n\}$ such that $h(i_m)\vee h(j_m) = v_m$. Let $s_0 = 1$ and for each $1\le m \le \ell$, let $s_m = I_{i_mj_m}$. Next, for $1\le m \le \ell$, let 
\[
        t_m := \left\{ \begin{array}{ll} s_{m-1}/s_{m}, & \text{if } f(v_m) \sim f(v_{m-1}) \\
        0& \text{otherwise}.\end{array} \right.
\] 
Finally, for any $i,j\in \{1,\ldots, n\}$ such that $h(i)\vee h(j)$ is on level $m$, let 
\[
    z_{ij} := \frac{I_{ij}}{s_m} = \frac{1}{s_m} \int_{n_i}^{n_j}\omega_{f(v_m)}
\] 
for all $i,j \in \{1,\ldots, n\}$. Let $A$ denote the set of pairs $i<j$ such that $h(i)\ne h(j)$ in $V(\Gamma)$ and $B$ the set of pairs with $h(i) = h(j)$. 

\begin{lem}
\label{L:equivalentindices}
    There is a well-defined map
    \begin{equation}
    \label{E:Mn_coordinates}
        (z_{ij},t_m)_{\substack{1\leq i<j \leq n\\ 1 \leq m \leq \ell}} : U_\Gamma \to (\bC^\ast)^A \times \bC^B \times \bC^\ell.
    \end{equation}
    If different vertices $v_m$ and indices $i_m,j_m$ are chosen for each $1\le m \le \ell$, the resulting function differs from the original by composition with a monomial automorphism of $(\bC^\ast)^A \times \bC^B \times \bC^\ell$.
\end{lem}

\begin{proof}
    It suffices to prove the result changing one pair of indices $(i_m,j_m) = (i,j)$ to $(i_m',j_m') = (k,\ell)$ on some level $1\le m \le \ell$. Write $z_{ij}'$ for the coordinates defined with respect to the indices $(i_m',j_m')$. By inspection, $z_{\alpha\beta} = z_{\alpha\beta}'$ for all $\alpha,\beta$ with $h(\alpha)\vee h(\beta)$ not on level $m$. Otherwise, we have $z_{k\ell}' = z_{ij}$ and $z_{\alpha \beta}' = z_{\alpha\beta}/z_{k\ell}$. For $k\not\in \{m,m+1\}$, one has $t_k' = t_k$. Finally, $t_m' = t_m\cdot z_{k\ell}$ and $t_{m+1}' = t_{m+1}/z_{k\ell}$. It is an exercise to verify that this map is invertible.\endnote{Since $z_{\alpha\beta} = I_{\alpha\beta}/I_{ij}$ and $z_{\alpha\beta}' = I_{\alpha\beta}/I_{k\ell}$, the inverse map is given by $z_{\alpha\beta} = z_{\alpha\beta}'\cdot (z_{ij}')^{-1}$. Furthermore, one can verify that $t_k = t_k'$ for $k\ne m,m+1$, $t_{m+1} = t_{m+1}'/z_{ij}'$, and $t_m = t_m'\cdot z_{ij}'.$}
\end{proof}

The functions $(z_{ij},t_m)$ depend on a choice of level tree $\Gamma$ which defines the open set $U_\Gamma$ and indices $\{(i_m,j_m)\}_{m=1}^{\ell}$. We will always suppress the choice of indices from the notation, however to emphasize $\Gamma$ we may write $(z_{ij}^\Gamma,t_m^\Gamma)$. By a \emph{choice of indices for} $\Gamma$, we mean a choice of $\{(i_m,j_m)\}_{m=1}^{\ell}$ as above.\endnote{By \Cref{L:equivalentindices}, the choice of indices only alters the map $U_\Gamma \to (\bC^*)^A\times \bC^B \times \bC^\ell$ by a simple type of automorphism. However, these indices are analogous to choosing a basis of a vector space and different choices are technically useful in different scenarios.} The default notation will be to write an alternative choice of indices as $\{(i_m',j_m')\}_{m=1}^{\ell}$ and the resulting coordinates as $(z_{ij}',t_m')$.

\begin{rem}
\label{R:periodsintermsofzij}
    We let $*$ denote the tree with one vertex and note that $U_* = \Mmscbar_n^\circ$. There are no $t$-coordinates for $\Gamma = *$ and under the identification $U_* = \Mmscbar_n^\circ = \bC^n/\bC$, $z_{ij}^*$ corresponds to $\Pi_{ij}$.
\end{rem}

\begin{cor}
\label{C:determinedbyz}
    Let $\Sigma$ be given whose equivalence class lies in $U_\Gamma$. $\{z_{ij}(\Sigma),t_m(\Sigma)\}$ depend only on the complex projective isomorphism class of $\Sigma$. Furthermore, if $\Gamma(\Sigma) = \Gamma$, then $\Sigma$ is uniquely determined up to complex projective isomorphism by $\{z_{ij}(\Sigma)\}_{1\le i<j\le n}$.
\end{cor}

\begin{proof}
    Since each $z_{ij}$ is defined as a ratio $I_{ij}/I_{i_mj_m}$ where $h(i) \vee h(j)$ is on level $m$, it depends only on the complex projective isomorphism class of $\Sigma$. For each $1\le m \le \ell$, choose $n_v\in N_v$ for each $v\in \lambda^{-1}(m)$ such that $n_v$ is the node connecting $\Sigma_v$ to $\Sigma_{h(i_m)}$ when $v = v_m$. $s_m = \int_{n_v}^{n}\omega_{v_m}$ for some $n\in N_{v_m}$, by definition. Choose $1\le m \le \ell$, $v\in \lambda^{-1}(m)$, and $n'\in N_v\setminus \{n_v\}$. Then $z_{ij}(\Sigma) = \int_{n_v}^{n'} \omega_v/\int_{n_{v_m}}^{n}\omega_{v_m}$, where $h(i)\vee h(j) = v$, $n_v$ connects $\Sigma_v$ to $\Sigma_{h(i)}$, and $n'$ connects $\Sigma_v$ to $\Sigma_{h(j)}$. So, all ratios of homogeneous coordinates in \eqref{E:integrals2} can be recovered from $\{z_{ij}(\Sigma)\}$. Similar reasoning shows that each $\int_{n_v}^n \omega_v$ for $v\in \lambda^{-1}(0)$ equals some $z_{ij}(\Sigma)$. The result now follows from \Cref{L:determinedbyintegrals}.
\end{proof}

\begin{lem}
\label{L:intersection}
    Suppose $f:\Gamma\twoheadrightarrow \Gamma'$ is the contraction of levels $1\le i_1<\cdots< i_k\le \ell$. Then $U_\Gamma \cap U_{\Gamma'} = D(t_{i_1}^\Gamma\cdots t_{i_k}^\Gamma)$.
\end{lem}

\begin{proof}
    One verifies that $U_\Gamma \cap U_{\Gamma'} = \{\Sigma \in U_\Gamma: \exists \text { contraction } \Gamma'\twoheadrightarrow\Gamma(\Sigma)\}$.\endnote{By definition, $U_\Gamma \cap U_{\Gamma'}$ consists of those $\Sigma$ such that there exist contractions $\Gamma \twoheadrightarrow \Gamma(\Sigma)$ and $\Gamma'\twoheadrightarrow \Gamma(\Sigma)$ and $\Pi_{ij}(\Sigma)\ne 0$ if $h(i) \ne h(j)$ or $h'(i) \ne h'(j)$. However, $h'(i)\ne h'(j)$ implies $h(i) \ne h(j)$ since $\Gamma'$ is a contraction of $\Gamma$. Therefore, if $\Sigma \in U_\Gamma$, the condition of being in $U_{\Gamma'}$ is just that there exists a contraction $\Gamma'\twoheadrightarrow\Gamma(\Sigma)$.} This is equivalent to saying that $g:\Gamma\twoheadrightarrow\Gamma(\Sigma)$ factors through $f$. This, in turn, is equivalent to the statement: if $f$ contracts a level $m$ then $g$ contracts the level $m$. This happens exactly on the locus where $t_{i_1}^\Gamma\cdots t_{i_k}^\Gamma$ does not vanish.
\end{proof}

We use the term \emph{interval} to refer to a set of consecutive integers when it is clear from the context. Such intervals are all of the form $I = \{a,\ldots, a+k\}$ for $a\in \bZ$ and $k\in \bN$. Given such an $I$, let $I^+ := \{a+1,\ldots, a+k\}$.

\begin{lem}
\label{L:goodindicesCOV}
Suppose $\Sigma \in U_\Gamma \cap U_{\Gamma'}$ where $\Gamma \twoheadrightarrow \Gamma'$ is the contraction corresponding to $\alpha:[\ell]\twoheadrightarrow [\ell-k]$, deleting levels $1\le i_1< \cdots < i_k\le \ell$. Given a choice of indices for $\Gamma$, there is a choice of indices for $\Gamma'$ such that 
    \begin{enumerate}
        \item if $h(i)\vee h(j)$ is on level $m$ of $\Gamma$ and 
        \begin{enumerate} 
            \item $k$ is the largest integer such that $\{m-k,\ldots, m\}$ are contracted by $\alpha$ then $z_{ij}' = z_{ij}/t_{m-k}\cdots t_m$; and 
            \item if $m$ is not contracted then $z_{ij}' = z_{ij}$; and 
        \end{enumerate}
        \item for all nonzero $j\in [\ell-k],$ one has $t_j' = \prod_{i\in (\alpha^{-1}\{j-1,j\})_+} t_{i}$.
    \end{enumerate}
\end{lem}

\begin{proof}
    For $m\in [\ell-k]$, let $(i_m',j_m') = (i_{\min \alpha^{-1}(m)},j_{\min \alpha^{-1}(m)})$. This is sensible since if $h(i)\vee h(j)$ lies on a level in $\alpha^{-1}(m)$, then $h'(i)\vee h'(j)$ lies on level $m$. As a consequence, we have $s_m' := s_{\min \alpha^{-1}(m)}$.

    Suppose a single level $m$ is contracted. In this case, a direct calculation shows that $z_{ij}' = z_{ij}/t_m$ if $h(i)\vee h(j)$ is on level $m$ and that $z_{ij}' = z_{ij}$ otherwise. Also, one can verify that $t_p' = t_p$ for all $p<m$, $t_m' = t_m\cdot t_{m+1}$, and $t_p' = t_{p+1}$ for all $p>m$.\endnote{Suppose first that $h(i)\vee h(j)$ is on level $m$. In that case, $h'(i)\vee h'(j)$ is on level $m-1$ and one has 
    \[
        z_{ij}' = \frac{I_{ij}}{s_m'} = \frac{I_{ij}}{s_{m-1}} = \frac{I_{ij}}{s_m}\cdot \frac{s_m}{s_{m-1}} = \frac{z_{ij}}{t_m}
    \]
    as claimed. If $h(i)\vee h(j)$ is on level $p<m$ then so is $h'(i)\vee h'(j)$ and so $z_{ij}' = I_{ij}/s_p' = I_{ij}/s_p = z_{ij}$. If $h(i)\vee h(j)$ is on level $p>m$, then $h'(i)\vee h'(j)$ is on level $p-1$ and one has $z_{ij}' = I_{ij}/s_{p-1}' = I_{ij}/s_{p} = z_{ij}.$

    Next, for $p\le m-1$, $t_p' = s_{p-1}'/s_p' = s_{p-1}/s_p = t_p$. For $p = m$, $t_m' = s_{m-1}'/s_m' = s_{m-1}/s_{m+1} = t_m\cdot t_{m+1}$. Finally, when $p>m$, $t_p' = s_{p-1}'/s_p' = s_p/s_{p+1} = t_{p+1}$.}

    Consider the contraction of levels $i_1<\cdots<i_k$, $\Gamma\twoheadrightarrow \Gamma''$, recalling from \Cref{ex:contractinglevels} that all contractions are of this form. It factorizes as $\Gamma\twoheadrightarrow\Gamma'\twoheadrightarrow\Gamma''$ where $\Gamma\twoheadrightarrow\Gamma'$ is the contraction of levels $i_2<\cdots<i_{k}$ and corresponds to $\beta:[\ell]\twoheadrightarrow [\ell-k+1]$ and $\Gamma'\twoheadrightarrow\Gamma''$ contracts $\beta(i_1) = i_1$ and corresponds to $\gamma:[\ell-k+1]\twoheadrightarrow[\ell-k]$.\endnote{In the notation of \Cref{ex:contractinglevels}, $\beta$ can be written as $\sigma_{i_2-1}\circ \cdots \circ \sigma_{i_k-1}$ (with some indices suppressed) and since $i_1\le i_2-1<\cdots<i_k-1$ we have that $\beta(i_1) = i_1$.} 
    The general result follows from induction on the number of levels contracted, using the factorization of $\alpha$ as $\gamma \circ \beta$ like above.\endnote{If $p<i_1$ is given then $\gamma^{-1}(\{p-1,p\}) = \{p-1,p\}$ and so $t_p'' = t_p'$. Similar reasoning shows that $t_p' = t_p$. Next, $\gamma^{-1}(\{i_1-1,i_1\}) = \{i_1-1,i_1,i_1+1\}$ and $t_{i_1}'' = t_{i_1}'t_{i_1+1}'$. Next, $\beta^{-1}(\{i_1-1,i_1\}) = \{i_1-1,i_1\}$ so $t_{i_1}' = t_{i_1}$. By induction, $t_{i_1+1}' = \prod_{p\in \beta^{-1}(\{i_1,i_1+1\})_+} t_p$. In particular, $t_{i_1}'' = \left(\prod_{p\in \beta^{-1}(\{i_1,i_1+1\})_+} t_p\right) \cdot t_{i_1}$. Since $\alpha^{-1}(\{i_1-1,i_1\}) = \beta^{-1}(\{i_1-1,i_1,i_1+1\}) = \{i_1-1\}\sqcup \beta^{-1}(\{i_1,i_1+1\})$ the claimed formula follows for $t_{i_1}''$. If $p>i_1$, then $t_p'' = t_{p+1}' = \prod_{k\in \beta^{-1}(\{p+1,p+2\})_+} t_k$ by induction.
    
    Finally, if $h(i)\vee h(j)$ is on level $m<i_1$ then $z_{ij}'' = z_{ij}' = z_{ij}$. If $h(i)\vee h(j)$ is on level $m = i_1$, then $z_{ij}'' = z_{ij}'/t_{i_1}' = z_{ij}/t_{i_1}$. Finally, if $h(i)\vee h(j)$ is on level $m > i_1$ there are two cases. If $m-k = i_1$, then $z_{ij}'' = z_{ij}'/t_{i_1}' = z_{ij}/t_{i_1}'t_{m-k+1}\cdots t_m = z_{ij}/t_{m-k}\cdots t_m$. If $m-k>i_1$, then $z_{ij}'' = z_{ij}' = z_{ij}/t_{m-k}\cdots t_m$.}
\end{proof}

\begin{cor}
\label{C:goodCOVgeneric}
    For any $\Gamma$, on $U_* \cap U_\Gamma$ one has $\Pi_{ij} = z_{ij}/t_{1}\cdots t_m$ where $h(i)\vee h(j)$ is on level $m\ge 1$ of $\Gamma$. If $h(i)\vee h(j)$ is on level $0$ then $\Pi_{ij} = z_{ij}$.
\end{cor}

\begin{proof}
    This is an immediate consequence of \Cref{L:goodindicesCOV}.
\end{proof}

\begin{prop}
\label{P:identification}
    For any dual level tree $\Gamma$ and choice of indices, \eqref{E:Mn_coordinates} is injective and identifies $U_\Gamma$ with a smooth algebraic variety of dimension $n-1$. 
\end{prop}

\begin{proof}
    For each $v\in V(\Gamma)$, consider $N_v$ as in the discussion preceding \Cref{L:determinedbyintegrals}. When $v = v_m$ for $1\le m\le \ell$, we choose $n_v$ to be the node connecting $\Sigma_{v_m}$ to $\Sigma_{h(i_m)}$. For each non-terminal $v$ and each $n \in N_v$, choose $\iota_n \in \{1,\ldots, n\}$ such that $n$ connects $\Sigma_v$ to $\Sigma_{h(\iota_n)}$. We write $\iota_v = \iota_{n_v}$. When $v = v_m$, choose $\{\iota_n:n \in N_{v_m}\}$ such that it contains $i_m,j_m$. When $v$ is terminal, $N_v = \{p_i:h(i) = v\}$ and we define $\iota_{p_i} = i$. Consider 
    \[
    A' := \left(\bigcup_{v \ne \lambda^{-1}(0)}\{(\iota_v,\iota_n)|n\in N_v\setminus n_v\}\right) \setminus \{(i_m,j_m)\}_{m=1}^{\ell}, \text{ and}\]
    \[
        B' := \bigcup_{v\in \lambda^{-1}(0)}\{(\iota_v,\iota_i)|p_i\in N_v\setminus n_v\}. 
    \]
    One has $A'\subset A$ and $B'\subset B$ and $|A'\cup B'| = n - 1 - \ell$. Consider $\bC^A := \Spec \bC[Z_{ij}:(i,j)\in A]$, $\bC^B := \Spec \bC[Z_{ij}:(i,j) \in B]$ and $\bC^\ell := \Spec \bC[T_1,\ldots, T_\ell]$. If we define $\bC^{A'}$ and $\bC^{B'}$ in the analogous fashion, then there is a natural inclusion map $\bC^{A'}\times \bC^{B'}\times \bC^\ell \hookrightarrow \bC^A\times \bC^B\times \bC^\ell$. Define $\varphi: U_\Gamma \to \bC^{A'}\times \bC^{B'}\times \bC^\ell$ by $\varphi^*(Z_{ij}) = z_{ij}$ and $\varphi^*(T_m) = t_m$ for all relevant indices. One can verify that $V := \im(\varphi)$ is the complement of a hyperplane arrangement defined by the conditions $Z_{\iota_v\iota_n} \ne 0$ and $Z_{\iota_v\iota_n} \ne Z_{\iota_v\iota_{n'}}$ for all $v$ and $n,n'\in N_v\setminus \{n_v\}$.\endnote{Elements of $\im(\varphi)$ satisfy these nonvanishing conditions, by the description of the elements of $U_\Gamma$. Indeed, suppose $\Gamma\twoheadrightarrow\Gamma'$ is a contraction and $\Sigma$ is given with $\Gamma(\Sigma) = \Gamma'$. If $z_{\iota_v\iota_n}(\Sigma) = 0$, it must be the case that $h'(\iota_v)= h'(\iota_n)$. Now, if $h(\iota_v) = h(\iota_n)$, then this does not correspond to one of the hyperplanes which has been removed. If $h(\iota_v) \ne h(\iota_n)$, the condition $z_{\iota_v\iota_n}(\Sigma) = 0$ is equivalent to $\Pi_{\iota_v\iota_n}(\Sigma) = 0$ which is prohibited by definition of $U_\Gamma$. Analogous reasoning applies for the conditions $z_{\iota_v\iota_n} = z_{\iota_v\iota_{n'}}$. On the other hand, given a set of values $\{z_{ij}(\Sigma),t_m(\Sigma)\} \in V$ where $(i,j)\in A'\cup B'$ one can freely construct $\Sigma \in U_\Gamma$ with these coordinates.} In particular, $V$ is a quasi-affine variety.

    Finally, we define a morphism $g:\bC^{A'}\times \bC^{B'}\times \bC^\ell \to \bC^{A\setminus A'}\times \bC^{B\setminus B'}$: For all $(i,j)\in (A\setminus A')\cup (B\setminus B')$ such that $h(i)\vee h(j) = v$, let $g^*(Z_{ij}) = Z_{\iota_v\iota_{n'}} - Z_{\iota_v\iota_n}$ where $\Sigma_v$ is connected to $\Sigma_{h(j)}$ by $n'$ and to $\Sigma_{h(i)}$ by $n$. In particular, $g$ is a smooth morphism which restricts to a map $g:V\to (\bC^*)^{A\setminus A'}\times \bC^{B\setminus B'}$. It follows that $(\id,g)\circ \varphi:U_\Gamma \to (\bC^*)^A\times \bC^B\times \bC^\ell$ equals \eqref{E:Mn_coordinates} and in particular this identifies the image of $U_\Gamma$ with the graph of a smooth morphism of quasi-affine varieties. Projection to $V$ implies that the image is dimension $n-1$.\endnote{Consider a morphism $f:V\to W$ of quasi-affine varieties. $\Gamma_f\subset V\times W$ is defined by $\{(x,f(x)):x\in V\}$. $\Gamma_f$ is a quasi-affine variety itself; writing $V\subset \bA^n$ and $W\subset \bA^m$ and $f$ as $f(x) = (f_1(x),\ldots, f_m(x))$ one has that $\Gamma_f$ is defined by the vanishing of the equations $\{y_i-f_i(x)\}_{i=1}^m$.
    
    Next, there is a canonical map $(\id,f):V\to \Gamma_f$ by $x\mapsto (x,f(x))$ which is surjective. In our situation, we need to verify that $(\id,g)\circ \varphi$ is the same as \eqref{E:Mn_coordinates}. If $(i,j) \in A'\cup B'$, one has $(\id,g)^*(Z_{ij}) = Z_{ij}$, $\varphi^*(Z_{ij}) = z_{ij}$. If $(i,j)\in (A\setminus A')\cup (B\setminus B')$, then $(\id,g)^*(Z_{ij}) = Z_{\iota_v\iota_{n'}} - Z_{\iota_v\iota_n}$, where $n'$ and $n$ are as in the body of the proof, and so $\varphi^*((\id,g)^*(Z_{ij})) = z_{\iota_vj} - z_{\iota_vi} = z_{ij}$ and we are done.}
\end{proof}

\begin{rem}
\label{R:projectioncoords}
    In the proof of \Cref{P:identification}, we have shown more\-over that each $U_\Gamma$ is algebraically isomorphic to the complement of a hyper\-plane arrangement in $\bC^{n-1}$: indeed, projection from the image of $U_\Gamma$ to $V$ induces the isomorphism and $\{z_{ij},t_m:(i,j)\in A'\cup B',1\le m\le \ell\}$ form a system of algebraic coordinates on $U_\Gamma$. 
\end{rem}

\begin{thm}
\label{T:spaceconstruction}
The set $\mscbar_n$ admits a unique structure of a complex algebraic variety such that every $U_\Gamma$ is a Zariski open subset, and the functions \eqref{E:Mn_coordinates} are closed immersions. With respect to this structure:
\begin{enumerate}
    \item the locus consisting of irreducible curves is open dense;
    \item the functions $\Pi_{ij}$ on $\bC^n/\bC$ extend to morphisms $\Mmscbar_n \to \bP^1$; and
    \item $\Mmscbar_n$ is smooth and proper of dimension $n-1$.
\end{enumerate}
\end{thm}

\begin{proof}
    To emphasize dependence on $\Gamma$, we denote \eqref{E:Mn_coordinates} by $\psi_\Gamma$. When $\Gamma'$ is a coarsening of $\Gamma$, it follows from \Cref{L:goodindicesCOV} that $\psi_{\Gamma'}\circ \psi_{\Gamma}^{-1}$ is algebraic. For a general $\Gamma'$, $U_\Gamma \cap U_{\Gamma'}$ is covered by $U_\Gamma \cap U_{\Gamma'} \cap U_{\Gamma''}$ where $\Gamma''$ ranges over level trees that are a coarsening of both $\Gamma'$ and $\Gamma$. However, on $U_{\Gamma}\cap U_{\Gamma'}\cap U_{\Gamma''}$ one has $\psi_{\Gamma'}\circ \psi_{\Gamma}^{-1} = (\psi_{\Gamma'}\circ \psi_{\Gamma''}^{-1}) \circ (\psi_{\Gamma''}\circ \psi_{\Gamma}^{-1})$ and as the composition of algebraic maps is algebraic, the result follows. The locus of irreducible curves is $U_*$ and therefore open. For any $\Gamma$, $U_*\cap U_\Gamma = D(t_1\cdots t_\ell)$ by \Cref{L:intersection} and density of $U_*$ follows. 
    
    For the remaining claims, we consider $\Mmscbar_n$ in the analytic topology. Consider a convergent net $(\Sigma_\alpha)\to \Sigma$ with $\Sigma$ reducible. On $U_{\Gamma(\Sigma)}\cap U_*$, we have $\Pi_{ij} = z_{ij}/t_{1}\cdots t_m$ as in \Cref{C:goodCOVgeneric}, where $h(i)\vee h(j)$ is on level $m$. For all $1\le k \le \ell$, one has $\lim_\alpha t_k(\Sigma_\alpha) =0$ and similarly for all $i,j$ where $h(i)\ne h(j)$, one has $\lim_\alpha z_{ij}(\Sigma_\alpha) = z_{ij}(\Sigma) \ne 0$. Thus, $\lim_\alpha \Pi_{ij}(\Sigma_\alpha) = \infty$ if and only if $h(i)\ne h(j)$ and if $h(i) = h(j)$, then $\lim_\alpha \Pi_{ij}(\Sigma_\alpha) = \Pi_{ij}(\Sigma)$. Therefore, $\Pi_{ij}:\Mmscbar_n\to \bP^1$ is continuous.

    A topological space is Hausdorff if and only if there is a dense subspace $Y\subset X$ such that any net $(y_\alpha)$ in $Y$ has at most one limit in $X$.\endnote{Suppose we are given a dense subspace $Y\subset X$ such that the claimed property holds. Suppose $(x_\alpha)_{\alpha \in A}$ is a net in $X$ with limit-points $x$ and $x'$. Choose a net $(y_{\alpha}^\beta)_{\beta \in B_\alpha}$ such that $\lim_\beta y_\alpha^\beta = x_\alpha$ for each $\alpha \in A$. By the theorem on iterated limits of nets \cite{KelleyTopology}*{p.69} it follows that the net $(y_\alpha^\beta)$ indexed by the product directed set $A\times \prod_{\alpha\in A} B_\alpha$ converges to $x$ and $x'$. Therefore, $x=x'$. It follows that $X$ is Hausdorff. Conversely, take $Y=X$.} In our case, let $Y\subset U_*$ denote the open subspace where $\Pi_{ij} \ne 0$ for all $i<j$.

    First, consider a net $(\Sigma_\alpha)$ in $Y$ such that for all pairs $i<j$ and $k<l$ the nets $\Pi_{ij}^\alpha := \Pi_{ij}(\Sigma_\alpha)$ and $\Pi_{kl}^\alpha/\Pi_{ij}^\alpha$ converge in $\bP^1$. Write $\lim_\alpha \Pi_{ij}^\alpha = \Pi_{ij}$. We define an $n$-marked rooted level tree from these data as follows: define the level 0 vertex set by $\{1,\ldots,n\}/{\sim}$, where $i\sim j$ if $\lim_\alpha\Pi_{ij}^\alpha \in \bC$. Define the markings by $h(i) = [i]$ for all $1\le i \le n$. We define a total preorder on $\{(i,j):h(i)\ne h(j)\}$ by $(i,j)\le (k,l)$ if $\lim_\alpha \Pi_{ij}^\alpha/\Pi_{kl}^\alpha \in \bC$.\endnote{Totality follows from the assumption that $\lim_\alpha \Pi_{ij}^\alpha/\Pi_{kl}^\alpha$ converges in $\bP^1$ for all $i<j$ and $k<l$. Reflexivity is by $\lim_\alpha \Pi_{ij}^\alpha/\Pi_{ij}^\alpha = 1$ and transitivity follows from the identity $\frac{\Pi_{ij}^\alpha}{\Pi_{kl}^\alpha}\cdot \frac{\Pi_{kl}^\alpha}{\Pi_{mn}^\alpha } = \frac{\Pi_{ij}^\alpha}{\Pi_{mn}^\alpha}$.
    } Note that this preorder depends only on the classes of $i,j,k,l$ under $\sim$.\endnote{Suppose $i\sim i'$. Then $\Pi_{i'j}^\alpha/\Pi_{kl}^\alpha = (\Pi^\alpha_{i'i} + \Pi^\alpha_{ij})/\Pi_{kl}^\alpha$. Now, as $h(k)\ne h(l)$, $\Pi_{kl}^\alpha \to \infty$ and $\Pi_{i'i}^\alpha$ converges in $\bC$. Therefore, $\Pi_{i'i}^\alpha/\Pi_{kl}^\alpha \to 0$ and we have $\lim_\alpha \Pi_{i'j}^\alpha/\Pi_{kl}^\alpha = \lim_\alpha\Pi_{ij}^\alpha/\Pi_{kl}^\alpha$. A similar argument applies for the indices in the denominator.}

    For $k \in \bZ_{\ge 0}$, suppose a level forest $\Gamma_{\le k}$ has been constructed with levels $0,\ldots, k$. Consider $\cF_{k} = \{([i],[j]):\nexists \: [i]\vee [j] \in \Gamma_{\le k}\}$. For each pair $([i],[j]) \in \cF_k$ minimal with respect to the restriction of $\le$ to $\cF_k$, define a level $k+1$ vertex denoted $[i]\vee[j]$. Define edges between $[i]\vee[j]$ and the vertices of maximal level in $\Gamma_{\le k}$ to which $[i]$ and $[j]$ are connected, respectively. It follows that the output $\Gamma_{\le k+1}$ is a level forest with levels $0,\ldots, k+1$.\endnote{$\Gamma_{\le k}$ is a level forest and we have connected each new vertex to distinct components of the forest. Therefore, each pair of vertices is connected by a unique path in $\Gamma_{\le k+1}$ as needed.}
    
    Since a preorder on a finite set always has minimal elements, $|\cF_{k+1}|<|\cF_{k}|$ and this process eventually terminates. At the last stage, on the maximal level $\ell$, one vertex is attached which we define to be the root. $\Gamma_{\le \ell} =: \Gamma$ is an $n$-marked rooted level tree: by induction $\Gamma$ is an $n$-marked level forest and by construction for all terminal vertices $[i]$ and $[j]$, their join, $[i]\vee [j]$, exists. 

    Now, suppose $(\Sigma_\alpha)\to \Sigma$. Up to passing to a cofinal subnet, we may assume that all $\Pi_{ij}^\alpha$ and $\Pi_{kl}^\alpha/\Pi_{ij}^\alpha$ converge, by compactness of $\bP^1$. We claim that $\Gamma = \Gamma(\Sigma)$, where $\Gamma$ is the level tree constructed in the previous paragraphs. $p_i$ and $p_j$ are on the same terminal component of $\Sigma$ iff $\Pi_{ij}(\Sigma) \ne \infty$, so level $0$ of the trees agrees. Suppose the trees agree up to level $k\ge 0$. $i$ and $j$ meet on level $k$ of $\Gamma(\Sigma)$ if and only if $(i,j)$ is minimal among pairs that do not meet on level $\le k$ with respect to $\le$ as above. Thus, $\Gamma$ and $\Gamma(\Sigma)$ agree up to level $k+1$.\endnote{By \Cref{C:goodCOVgeneric}, if $h(i)\vee h(j)$ is on level $k$ of $\Gamma(\Sigma)$ then in $U_\Gamma$ coordinates we have $\Pi_{ij}^\alpha = z_{ij}^\alpha/t_{1}^\alpha \cdots t_k^\alpha$. So, if $h(k) \vee h(l)$ is on level $p$ one has 
    \[
    \Pi_{ij}^\alpha/\Pi_{kl}^\alpha = \frac{z_{ij}^\alpha}{z_{kl}^\alpha}\cdot r_\alpha, \text{ where } r_\alpha = 
    \begin{cases}
    t_{m+1}^\alpha\cdots t_p^\alpha & p>m\\
    1& p = m\\
    (t_{p+1}^\alpha \cdots t_m^\alpha)^{-1}& p<m.
    \end{cases}
    \]
    In particular, $(i,j)$ and $(k,l)$ are equivalent with respect to the relation induced by $\le$ if and only if $p=m$, while $(i,j)<(k,l)$ if and only if $p>m$.} By induction, $\Gamma = \Gamma(\Sigma)$. 

    Now, $\Sigma_\alpha \in U_\Gamma \cap U_*$ if and only if $\Pi_{ij}(\Sigma_\alpha) \ne 0$ for all pairs $i,j$ such that $h(i)\ne h(j)$. However, $h(i) = h(j)$ if and only if $\lim_\alpha \Pi_{ij}^\alpha \in \bC$. So, if $h(i)\ne h(j)$ then $\lim_{\alpha} \Pi_{ij}^\alpha =\infty$ and so there exists $\alpha_0$ such that $\alpha \ge\alpha_0$ implies $\Sigma_\alpha \in U_\Gamma \cap U_*$. 

    Henceforth suppose $\alpha \ge \alpha_0$. Since $\Sigma_\alpha$ is convergent in $U_\Gamma$, the associated coordinate nets $\{z_{ij}^\alpha,t_m^\alpha\}$ converge and determine $\Sigma$ uniquely by \Cref{C:determinedbyz}. This implies that $\Mmscbar_n$ is Hausdorff with its analytic topology. $\Mmscbar_n$ is second countable because it has a finite cover by Euclidean spaces, namely the $U_\Gamma$. Thus, $\Mmscbar_n$ has the structure of a complex manifold. It follows that compactness is equivalent to every sequence in $Y\subset \Mmscbar_n$ admitting a convergent subsequence in $\Mmscbar_n$.\endnote{Suppose $Y\subset X$ is given, where $X$ is a manifold of positive dimension and $Y$ is dense. Suppose every sequence in $Y$ admits a subsequence convergent in $X$. $X$ is metrizable so choose some metric $d$ and consider a sequence $(x_n)$ in $X$. For each $n\in \bN$ choose $y_n \in Y$ such that $d(y_n,x_n) < 1/n$. $(y_n)$ admits a convergent subsequence with limit $y$, which we index again as $y_n$ without loss of generality. Now, $d(x_n,y) \le d(x_n,y_n) + d(y_n,y)$ and in particular $(x_n)\to y$.}

    Consider a sequence $(\Sigma_\alpha)_{\alpha\in \bN}$ in $Y$ and up to passing to a subsequence suppose that $\Pi_{ij}^\alpha$ and $\Pi_{kl}^\alpha/\Pi_{ij}^\alpha$ converge in $\bP^1$ for all $i<j$ and $k<l$. By the previous discussion, we may associate to $\{\Pi_{ij}^\alpha,\Pi_{kl}^\alpha/\Pi_{ij}^\alpha\}$ a level tree $\Gamma$. By construction, for all $\alpha$ sufficiently large, $\Sigma_\alpha \in U_\Gamma$. By \Cref{C:goodCOVgeneric}, if $h(i) \vee h(j)$ is on level $1\le p \le \ell$, one has $\Pi_{ij}= z_{ij}/t_{1} \cdots t_p$ and in particular $\Pi_{i_mj_m}^\alpha = 1/t_{1}^\alpha\cdots t_m^\alpha$. By construction of $\Gamma$, $\lim_\alpha \Pi^\alpha_{i_{m-1}j_{m-1}}/\Pi^\alpha_{i_{m}j_{m}} = \lim_\alpha t_m^\alpha = 0$ for each $1\le m \le \ell$. Also, for any $i,j$ with $h(i)\vee h(j)$ on level $m$, $\lim_\alpha \Pi_{ij}^\alpha/\Pi_{i_mj_m}^\alpha = \lim_\alpha z_{ij}^\alpha \in \bC^*$ for $m>0$ and is in $\bC$ for $m = 0$. So, $(\Sigma_\alpha)\to \Sigma \in U_\Gamma$ with $\Gamma(\Sigma) = \Gamma$. This gives compactness.

    Finally, a coordinate calculation shows that $\Pi_{ij}:\Mmscbar_n \to \bP^1$ is holomorphic\endnote{Consider $\Sigma$ where $\Pi_{ij}(\Sigma) = \infty$. Suppose $h(i)\vee h(j)$ is on level $p>0$. Then on $\Gamma = \Gamma(\Sigma)$ one has $h(i) \ne h(j)$ and so on $U_\Gamma$ we have $\Pi_{ij} \ne 0$. In particular, $\Pi_{ij}^{-1} = t_{1}\cdots t_{p}/z_{ij}$ where $z_{ij}$ is nonvanishing on $U_\Gamma$. Therefore, $\Pi_{ij}^{-1}$ is holomorphic at $\Sigma$ as needed.} and by proper GAGA \cite{SGA1}*{Expos\'{e} XII, Cor. 4.5}, it follows that $\Pi_{ij}:\Mmscbar_n \to \bP^1$ is algebraic.
\end{proof}

For use in what follows, we record some properties of $\Mmscbar_n$. Let $S_\Gamma^\circ \subset \Mmscbar_n$ denote the set of (classes of) curves with dual level tree $\Gamma$. Note that here we abuse notation by suppressing the level structure and marking from the notation, but they are implicit; indeed, different choices of level structure or marking determine distinct sets $S_\Gamma^\circ$. Let $S_\Gamma$ denote the closure of $S_\Gamma^\circ$ in $\Mmscbar_n$. Let $\partial$ denote the union over all $S_\Gamma$ for $\Gamma$ of length $1$.

\begin{prop}
\label{P:propertiesofSGamma}
    The $S_\Gamma$ satisfy the following properties:
    \begin{enumerate}
        \item Each $S_\Gamma$ is a subvariety of $\Mmscbar_n$ of  codimension equal to the length of $\Gamma$ and $S_\Gamma^\circ$ is an open subset of $S_\Gamma$.
        \item $S_\Gamma \setminus S_\Gamma^\circ$ is the union over all $S_{\Gamma'}$ such that there exists a contrac\-tion $\Gamma'\twoheadrightarrow \Gamma$ and in particular $S_\Gamma^\circ$ form the open strata of a smooth quasi-affine stratification of $\Mmscbar_n$.
        \item $\partial$ is a simple normal crossings divisor and $\Mmscbar_n = \Mmscbar_n^\circ \sqcup \: \partial$.
    \end{enumerate}
\end{prop}

\begin{proof}
    Argue using local coordinates or see \cite{SMLWC}*{Lem. 3.17} for (1) and (2) and \cite{SMLWC}*{Lem. 3.18(4)} for (3).
\end{proof}

\begin{rem}
    $\Mmscbar_n$ is a proper algebraic variety by \Cref{T:spaceconstruction}, however it is \emph{a priori} unclear whether or not it is projective. In \cite{SMLWC}, $\Mmscbar_n$ is realized as a blowup of a subspace arrangement of $\bP^{n-1}$, giving a proof of projectivity. $\Mmscbar_n$ also admits a $\bC^n/\Delta$-action, where $\Delta \subset \bC^n$ is the small diagonal subgroup, extending the natural one on $\Mmscbar_n^\circ$. 
\end{rem}

\begin{rem}
    As mentioned in the introduction, multi-scale lines and the notion of complex projective isomorphism are inspired by \cite{BCGGM}. In fact, there are only a few differences between the $n$-marked multi-scale lines considered here and the genus $0$ multi-scale diff\-erentials of type $(0^n,-2)$ studied in \cite{BCGGM}. The most obvious difference is that multi-scale lines are allowed to have colliding marked points whereas multi-scale differentials are not. In \cite{msdandwonderfulmodels}, it is proven that the moduli space of genus $0$ multi-scale differentials of type $(0^n,-2)$ constructed in \cite{BCGGM} embeds in $\cA_n$ as the boundary divisor corresponding to the level tree with two levels and $n$ vertices on the bottom level.
\end{rem}

\subsection{The moduli space of real multi-scale lines}

In this section, we consider the set of real oriented isomorphism classes of $n$-marked multi-scale lines (see \Cref{D:multi-scaleline}), denoted $\rmscbar_n$. There is a canonical map $p:\rmscbar_n\to \Mmscbar_n$ which sends the real oriented isomorphism class of a multi-scale line to its complex projective isomor\-phism class. In what follows, we will equip $\rmscbar_n$ with the structure of a manifold with corners such that $p:\rmscbar_n\to \Mmscbar_n$ is a $C^\infty$ map of manifolds restricting to a diffeomorphism over $\Mmscbar_n^\circ$.

We define some functions on the set of $n$-marked multi-scale lines which depend only on the real oriented isomorphism class. Given an $n$-marked multi-scale line $\Sigma$, consider $I_{ij}$ as in \eqref{E:Iijdefn}. The ratio $I_{ij}/\lvert I_{ij}\rvert$ is defined on real oriented isomorphism classes as long as $I_{ij} \ne 0$. So, for $i<j$ we let

\[
    \mathfrak{p}_{ij} = 
        \left\{ \begin{array}{ll} I_{ij}/\lvert I_{ij}\rvert, &\text{ if } h(i)\ne h(j), \text{ and}\\
        \frac{\Pi_{ij}}{1+|\Pi_{ij}|} &\text{otherwise.}\end{array} \right. 
\]
$\mathfrak{p}_{ij}$ defines a function $\rmscbar_n\to \bD$, where $\bD$ denotes the closed unit disk in $\bC$. Note that $\mathfrak{p}_{ij}$ maps $\Sigma$ to the interior of $\bD$ if and only if $h(i) = h(j)$. Also, we identify $\Pi_{ij}:\Mmscbar_n\to \bP^1$ with its pullback $p^*(\Pi_{ij}):\cA_n^{\bR} \to \bP^1$.

\begin{lem}
\label{L:determinedbypij}
    Let an $n$-marked multi-scale line $\Sigma$ be given with dual level tree $\Gamma$. For any choice of indices $\{(i_m,j_m)\}_{m=1}^\ell$ for $\Gamma$, the real oriented isomorphism class of $\Sigma$ is determined by $\{\mathfrak{p}_{i_mj_m}(\Sigma)\}_{m=1}^\ell$ and its complex projective isomorphism class.
\end{lem}

\begin{proof}
    Consider $n$-marked multi-scale lines $\Sigma$ and $\Sigma'$ which are complex projectively iso\-morphic and such that $\mathfrak{p}_{ij}(\Sigma) = \mathfrak{p}_{ij}(\Sigma')$ for all $i<j$ such that $h(i)\ne h(j)$. It follows that the constants $c_v\in \bC^*$ of \Cref{D:multi-scaleline} actually lie in $\bR_{>0}$ and in particular that $\Sigma$ and $\Sigma'$ are real orientedly isomorphic. It remains to prove that $\{\mathfrak{p}_{ij}(\Sigma)|h(i)\ne h(j)\}$ can be recovered from $\{\mathfrak{p}_{i_mj_m}(\Sigma)\}_{m=1}^\ell$ and the complex projective isomorphism class of $\Sigma$. For this, suppose $h(i)\vee h(j)$ is on level $m\ge 1$. Then one computes $(\Pi_{ij}/\Pi_{i_mj_m})(\Sigma) = (I_{ij}/I_{i_mj_m})(\Sigma)$ by definition. Thus, 
    \[
    \frac{\mathfrak{p}_{ij}(\Sigma)}{\mathfrak{p}_{i_mj_m}(\Sigma)} = \frac{I_{ij}(\Sigma)}{I_{i_mj_m}(\Sigma)}\cdot \frac{\lvert I_{i_mj_m}(\Sigma)\rvert}{\lvert I_{ij}(\Sigma)\rvert}
    \]
    and $\mathfrak{p}_{ij}(\Sigma)$ is determined by $\mathfrak{p}_{i_mj_m}(\Sigma)$ and the complex projective isomorphism class of $\Sigma$. 
\end{proof}

By \Cref{P:propertiesofSGamma}, $\Mmscbar_n$ has a simple normal crossings boundary divisor $\partial$. We perform the real oriented blowup of $\Mmscbar_n$ along $\partial$ to obtain a $C^\infty$-manifold with corners, $\Bl_\partial^{\bR}(\Mmscbar_n)$, equipped with a smooth map $\pi:\Bl_\partial^{\bR}(\Mmscbar_n)\to \Mmscbar_n$ restricting to a diffeomorphism over $\pi^{-1}(\Mmscbar_n^\circ)$  --- see \cite{HandbookArticle}*{\S 8} for background.\endnote{Let $X$ be a complex manifold and $D$ a Cartier divisor. Defining equations of $D$ yield a section $f$ of the line bundle $\Gamma(X,\cal{O}_X(D))$. If $i:D\hookrightarrow X$ is a smooth divisor, $i^*\cal{O}_X(D) = \cal{N}_{D|X}$. We regard $\pi:\cal{O}_X(D)\to X$ as a geometric line bundle. Define $B_{\cal{O}_X(D),f} = \{(p,v): \lVert v\rVert f(p) = \lVert f(p)\rVert v\}$ where we have chosen some Hermitian metric on $\cal{O}_X(D)$. If $f(p) \ne 0$, the condition gives those $v\in \cal{O}_X(D)|_p$ such that $v/\lVert v\rVert = f(p)/\lVert f(p)\rVert$. For $f(p) = 0$, the condition yields all $v\in \cal{O}_X(D)|_p$. We see that the zero section $X_0\subset B_{\cal{O}_X(D),f}$. We define $B_{\cal{O}_X(D),f}^* = B_{\cal{O}_X(D),f}\setminus X_0$. Analyzing the definition shows that $(B_{\cal{O}_X(D),f}^*)_p$ is a $\bb{R}_{>0}$-torsor for $p \not \in Z(f) = D$ and $(B_{\cal{O}_X(D),f}^*)_p$ is a $\bb{C}^*$-torsor for $f(p) = 0$. We form the quotient space $\rm{Bl}^{\rm{o}}_D(X) = B_{\cal{O}_X(D),f}^*/\bb{R}_{>0}$. Note that this is a closed subspace of $S^1(\cal{O}_X(D)) = (\cal{O}_X(D)\setminus X_0)/\bb{R}_{>0}$. Hence, it is a compact manifold with boundary which we term the \emph{real oriented blowup of} $X$ \emph{along} $D$. There is a canonical map $p:\rm{Bl}^{\rm{o}}_D(X) \to X$ which is a diffeomorphism over $X\setminus D$. Furthermore, there is a Cartesian square 
$$
\begin{tikzcd}[ampersand replacement=\&]
S^1(\cal{N}_{D|X})\arrow[d] \arrow[r] \& \rm{Bl}^{\rm{o}}_D(X)\arrow[d,"p"]\\
D\arrow[r,"i"]\&X
\end{tikzcd}
$$
where we note furthermore that $i^*\rm{Bl}_D^{\rm{o}}(X) = S^1(\cal{N}_{D|X})$ on the nose. We define the real oriented blowup for a more general Cartier divisor. Suppose $D = D_1\cup \cdots \cup D_n$ is an snc divisor. The corresponding line bundle is $\cal{O}_X(D_1+\cdots + D_n) = \bigotimes_{i=1}^n\cal{O}_X(D_i)$. Choose local equations $f_1,\ldots, f_n$ for $D$ and define 
$$\rm{Bl}^{\rm{o}}_D(X) = \rm{Bl}^{\rm{o}}_{D_1}(X)\times_X \cdots \times_X \rm{Bl}^{\rm{o}}_{D_n}(X)\to X.$$
For $p\in D_{i_1}\cap \cdots \cap D_{i_k}$, $\rm{Bl}^{\rm{o}}_D(X)|_p = S^1(\cal{N}_{D_{i_1}|X}|_p)\times \cdots \times S^1(\cal{N}_{D_{i_k}|X}|_p).$ The space $\rm{Bl}_D^{\rm{o}}(X)$ is a manifold with corners, referred to as the \emph{real oriented blowup} of $D$.}

Let $X$ be a complex manifold with $D$ an simple normal crossings divisor. Locally, there exist coordinates $(U,t_1,\ldots, t_\ell,z_1,\ldots, z_{n-\ell})$ such that $D = Z(t_1\cdots t_\ell)$. Locally, the real oriented blowup is given by
\begin{equation}
\label{E:Rblowuplocally}
    \Bl_D^{\bR}(U) = \{((t,z),T) \in U\times (S^1)^\ell: \lvert t_i\rvert T_i = t_i\}
\end{equation}
and the projection $\Bl_D^{\bR}(U)\to U$ is given by restricting the projection $U\times (S^1)^\ell \to U$. One can see that $T_i$ defines a smooth extension of $t_i/\lvert t_i\rvert$ over the locus where $t_i = 0$ \cite{HandbookArticle}*{p. 37}. Consequently, we regard $t_i/\lvert t_i\rvert$ as a smooth function extending to the boundary of $\Bl_D^{\bR}(U)$ for all $1\le i \le \ell$.

For any $i<j$, we identify $\Pi_{ij}:\Mmscbar_n\to \bP^1$ with its pullback $\pi^*(\Pi_{ij}):\Bl_{\partial}^{\bR}(\cA_n) \to \bP^1$. Let $U_{ij}$ denote the set of points in $\Mmscbar_n$ such that $\Pi_{ij} \ne 0$ and let $U_{ij}^{\bR} = \pi^{-1}(U_{ij})$.   

\begin{lem}
\label{L:continuousextensiontoboundary}
    $\Pi_{ij}/(1+\lvert \Pi_{ij}\rvert)$ extends continuously to the boundary of $\Bl_{\partial}^{\bR}(\cA_n)$ and $\Pi_{ij}/\lvert \Pi_{ij}\rvert$ extends continuously to $U_{ij}^{\bR}$. Furthermore, given $\Gamma$ such that $h(i)\ne h(j)$ one has that $\Pi_{ij}/\lvert \Pi_{ij}\rvert$ and $\Pi_{ij}/(1+\lvert \Pi_{ij}\rvert)$ coincide over $\pi^{-1}(S_\Gamma^\circ)$.
\end{lem}

\begin{proof}
    Consider a dual level tree $\Gamma$ and the associated open set $U_\Gamma$. We can choose coordinates as in \Cref{C:goodCOVgeneric} such that if $h(i)\vee h(j)$ is on level $m\ge 1$ one has $\Pi_{ij} = z_{ij}/t_1\cdots t_m$. Then, expanding in these coordinates gives 
    \begin{equation}
    \label{E:periodcomp}
    \frac{\Pi_{ij}}{1+\lvert \Pi_{ij}\rvert} = \frac{z_{ij}/t_1\cdots t_m}{1+\lvert z_{ij}/t_1\cdots t_m\rvert} = \frac{z_{ij}}{\lvert z_{ij}\rvert} \cdot \frac{\lvert t_1\cdots t_m\rvert}{t_1\cdots t_m} = \frac{\Pi_{ij}}{\lvert \Pi_{ij}\rvert}
    \end{equation}
    and the result follows.
\end{proof}

\begin{lem}
\label{L:bijectionoffiber}
    Consider $x\in \partial$ with dual tree $\Gamma$ of length $\ell$. Given a choice of indices $\{(i_m,j_m)\}_{m=1}^\ell$, the functions $\Pi_{i_mj_m}/(1+\Pi_{i_mj_m})$ define a homeomorphism $\pi^{-1}(x) \to (S^1)^\ell$.
\end{lem}

\begin{proof}
    Consider the functions $z_{ij}$ and $t_k$ associated to $\Gamma$ and the choice of indices. The induced map is a bijection because $\Pi_{i_mj_m} = 1/t_1\cdots t_m$ for each $1\le m \le \ell$ and $t_i/\lvert t_i\rvert$ restrict to fiber coordinates in the local model \eqref{E:Rblowuplocally}.\endnote{More precisely, since $\Pi_{i_mj_m}/(1+\lvert \Pi_{i_mj_m}\rvert) = \Pi_{i_mj_m}/\lvert \Pi_{i_mj_m}\rvert = \lvert t_1\cdots t_m\rvert/t_1\cdots t_m$ the map defined by sending $y\in \pi^{-1}(x) \mapsto ((\Pi_{i_mj_m}/1+\lvert \Pi_{i_mj_m}\rvert)(y))_{m=1}^\ell$ is the composition of the following bijections:
    \begin{enumerate}
        \item $\pi^{-1}(x)\to (S^1)^\ell$ given by $y\mapsto ((t_i/\lvert t_i\rvert)(y))_{i=1}^m$
        \item $(S^1)^\ell \to (S^1)^\ell$ given by $(z_1,\ldots, z_\ell)\mapsto (z_1,z_1z_2,\ldots, z_1\cdots z_\ell)$; and 
        \item $(S^1)^\ell \to (S^1)^\ell$ given by $(z_1,\ldots, z_\ell)\mapsto (z_1^{-1},\ldots, z_\ell^{-1})$.
    \end{enumerate}} The induced map is automatically a homeomorphism since it is a bijection between compact Hausdorff spaces.
\end{proof}

\begin{thm}
\label{T:bijection}
    There exists a unique bijection $f:\rmscbar_n \to \Bl_\partial^{\bR}(\Mmscbar_n)$ such that $\pi \circ f = p$ and $f^*(\Pi_{ij}/1+\lvert \Pi_{ij}\rvert) = \mathfrak{p}_{ij}$. 
\end{thm}

\begin{proof}
    Since $\pi$ and $p$ restrict to bijections over $\Mmscbar_n^\circ$, $f$ is uniquely characterized over $\Mmscbar_n^\circ$ by $\pi \circ f = p$. Consider $x\in S_\Gamma^\circ$ for $\Gamma$ a nontrivial level tree. Choose a set of indices $\{(i_m,j_m)\}_{m=1}^\ell$ for $\Gamma$. $\Sigma \in p^{-1}(x)$ is determined by $\{\mathfrak{p}_{i_mj_m}(\Sigma)\}_{m=1}^\ell$, by \Cref{L:determinedbypij}. On the other hand, there is a unique $t \in \pi^{-1}(x)$ such that $(\Pi_{i_mj_m}/(1+\lvert \Pi_{i_mj_m}\rvert))(t) = \mathfrak{p}_{i_mj_m}(\Sigma)$ for all $1\le m \le \ell$ by \Cref{L:bijectionoffiber}. Set $f(\Sigma) = t$.

    This defines the extension of $f$ to $p^{-1}(S_\Gamma^\circ)$, which maps bijectively onto $\pi^{-1}(S_\Gamma^\circ)$. We show that over $\pi^{-1}(S_\Gamma^\circ)$ we have $f^*(\Pi_{ij}/1+\lvert \Pi_{ij}\rvert) = \mathfrak{p}_{ij}$ for all $1\le i < j \le n$. Consider $y\in p^{-1}(x)$ and suppose $h(i)\vee h(j)$ is on level $m$ for $1\le m \le \ell$. $\Pi_{ij}/(1+\lvert \Pi_{ij}\rvert) = \Pi_{ij}/\lvert \Pi_{ij}\rvert$ over $\pi^{-1}(S_\Gamma^\circ)$ by \Cref{L:continuousextensiontoboundary}. Then, 
    \begin{align*}
        f^*\left(\frac{\Pi_{ij}}{\lvert \Pi_{ij}\rvert}\right)(y) = \frac{\Pi_{ij}}{\lvert \Pi_{ij}\rvert}(f(y)) & =  \frac{\Pi_{ij}}{\Pi_{i_mj_m}}(x) \cdot \frac{ \Pi_{i_mj_m}}{\lvert\Pi_{i_mj_m}\rvert}(f(y)) \cdot \frac{\lvert \Pi_{i_mj_m}\rvert}{\lvert \Pi_{ij}\rvert}(x)\\
        & = \frac{I_{ij}}{I_{i_mj_m}}(x) \cdot \mathfrak{p}_{i_mj_m}(y) \cdot \frac{\lvert I_{i_mj_m}\rvert}{\lvert I_{ij}\rvert}(x)\\
        & = \mathfrak{p}_{ij}(y)
    \end{align*}
    Thus, $f^*(\Pi_{ij}/(1+\lvert \Pi_{ij}\rvert)) = \mathfrak{p}_{ij}$ when $h(i)\ne h(j)$. When $h(i) = h(j)$, $\Pi_{ij}$ is regular at $x$ and so by definition $f^*(\Pi_{ij}/(1+\lvert \Pi_{ij}\rvert)) = \mathfrak{p}_{ij}$. Uniqueness follows from the proof above, since there was no choice allowed in how $f$ was defined.
\end{proof}

\begin{defn}
    Equip $\rmscbar_n$ with the unique structure of a $C^\infty$ manifold with corners such that $f$ of \Cref{T:bijection} is a diffeomorphism. 
\end{defn}

With respect to this manifold structure, the functions $\mathfrak{p}_{ij}:\rmscbar_n\to \bD$ are smooth and $\pi \circ f = p$. Given a net in $\Mmscbar_n^\circ$ converging in $\Mmscbar_n$, its lift to a net in $\rmscbar_n$ may not converge.\endnote{Indeed, consider the real oriented blowup of $\bb{C}$ at the origin. A sequence $(z_m)$ in $\bb{C}^*$ converging to $0$ will now converge in the real oriented blowup if and only if $z_m/\lvert z_m\rvert$ converges. E.g. the sequence $z_m = \frac{(-1)^m}{m}$ converges in $\bb{C}$ but not in the real oriented blowup.}

\begin{prop}
\label{P:angularconvergence}
Consider a net $(x_\alpha)_{\alpha \in A}$ in $\Mmscbar_n^\circ$ which converges to $x\in \partial$. $(x_\alpha)_{\alpha \in A}$ converges in $\rmscbar_n$ if and only if 
\begin{equation} 
\label{E:modifiedperiod}
    \lim_{\alpha} \frac{\Pi_{ij}^\alpha}{1+\lvert \Pi_{ij}^\alpha\rvert}
\end{equation}
exists for all $i,j$.
\end{prop}

\begin{proof}
    This follows from a coordinate calculation using \Cref{C:goodCOVgeneric}.\endnote{Consider $\Gamma = \Gamma(x)$ and a choice of indices $\{(i_m,j_m)\}_{m=1}^\ell$ as in \Cref{C:goodCOVgeneric}. Since $(x_\alpha)\to x$, we know that $\lim_\alpha t_m^\alpha = t_m(x) = 0$ for all $1\le m \le \ell$ and $\lim_{\alpha} z_{ij}^\alpha = z_{ij}(x)$ for all $i<j$. By \eqref{E:Rblowuplocally}, $(x_\alpha)\to \widetilde{x} \in \rmscbar_n$ if and only if $\lim_\alpha t_m^\alpha/\lvert t_m^\alpha \rvert$ exists for each $1\le m \le \ell$. 

    Suppose first that for all $i<j$ \eqref{E:modifiedperiod} converges. Then, taking $i = i_1$ and $j = j_1$, we see by \eqref{E:periodcomp} that $\lim_\alpha t_1^\alpha/\lvert t_1^\alpha\rvert$ converges. Taking $i = i_m$ and $j = j_m$ by \eqref{E:periodcomp} we see that convergence of \eqref{E:modifiedperiod} is equivalent to existence of $\lim_\alpha t_1^\alpha\cdots t_m^\alpha/\lvert t_1^\alpha \cdots t_m^\alpha \rvert$. By induction, it follows that each $\lim_\alpha t_m^\alpha/\lvert t_m^\alpha\rvert$ exists. Conversely, suppose that each $\lim_\alpha t_m^\alpha/\lvert t_m^\alpha\rvert$ exists for $1\le m \le \ell$. \eqref{E:periodcomp} implies that each of the limits of \eqref{E:modifiedperiod} exists.}
\end{proof}

The following proposition will be used multiple times in the sequel. As a matter of notation, given an element $x$ of $\Mmscbar_n$ we write $\lvert x\rvert$ to denote the underlying (complex projective isomorphism class of) multi-scale line with the markings forgotten. Let $x^1,\ldots, x^n$ denote the marked points of $x$.

\begin{prop}
\label{P:addingpointsconvergence}
    Let nets $(x_\alpha) \in \mscbar_n$ (resp. $\rmscbar_n)$ and $(y_\alpha) \in \mscbar_{n+k}$ (resp. $\rmscbar_{n+k}$) be given such that 
    \begin{enumerate}
        \item for all $\alpha,$ there is an isomorphism $\lvert x_\alpha\rvert \cong \lvert y_\alpha\rvert$ taking $x_\alpha^i$ to $y_\alpha^i$ for all $1\le i \le n$; and
        \item for all $m > n$, $\exists j \in [1,n]$ such that $\lim_\alpha y_\alpha^{m} - x_\alpha^{j}$ exists.
    \end{enumerate}
    Then $(x_\alpha)\to x\in \mscbar_n$ (resp. $\rmscbar_n$) if and only if $(y_\alpha)\to y \in \mscbar_{n+k}$ (resp. $\rmscbar_{n+k}$). In this case, $\lvert x\rvert \cong \lvert y\rvert$ via a complex projective (resp. real oriented) isomorphism  that identifies $x^i$ with $y^i$ for all $1\le i \le n$, and $y^{m}$ lies on the same component as $x^j$ for $m$ and $j$ as in (2).
\end{prop}

\begin{proof}
We omit the proof, noting that it follows from the definition of the coordinate systems on $\cA_n$ as in \eqref{E:Mn_coordinates} and the local description of the real oriented blowup as in \eqref{E:Rblowuplocally}.\endnote{For both $\mscbar_n$ and $\rmscbar_n$ it suffices to prove the $k=1$ case. We begin with $\mscbar_n$. Suppose $x$ has dual tree $\Gamma$. Consider $\Gamma'$ defined to be the dual tree of $x$ with an extra marked point $x^{n+1}$ attached to the component of some $x^k$. Associated to $\Gamma$ and a choice of indices $\{(i_m,j_m)\}_{m=1}^\ell$ is an open $U_\Gamma$ and a set of functions $\{z_{ij},t_m\}$ as in \eqref{E:Mn_coordinates}.

The same choice of indices defines a system of functions $\{w_{ij},u_m\}$ on $U_{\Gamma'}\subset \Mmscbar_{n+1}$. If $1\le i,j\le n$ then $z_{ij}^\alpha = w_{ij}^\alpha$ so $\lim_\alpha w_{ij}^\alpha = z_{ij}(x)$ for all such $i,j$. Also, $u_m^\alpha = t_m^\alpha$ for all $\alpha$ and $\lim_\alpha u_m^\alpha = 0$ for all $1\le m \le \ell$. 

Now, if $j = n+1$ and $h(i)\ne h(k)$ then $z_{ik}^\alpha = z_{i,n+1}^\alpha$ for all $\alpha$ and so $\lim_\alpha z_{i,n+1}^\alpha = z_{ik}(x)$. If $h(i) = h(k)$ then $z_{i,n+1}^\alpha = z_{i,k}^\alpha + z_{k,n+1}^\alpha$ and $\lim_{\alpha} z_{i,n+1}^\alpha$ exists if and only if $\lim_\alpha z_{k,n+1}^\alpha$ exists.

Under the projection $\rmscbar_{n+1}\to \mscbar_{n+1}$, $(y_\alpha)$ defines a net $(y_\alpha')$, which converges to $y'\in \mscbar_{n+1}$ with dual graph $\Gamma'$ by the first part. There is a map $q:U_{\Gamma'}\to U_\Gamma$ given by deleting the last marked point. In terms of the functions $\{z_{ij},t_m\}$ from before this is given by $(t,z,z_*)\mapsto (t,z)$ where $z_*$ is shorthand for the functions $z_{ij}$ with $j = n+1$ and $z$ is the remaining functions. 

We define $U_\Gamma^{\bR}$ as the preimage of $U_\Gamma$ under $\rmscbar_n \to \Mmscbar_n$ and we define $U_{\Gamma'}^{\bR} \subset \Mmscbar_{n+1}$ analogously. $U_{\Gamma}^{\bR} \to U_\Gamma$ is given by projection as described in \eqref{E:Rblowuplocally}. Explicitly, $U_{\Gamma}^{\bR} = \{(t,z,T):\lvert t_i\rvert T_i = t_i\}$ and $U_{\Gamma'}^{\bR} = \{(t,z,z^*,T): \lvert t_i\rvert T_i = t_i\}$. Define $\widetilde{q}(t,z,z^*,T) = (t,z,T)$. There is a commutative diagram
\[
    \begin{tikzcd}[ampersand replacement = \&] U_{\Gamma'}^{\bR}\arrow[d]\arrow[r,"\widetilde{q}"]\&U_\Gamma^{\bR}\arrow[d]\\
    U_{\Gamma'}\arrow[r,"q"]\& U_\Gamma.
    \end{tikzcd}
\]
Write $y_\alpha = (t_\alpha,z_\alpha,z^*_\alpha,T_\alpha)$ and $y_\alpha' = (t_\alpha,z_\alpha,z^*_\alpha)$. Since $\lim_\alpha y_\alpha' = y'$, the limit of each component of $y_\alpha'$ exists. On the other hand, $x_\alpha = \widetilde{q}(y_\alpha) = (t_\alpha,z_\alpha,T_\alpha)$ by its definition. Since $\lim_\alpha x_\alpha = x$, similar reasoning yields that $\lim_\alpha y_\alpha = y$ exists. The claimed properties are a consequence of the first part and $\widetilde{q}(y) = x$.}
\end{proof}

\section{The space of augmented stability conditions}

\subsection{The weak topology}

For any multi-scale decomposition $\langle \cC_\bullet \rangle_{\Sigma}$, $U(\langle \cC_\bullet \rangle_\Sigma) \subset \Astab(\cC)$ denotes the set of points with underlying multi-scale decomposition a coarsening of $\langle \cC_\bullet \rangle_\Sigma$. In what follows, we use some functions defined in \Cref{S:background}, which the reader may wish to review: see \eqref{E:massmeasure}, \eqref{E:smoothedphase}, \eqref{E:logcentralchargefunction}, and \Cref{S:C_invariant_functions}.

\begin{defn}[Weak topology]\label{D:weak_topology}
We define the \emph{weak topology} on $\Astab(\cC)$ to be the weakest topology such that for any multi-scale decomposition $\langle \cC_\bullet \rangle_{\Sigma}$ and collection of $t$-well-placed objects $E_1,\ldots,E_N$ with every terminal vertex $v$ occurring as $\dom(E_j)$ for some $j$:
\begin{enumerate}
    \item $U(\langle \cC_\bullet \rangle_\Sigma)$ is open; and
    \item the map $\ell_{(-)}^t(E_1,\ldots,E_N) : U(\langle \cC_\bullet \rangle_\Sigma) \to \rmscbar_N$ taking $\langle \cC'_\bullet| \sigma'_\bullet \rangle_{\Sigma'}$ to the marked multi-scale line $(\Sigma',\ell^t_{\sigma'}(E_1),\ldots,\ell^t_{\sigma'}(E_N))$ is cont\-inuous.
\end{enumerate}
Note that $\ell_{(-)}^t(E_1,\ldots, E_N): U(\langle \cC_\bullet\rangle_\Sigma)\to \rmscbar_N$ is defined by virtue of the fact that the $E_i$ are $0$-well-placed (\Cref{D:twellplaced}) for all elements of $U(\langle \cC_\bullet\rangle_\Sigma)$ by \Cref{L:0wpcoarsening}. The reader may wish to review the notation of \eqref{E:functionsastab} and \Cref{D:terminology}.
\end{defn}

Recall that a \emph{net} in a set $X$ is a directed set $(I,\le)$ along with a function $I \to X$. 

\begin{defn}[Weak convergence criteria] \label{D:convergence_conditions}
Given an augmented stab\-ility condition $\sigma = \langle \cC_\bullet | \sigma_\bullet\rangle_{\Sigma}$, a labeling $v_1,\ldots,v_n$ of the terminal components of $\Sigma$, and a choice of $P_i \in \cC^{\rm{ss}}_{\le v_i}$ for all terminal $i$, we will consider the following conditions on a net $\{\sigma_\alpha\}_{\alpha \in I}$ in $\Stab(\cC)/\bC$:
\begin{enumerate}
 
    \item The multi-scale lines $(\bP^1, \ell_\alpha(P_1),\ldots,\ell_\alpha(P_n))$ converge to $(\Sigma, \ell_\sigma(P_1),\ldots,\ell_\sigma(P_n))$ in $\rmscbar_n$. \vspace{-4mm}\label{I:marked_line_convergence_1}
    
    \item $\forall i$, $\forall E \in \cC_{\le v_i}^{\rm{ss}}$, and $\forall t>0$,
    
    \begin{enumerate}
        \item $\lim_\alpha c^t_\alpha(E) = 1$, and \label{I:cosh_bound} 
        \item $\lim_\alpha  \ell_\alpha(E/P_i) = \ell_\sigma(E/P_i)$. \label{I:log_central_charge}
    \end{enumerate} 
\end{enumerate}
\end{defn}

\begin{rem}
    The conditions of \Cref{D:convergence_conditions} are independent of $\{P_i\}_{i=1}^n$. Indeed, given another collection of $\{P_i' \in \cC_{\le v_i}^{\rm{ss}} : i=1,\ldots,n\}$, \eqref{I:log_central_charge} implies that $\lim_\alpha \ell_\alpha(P_i'/P_i) = \ell_\sigma(P_i'/P_i)$. Thus, \Cref{P:addingpointsconvergence} implies $(\bP^1,\ell_\alpha(\{P_i'\}))$ converges in $\rmscbar_n$ to $(\Sigma',\ell_\sigma(\{P_i'\})),$ which is real-orientedly isomorphic to $(\Sigma,\ell_\infty(\{P_i\}))$. One concludes by deducing \eqref{I:log_central_charge} with respect to $\{P_i'\}$ using the identity $\ell(E/P_i') = \ell(E/P_i) + \ell(P_i/P_i')$. 
\end{rem}

\begin{rem} \label{R:almost_semistable_convergence}
    The identity \eqref{E:cosh_expression} combined with Jensen's inequality $m^s_\alpha(E/E) \geq 1$ implies that \eqref{I:cosh_bound} is equivalent to the condition that $\lim_\alpha m^s_\alpha(E/E) = 1$ for $s= t$ and $s=-t$. The inequalities $1 \leq c^s_\alpha(E) \leq c^t_\alpha(E)$ for $\lvert s \rvert <t$ imply that if \eqref{I:cosh_bound} holds for $t$, then it holds for any $|s|<t$.
\end{rem}

\begin{defn}
\label{D:almostsemistable}
    Given a net $\{\sigma_\alpha\}_{\alpha \in I}$ satisfying the conditions of \Cref{D:convergence_conditions}, we call an object $E \in \cC$ \emph{almost semistable} with respect to the net if $\lim_\alpha c^t_\alpha(E)=1$ for all $t>0$.
\end{defn}

\begin{thm} \label{T:unique_limit}
A net $\{\sigma_\alpha\}_{\alpha \in I}$ in $\Stab(\cC)/\bC$ converges to $\sigma = \langle \cC_\bullet| \sigma_\bullet \rangle_\Sigma \in \Astab(\cC)$ in the weak topology if and only if the conditions of \Cref{D:convergence_conditions} hold, and in this case $\sigma$ is uniquely determined by the net as follows: Let $\{Q_a\}$ be a maximal collection of almost semistable objects such that $\lim_\alpha |\ell_\alpha(Q_i/Q_j)| = \infty$ for all $i \neq j$. Then:
\begin{enumerate}
    \item $\{Q_a\}$ is finite and $(\bP^1, \ell_\alpha(Q_1),\ldots,\ell_\alpha(Q_n))$ converges in $\rmscbar_n$ to a point whose under\-lying multi-scale line is $\Sigma$, and such that each terminal comp\-onent of $\Sigma$ contains exactly one marked point $\ell_\sigma(Q_i)$. Let $v_i$ denote the terminal component containing $\ell_\sigma(Q_i)$. 
    
    \item For any almost semistable $E \in \cC$, there is a unique index $i$ such that $\lim_\alpha \ell_\alpha(E/Q_i)$ exists. Furthermore, $\dom(E) = v_i$. 
    
    \item For any terminal component $v_i$, $\cC_{\leq v_i}$ is the smallest triangulated subcategory of $\cC$ that contains every almost semistable object $E$ with $\dom(E) \leq_{1,\infty} v_i$. 

    \item The $\sigma_{v_i}$-semistable objects in $\gr_{v_i}(\cC_\bullet)$ are the objects of the form $\Pi_{v_i}(E)$ where $E \in \cC$ is almost semistable with $\dom(E)=v_i$. The central charge of $\sigma_{v_i}$ on such an $E \in \cC_{\leq v_i}$ is given, up to the action of $\bC$, by
    \[
    Z_{v_i}(\Pi_{v_i}(E)) = e^{\lim_\alpha \ell_\alpha(E/Q_i)}.
    \]
\end{enumerate}
\end{thm}

We will prove this theorem at the end of the section, after establishing several preliminary results. For the rest of the section we will fix a choice of net $\{\sigma_\alpha\}_{\alpha \in I}$ in $\Stab(\cC)/\bC$ that satisfies the conditions of \Cref{D:convergence_conditions}.

For nets of positive real numbers $\{x_\alpha\}$ and $\{y_\alpha\}$, indexed by the same directed set, we say $x_\alpha \sim y_\alpha$ if $\lim_\alpha x_\alpha / y_\alpha = 1$. Note that this an equivalence relation, and $x_\alpha \sim y_\alpha$ implies both $x_\alpha z_\alpha \sim y_\alpha z_\alpha$ and $x_\alpha + z_\alpha \sim y_\alpha + z_\alpha$ for any other net of positive numbers $\{z_\alpha\}$.\endnote{The fact that $x_\alpha z_\alpha \sim y_\alpha z_\alpha$ is immediate from the definition, but the additive version requires some justification. Observe that \[\frac{x_\alpha+z_\alpha}{y_\alpha+z_\alpha} - 1 = \frac{y_\alpha}{y_\alpha+z_\alpha} \left(\frac{x_\alpha}{y_\alpha} - 1\right).\] The hypothesis that $y_\alpha$ and $z_\alpha$ are positive implies that $\frac{y_\alpha}{y_\alpha+z_\alpha} < 1$, and this implies that the limit of the right hand side is $0$.} We use the notation $x_\alpha \lesssim y_\alpha$ to mean that $\liminf_\alpha y_\alpha/x_\alpha \geq 1$.

For any $E \in \cC$, the scale filtration (\Cref{L:multi-scale_decomp_filtration}) can be refined by further equipping each $\gr_i(E)$ with a filtration that lifts the $\sigma_{s_i}$-HN-filtration of $\Pi_{s_i}(\gr_i(E))$ under the projection $\cC_{\leq s_i} \to \gr_{s_i}(\cC_\bullet)$. We refer to this finer filtration as a \emph{combined HN filtration of $E$} with respect to $\sigma$. Although the filtration is non-unique, the dominant projections of $\gr_i(E)$ are unique up to canonical isomorphism.

\begin{prop} \label{P:HN_mass_comparison}
If $E \in \cC$ and $F_1,\ldots,F_N$ are the associated graded objects of a combined HN filtration of $E$ with respect to $\sigma$, then
\begin{equation} \label{E:mass_additive_filtration}
    m_\alpha^t(E) \sim m_\alpha^t(F_1)+\cdots+m_\alpha^t(F_N).
\end{equation}
\end{prop}

\begin{proof}
Let $P_j = P_{\dom(F_j)}$ be as in \Cref{D:convergence_conditions} for the dominant vertex of each $F_j$. The key observation is that \Cref{D:convergence_conditions}\eqref{I:log_central_charge} implies that for all $j=1,\ldots,N-1$,
\[
    \liminf_\alpha \phi_\alpha(F_{j+1}/P_{j+1}) - \phi_\alpha(F_j/P_j) > 0.
\]
So we may choose a small positive $\delta$ and a cofinal subnet such that $\phi_\alpha(F_{j+1}/P_{j+1}) - \phi_\alpha(F_j/P_j) > 6 \delta$ for all $\alpha$ in the subnet and all $j$. Then \Cref{D:convergence_conditions}\eqref{I:cosh_bound} implies that for any fixed $t'$ with $t' > \lvert t\rvert$ and any $\epsilon>0$, we can find a further cofinal subnet such that 
\[
    c^\alpha_{t'}(F_j) < \left( \frac{1}{2} + \sqrt{\frac{1}{4} + \epsilon \frac{\cosh(t' \delta)}{\cosh(t\delta)}-\epsilon} \right)
\]
for all $\alpha$ in the subnet and all $j=1,\ldots,N$. Now let $a_j := (\phi_\alpha(F_{j}/P_j)+\phi_\alpha(F_{j+1}/P_{j+1}))/2$ for $j=1,\ldots,N-1$, $a_0=-\infty$, and $a_N=\infty$, so that for $j=1,\ldots,N$
\[
a_j-\delta>\phi_\alpha(F_j/P_i)+\delta \text{ and } a_{j-1} < \phi_\alpha(F_{j}/P_i)-\delta.
\]
Now, \cite{massmeasures}*{Lem. 6.5} implies that for $j = 1,\ldots,N$,
\begin{align*}
m^t_\alpha(F_j^{\leq a_{j-1}}) + m^t_\alpha(F_j^{>a_j-\delta}) 
&\leq m_\alpha^t(F_j^{\leq \phi_\alpha(F_j/P_i)-\delta})+m_\alpha^t(F_j^{\geq \phi_\alpha(F_j/P_i)+\delta}) \\
&\leq \epsilon m_\alpha^t(F_j),
\end{align*}
We now apply \Cref{T:filtration_inequality} to conclude
\[
m_\alpha^t(E) \leq \sum_{j=1}^N m_\alpha^t(F_j) \leq m_\alpha^t(E)+ \epsilon C_{N,\delta,t} \sum_{j=1}^N m_\alpha^t(F_j),
\]
which implies \eqref{E:mass_additive_filtration}, because $\epsilon$ can be chosen arbitrarily small. Note that we have dropped the normalization factor of $|e^{(1-\frac{it}{\pi})\ell_\alpha(P_i)}|$ from these inequalities, because the inequalities themselves are already independent of the choice of stability condition representing $\sigma_\alpha \in \Stab(\cC)/\bC$.
\end{proof}

\begin{lem}\label{L:semistable_divergence}
Let $E \in \cC_{\le v}^{\rm{ss}}$ and $F \in \cC_{\le w}^{\rm{ss}}$ be given. If $\Re((1-\frac{it}{\pi}) \mathfrak{p}(v,w))>0$, then $\lim_\alpha m^t_\alpha(E)/m^t_\alpha(F) = 0$. 
\end{lem}

\begin{proof}
\Cref{D:convergence_conditions}\eqref{I:cosh_bound} combined with \Cref{R:almost_semistable_convergence} implies that $m^t_\alpha(E) \sim |e^{(1-\frac{it}{\pi})\ell_\alpha(E)}|$ and likewise for $F$, so
\[
\frac{m^t_\alpha(E)}{m^t_\alpha(F)} \sim \left|e^{(1-\frac{it}{\pi}) (\ell_\alpha(E/P_v) - \ell_\alpha(F/P_w) - \ell_\alpha(P_w/P_v))} \right|.
\]
\Cref{D:convergence_conditions}\eqref{I:log_central_charge} now implies the claim: $\lim_\alpha \ell_\alpha(E/P_v) = \ell_\sigma(E/P_v)$ and $\lim_\alpha\ell_\alpha(F/P_w) = \ell_\sigma(F/P_w)$. The result now follows from the hypothesis $\Re((1-\frac{it}{\pi}))\mathfrak{p}(v,w) > 0$, since 
\[
    \lim_\alpha \frac{\ell_\alpha(P_w/P_v)}{\lvert \ell_\alpha(P_w/P_v)\rvert}  = \mathfrak{p}(v,w)
\] 
and thus $\lim_\alpha \Re(-(1-\tfrac{it}{\pi})\ell_\alpha(P_w/P_v)) = -\infty.$
\end{proof}

\begin{cor} \label{C:well_placed_convergence}
If $E \in \cC$ is $t$-well-placed with $\dom(E)=v$, then for any $|s| \leq t$,
\begin{gather*}
\lim_\alpha m^s_\alpha(E/P_v) = m^s_\sigma(\Pidom(E)/\Pidom(P_v)), \text{ and}\\
\lim_\alpha \ell_\alpha^s(E/P_v) = \ell_\sigma^s(\Pidom(E)/\Pidom(P_v)).
\end{gather*}
\end{cor}

\begin{proof}
Let $F_1,\ldots,F_N$ be the associated graded pieces of the combined HN filtration of $E$, and assume they are indexed such that $\dom(F_j)=v$ for $j=1,\ldots,m$ and $\dom(F_j) \neq v$ for $j>m$. \Cref{R:almost_semistable_convergence} implies that for $j=1,\ldots,m$,
\[
    m^s_\alpha(F_j/P_v) \sim |e^{(1-\frac{is}{\pi}) \ell_\alpha(F_j/P_v)}| \sim |e^{(1-\frac{is}{\pi}) \ell_\sigma(F_j/P_v)}| = m^s_\sigma(F_j/P_v).
\]
\Cref{P:HN_mass_comparison} implies that $m^s_\alpha(E/P_v) \sim \sum_{j=1}^N m^s_\alpha(F_j/P_v)$. The fact that $E$ is $t$-well-placed and \Cref{L:semistable_divergence} then imply that the terms with $j>m$ can be ignored, so
\[
    m^s_\alpha(E/P_v) \sim \sum_{j=1}^m m^s_\alpha(F_j/P_v) \sim \sum_{j=1}^{m} m^s_\sigma(F_j/P_v) = m^s_\sigma(\Pidom(E)/\Pidom(P_v)).
\]
The special case where $s=0$ shows that $\Re(\ell_\alpha(E/P_v))$ converges as desired. Also, for $s \neq 0$, we see that
\begin{align*}
    \lim_\alpha \phi^s_\alpha(E/P_v) &= \lim_\alpha \frac{1}{s} \log \left( \frac{m^s_\alpha(E/P_v)}{m_\alpha(E/P_v)} \right) \\
    &= \frac{1}{s} \log \left( \frac{m^s_\sigma(E/P_v)}{m_\sigma(E/P_v)} \right) = \phi^s_\sigma(E/P_v),
\end{align*}
so we have convergence for $\Im(\ell^s_\alpha(E/P_v))$. Finally for $\phi^0$ the inequalities $\phi^{-\epsilon}_\alpha(E/P_v) \leq \phi^0_\alpha(E/P_v) \leq \phi^{\epsilon}_\alpha(E/P_v)$ and the analogous inequalities for $\phi^s_\sigma$, combined with the fact that $\phi_\sigma^s(E/P_v)$ is continuous as a function of $s$, imply that $\lim_\alpha \phi^0_\alpha(E/P_v) = \phi^0_\sigma(E/P_v)$.
\end{proof}

\begin{lem} \label{L:simple_well_placed_criterion}
An object $E\in \cC$ is $t$-well-placed ($t \geq 0$) with $\dom(E)=v$ if and only if one has $\limsup_\alpha m^s_\alpha(E/P_v) < \infty$ for some $s>t$ and some $s<-t$.
\end{lem}

\begin{proof}
Let $F_1,\ldots,F_n$ be the associated graded pieces of a combined HN filtration of $E$, and let $v_j = \dom(F_j)$. By \Cref{P:HN_mass_comparison} and \Cref{C:well_placed_convergence},
\[
m^s_\alpha(E/P_v) \sim \sum_j m^s_\alpha(F_j/P_v) \sim \sum_j |e^{(1-\frac{is}{\pi}) \ell_\alpha(P_{v_j}/P_v)}| m^s_\sigma(F_j/P_{v_j}).
\]
Because $\ell_\alpha(P_{v_j}/P_v) \sim \mathfrak{p}(v,v_j) |\ell_\alpha(P_{v_j}/P_v)|$ and $|\ell_\alpha(P_{v_j}/P_v)| \to \infty$, the $\limsup$ of this ex\-pression is finite if and only if $\Re((1-\frac{is}{\pi}) \mathfrak{p}(v,v_j)) \leq 0$ for all $j$. Applying this for $s>t$ and $s<-t$ gives the definition of $t$-well-placed as in \Cref{D:twellplaced}.
\end{proof}

\begin{prop} \label{P:identify_well_placed}
The following are equivalent for an object $E \in \cC$ and $t\geq 0$:
\begin{enumerate}
    \item $E$ is $t$-well-placed.
    \item $\limsup_\alpha c_\alpha^{t'}(E) < \infty$ for some $t'>t$.
\end{enumerate}
Further, if these conditions hold then $\Pidom(E)$ is semistable if and only if the $\limsup$ in (2) equals $1$.
\end{prop}

\begin{proof}
Note that if $E$ is $t$-well-placed, then by definition it is also $t'$-well-placed for some $t'>t$. \Cref{C:well_placed_convergence} combined with \eqref{E:cosh_expression} then implies that $\lim_\alpha c^t_\alpha(E) = c^t_\sigma(\Pidom(E))$ exists. Also note that $\Pidom(E)$ is semistable if and only if this limit is $1$. It follows that $(1) \Rightarrow (2)$ and that $\Pi(E)$ is semistable if and only if $\limsup_\alpha c_\alpha^{t'}(E) = 1$. 

\medskip
\noindent \textit{Proof that $(2) \Rightarrow (1)$:}
\medskip

Let $F_1,\ldots,F_N \in \cC$ be the associated graded pieces of a scale filtration of $E$ with respect to $\cC=\langle \cC_\bullet \rangle_{\Sigma}$. Each $F_i$ is $\infty$-well-placed by definition, so we only need to prove this assuming $N>1$. Then \Cref{P:HN_mass_comparison} and \Cref{C:well_placed_convergence} imply that
\[
    m^t_\alpha(E/E) \sim \sum_j m^t_\alpha(F_j/E) \sim \sum_j m^t_\sigma(F_j/F_j) |e^{(1-\frac{it}{\pi})\ell_\alpha(F_j/E)}|
\]
Applying Jensen's inequality $m^t_\sigma(F_j/F_j) \geq 1$, using the estimate $m_\alpha(E) \sim \sum_{j} m_\alpha(F_j)$ from \Cref{P:HN_mass_comparison}, and adding the estimate for $m^t_\alpha$ and $m^{-t}_\alpha$ gives\endnote{\begin{align*} c^t_\alpha(E) &:= (m^t_\alpha(E/E) + m^{-t}_\alpha(E/E))/2 \\
&\sim \frac{1}{2} \sum_j m^t_\sigma(F_j/F_j) |e^{(1-\frac{it}{\pi}) \ell_\alpha(F_j/E)}| + m^{-t}_\sigma(F_j/F_j) |e^{(1+\frac{it}{\pi}) \ell_\alpha(F_j/E)}| \\
&\gtrsim \frac{1}{2} \sum_j |e^{(1-\frac{it}{\pi}) \ell_\alpha(F_j/E)}| + |e^{(1+\frac{it}{\pi}) \ell_\alpha(F_j/E)}|\\
&= \frac{1}{2} \sum_j \frac{m_\alpha(F_j)}{m_\alpha(E)} ( e^{t \phi_\alpha(F_j/E)} + e^{-t \phi_\alpha(F_j/E)}) \\
&\sim \sum_j \frac{m_\alpha(F_j)}{\sum_k m_\alpha(F_k)} \cosh(t \phi_\alpha(F_j/E))
\end{align*}}
\begin{equation} \label{E:key_cosh_bound}
    c^t_\alpha(E) \gtrsim \sum_{j} \frac{m_\alpha(F_j)}{\sum_{k} m_\alpha(F_k)} \cosh(t \phi_\alpha(F_j/E))
\end{equation}
Let $p_1(\alpha),p_2(\alpha) \in \{1,\ldots,N\}$ be the indices with the largest and second largest value of $m_\alpha(F_j)$, respectively. Likewise, let $n_1(\alpha)$ and $n_2(\alpha)$ be the indices with the smallest and second smallest value of $|\phi_\alpha(F_j/E)|$. Because $p_1, p_2, n_1,$ and $n_2$ can only take on finitely many values, we can pass to a cofinal subnet on which these indices are constant in $\alpha$.\endnote{If $D$ is an infinite directed set, and $D = \bigsqcup_{i=1}^n D_i$ is a finite partition into disjoint subsets, then one of the $D_i$ must be cofinal. To see this, assume the opposite, that is, for each $i$ there is an $\alpha_i \in D$ such that $D_i \cap \{\alpha>\alpha_i\} = \emptyset$. Then one could find an $\alpha'$ such that $\alpha' \geq \alpha_i$ for all $i$. It would follow that $D_i \cap \{\alpha > \alpha'\} = \emptyset$ for all $i$, which is absurd.} Cofinality implies that the bound \eqref{E:key_cosh_bound}, as well as other $\limsup_\alpha$ and $\liminf_\alpha$ bounds for the original net, continue to hold for the subnet. One can now bound the right-hand side of \eqref{E:key_cosh_bound} below to obtain\endnote{This is a consequence of the fact that if $a = (a_1,\ldots,a_n)$ and $b = (b_1,\ldots,b_n)$ are tuples of nonnegative numbers, and $a_1 \geq a_2 \geq \cdots \geq a_n$, then the permutation of the entries of $b$ that minimizes the dot product $a \cdot b$ is the one for which $b_1 \leq b_2 \leq \cdots \leq b_n$.}
\[
    c^t_\alpha(E) \gtrsim  \frac{m_\alpha(F_{p_1})}{\sum_{k} m_\alpha(F_k)} \cosh(t \phi_\alpha(F_{n_1}/E)) + \frac{m_\alpha(F_{p_2})}{\sum_k m_\alpha(F_k)} \cosh(t\phi_\alpha(F_{n_2}/E))
\]
We know that $\Im(\mathfrak{p}(n_1,n_2)) \neq 0$ by construction of the scale filtration, which then implies that $\liminf_\alpha |\phi_\alpha(F_{n_1}/F_{n_2})| = \infty$, and hence $\liminf_\alpha |\phi_\alpha(F_{n_2}/E)| = \infty$ by the triangle inequality.\endnote{Indeed, we have 
\[
    \lvert \phi_\alpha(F_{n_2}/E)\rvert \ge \big\lvert \lvert \phi_\alpha (F_{n_1}/E)\rvert - \lvert \phi_\alpha(F_{n_2}/F_{n_1})\rvert \big\rvert \ge \max\big\{\lvert \phi_\alpha(F_{n_1}/E)\rvert , \lvert \phi_\alpha(F_{n_2}/F_{n_1})\rvert \big\} 
\]
where the rightmost term tends to infinity.} Condition (2) of the proposition then implies 
\[
    \limsup_\alpha \frac{m_\alpha(F_{p_2})}{\sum_k m_\alpha(F_k)} = 0\quad \text{ and hence } \quad \liminf_\alpha \frac{m_\alpha(F_{p_1})}{\sum_k m_\alpha(F_k)} = 1.
\]
The original asymptotic inequality \eqref{E:key_cosh_bound} shows that con\-dition (2) of the proposition implies that $\limsup_\alpha |\phi_\alpha(F_{p_1}/E)| < \infty$, and in fact $p_1=n_1$. 

Thus after passing to a cofinal subnet we have identified an index $j^\ast \in \{1,\ldots,N\}$ such that $\limsup_\alpha |\ell_\alpha(E/F_{j^\ast})| = c < \infty$. Let $v = \dom(F_{j^\ast})$. Then \Cref{C:well_placed_convergence} and condition (2) of the proposition imply that for $s=t'$ or $s=-t'$, we have
\begin{align*}
    \limsup_\alpha m^s_\alpha(E/P_v) &=  |e^{(1-\frac{is}{\pi})\ell_\sigma(F_{j^\ast}/P_v)}| \limsup_\alpha \left( m^s_\alpha(E/E)|e^{(1-\frac{is}{\pi})\ell_\alpha(E/F_{j^\ast})}| \right)\\
    &\leq e^{c \sqrt{1+s^2/\pi^2}} |e^{(1-\frac{is}{\pi})\ell_\sigma(F_{j^\ast}/P_v)}| \limsup_\alpha m^s_\alpha(E/E) < \infty.
\end{align*}
It follows from \Cref{L:simple_well_placed_criterion} that $E$ is $t$-well-placed and $\dom(E)=v$.
\end{proof}

\begin{proof}[Proof of \Cref{T:unique_limit}]
First, by \Cref{P:identify_well_placed} any almost semistable object $E$ is $\infty$-well-placed and thus $E\in \cC_{\le \dom(E)}$. Second, for any almost semistable object $E$, $\dom(E) = v$ is the unique terminal vertex for which $\lim_\alpha \ell_\alpha(E/P_v)$ exists: Indeed, \Cref{C:well_placed_convergence} implies that $\lim_\alpha \ell_\alpha(E/P_v) = \ell_\sigma(\Pi(E)/\Pi(P_v)).$ On the other hand, given $P_w$ with $v\ne w$, we have 
\[
    \lim_\alpha \ell_\alpha(E/P_w) = \lim_\alpha \ell_\alpha(E/P_v) + \lim_\alpha \ell_\alpha(P_v/P_w)
\]
which diverges to $\infty$. Next, we verify the claim that $\{\sigma_\alpha\}_{\alpha \in I}$ converges with respect to the weak topology if and only if the conditions of \Cref{D:convergence_conditions} hold. Suppose $\{\sigma_\alpha\}$ converges with respect to the weak topology. The observation that each $P_i\in \cC_{\le v_i}^{\rm{ss}}$ is $\infty$-well-placed gives that \Cref{D:weak_topology}(2) implies \Cref{D:convergence_conditions}(1) and \eqref{I:log_central_charge}. By \Cref{R:almost_semistable_convergence}, \eqref{I:cosh_bound} follows from the fact that for any $i$, $t\ne 0$, and $E\in \cC_{\le v_i}^{\rm{ss}}$ we have $\lim_\alpha m_\alpha^t(E/E) = 1$.\endnote{The claim that $\lim_\alpha m_\alpha^t(E/E) = 1$ is the claim that $m_\alpha^t(E) \sim m_\alpha(E) e^{t \phi_\alpha(E)}$. Convergence of the net of stability conditions in the weak topology implies that $\lim_\alpha \ell^t_\alpha(E/E)=i \pi (\phi^t_\sigma(E)-\phi_\sigma(E)) = 0$, because $E$ is semistable in the quotient category, so \[1 = \lim_\alpha |e^{(1-\frac{it}{\pi}) \ell^t_\alpha(E/E)}| = \lim_\alpha e^{t(\phi_\alpha^t(E)-\phi_\alpha(E))} = \lim_\alpha \frac{m^t_\alpha(E)}{m_\alpha(E)e^{t \phi_\alpha(E)}} = \lim_\alpha \frac{m^t_\alpha(E)}{|e^{(1-\frac{it}{\pi}) \ell_\alpha(E)}|}.\]}

Conversely, suppose the conditions of \Cref{D:convergence_conditions} hold. We check that $\{\sigma_\alpha\}$ converges in the weak topology to $\sigma$. For this, the reader can verify that it suffices to check that (I) $\{\sigma_\alpha\}$ is eventually in $U(\langle \cC_\bullet\rangle_\Sigma)$ and (II) for all $E_1,\ldots, E_N$ as in \Cref{D:weak_topology}, one has $\lim_\alpha \ell^t_\alpha(E_1,\ldots, E_n) = \ell^t_\sigma(E_1,\ldots, E_N)$.\endnote{This follows from characterization of the topology on a set $X$ generated by a collection of sets $\cU$. This topology has a base constructed by taking all finite intersections of members of $\cU$. In particular, the base of the weak topology on $\Astab(\cC)$ we are considering consists of finite intersections of sets of the form $U(\langle \cC_\bullet\rangle_\Sigma)$ and preimages of open sets under the maps $\ell_{(-)}^t(E_1,\ldots, E_N)$.} (I) is immediate from the fact that $\sigma_\alpha \in U(\langle \cC_\bullet\rangle_\Sigma)$ is an open condition on the underlying multi-scale line combined with \Cref{D:convergence_conditions}\eqref{I:marked_line_convergence_1}. (II) follows from \Cref{C:well_placed_convergence}.

Next, we verify that (1)-(4) hold. By the first paragraph, it follows that $\dom:\{Q_a\}\to \{v_1,\ldots, v_n\}$ is a bijection. Index the $Q_i$ such that $\dom(Q_i) = v_i$ and so $\lim_\alpha \ell_\alpha(Q_i/P_i)$ exists for all $1\le i \le n$. \Cref{P:addingpointsconvergence} implies that $(\bP^1,\ell_\alpha(Q_1),\ldots, \ell_\alpha(Q_n))$ converges in $\rmscbar_n$ to a point whose underlying multi-scale line is equivalent to $\Sigma$ up to real-oriented isomorphism and such that $\Sigma_{v_i}$ contains a single marked point $\ell_\sigma(Q_i)$ which gives (1). Item (2) is a consequence of the second statement of the first paragraph.

Next, consider an almost semistable object $E$ with $\dom(E) = v_i$. By \Cref{P:identify_well_placed}, $\Pi(E)$ is $\sigma_{v_i}$-semistable. Conversely, consider a nonzero object $E\in \gr_{v_i}(\cC_\bullet)$ that is $\sigma_{v_i}$-semistable and its lift to $E'\in \cC_{\le v_i}$. $E'$ is $\infty$-well-placed, so by \Cref{P:identify_well_placed} one has $\lim_\alpha c_\alpha^t(E') = c_\sigma^t(E) = 1$ for all $t>0$. As a consequence, $E$ is almost semistable and the first part of (4) follows. The remaining part of (4) is by \Cref{D:generalized_stability_condition} and \Cref{P:logZ_functions}.

For (3), let a terminal component $v$ be given that is minimal for the $\le_{1,\infty}$ order so that $\cC_{<v} = 0$ and $\gr_v(\cC_\bullet) = \cC_{\le v}$. The $\sigma_v$-semistable objects are exactly the almost semistable objects $E \in \cC_{\le v}$ and by existence of HN-filtrations for $\sigma_v$ the claim follows for $v$. In general, consider a maximal chain $v_1\le_{1,\infty}\cdots \le_{1,\infty}v_n =: v$ of terminal components. Let $\cT$ denote the least triangulated subcategory of $\cC$ containing the almost semistable objects $E$ with $\dom(E) \le_{1,\infty} v$. By induction, it follows that $\cC_{<v}\subset \cT$ and by the description of the $\sigma_v$-semistable objects, we see that $\cT$ maps to all of $\gr_v(\cC_\bullet)$. Consequently, by \cite{Verdierquotient}*{Prop. 2.3.1}, $\cT = \cC_{\le v}$.
\end{proof}

\subsection{The topology defined by convergent nets}

Recall that given an augmented stability condition $\sigma = \langle \cC_\bullet |\sigma_\bullet\rangle_\Sigma$ and $v\in V(\Sigma)_{\rm{term}}$ we denote by $\cC_{\le v}^{\rm{ss}}$ the full subcategory of $\cC_{\le v}$ consisting of objects that project to $\sigma_v$-semistable objects in $\gr_v(\cC_\bullet)$.

\begin{defn}[Directed distance]
\label{D:directeddistance}
    Let $\sigma = \langle \cC_\bullet | \sigma_\bullet \rangle_{\Sigma}$ and $\tau = \langle \cB_\bullet | \tau_\bullet \rangle_{\Sigma'}$ be a pair in $\Astab(\cC)$ such that $\sigma \rightsquigarrow \tau$. For each $v\in V(\Sigma)_{\rm{term}}$ choose $P_v\in \cC_{\le v}^{\rm{ss}}$. Given $E\in \cC_{\le v} \setminus \cC_{<v}$ and $E' \equiv E \mod \cC_{<v}$, we let 
    \[
        \vec{d}_{E,E'}(\sigma,\tau) = \max\left\{ \begin{array}{c} | \phi^-_\sigma(E/P_v) - \phi^-_\tau(E'/P_v) |, \\ | \phi^+_\sigma(E/P_v) - \phi^+_\tau(E'/P_v) |, \\ | \log m_\sigma(E/P_v) - \log m_\tau(E'/P_v) | \end{array} \right\}.
    \]
    Next, we define 
    \begin{equation}
        \vec{d}(\sigma,\tau) = \sup_{\substack{v \in \Gamma(\Sigma)_{\rm{term}}, \\ E \in \cC_{\le v} \setminus \cC_{<v}}} \left( \inf_{ E'\equiv E\:\rm{mod}\:\cC_{<v}} \vec{d}_{E,E'}(\sigma,\tau)\right)
    \end{equation}
\end{defn}

We will show in \Cref{L:directedtriangleinequality} that $\vec{d}$ satisfies a directed triangle inequality necessary for the definition of a topology on $\Astab(\cC)$. 

\begin{rem}
\label{R:localmetric}
    If $\sigma$ lies in $\Stab(\cC)/\bC$ then so does any coarsening $\tau$. Choose $P \in \cC$ that is $\sigma$-stable. There is a single terminal vertex $v$ associated to $\sigma$ (namely the root) and $\cC_{\le v} = \cC$ and $\cC_{<v} = 0$. For any $E\in \cC$, choosing $E'\equiv E$ simply means choosing $E'\cong E$. Then, since all quantities in sight are well-defined on isomorphism classes of objects, $\vec{d}(\sigma,\tau) = d_P(\sigma,\tau)$, where $d_P(\sigma,\tau)$ is a metric on an open neighborhood of $\sigma$ in $\Stab(\cC)/\bC$ which we now define:

    By \cite{BridgelandSmith}*{Prop. 7.6}, there exists an open neighborhood $U$ of $\sigma$ in $\Stab(\cC)/\bC$ on which $P$ is stable. There is a holomorphic section $s_P$ of $\pi:\Stab(\cC)\to \Stab(\cC)/\bC$ over $U$ given by sending $\tau \in U$ to the unique element of $\pi^{-1}(\tau)$ for which $Z(P) = 1$. We define $d_P = s_P^*(d)$, where $d$ is the usual metric on $\Stab(\cC)$. For any $\tau,\eta \in U$ we have 
    \[
        d_P(\tau,\eta) = \sup_{0\ne E \in \cC}\left\{\lvert \phi_\tau^+(E/P) - \phi_\eta^+(E/P)\rvert, \lvert \phi_\tau^-(E/P) - \phi_\eta^-(E/P)\rvert, \lvert \log \tfrac{m_\tau(E)}{m_\tau(P)} - \log \tfrac{m_\eta(E)}{m_\eta(P)}\rvert \right\}.
    \]
    These local metrics $d_P$ characterize convergence near $\sigma$. The directed distance function $\vec{d}$ of \Cref{D:directeddistance} is intended as a generalization of this observation.
\end{rem}

\begin{ex}
    In the case where $\sigma = \langle \cC_\bullet|\sigma_\bullet\rangle_\Sigma$ is generic, in the sense that $\le_{i,0}$ is a total order on the terminal components, \Cref{D:directeddistance} simplifies significantly. Indeed, for each $v\in V(\Sigma)_{\rm{term}}$, $\cC_{<v} = 0$ and thus 
    \begin{equation}
    \label{E:simplifiedddistance}
    \vec{d}(\sigma,\tau) = \sup\left\{\vec{d}_{E,E}(\sigma,\tau): v\in V(\Gamma)_{\rm{term}}, E\in \cC_{\le v}\right\}.
    \end{equation}
\end{ex}

\begin{lem}
[Directed triangle inequality] 
\label{L:directedtriangleinequality} Consider $\sigma,\tau,\upsilon \in \Astab(\cC)$ such that $\sigma \rightsquigarrow \tau \rightsquigarrow \upsilon$. We have $\vec{d}(\sigma,\upsilon) \le \vec{d}(\sigma,\tau) + \vec{d}(\tau,\upsilon)$. 
\end{lem}

\begin{proof}
    Write the contraction of trees as $f:\Gamma(\Sigma_\sigma) \twoheadrightarrow \Gamma(\Sigma_\tau)$. The claim is trivial when at least one of $d_1:=\vec{d}(\sigma,\tau)$ or $d_2:=\vec{d}(\tau,\upsilon)$ is $\infty$. So, suppose both are finite. Let $\epsilon>0$ and $v\in V(\Sigma_\sigma)_{\rm{term}}$ be given and consider $E \in \cC_{\le v}^\sigma \setminus \cC_{<v}^\sigma$ and $E' \equiv E\mod \cC_{<v}^\sigma$ such that $\vec{d}_{E,E'}(\sigma,\tau) \le d_1 + \epsilon$. Write $w = f(v)$. It follows that $E' \in \cC_{\le v}^\sigma$ and thus in $\cC_{\le f^\dagger (w)}^\sigma$ by $v\subset f^\dagger (w)$. So, using the equivalence $\gr_w(\cC_\bullet^\tau)\simeq \gr_{f^\dagger(w)}(\cC_\bullet^\sigma)$, we can find $F \in \cC_{\le w}^\tau \setminus \cC_{<w}^\tau$ such that $F \equiv E' \mod \cC_{<w}^\tau$ and $E'' \equiv F \mod \cC_{<w}^\tau$ such that $\vec{d}_{E',E''}(\tau,\upsilon) \le d_2 + \epsilon$. The key fact is that since $\cC_{<w}^\tau \subset \cC_{<f^\dagger(w)}^\sigma \subset \cC_{<v}^\sigma$, $E''\equiv E' \mod \cC_{<w}^\tau$ and $E'\equiv E \mod \cC_{<v}^\sigma$ implies $E''\equiv E \mod \cC_{<v}^\sigma$. Now, one can verify the inequality
    \[
        \vec{d}_{E,E''}(\sigma,\upsilon) \le \vec{d}_{E,E'}(\sigma,\tau) + \vec{d}_{E',E''}(\tau,\upsilon) \le d_1 + d_2 + 2\epsilon
    \]
    using the fact that $F$ and $E'$ can be used interchangeably in formulas since $\dom_\tau(F) = \dom_\tau(E') = w$ by \Cref{L:0wpcoarsening} and $F\equiv E \mod \cC_{<w}^\tau$.\endnote{We consider the case of $\phi^+$. Since $\vec{d}_{E,E'}(\sigma,\tau) \le d_1 + \epsilon$, it follows that $\lvert\phi_\sigma^+(E/P_v) - \phi_\tau^+(E'/P_v)\rvert \le d_1 + \epsilon$. Similarly, by $\vec{d}_{F,E''}(\tau,\upsilon) \le d_2 + \epsilon$, it follows that $\lvert \phi_\tau^+(F/P_v) - \phi_\upsilon^+(E''/P_v)\rvert \le d_2 + \epsilon$. Next, since $\dom(F) = \dom(E') = w$ and $F\equiv E'\mod \cC_{<w}^\tau$, we have $\phi_\tau^+(F/P_v) = \phi_\tau(E'/P_v)$ and consequently, 
    \begin{align*}
        \lvert \phi_\sigma^+(E/P_v) - \phi_\upsilon^+(E''/P_v)\rvert & \le \lvert \phi_\sigma^+(E/P_v) - \phi_\tau^+(E'/P_v)\rvert + \lvert \phi_\tau^+(E'/P_v) - \phi_\tau^+(F/P_v)\rvert \\
        & \quad + \lvert \phi_\tau^+(F/P_v) + \phi_\upsilon^+(E''/P_v) \rvert \le d_1+d_2+2\epsilon.
    \end{align*}
    Since $\epsilon$ is arbitrary the result follows.} 
\end{proof}

The following is one of our main theorems:

\begin{thm}[The topology on $\Astab(\cC)$]\label{T:topology}
There is a unique topology on $\Astab(\cC)$ such that a net $\{\langle \cC_\bullet^\alpha | \sigma^\alpha_\bullet \rangle_{\Sigma_\alpha}\}_{\alpha \in I}$ \emph{converges} to $\sigma = \langle \cC_\bullet | \sigma_\bullet\rangle_{\Sigma}$ precisely if:
\begin{enumerate}
    \item \label{I:weaktopology}
    It converges in the weak topology (\Cref{D:weak_topology}). In particular, after passing to a subnet $\{\alpha\:|\: \alpha \geq \alpha_0\}$, we may assume that every $\langle \cC^\alpha_\bullet \rangle_{\Sigma_\alpha}$ is a coarsening of $\langle \cC_\bullet \rangle_{\Sigma}$.\smallskip

    \item \label{I:phasewidthunif}
    For any $\epsilon>0$, $\exists \alpha_0$ such that $\forall \alpha \geq \alpha_0$, $\Vec{d}(\sigma_\alpha,\sigma)<\epsilon$. \smallskip

    \item \label{I:stability} $\forall$ $\sigma$-stable $E \in \cC_{\leq v}$, $\exists \alpha_0$ such that $\forall \alpha \ge \alpha_0$, $E$ is equivalent to a $\sigma_\alpha$-stable object modulo $\cC_{<v}$.
    
\end{enumerate}
This topology is Hausdorff, the inclusion $\Stab(\cC)/\bC \hookrightarrow \Astab(\cC)$ is a homeomorphism onto an open subspace, and the subgroup of $\Aut(\cC)$ of autoequivalences that factor through an automorphism\endnote{More precisely, we mean that $F : \cC \to \cC$ is an automorphism such that there is a group homomorphism $g : \Lambda \to \Lambda$ making the following diagram commute: \[\xymatrix{ \rm{K}_0(\cC) \ar[d]^F \ar[r]^\ch & \Lambda \ar[d]^g \\  \rm{K}_0(\cC) \ar[r]^\ch & \Lambda}. \]Because $\ch$ is surjective, such a $g$ is unique if it exists, so this does not constitute additional data.} of $\Lambda$ acts by homeomorphisms on $\Astab(\cC)$.
\end{thm}

The reader may wish to consult \Cref{S:disjointpoints} where the case of $\Astab(\DCoh(\pt)^{\oplus n})$ is worked out in detail. We will prove \Cref{T:topology} at the end of this subsection.

\begin{lem}
\label{L:constanttype}
    Let $\{\sigma_\alpha\}_{\alpha \in I}$ be a net in $\Astab(\cC)$ satisfying \Cref{T:topology}\eqref{I:weaktopology}. Then $\exists \alpha_0 \in I$ and a finite decomposition $\{\alpha \in I | \alpha \geq \alpha_0 \} = I_1 \cup \cdots \cup I_K,$ where each $I_j$ is cofinal, and for any $j=1,\ldots,K$, the contraction $\Gamma(\Sigma) \twoheadrightarrow \Gamma(\Sigma_\alpha)$, the signature of $\Sigma_\alpha$, and the categories $\cC^\alpha_{\leq v}$ are constant for $\alpha \in T_j$.
\end{lem}

\begin{proof}
    There is a finite set of rooted level trees that are coarsenings of $\Gamma(\Sigma_\infty)$. Given such a tree $\Gamma$, there are three possible values of the sign of $\mathfrak{p}_{\Sigma'}(u,v)$ for $u,v\in V(\Gamma)_{\rm{term}}$ and $\Sigma'$ such that $\Gamma(\Sigma') = \Gamma$. By \Cref{L:coarseninguniquelydetermined}, the categories $\cC_{\le v}^\alpha$ are uniquely determined by these data. Therefore, there are finitely many possible collections $D=\{\Gamma(\Sigma'),\Re(\mathfrak{p}'_{u,v}),\Im(\mathfrak{p}_{u,v}'),\cC_{\le v}'\}$ where $u,v\in V(\Sigma')_{\rm{term}}$ and such that $\langle \cC_\bullet'\rangle_{\Sigma'}$ is a multi-scale decomposition coarsening $\langle \cC_\bullet\rangle_{\Sigma}$. So, we can form a finite partition of $I$ according to which element of $D$ underlies $\sigma_\alpha$, i.e. $I = \bigsqcup_{d\in D} I_d$. At least one of the $I_d$ is cofinal and the result follows.
\end{proof}

    The conditions of \Cref{T:topology} are recursive in the following sense:

\begin{lem}
\label{L:descendentconv}
    If a net $\{ \sigma_\alpha\}_{\alpha \in I}$ satisfies \Cref{D:convergence_conditions}\eqref{I:marked_line_convergence_1}, and the contraction $f:\Gamma(\Sigma) \twoheadrightarrow \Gamma(\Sigma_\alpha)$, the signature of $\Sigma_\alpha$, and the categories $\cC^\alpha_{\leq v}$ are constant in $\alpha$, then the net satisfies the remaining conditions of \Cref{T:topology} if and only if for every $u \in V(\Sigma_\alpha)_{\rm{term}}$, the net $\sigma^\alpha_u \in \Stab(\gr_u(\cC_\bullet^\alpha))/\bC$ converges to $\gr_{f^\dagger(u)}(\sigma)$ (\Cref{D:terminology}).
\end{lem}

\begin{proof}
    For any $u \in V(\Sigma_\alpha)_{\rm{term}}$, the natural choice of reference objects is $\{P_j: v_j\subseteq u\}$. By the definition of convergence in $\rmscbar_n$, one can check $(\bP^1, \{\ell_\alpha(P_j):v_j\subseteq u\})$ converges to $(\Sigma_{\subseteq u},\{\ell_\sigma(P_j):v_j\subset u\})$ and so \Cref{D:convergence_conditions}(1) follows. Now, we show that $\{\sigma_\alpha\}_{\alpha \in I}$ satisfies each remaining condition of \Cref{T:topology} with respect to $\sigma$ if and only if for each $u\in V(\Sigma)_{\rm{term}}$, $\{\sigma_u^\alpha\}$ satisfies the same condition for $\gr_{f^\dagger(u)}(\sigma)$.

    Consider $v\in V(\Sigma)_{\rm{term}}$ and let $u = f(v)$. Using the argument of the proof of \Cref{L:0wpcoarsening}, one can show that if $E\in \cC_{\le v}$, then its $\langle\cC_\bullet^\alpha\rangle_{\Sigma_\alpha}$-dominant projection in $\gr_u(\cC_\bullet^\alpha)$ corresponds to its image in $\gr_{f^\dagger(u)}(\cC_\bullet)$ under the equivalence $\gr_{f^\dagger(u)}(\cC_\bullet)\simeq \gr_u(\cC_\bullet^\alpha)$.\endnote{Slightly more generally, consider a coarsening $f:\langle \cC_\bullet\rangle_\Sigma\rightsquigarrow \langle \cB_\bullet\rangle_{\Sigma'}$. Suppose that $E$ is in $\cC_{\le v}$ with $v\in V(\Sigma)_{\rm{term}}$ and let $u = f(v)$. Consider the scale filtration $E_\bullet$ of $E$ with respect to $\langle \cB_\bullet\rangle_{\Sigma'}$ and denote the associated terminal vertices $s_1,\ldots, s_m$. As detailed in the proof of \Cref{L:0wpcoarsening}, the first step in constructing the $\langle \cC_\bullet\rangle_\Sigma$-scale filtration of $E$ is to delete the $a^{\rm{th}}$ step in $E_\bullet$ if there exists $b\ne a$ with $f^\dagger(s_a)\le_{1,\infty} f^\dagger(s_b)$.
    
    However, since $E\in \cC_{\le v}$, it follows that only $f^\dagger(u)$ survives the deletion process (otherwise, we would get a contradiction to $E$ being $\infty$-well-placed for $\langle \cC_\bullet\rangle_\Sigma$ by \Cref{R:inftywellplaced}). So, the new filtration has one step $E_1' = E$ and the $\langle \cC_\bullet\rangle_\Sigma$-scale filtration is built by lifting the scale filtration of $\Pi(E_1') = \Pi_{\rm{dom}}^{\langle \cB_\bullet\rangle}(E) \in \gr_{u}(\cB_\bullet)\simeq \gr_{f^\dagger(u)}(\cC_\bullet)$. 
    
    In particular, $\Pi_{\rm{dom}}^\alpha(E)$ is isomorphic to the image of $E$ in $\gr_{f^\dagger(u)}(\cC_\bullet)$ under the composite $\cC_{\le v}\hookrightarrow\cC_{\le f^\dagger(u)}\twoheadrightarrow \gr_{f^\dagger(u)}(\cC_\bullet)$.}

    Suppose we are given $E_1,\ldots, E_N$ as in \Cref{D:weak_topology}. If $\dom(E_j) = v$ with $f(v) = u \in V(\Sigma_\alpha)_{\rm{term}}$, then $E_j$ is $t$-well-placed for $\langle \cC_\bullet^{f^\dagger(u)}\rangle_{\Sigma_{\subseteq f^\dagger(u)}}$ if and only if it is $t$-well-placed for $\langle \cC_\bullet\rangle_\Sigma$. Now, an argument using the observation of the previous paragraph implies that for all $E_j$ with $\dom(E_j)\subseteq f^\dagger(u)$, $\ell_{\sigma_u^\alpha}^t(E_j) = \ell_\alpha^t(E_j)$ for all $\alpha$ and $t$ and it follows that $\{\sigma_\alpha\}_{\alpha \in I}$ converges in the weak topology if and only if each $\{\sigma_u^\alpha\}$ does.\endnote{We check the claim that $\ell^t_{\sigma_u^\alpha}(E_j) = \ell_\alpha^t(E_j)$. Up to replacing $E_j$ by the subquotient in its $\langle \cC_\bullet\rangle_\Sigma$-scale filtration, we may assume that $E\in \cC_{\le v}$ with $f(v) = u \in V(\Sigma)_{\rm{term}}$. By definition, $\ell^t_{\sigma_u^\alpha}(E_j)$ comes from evaluation on $E_j$ regarded as an element of $\gr_{f^\dagger(u)}(\cC_\bullet)\simeq \gr_u(\cC_\bullet^\alpha)$. On the other hand, $\ell_\alpha^t(E)$ comes from taking $\ell_{\sigma_u^\alpha}^t(\Pi_{\rm{dom}}^\alpha)$, but this is simply the image of $E$ in $\gr_{f^\dagger(u)}(\cC_\bullet)\simeq \gr_u(\cC_\bullet^\alpha)$ as observed in the body of the proof.} \Cref{T:topology}\eqref{I:phasewidthunif} goes along the same lines using the following observations:\vspace{-2mm}
    \begin{itemize}
        \item $P_v\in \cC_{\le v}^{\rm{ss}}$ if and only if $P_v\in \cC_{\le v}^{f^\dagger(u),\text{ ss}}$,   \vspace{-2mm}
        \item $E\in \cC_{\le v}\setminus \cC_{<v}$ and $E'\equiv E \mod \cC_{<v}$ if and only if $E\in \cC_{\le v}^{f^\dagger(u)}\setminus \cC_{<v}^{f^\dagger(u)}$ and $E'\equiv E \mod \cC_{<v}^{f^\dagger(u)}$, and\vspace{-2mm}
        \item in this notation, $\phi_\sigma^{\pm}(E/P_v) = \phi^{\pm}_{\gr_{f^\dagger(u)}(\sigma)}(E/P_v)$, $\phi_\alpha^{\pm}(E/P_v) = \phi_{\sigma_\alpha^u}^{\pm}(E'/P_v)$ for all $\alpha$, and analog\-ously for the mass functions.\endnote{The first item follows from \eqref{E:vertex_category_containment}, which gives $\cC_{<f^\dagger(u)}\subseteq \cC_{<v}\subseteq \cC_{\le v}\subseteq \cC_{\le f^\dagger(u)}$. In particular, $\cC_{\le v}^{f^\dagger(u)}/\cC_{<v}^{f^\dagger(u)}\simeq \gr_v(\cC_\bullet)$. Furthermore, the stability condition at vertex $v\in V(\Sigma_{\subseteq u})_{\rm{term}}$ of the descendent augmented stability condition (\Cref{D:terminology}) corresponds to that of $v\in V(\Sigma)_{\rm{term}}$ under this equivalence.
        
        The second item is along similar lines. The third item follows from the fact that given $E\in \cC_{\le v}$ with $f(v) = u$, we have $\Pi_{\rm{dom}}^\alpha(E)$ is the image of $E$ in $\gr_{f^\dagger(u)}(\cC_\bullet)\simeq \gr_u(\cC_\bullet^\alpha)$.} \vspace{-2mm}
    \end{itemize}

    Finally, we consider \eqref{I:stability}. First, note that $E\in \cC_{\le v}$ is $\sigma$-stable if and only if it is $\gr_{f^\dagger(u)}(\sigma)$-stable, when regarded as an element of $\cC_{\le v}^{f^\dagger(u)}$. The result now follows from the observation that $E'\equiv E \mod \cC_{<v}$ if and only if $E'\equiv E \mod \cC_{<v}^{f^\dagger(u)}$, regarded as objects of $\cC_{\le v}^{f^\dagger(u)}$.
\end{proof}

\begin{proof}[Proof of \Cref{T:topology}]
By a slight modification of \cite{KelleyTopology}*{Thm. 9, p.27},\endnote{
For the reader's convenience, we recall some of the relevant definitions and theorems involved in the definition of the topology. \begin{defn}
[\cite{KelleyTopology} Ch. 2] A \emph{net} in a set $X$ consists of a directed set $(D,\le)$ equipped with a function $S:D\to X$. We often refer to the net by the function $S$ only, the rest of the data being implicit. A \emph{subnet} of a net $S$ in $X$ is a net $T:D'\to X$ equipped with a function $N:D'\to D$ so that 
\begin{enumerate}
    \item $T = S\circ N$; and 
    \item for each $\alpha \in D$ there is a $\beta \in D'$ such that $\beta'\ge \beta$ implies $N(\beta')\ge \alpha$. 
\end{enumerate}
\end{defn}

Note that the second condition in Kelley's definition of a subnet is sometimes called cofinality. We have indicated this in parentheses in the body of the paper. 

\begin{defn}
[\cite{KelleyTopology} p. 74] Let $X$ denote a set. Let $\mathscr{C}$ denote a class consisting of pairs of $(S,s)$, where $S$ is a net in $X$ and $s$ is a point. We say that $S$ converges $(\mathscr{C})$ to $s$ if $(S,s)\in \mathscr{C}$. This is also denoted $\lim_\alpha S(\alpha) = s\: (\mathscr{C})$. $\mathscr{C}$ is called a \emph{convergence class} for $X$ if :
\begin{enumerate}
    \item If $S$ is a net such that $S(\alpha) = s$ for each $n$, then $S$ converges $(\mathscr{C})$ to $s$.
    \item If $S$ converges $(\mathscr{C})$ to $s$, then so does each subnet of $S$.
    \item If $S$ does not converge $(\mathscr{C})$ to $s$, then there is a subnet of $S$, no subnet of which converges $(\mathscr{C})$ to $s$. 
    \item Let $D$ be a directed set, let $E_\alpha$ be a directed set for each $\alpha\in D$, let $F = D\times \prod_{\alpha\in D} E_\alpha$, and for $(\alpha,f)\in F$, let $R(\alpha,f) = (\alpha,f(\alpha))$. If $\lim_\alpha \lim_\beta S(\alpha,\beta) = s \:(\mathscr{C})$, then $S\circ R$ converges $(\mathscr{C})$ to $s$. 
\end{enumerate}
\end{defn}

We state a slight modification of Theorem 9 of \cite{KelleyTopology} (p.74). 

\begin{thm}
\label{T:topologyvianets}
    Let $\mathscr{C}$ denote a convergence class on a set $X$. There is a unique topology $\mathscr{T}$ on $X$ with respect to which $\mathscr{C}$ is the class of convergent nets.
    \begin{proof}
    By Theorem 9 of loc. cit., the map $(-)^c:2^X\to 2^X$ given by sending $A\mapsto A^c$, defined to be the set of points $s\in X$ such that some net $S$ in $A$ converges to $s$ with respect to $\mathscr{C}$, is a Kuratowski closure operator. By Theorem 8 of p. 43, associated to such an operator, there is a unique topology with respect to which the complements of sets of the form $A^c$ are the open sets. In particular, given $(S,s)\in \mathscr{C}$, $s$ is the limit of $S$ with respect to this topology and conversely.
    \end{proof}
\end{thm}} it suffices to verify the following enumerated conditions. 
\begin{enumerate}[label=(\roman*)]
    \item All constant nets satisfy the conditions of \Cref{T:topology};
    \item If $x_\alpha$ is convergent in the sense of \Cref{T:topology}, then so are all of its cofinal subnets;
    \item If $x_\alpha$ does not converge to some point $x\in \Astab(\cC)$, then there is a cofinal subnet no subnet of which converges to $x$;
    \item Let $D$ be a directed set, let $\{x_\alpha\}_{\alpha \in D} \to x$ be a convergent net in $\Astab(\cC)$, and let $\{y_\beta\}_{\beta \in F_\alpha} \to x_\alpha$ be a convergent net for each $\alpha$, indexed by different directed sets $F_\alpha$. Then the net $\{ y_{f(\alpha)} \}_{(\alpha,f) \in G}$ indexed by the product directed set $G := D \times \prod_{\alpha \in D} F_\alpha$\endnote{Given a pair of directed sets $(D,\le)$ and $(E,\le)$, the product set $D\times E$ is directed by $(d,e)\le (d',e')$ iff $d\le d'$ and $e\le e'$. In this specific situation, $\prod_{\alpha \in D} E_\alpha$ is equipped with the pointwise preorder in the following sense: an element of $\prod_{\alpha \in D} E_\alpha$ is a function $f: D\to \coprod_{\alpha \in D}E_\alpha$ such that $f(\alpha) \in E_\alpha$ for all $\alpha$. We say $f\le g$ if $f(\alpha)\le g(\alpha)$ for all $\alpha \in D$. So, $F = D\times \prod_{\alpha \in D}E_\alpha$ has the preorder defined by $(d,f)\le (e,g)$ iff $d\le e$ and $f\le g$ in the sense just defined.} converges to $x$.
\end{enumerate}

    For the rest of the proof, $\sigma_\alpha = \langle \cC_\bullet^\alpha | \sigma^\alpha_\bullet \rangle_{\Sigma_\alpha}$ denotes a net in $\Astab(\cC)$. Condition (1) of \Cref{T:topology} defines the class of convergent nets for the weak topology, so it satisfies conditions (i)-(iv) (see the discussion in \cite{KelleyTopology}*{Ch. 2}). 
    
    For \eqref{I:phasewidthunif}, conditions (i)-(iii) are easily verified so we focus on (iv). Let $\epsilon>0$ be given and suppose $\{\sigma_\alpha\}$ satisfies \eqref{I:phasewidthunif} with respect to $\sigma$. Thus, there is an $\alpha_0$ such that $\alpha \ge \alpha_0$ implies $\vec{d}(\sigma,\sigma_\alpha)<\epsilon$. For each $\alpha$, consider a net $\{\tau_\beta\}_{\beta \in F_\alpha}$ satisfying \eqref{I:phasewidthunif} with respect to $\sigma_\alpha$. For each $\alpha$, choose $f_0(\alpha)\in F_\alpha$ such that $\beta \ge f_0(\alpha)$ implies $\vec{d}(\sigma_\alpha,\tau_\beta) < \epsilon$. Now, suppose $(\alpha,f)\ge (\alpha_0,f_0)$ in $G$. It follows that $\vec{d}(\sigma,\tau_{f(\alpha)}) \le \vec{d}(\sigma,\sigma_\alpha) + \vec{d}(\sigma_\alpha,\tau_{f(\alpha)}) < 2\epsilon$ by \Cref{L:directedtriangleinequality}. 

    Finally, we consider \eqref{I:stability}. We leave (i)-(iii) as exercises for the reader\endnote{(i) If $\{\sigma_\alpha\}$ is constant, then we can take any $\sigma_v^\infty$-stable $E\in \gr_v^\infty(\cC)$ and choose any lift $E'\in \cC_{\le v}$. (ii) Consider a cofinal subnet $\{\sigma_\beta\}_{\beta \in B}$ of $\{\sigma_\alpha\}_{\alpha \in A}$ where $\{\sigma_\alpha\}$ has property \eqref{I:stability}. Let $\alpha_0 \in A$ be as in \eqref{I:stability}. Because $B$ is cofinal, choose $\beta_0\ge \alpha_0$ with $\beta_0\in B$. It now follows that for any $\beta \ge \beta_0$ and any $E\in \gr_v^\infty(\cC)$ which is stable we can find a suitable lift $E'$. (iii) Suppose $\{\sigma_\alpha\}$ does not satisfy \eqref{I:stability} with respect to $\sigma$. Then for all $\alpha_0 \in A$ there exists $\alpha_0'\ge \alpha_0$ and a stable $E\in \gr_v^\infty(\cC)$ without a suitable lift. Collecting all of these $\alpha_0'$ defines a cofinal subnet $\{\sigma_\alpha\}_{\alpha \in A'}$ for which \eqref{I:stability} does not hold.} and focus on (iv). Using the previous notation, let $f_\alpha: \Gamma(\Sigma)\twoheadrightarrow \Gamma(\Sigma_\alpha)$ denote the contraction associated to $\sigma \rightsquigarrow \sigma_\alpha$ and $g_\beta:\Gamma(\Sigma_\alpha)\twoheadrightarrow \Gamma(\Sigma_\beta)$ that associated to $\sigma_\alpha \rightsquigarrow \tau_\beta$. Consider $E\in \gr_v(\cC_\bullet)$ which is $\sigma_v$-stable. Suppose there exists $\alpha_0$ such that $\alpha \ge \alpha_0$ implies that $E$ is equivalent to a $\sigma_\alpha$-stable object $E_\alpha$ modulo $\cC_{<v}$. 
    
    Next, suppose for each $\alpha$ there exists $f_0(\alpha)\in F_\alpha$ such that $\beta \ge f_0(\alpha)$ implies that there exists a $\tau_\beta$-stable object $E_\beta$ which equivalent to $E_\alpha$ modulo $w = f_\alpha(v)$. Since $\cC_{<w}^\alpha \subset \cC_{<f^\dagger(w)} \subset \cC_{<v}$, $E_\beta \equiv E_\alpha \mod \cC_{<w}^\alpha$ implies that $E_\beta \equiv E_\alpha \mod \cC_{<v}$ and hence $E_\beta\equiv E \mod \cC_{<v}$. Therefore, $E_\beta \in \cC_{\le v}$ and we are done.

    So, there is a unique topology on $\Astab(\cC)$ with convergent nets those defined in \Cref{T:topology}. Next, we prove that $\Astab(\cC)$ is Hausdorff. Given a convergent net $\{\sigma_\alpha\}\to \sigma$ any of its cofinal subnets also converges to $\sigma$ and so by \Cref{L:constanttype} we may suppose that $\Gamma(\Sigma)\twoheadrightarrow\Gamma(\Sigma_\alpha)$, the signature of $\Sigma_\alpha$, and $\cC_{\le v}^\alpha$ are constant in $\alpha$. By \Cref{L:descendentconv}, for each $u\in V(\Sigma_\alpha)_{\rm{term}}$ the net $\{\sigma_u^\alpha\}$ in $\Stab(\gr_u(\cC_\bullet))/\bC$ converges to $\gr_{f^\dagger(u)}(\sigma)$, which is thus uniquely determined by \Cref{T:unique_limit}. Now, by \Cref{L:recovercategories} we can recover $\sigma$ uniquely from these data.
    
    We next verify that $\Stab(\cC)/\bC \hookrightarrow \Astab(\cC)$ is a homeo\-morphism onto its image. Suppose $\{\sigma_\alpha\} \to \sigma \in \Stab(\cC)/\bC$ with respect to \Cref{D:convergence_conditions}. For any $\sigma$-stable object $P$, there exists an open neighborhood $U$ of $\sigma$ in $\Stab(\cC)/\bC$ on which $P$ is stable. We may choose $\alpha_0$ such that $\alpha \ge \alpha_0\Rightarrow$ $\sigma_\alpha \in U$. By \Cref{R:localmetric}, $\{\sigma_\alpha\}\to \sigma$ if and only if $d_P(\sigma_\alpha,\sigma)\to 0$, however this is exactly \eqref{I:phasewidthunif}. To see that the image of $\Stab(\cC)/\bC$ is open, simply note that its elements are uniquely characterized by the property that their underlying curve $\Sigma$ is irreducible. Consequently, if $\{\sigma_\alpha \}\to\sigma \in \Stab(\cC)/\bC$ with respect to the topology of \Cref{D:convergence_conditions}, by \eqref{I:weaktopology} it follows that $\{\sigma_\alpha\}$ is eventually in $U(\langle \cC\rangle_{\bP^1})\subset \Stab(\cC)/\bC$. Therefore, $\Stab(\cC)/\bC$ is open in this topology.

    Finally, the action $\Aut(\cC)\times \Astab(\cC)\to \Astab(\cC)$ is given by $(\psi,\langle \cC_\bullet|\sigma_\bullet\rangle_\Sigma)\mapsto \langle \cC_\bullet^\psi|\sigma_\bullet^\psi\rangle_\Sigma$ where 
    \begin{enumerate}
        \item the multi-scale decomposition $\langle \cC_\bullet^\psi\rangle_\Sigma$ is specified by setting $\cC_{\le v}^\psi := \psi(\cC_{\le v})$ for each $v\in V(\Sigma)_{\rm{term}}$; and 
        \item for each $v \in V(\Sigma)_{\rm{term}}$, $\sigma_v^\psi \in \Stab(\gr_v(\cC^{\psi}_\bullet))/\bC$ is the stability condition associated to $\sigma_v \in \Stab(\gr_v(\cC_\bullet))/\bC$ under the equivalence $\gr_v(\cC_\bullet) \to \gr_v(\cC_\bullet^\psi)$ induced by $\psi$. 
    \end{enumerate} 
    We leave to the reader to verify that this action is continuous, using the definition of convergent nets for the topology on $\Astab(\cC)$ in \Cref{T:topology}.\endnote{We omit the verification that $\Aut(\cC)\times \Astab(\cC)\to \Astab(\cC)$ is an action in the category of sets, besides showing that $\langle \cC_\bullet^\psi|\sigma_\bullet^\psi\rangle_\Sigma$ is an augmented stability condition. We verify the conditions of \Cref{D:multi-scale_decomposition}. 
    
    \Cref{D:multi-scale_decomposition}(1) follows from the fact that given $\cT = \Span\{\cA_1,\ldots, \cA_n\}$, all of which are subcategories of $\cC$, one has $\cT^\psi = \Span\{\cA_1^\psi,\ldots, \cA_n^\psi\}$. For \Cref{D:multi-scale_decomposition}(2), given any $w\le_{i,0} v$ for $w$ and $v$ terminal vertices there is an induced equivalence $\psi:\cC/\cC_{\le v} \cap \cC_{\le w}\to \cC/\cC_{\le v}^\psi \cap \cC_{\le w}^\psi$ under which $\cC_{\le v}/\cC_{\le v}\cap \cC_{\le w}$ is sent to $\cC_{\le v}^\psi/\cC_{\le v}^\psi \cap \cC_{\le w}^\psi$ and likewise for $\cC_{\le w}/\cC_{\le v}\cap \cC_{\le w}$. \Cref{D:multi-scale_decomposition}(3) and the second part of (4) follow from the first sentence of this paragraph. For the first part of \Cref{D:multi-scale_decomposition}(4), simply note that $\cC_{\le v}^\psi /\cC_{<v}^\psi \simeq \cC_{\le v}/\cC_{<v}$.

    Given $\psi \in \Aut(\cC)$ which factors through an automorphism of $\Lambda$, we show that the map $\psi:\Astab(\cC)\to \Astab(\cC)$ is continuous. Consider a convergent net $\sigma_\alpha =\langle \cC_\bullet^\alpha|\sigma_\bullet^\alpha\rangle_{\Sigma_\alpha}\to \langle \cC_\bullet|\sigma_\bullet\rangle_{\Sigma} = \sigma$. We claim $\sigma_\alpha^\psi = \langle \cC_\bullet^{\alpha,\psi}|\sigma_\bullet^{\alpha,\psi}\rangle_{\Sigma_\alpha}$ converges to $\langle \cC^\psi_\bullet|\sigma_\bullet^\psi\rangle_{\Sigma} = \sigma^\psi$.

    We check convergence with respect to the weak topology first (\Cref{D:weak_topology}). There exists an $\alpha_0$ such that $\alpha \ge \alpha_0$ implies $\langle \cC_\bullet^\alpha\rangle_{\Sigma_\alpha}$ is a coarsening of $\langle \cC_\bullet\rangle_{\Sigma}$. In particular, there is a contraction $\Gamma(\Sigma)\to \Gamma(\Sigma_\alpha)$ for each $\alpha$. Now, $\cC_{\le v}^\alpha \subset \cC_{\le f^\dagger(v)}$ for all $v\in V(\Sigma_\alpha)_{\rm{term}}$ and consequently $\cC_{\le v}^{\alpha,\psi} \subset \cC_{\le f^\dagger(v)}^{\psi}$. It also follows that this latter inclusion induces an equivalence $\cC_{\le v}^{\alpha,\psi}/\cC_{<v}^{\alpha,\psi} \to \gr_{f^\dagger(v)}(\cC_\bullet^\psi)$. 

    Next, given any collection of $t$-well-placed objects $E_1,\ldots, E_N$ as in \Cref{D:weak_topology}, let $E_i^\psi = \psi(E_i)$ for each $i=1,\ldots, N$. One can verify that $E_1^\psi,\ldots, E_N^\psi$ are $t$-well-placed for $\langle \cC_\bullet^\psi\rangle_\Sigma$. Now, $\ell_\alpha^{t,\psi}(E_i) = \ell_\alpha^{t}(\psi^{-1}(E_i))$ and it follows that $\ell_\alpha^{t,\psi}(E_1^\psi,\ldots, E_N^\psi) = \ell_\alpha^t(E_1,\ldots, E_N)$ converges to $(\Sigma,\ell_{\sigma^\psi}^t(E_1^\psi,\ldots, E_N^\psi)) = (\Sigma,\ell_\sigma^t(E_1,\ldots, E_N))$. Thus, $\sigma_\alpha^\psi$ converges to $\sigma^\psi$ in the weak topology. \Cref{T:topology}\eqref{I:phasewidthunif} and \eqref{I:stability} are proved along similar lines. Let us consider \eqref{I:stability} for simplicity. If $E\in \cC_{\le v}$ is $\sigma$-stable, then $\psi(E) = E^\psi$ is $\sigma^\psi$-stable. If there exists an $F\equiv E \mod\cC_{<v}$ that is $\sigma_\alpha$-stable, then $F^\psi$ is $\sigma_\alpha^\psi$-stable and $F^\psi \equiv E^\psi \mod \cC_{<v}^\psi$.} Note that this definition extends the usual left action of $\Aut(\cC)$ on $\Stab(\cC)/\bC$. 
\end{proof}

\subsection{Quasi-convergence and convergence} 
In \cite{quasiconvergence}, the notion of a \emph{quasi-con\-vergent} path in $\Stab_\Lambda(\cC)/\bC$ is introduced. These are paths that either converge for the topology on $\Stab(\cC)/\bC$, or diverge ``to infinity'' in a controlled manner \cite{quasiconvergence}*{Def. 2.8}. In this section, we connect the notion of quasi-convergence to convergence in the topology on $\Astab(\cC)$ of \Cref{T:topology}. This section uses the notation and terminology of \cite{quasiconvergence}. The reader unfamiliar with the contents of \cite{quasiconvergence} can safely skip this section on a first reading.

Quasi-convergent paths are rather general. To enforce quasi-convergent paths to converge in the topology on $\Astab(\cC)$ we have to add more hypotheses. Consider a quasi-convergent path $\sigma_s:[0,\infty)\to \Stab(\cC)/\bC$ that is numerical \cite{quasiconvergence}*{Def. 2.34} and satisfies the support property for paths of \cite{quasiconvergence}*{Def. 2.36}. 

By Cor. 2.38 of \emph{loc. cit.}, $\sigma_s$ has a finite asymptotic indexing set $P_1,\ldots, P_n$. Furthermore, Thm. 2.37 of \emph{loc. cit.} states that there are categories $\cC_{\preceq P_i}^{P_i}$ associated to $\sigma_s$ as well as certain quotient categories $\cC_{\preceq P_i}^{P_i}/\cC_{\prec P_i}^{P_i}$ with stability conditions $\sigma_i \in \Stab(\cC_{\preceq P_i}^{P_i}/\cC_{\prec P_i}^{P_i})/\bC$ for each $i$. The support property and numericity hypotheses imply that each $\sigma_i$ satisfies the support property with respect to $\Lambda_i = v(\cC_{\preceq P_i})/v(\cC_{\prec P_i})^{\rm{sat}}$ for all $i = 1,\ldots, n$.

\begin{defn}
\label{D:stronglyqconv}
    Let $\sigma_s$ be a quasi-convergent path in $\Stab(\cC)/\bC$ with asymptotic indexing set $P_1,\ldots, P_n$. $\sigma_s$ is \emph{strongly quasi-convergent} if 
    \begin{enumerate} 
        \item for all $1\le i,j,k,l\le n$, the following limit exists in $\bP^1$:
        \[ 
            \lim_{s\to \infty} \frac{\ell_s(P_j) - \ell_s(P_i)}{\ell_s(P_l)-\ell_s(P_k)}.
        \] 
        \item for all $i$ and all $\sigma_i$-stable objects $E\in \cC_{\preceq P_i}^{P_i}/\cC_{\prec P_i}^{P_i}$ there exists $s_E$ such that $s\ge s_E \Rightarrow$ there exists a lift of $E$ to $\cC_{\preceq P_i}^{P_i}$ that is $\sigma_s$-stable; and
        \item for all $\epsilon>0$ there exists a $s_0$ such that $s\ge s_0 \Rightarrow$ for any $i=1,\ldots,n$ and any nonzero $E\in \cC_{\preceq P_i}^{P_i}/\cC_{\prec P_i}^{P_i}$ there exists a lift to $E'\in \cC_{\preceq P_i}^{P_i}$ such that 
    \begin{equation}
    \label{E:extracondition}
        \max\left\{
        \begin{array}{c}
            \lvert \phi_s^+(E'/P_i) - \phi_\infty^+(E/P_i)\rvert,\\
            \lvert \phi_s^-(E'/P_i) - \phi_\infty^-(E/P_i)\rvert,\\
            \lvert\log \tfrac{m_s(E')}{m_s(P_i)} - \log \tfrac{m_\infty(E)}{m_\infty(P_i)}\rvert
        \end{array}
    \right\} < \epsilon.
    \end{equation}
    Here, $m_\infty(E)/m_\infty(P_i):=\lim_{s\to\infty} m_s(E)/m_s(P_i)$ and similarly for $\phi^{\pm}_\infty(E/P_i)$. These limits exist by definition of the categories $\cC_{\preceq P_i}^{P_i}$ and $\cC_{\prec P_i}^{P_i}$.
    \end{enumerate}
\end{defn}

\begin{lem}
\label{L:convergenceinrmscbar}
    Suppose that $\sigma_s$ as above is strongly quasi-convergent. $(\bP^1,\infty,dz,\{\ell_t(P_i)\})$ converges in $\rmscbar_n$ to $(\Sigma,\{\ell_\infty(P_i)\})$ where $\Sigma$ has exactly $n$ terminal comp\-onents.
\end{lem}

\begin{proof}
    Since the limits of $\ell_s(P_j/P_i)/\ell_s(P_l/P_k)$ exist in $\bP^1$ for all $1\le i,j,k,l\le n$, the proof of \Cref{T:spaceconstruction} implies that $(\bP^1,\infty,dz,\{\ell_s(P_i)\})$ converges in $\Mmscbar_n$ to $(\Sigma,\{\ell_\infty(P_i)\})$. Since we know $\lim_s \ell_s(P_j/P_i) = \infty$ for all $i\ne j$, it follows that there is exactly one marked point per terminal component of $\Sigma$. Finally, by \Cref{P:angularconvergence}, $(\bP^1,\omega,dz,\{\ell_s(P_i)\})$ con\-verges in $\rmscbar_n$ since $\lim_s \ell_s(P_j/P_i)/(1+\lvert \ell_s(P_j/P_i)\rvert)$ exists for all $1\le i,j\le n$ by \cite{quasiconvergence}*{Def. 2.8}.
\end{proof}

\Cref{L:convergenceinrmscbar} gives a bijection between $V(\Sigma)_{\rm{term}}$ and $\{P_1,\ldots P_n\}$. Denote the terminal vertex of $\Gamma(\Sigma)$ corresponding to $\ell_\infty(P_i)$ by $v_i$ for all $1\le i \le n$ and set $\cC_{\le v_i} := \cC_{\preceq P_i}^{P_i}$. Next, $\cC_{\prec P_i}^{P_i}$ is generated by $\cC_{\preceq P_j}^{P_j}$ for $P_i\sim^{\rm{i}} P_j$ and $P_j\prec P_i$. On the other hand, $P_j$ satisfies these conditions if and only if $\mathfrak{p}(v_j,v_i) = 1$. So, $\cC_{<v_i} = \cC_{\prec P_i}^{P_i}$. For each $1\le i \le n$, we associate to $v_i$ the stability condition $\sigma_i$ on $\cC_{\preceq P_i}^{P_i}/\cC_{\prec P_i}^{P_i} =: \gr_{v_i}(\cC_\bullet)$.

For each $v_i \in V(\Sigma)_{\rm{term}}$, we define $\Lambda_{\le v_i} = v(\cC_{\preceq P_i})^{\rm{sat}}$. By the previous discussion, $\Lambda_{<v_i}$, as defined in \Cref{S:augmentedstabilityconditions}, equals $v(\cC_{\prec P_i})^{\rm{sat}}$. Consequently, $\gr_{v_i}(\Lambda) = \Lambda_i$.

\begin{prop}
\label{P:quasiconvconv}
    Let $\sigma_s$ be a numerical quasi-convergent path in $\Stab(\cC)/\bC$ with the support property. The data $\Sigma$, $\{\cC_{\le v_i}\}$, $\{\Lambda_{\le v_i}\}$, and $\{\sigma_i\}$ given above define an augmented stability condition $\langle \cC_\bullet|\sigma_\bullet\rangle_\Sigma$. 
    \begin{enumerate} 
        \item If $\sigma_s$ satisfies \Cref{D:stronglyqconv}(1) then it converges to $\langle \cC_\bullet|\sigma_\bullet\rangle_\Sigma$ in the weak topology on $\Astab(\cC)$. 
        \item If, in addition, $\sigma_s$ is strongly quasi-convergent, then it converges to $\langle \cC_\bullet|\sigma_\bullet\rangle_\Sigma$ in the topology of \Cref{T:topology}.
    \end{enumerate}
\end{prop}

\begin{proof}
    First, we check that $\{\Lambda_{\le v_i}\}$ satisfies the conditions of \Cref{D:generalized_stability_condition}. By \cite{quasiconvergence}*{Lem. 2.35}, there is a direct sum decomposition $\Lambda_{\bQ} = \bigoplus_{F\in \cP/{\sim^{\rm{i}}}} v(\cC^F)_{\bQ}$. Each $\cC^F$ is filtered by $0\subsetneq \cC_{P_{i_1}}^F\subsetneq \cdots \subsetneq \cC_{P_{i_k}}^F = \cC.$ Choosing a sequence of subspaces $M_{v_{i_1}},\ldots, M_{v_{i_k}}$ splitting the filtration 
    \[
        0\subsetneq v(\cC_{\preceq P_{i_1}}^F)_{\bQ}\subsetneq \cdots \subsetneq v(\cC_{\preceq P_{i_k}}^F)_{\bQ} = v(\cC^F)_{\bQ}
    \]
    defines $M_v$ for all $v\in V(\Sigma)_{\rm{term}}$ such that the conditions of \Cref{D:generalized_stability_condition} hold. It remains to verify that $\langle \cC_\bullet\rangle_\Sigma$ satisfies \Cref{D:multi-scale_decomposition}.

    First, $\cC_{\le v_i}\cap \cC_{\le v_j} = 0$ if $P_i\not \sim^{\rm{i}} P_j$. In this case, since $v_k \le_{1,\infty} v_i$ if and only if $P_k \preceq P_i$ and $P_k\sim^{\mathrm{i}} P_i$ and analogously for $v_j$, there are no $v_k$ such that $v_k \le_{1,\infty} v_i$ and $v_k\le_{1,\infty}v_j$. If $P_i \sim^{\mathrm{i}} P_j$, then if $P_i \prec P_j$ we have $\cC_{\le v_i}\cap \cC_{\le v_j} = \Span\{\cC_{\le v_k}\:|\: v_k\le_{1,\infty}v_i\} = \cC_{\le v_i}$ and so \Cref{D:multi-scale_decomposition}(1) holds. For \Cref{D:multi-scale_decomposition}(2), $v_i\le_{i,0}v_j$ if and only if $P_i \prec^{\mathrm{i}} P_j$ and so $\Hom(\cC_{\le v_j},\cC_{\le v_i}) = 0$ by \cite{quasiconvergence}*{Prop. 2.20}. \Cref{D:multi-scale_decomposition}(3) follows from Lem. 2.26 \emph{ibid}.\endnote{Consider an enumeration $\cP/{\sim^{\rm{i}}} = \{F_1,\ldots, F_k\}$. Any of the categories that we are considering can be written in the form $\cT = \langle \cC_{\preceq E_1}^{F_1},\ldots, \cC_{\preceq E_k}^{F_k}\rangle$, where the factors are allowed to be zero. \cite{quasiconvergence}*{Lem. 2.26} implies that each semiorthogonal factor is thick in $\cC^{F_i}$ for $i=1,\ldots, k$. We also have a semiorthogonal decomposition $\cC = \langle \cC^{F_1},\ldots, \cC^{F_k}\rangle$. Now, the condition of $X\in \cT$ is equivalent to $\Pi_i(X) \in \cC_{\preceq E_i}^{F_i}$ for $i=1,\ldots, k$. However, $\Pi_i(X\oplus Y) = \Pi_i(X)\oplus \Pi_i(Y)$ and so thickness of $\cT$ follows from thickness of $\cC_{\preceq E_i}^{F_i}$ for each $i=1,\ldots, k$. 
    
    The fact that $\Pi_i(X\oplus Y) = \Pi_i(X)\oplus \Pi_i(Y)$ as follows. Consider the canonical filtrations $X_n\to X_{n-1}\to \cdots \to X_1 = X$ and $Y_{n}\to Y_{n-1}\to\cdots Y_1 = Y$. Taking the sum of these filtrations we have $X_n\oplus Y_n \to X_{n-1}\oplus Y_{n-1}\to \cdots \to X_1\oplus Y_1 = X\oplus Y$. Since 
    \[
        \Cone(X_{i+1}\oplus Y_{i+1} \to X_i \oplus Y_i) = \Cone(X_{i+1}\to X_i) \oplus \Cone(Y_{i+1} \to Y_i)
    \] 
    by \cite{stacks-project}*{\href{https://stacks.math.columbia.edu/tag/05QS}{Tag 05QS}}, the result follows since, for instance, $\Pi_i(X) = \Cone(X_{i+1}\to X_i)$.} Finally, the first part of \Cref{D:multi-scale_decomposition}(4) follows from the fact that $\gr_{v_i}(\cC_\bullet) \ne 0$ for each $i=1,\ldots, n$, since it admits a stability condition and \eqref{E:vertex_category_containment}. The second part is a combination of Prop. 2.20 and Thm. 2.29 \emph{ibid}.

    Next, we prove (1) of the proposition. \Cref{D:convergence_conditions}\eqref{I:marked_line_convergence_1} is by \Cref{L:convergenceinrmscbar}. To prove \eqref{I:cosh_bound} and \eqref{I:log_central_charge}, recall that $\cC_{\le v_i}^{\mathrm{ss}}$ consists of the limit semistable objects $E$ such that $E\sim P_i$.\endnote{By definition \cite{quasiconvergence}*{Thm. 2.30}, the semistable objects for $\sigma_{v_i}$ are the essential image of the limit semistable objects in $\cC_{\le v_i} = \cC_{\preceq P_i}^{P_i}$. The objects sent to zero are precisely those in $\cC_{\prec P_i}^{P_i} = \cC_{<v_i}$.} Consider such an $E$. For any $\delta>0$, there exists $s_0$ such that $s\ge s_0\Rightarrow \lvert \phi^+_s(E) - \phi^-_s(E)\rvert < \delta$. For $t>0$ and $\delta < 1/t$, one has $\lvert c_s^t(E) - 1\rvert < (1-\delta^2)^{-1}$ which implies \eqref{I:cosh_bound}.\endnote{One has the Taylor series expansion $\cosh(x) = \sum_{k=0}^\infty \tfrac{x^{2k}}{2k!}$. For $\lvert x \rvert \le \delta < 1$, this gives $\cosh(x) < \sum_{k\ge 0}(x^2)^k \le (1-\delta^2)^{-1}$. Suppose $\phi_s^+(E) - \phi_s^-(E) < \delta$ where $\delta$ is such that $t\delta < 1$. Then,
    \[
        c_s^t(E) = \frac{\sum_i \cosh t(\theta_i - \phi_s(E))\cdot  m_s(\gr_i(E))}{m_s(E)} < \frac{1}{1-\delta^2}.
    \]} \eqref{I:log_central_charge} follows immediately from \cite{quasiconvergence}*{Thm. 2.30}.\endnote{The central charge of $\sigma_{i}$ on $E$ is $Z_{i}(E) = \lim_{s\to \infty} \exp(\ell_s(E/P_i))$. Furthermore, the phase of $E$ with respect to $\sigma_i$ is $\lim_{s\to\infty} \phi_s(E) - \phi_s(P_i)$.} Finally, if $\sigma_s$ is strongly quasi-convergent, \Cref{T:topology}\eqref{I:phasewidthunif} is immediate from \Cref{D:stronglyqconv}(3), while \Cref{T:topology}\eqref{I:stability} is by \Cref{D:stronglyqconv}(2).\endnote{Let $\epsilon > 0$ be given. Choose $s_0$ such that $s\ge s_0$ implies that quantity of \eqref{E:extracondition} is less than $\epsilon$. For any $E\in \cC_{\le v_i} = \cC_{\preceq P_i}^{P_i}$, taking $E = E'$ in \eqref{I:phasewidthunif} gives that the expression of \eqref{I:phasewidthunif} is less than $\epsilon$.} 
\end{proof}

\section{First properties of \texorpdfstring{$\Astab(\cC)$}{AStab(C)}}

\subsection{The manifold-with-corners conjecture}

\begin{defn}[Complete coarsening]
    We call a coarsening $\langle \cC_\bullet \rangle_{\Sigma} \rightsquigarrow \langle \cB_\bullet \rangle_{\Sigma'}$ \emph{complete} if the underlying contraction of dual trees is bijective on terminal comp\-onents, and the terminal comp\-onents of $\Sigma'$ are totally ordered with respect to $\leq_{i,0}$.
\end{defn}

A complete coarsening gives a semi\-orthogonal decomposition $\cC = \langle \gr_{v_1}(\cC_\bullet),\ldots,\gr_{v_n}(\cC_\bullet) \rangle$, where we have indexed the terminal components so that $v_1 \leq_{i,0} v_2 \leq_{i,0} \cdots \leq_{i,0} v_n$ with respect to $\Sigma'$.

In \cite{CP10}, a notion of gluing is introduced which allows for the construction of stability conditions on a category $\cC$ in the presence of a semiorthogonal decomposition $\cC = \langle \cC_1,\ldots, \cC_n\rangle$ and stability conditions $(\sigma_i)_{i=1}^n \in \prod_{i=1}^n \Stab(\cC_i)$ whose slicings satisfy certain Hom-vanishing conditions \cite{CP10}*{Thm. 3.6}. In \cite{quasiconvergence}*{Def. 3.7}, a notion of \emph{strong gluing} is introduced which requires slightly stronger Hom-vanishing conditions. If $\sigma\in \Stab(\cC)$ is constructed from strong gluing of $(\sigma_i)_{i=1}^n \in \prod_{i=1}^n \Stab(\cC_i)$ then for each $\phi \in \bR$, we have $\cP(\phi) = \bigoplus_{i=1}^n \cP_{\sigma_i}(\phi)$ \cite{quasiconvergence}*{Lem. 3.5}.

\begin{defn}[Admissible augmented stability condition]\label{D:admissible_point}
A point $\langle \cC_\bullet | \ell_\bullet \rangle_\Sigma \in \Astab(\cC)$ is \emph{admissible} if its underlying multi-scale decomposition is admissible, and the following conditions hold:
\begin{enumerate}
    \item For any complete coarsening $\langle \cC_\bullet \rangle_{\Sigma} \rightsquigarrow \langle \cB_\bullet \rangle_{\Sigma'}$, there exist lifts of the stability con\-ditions $\sigma_{v_i} \in \Stab(\gr_{v_i}(\cC_\bullet))/\bC$, denoted $\sigma_{i}$, for each $i=1,\ldots, n$ such that the collection $(\sigma_i)_{i=1}^n \in \prod_{i=1}^n \Stab(\gr_{v_i}(\cC_\bullet))$ is strongly gluable.

    \item For any two complete coarsenings,
    and any $u,v \in \Gamma(\Sigma)_{\rm{term}}$ with $u \leq_{1,\infty} v$, the functor $\gr_v(\cC_\bullet) \to \gr_u(\cC_\bullet)$ obtained by composing the inclusion $\gr_v(\cC_\bullet) \hookrightarrow \cC$ of the first semi\-orthogonal decomposition and the projection $\cC \to \gr_u(\cC_\bullet)$ from the second semiorthogonal decomposition is mass-bounded in the sense of \Cref{D:boundedness}.\endnote{Note that this condition is tautologically true if $\leq_{i,0}$ is already a total ordering for the original multi-scale decomposition. That's because any contraction that is bijective on terminal components will give the same total ordering and hence the same semiorthogonal decomposition.}
\end{enumerate}
We denote the subset of admissible augmented stability conditions by $\Astab(\cC)^{\rm{adm}}$. One has $\Stab(\cC)/\bC \hookrightarrow \Astab(\cC)^{\rm{adm}} \subset \Astab(\cC)$.
\end{defn}

\begin{prop} \label{P:admissible_points_in_closure}
    Any admissible point of $\Astab(\cC)$ lies in the closure of $\Stab(\cC)/\bC$ in $\Astab(\cC)$.
\end{prop}

\begin{proof}
    We will show that every $\sigma = \langle \cC_\bullet|\sigma_\bullet\rangle_\Sigma$ such that $\Gamma(\Sigma)$ has two levels is the limit of a path $\sigma_s:[0,\infty)\to \Stab(\cC)/\bC$; the general case follows from induction. The curve $\Sigma$ has $n \ge 2$ nodes connecting the root component to the terminal components $v_1,\ldots,v_n$. Let the indices $0<i_1<i_2<\cdots<i_k = n$ be as in \Cref{const:SODfrommulti-scale}. Each semiorthogonal factor $\cC_{\le v_{i_a}}$ has an admissible filtration $\cC_{\le v_{i_{a-1} + 1}}\subsetneq \cdots \subsetneq \cC_{\le v_{i_a}}$ which we re-index as $\cC_{\le u_1}\subsetneq \cdots \subsetneq \cC_{\le u_r}$. 
    
    By definition, $\sigma$ contains the data of $\sigma_{u_j} \in \Stab(\gr_{u_j}(\cC_\bullet)/\bC$ for each $j=1,\ldots, r$. Let $\cC_{u_1} = \cC_{\le u_1}$ and $\cC_{u_j} = {}^\perp \cC_{u_{j-1}}\cap \cC_{\le u_j}$ for $j=2,\ldots, r$. This gives a semiorthogonal decomposition $\cC = \langle \cC_{u_1},\ldots, \cC_{u_r}\rangle$ such that the composite functor $\cC_{u_j}\hookrightarrow \cC_{\le u_j}\to \cC_{\le u_j}/\cC_{\le u_{j-1}}$ is an equivalence; thus, we regard $\sigma_{v_j}$ as an element of $\Stab(\cC_{v_j})/\bC$. Repeating this for each $a=1,\ldots, k$, we obtain a semiorthogonal decomposition $\cC = \langle \cC_1,\ldots, \cC_n\rangle$ with stability conditions $\sigma_i \in \Stab(\cC_i)/\bC$ for all $i=1,\ldots, n$.



    Next, choose a configuration of complex numbers $(z_1,\ldots, z_n)\in \bC^n$ which lifts the config\-uration of nodes, i.e. such that for each $i\ne j$ one has $\mathfrak{p}(v_i,v_j) = (z_j - z_i)/\lvert z_j-z_i\rvert$. For each $j$ such that  $i_{a-1} < j \le i_a$, we let 
    \[
        c_j(s) = z_j\cdot s + (j - i_{a-1} - 1)\log(s)\cdot i. 
    \]
    Regarded as a path in $\bC^n/\bC$, the limit of $(c_1(s),\ldots, c_n(s))$ is an element of $\cA_n^{\bR}$ with underlying multi-scale line $\Sigma$. For each $j=1,\ldots, n$, let $P_j$ denote a collection of $\sigma_j$-stable objects of $\cC_j$ such that $v(P_j)$ is a basis of $\gr_{v_j}(\Lambda)_{\bQ}$. For each $j=1,\ldots, n$, and $s\ge 0$, there is a unique lift of $\sigma_j$ to $\widetilde{\sigma}_j(s)\in \Stab(\cC_{v_j})$ given by putting the centroid of $\ell(P_{v_j})$ at $c_{j}(s)$. Because $i<j$ implies $\lim_{s\to\infty} \Im(c_j(s) - c_i(s)) = \infty$, it follows that there is $\sigma_s \in \Stab(\cC)$ strongly glued from $(\widetilde{\sigma}_1(s),\ldots, \widetilde{\sigma}_n(s))$ for all $s\gg0$.\endnote{See the argument of \cite{quasiconvergence}*{Lem. 3.11}.} The reader may verify that $\sigma_s$ converges to $\sigma = \langle \cC_\bullet|\sigma_\bullet\rangle_\Sigma$ with respect to the topology of \Cref{T:topology}.\endnote{We verify the conditions of \Cref{T:topology}. This involves verifying \Cref{D:convergence_conditions} first: \Cref{D:convergence_conditions}(1) is a consequence of the fact that $\lim_{s\to\infty} (c_1(s),\ldots, c_n(s))$ has underlying multi-scale line $\Sigma$ and using \Cref{P:addingpointsconvergence}. \Cref{D:convergence_conditions}(2) is the most involved criterion we have to check: we consider the case where $E\in \cC_{\le v_n}^{\rm{ss}}$. For simplicity, we assume that $i<j$ implies that $\mathfrak{p}(v_i,v_j) = 1$. In this case, $E$ has a canonical filtration with sub-quotients $E_i\in \cC_i$ corresponding to the semiorthogonal decomposition $\cC = \langle \cC_1,\ldots, \cC_n\rangle$, where $E_n$ is $\sigma_{n}(s)$-semistable for all $t\gg0$. On the other hand, for $1\le j \le n-1$, $E_j$ has a $\sigma_j(s)$-HN filtration with subquotients $\{E_{j,p}\}$ which are independent of $s$ for all $s\gg0$. It follows that 
    \begin{equation}
    \label{E:coshboundequation}
        c_s^t(E) = \frac{1}{\sum_{j,p}m_s(E_{j,p})}\left(\sum_{j,p} m_s(E_{j,p})\cosh t(\phi_s(E_{j,p})-\phi_s(E)) \right)
    \end{equation}
    First, consider the summand $\cosh t(\phi_s(E_n) - \phi_s(E))$. We claim that $\phi_s(E_n) - \phi_s(E) \to 0$ as $s\to\infty$. Note that here $\phi_s(E)$ refers to the average phase function so that
    \begin{align*}
        \lvert \phi_s(E) - \phi_s(E_n)\rvert  = \left\lvert \frac{1}{\sum m_s(E_{j,p})}\left(\sum_{j,p}\phi_s(E_{j,p})m_s(E_{j,p}) - m_s(E_{j,p})\phi_s(E_n)\right) \right\rvert.
    \end{align*}
    Let us consider the contribution of a summand corresponding to indices $(j,p)$ where $j<k$: 
    \begin{equation}
    \label{E:coshboundcomputation}
        \left\lvert \frac{m_s(E_{j,p})}{\sum_{j,p} m_s(E_{j,p})} (\phi_s(E_{j,p}) - \phi_s(E_n))\right\rvert \le \left\lvert \frac{m_s(E_{j,p})}{m_s(E_n)}(\phi_s(E_{j,p}) - \phi_s(E_n))\right\rvert.
    \end{equation}
    Because $\ell_s(E_k/E_{j,p})/(1+\lvert \ell_s(E_k/E_{j,p})\rvert) \to 1$, the right-hand term of \eqref{E:coshboundcomputation} tends to zero as $s\to \infty$. Consequently, the only remaining term is the $j=n$ term, which is $\phi_s(E_n)m_s(E_n) - m_s(E_n)\phi_s(E_n)$, and hence zero. So, we have proved that $\lim_{s\to\infty} \phi_s(E_n) - \phi_s(E) = 0$. Returning to our study of $c_s^t(E)$, we see that the $j = n$ term in \eqref{E:coshboundequation} tends to $1$ as $s\to\infty$, using the fact that $\lim_{s\to\infty} m_s(E_{j,p})/m_s(E_n) = 0$ for all $j\ne n$. It remains to see that the other terms do not contribute. Fix some $E_{j,p} =:F$ for $j<n$. The relevant term is 
    \begin{align}
    \label{E:coshboundfinal}
        \left\lvert \frac{m_s(F)}{m_s(E_n)}\cosh t(\phi_s(F) - \phi_s(E))\right\rvert &\le \left\lvert \frac{m_s(F)}{m_s(E_n)} \cosh t(\phi_s(F) - \phi_s(E_n))\right \rvert \\
        & = \left\lvert \frac{m_s(F)}{2m_s(E_n)}(e^{t(\phi_s(F) - \phi_s(E_n))} + e^{-t(\phi_s(F) - \phi_s(E_n))} \right \rvert 
    \end{align}
    By construction of the path, $\phi_s(F) - \phi_s(E_n) = C\log(s)$ for $C\in \bZ$ so that $e^{\pm t(\phi_s(F) - \phi_s(E_n))} = s^{\pm tC}$. On the other hand, $m_s(F)/m_s(E_n) = e^{\kappa s}$ for some $\kappa \ne 0$. It follows that \eqref{E:coshboundfinal} tends to zero. This implies \Cref{D:convergence_conditions}(2a). Along the way, we have proven enough to observe (2b): for all $E\in \cC_{\le v_n}^{\rm{ss}}$ we have $\lim_{s\to\infty} \ell_s(E/P_n) = \lim_{s\to\infty} \ell_s(E_k/P_n) = \ell_{\sigma_{n}}(E_n/P_{v_n})$. It remains to verify \Cref{T:topology}(2),(3). The idea here is that given any $E \in \cC_{\le v_n}$, it is equivalent to the object of $\cC_n = \cC_{\le v_n}/\cC_{<v_n}$ that is its projection $\Pi_n(E)$ to $\cC_n$ using the functors of the semiorthogonal decomposition. Then by construction of the path, for all $s\gg0$ (independent of $E$) we have that $\vec{d}_{E,\Pi_n(E)}(\sigma,\sigma_s) = 0$. In particular, $\vec{d}(\sigma,\sigma_s) = 0$ for all $s \gg 0$. This gives \Cref{T:topology}(2). Similarly, $E \in \cC_{\le v_n}$ is $\sigma_{v_n}$-semistable if and only if $\Pi_n(E)$ is $\sigma_{n}(s)$-semistable for all $s\gg0$, which gives \Cref{T:topology}(3).
    }
\end{proof}

\begin{conj}[Manifold-with-corners]
\label{conj:manifoldwithcorners}
    Let $\cC$ be a stable dg-category, let $\sigma = \langle \cC_\bullet | \ell_\bullet \rangle_\Sigma \in \Astab(\cC)$ be an admissible point, and let $P_1,\ldots,P_n \in \cC$ be $\sigma$-stable objects, where $n=\rank(\Lambda)$, such that the vectors $\ch(P_i)$ among $P_i$ with $\dom(P_i)=v$ span $\Lambda_v \otimes \bQ$. Then there is an open neighborhood $U$ of $\sigma$ such that the map $\ell_{(-)}(P_1,\ldots,P_n): U \to \rmscbar_n$ is a homeomorphism between $U$ and an open subset of $\rmscbar_n$.
\end{conj}

We will see in \Cref{T:genericmanifoldwithcorners} below that at generic admissible points of $\Astab(\cC)$, the gluing constr\-uction of \cite{CP10} suffices to prove \Cref{conj:manifoldwithcorners}. In a neighborhood of a non-generic boundary point, however, the conjecture predicts many more stability conditions than arise from this construction. To prove this conjecture, it may be necessary to develop a more general kind of gluing construction.

\begin{prop} \label{P:manifold_with_corners_implies_boundedness}
    \Cref{conj:manifold_with_corners_simplified} implies \Cref{conj:boundedness}.
\end{prop}
\begin{proof}
    We use a very special case of the manifold-with-corners conjecture. Let $\cC$ be a smooth and proper dg-category over a field. For any $E \in \cC$, let $\tilde{\cC} := \langle A, \cC \rangle$ be the stable idempotent completion of the dg-category containing $\cC$ as a full subcategory and with one additional exceptional object $A$ such that $\RHom_{\tilde{\cC}}(\cC,A) = 0$ and $\RHom_{\tilde{\cC}}(A,F) = \RHom_\cC(E,F)$ for any $F \in \cC$. $\tilde{\cC}$ is smooth and proper, because it is proper and has a semiorthogonal decomposition with smooth factors.

    Given a stability condition $\sigma$ on $\cC$, we construct two paths in $\Stab(\tilde{\cC})/\bC$ with the same limit-point in $\Astab(\tilde{\cC})$. First, using gluing theorem of \cite{CP10}, there is a unique stability condition $\sigma_1(r)$ on $\tilde{\cC}$ satisfying the following:
    \begin{enumerate}
    \item the $\sigma_1(r)$-stable objects are precisely the $\sigma$-stable objects in $\cC$ and $A$;
    \item $\logZ_{\sigma_1(r)}(F) = \logZ_\sigma(F)$ for any stable $F \in \cC$; and
    \item $\logZ_{\sigma_1(r)}(A) = r + i \pi \theta$ for some fixed $\theta \ll 0$.
    \end{enumerate}
    On the other hand, we can mutate $\tilde{\cC} = \langle A,\cC\rangle$ to $\tilde{\cC} = \langle \cC', A \rangle$, where $\cC \simeq \cC'$ via the mutation equivalence $F \mapsto \cofib(A \otimes \RHom(A,F) \to F)$. Let $\sigma'$ be the stability condition on $\cC'$ corresponding to $\sigma$ under this equivalence. Applying the gluing theorem again, we get another stability condition $\sigma_2(r)$ on $\tilde{\cC}$ satisfying the same conditions with $\sigma'$ in place of $\sigma$, except that $\theta \gg 0$.

    As $r \to -\infty$, both $\sigma_1(r)$ and $\sigma_2(r)$ approach the augmented stability condition $\sigma_\infty = \langle \cB_\bullet | \ell_\bullet \rangle_\Sigma$, where $\Sigma$ has one root and two terminal components $v_1,v_2$ with $\mathfrak{p}(v_1,v_2) = 1$, $\cB_{\leq v_1} = \Span\{A\}$, $\cB_{\leq v_2} = \tilde{\cC}$, and the stability condition on $\tilde{\cC} / \cB_{\leq v_1} \simeq \cC$ is $\sigma$.\endnote{Recall that there is a unique stability condition on $\cB_{\leq v_1} \cong \DCoh(k)$ up to the action of $\bC$.} Therefore for $r \ll 0$, \Cref{conj:manifold_with_corners_simplified} applied to the point $\sigma_\infty \in \Astab(\tilde{\cC})$ implies that $\sigma_1(r)$ and $\sigma_2(r)$ lie in a connected neighborhood of $\sigma_\infty$. In particular, they lie on the same path component of $\Stab(\tilde{\cC})$, so $d(\sigma_1(r),\sigma_2(r))<\infty$\endnote{This uses the manifold-with-corners conjecture in a very soft way. All that is needed that that the two glued stability conditions lie a finite distance away from each other!} --- see \Cref{fig:MWCdeformation}.

    \begin{figure}
        \centering
   \begin{tikzpicture}

  \def\xmin{-5}
  \def\xmax{0}
  \def\ymin{-3.5}
  \def\ymax{3.5}

  \begin{scope}
    \clip (\xmin,\ymin) rectangle (\xmax,\ymax);

    \fill[gray!20]
      (\xmin,\ymin) --
      plot[domain=\xmin:\xmax, samples=200, smooth, variable=\x] 
        ({\x}, {-ln(-(\x - 1))}) --
      (\xmax,\ymin) -- cycle;

    \fill[gray!20]
      (\xmin,\ymax) --
      plot[domain=\xmin:\xmax, samples=200, smooth, variable=\x] 
        ({\x}, {ln(-(\x - 1))}) --
      (\xmax,\ymax) -- cycle;

    \fill[blue!10]
      plot[domain=\xmin:\xmax, samples=200, smooth, variable=\x] 
        ({\x}, {ln(-(\x - 1))}) --
      plot[domain=\xmax:\xmin, samples=200, smooth, variable=\x] 
        ({\x}, {-ln(-(\x - 1))}) -- cycle;
  \end{scope}

  \draw[thick] (0,\ymin) -- (0,\ymax);

  \draw[thick, domain=\xmin:\xmax, samples=200, smooth, variable=\x] 
    plot ({\x}, {ln(-(\x - 1))});
    
  \draw[thick, domain=\xmin:\xmax, samples=200, smooth, variable=\x] 
    plot ({\x}, {-ln(-(\x - 1))});

  \draw[red, thick, domain=-4.5:0, samples=200, smooth, variable=\x]
    plot ({\x}, {2*ln(-(\x - 1))}); 

  \draw[red, thick, domain=-4.5:0, samples=200, smooth, variable=\x]
    plot ({\x}, {-2*ln(-(\x - 1))}); 

  \node[red] at (-1.6, 2.4) {$\sigma_1(r)$}; 
  \node[red] at (-1.6, -2.4) {$\sigma_2(r)$};

  \filldraw[black] (0,0) circle (2pt);
  \node[right] at (0,0) {$\sigma_\infty$};

  \begin{scope}[shift={(1.5,0)}] 
    \draw[fill=blue!10, draw=black] (0,0.4) rectangle +(0.5,0.4);
    \node[right] at (0.6,0.6) {Conjectural region};

    \draw[fill=gray!20, draw=black] (0,-0.4) rectangle +(0.5,0.4);
    \node[right] at (0.6,-0.2) {Glued regions};
  \end{scope}

\end{tikzpicture}
        \caption{The connected neighborhood of $\sigma_\infty$ used in the proof of \Cref{P:manifold_with_corners_implies_boundedness}. The gray glued regions are constructed using techniques from \cite{CP10}. By contrast, the existence of the blue ``conjectural'' region which allows us to deform $\sigma_1(r)$ to $\sigma_2(r)$ uses \Cref{conj:manifold_with_corners_simplified}.}
        \label{fig:MWCdeformation}
    \end{figure}

    Now, for any $\sigma$-stable object $F$ with phase in $(0,1]$, consider the exact triangle
    \[
    \RHom(A,F) \otimes A \to F \to F' \to
    \]
    Then $F$ is $\sigma_1(r)$-stable by definition, and the HN filtration of $F$ with respect to $\sigma_2(r)$ refines this exact triangle. More precisely, the HN factors of $F$ are $\Hom(A,F[i]) \otimes A[-i]$ and $F'$.\endnote{When we said $\theta \gg 0$ in the construction of $\sigma_2(r)$, it was so that the phase of $A[-i]$ is $>1$ for all $i$ such that $\Hom(A,F[i]) \neq 0$ for some $F \in \cC_{(0,1]}$.} It follows that
    \begin{align*}
        m_{\sigma_2(r)}(F) &= m_{\sigma_2(r)}(F')+\dim H^\ast(\RHom_{\tilde{\cC}}(A,F)) m_{\sigma_2(r)}(A) \\
        &= m_{\sigma_1(r)}(F) + e^r \dim H^\ast(\RHom_\cC (E,F)).
    \end{align*}
    Since $d(\sigma_1(r),\sigma_2(r))<\infty$, there is a constant $C_{E,r}>0$ such that $m_{\sigma_2(r)}(F) / m_{\sigma_1(r)}(F)<C_{E,r}$ for any such $F$. This gives $\dim H^\ast(\RHom_\cC (E,F)) < e^{-r} C_{E,r}' m_{\sigma_1(r)}(F) = e^{-r} C_{E,r}'m_\sigma(F)$, where $C_{E,r}' = (C_{E,r}+1)/C_{E,r}$.\endnote{We have 
    \[
        -m_{\sigma_1(r)}(F) < \frac{m_{\sigma_2(r)}(F)}{C_{E,r}} 
    \]
    and so
    \begin{align*}
        \dim H^*(\RHom_{\cC}(E,F) & = \left(m_{\sigma_2(r)}(F) - m_{\sigma_1(r)}(F)\right)e^{-r} \\
        &<e^{-r}(m_{\sigma_2(r)}(F) + m_{\sigma_2(r)}(F)/C_{E,r})\\
        & = e^{-r}\left(\frac{C_{E,r}+1}{C_{E,r}}\right)m_{\sigma_2(r)}(F).
    \end{align*}} By \Cref{L:simplify_mass_hom_bound}, we conclude that the original stability condition $\sigma$ on $\cC$ has a mass-Hom bound.
\end{proof}

\begin{prop}\label{P:mass_Hom_bound_implies_admissibility}
    If \Cref{conj:boundedness} holds, then for any smooth and proper $\cC$ over a field, a point $\langle \cC_\bullet | \sigma_\bullet \rangle_\Sigma \in \Astab(\cC)$ is admissible if and only if the underlying multi-scale decomp\-osition is admissible. In particular, \Cref{conj:boundedness} and \Cref{conj:manifoldwithcorners} imply \Cref{conj:manifold_with_corners_simplified}.
\end{prop}

\begin{proof}
    Given a semiorthogonal decomposition $\cC = \langle \cC_1,\ldots,\cC_n \rangle$, for any collection of stability conditions $(\sigma_j)_{j=1}^n \in \prod_{j=1}^n \Stab(\cC_j)$, there exist real numbers $t_1<\cdots<t_n$ such that $(e^{i\cdot t_j}\cdot \sigma_j)_{j=1}^n$ is strongly gluable \cite{quasiconvergence}*{Lem. 3.5}, so \Cref{D:admissible_point}(1) holds. Any semiorthogonal factor of a smooth and proper category is also smooth and proper, so \Cref{conj:boundedness} and \Cref{P:mass_hom_implies_bounded_functors} imply \Cref{D:admissible_point}(2).
\end{proof}

\subsection{Proof of the manifold-with-corners conjecture at generic points}

In this section, we prove \Cref{conj:manifoldwithcorners} for generic (as in \Cref{D:terminology}) admissible elements of $\Astab(\cC)$. 

\begin{notn} 
\label{N:genmanifoldwithcorners}
We fix a generic admissible point $\sigma = \langle \cC_\bullet|\ell_\bullet\rangle_\Sigma$ and enumerate $V(\Sigma)_{\rm{term}}$ such that $v_1\le_{i,0} \cdots \le_{i,0} v_k$. Let $P_1,\ldots, P_n \in \cC$ be $\sigma$-stable objects; i.e., each $P_i \in \gr_v(\cC_\bullet) = \cC_{\le v}$ for some $v\in V(\Sigma)_{\rm{term}}$ and is $\sigma_v$-stable. We also require that $P_v = \{P_i:P_i\in \cC_{\le v}\}$ is a basis for $\gr_v(\Lambda) \otimes \bQ$ and let $r_v = \rk\gr_v(\Lambda)$. To simplify notation, we write $\ell(-) = \ell_{(-)}(P_1,\ldots P_n).$
\end{notn}

\begin{thm}[Generic manifold-with-corners]
\label{T:genericmanifoldwithcorners}
     Suppose $\sigma = \langle \cC_\bullet|\sigma_\bullet\rangle_\Sigma \in \Astab^{\rm{adm}}(\cC)$ is generic and consider $\{P_1,\ldots, P_n\}$ as in \Cref{N:genmanifoldwithcorners}. There exists an open neighborhood $U$ of $\sigma$ such that $\ell:U\to \Mmscbar_n^{\bR}$ is a homeomorphism onto an open subset of $\cA_n^{\bR}$.
\end{thm}

Given $x,y\in \bC^n/\bC$, we write $x\approx_\epsilon y$ if when $\widetilde{x},\widetilde{y} \in \bC^n$ are the respective lifts of $x$ and $y$ with centroids at $0$, one has $\lvert \widetilde{x}_i - \widetilde{y}_i\rvert<\epsilon$ for all $i=1,\ldots, n$. For $z\in \bC^n$, let $\overline{\Delta}_\epsilon(z) := \{w\in \bC^n/\bC: z\approx_\epsilon w\}$. Note that $\overline{\Delta}_\epsilon(z)$ depends only on the class of $z$ in $\bC^n/\bC$.

\begin{lem}
\label{L:Brdeformation}
    Consider $\sigma \in \Stab(\cC)/\bC$ and $P_1,\ldots, P_n$ that are $\sigma$-stable such that $\{v(P_i)\}_{i=1}^n$ is a basis of $\Lambda_{\bQ}$. There exists an $\epsilon>0$ such that $\ell$ maps an open neighborhood of $\sigma$ biholomorphically onto $\overline{\Delta}_\epsilon(\ell(\sigma))$.
\end{lem}

\begin{proof}
    Choose a lift of $\widetilde{\sigma} \in \Stab(\cC)$ of $\sigma$. By \cite{BridgelandSmith}*{Prop. 7.6}, there is an open neighborhood $U$ of $\widetilde{\sigma}$ on which $\{P_1,\ldots, P_n\}$ are stable. For all $\tau \in \pi(U)$, $\ell(\tau) = [\logZ_\tau(P_1),\ldots, \logZ_\tau(P_n)] \in \bC^n/\bC$. Also, since $\pi:\Stab(\cC)\to \Stab(\cC)/\bC$ is open, \cite{Br07}*{Thm. 1.2} implies there exists an open $V\subset \pi(U)$ containing $\sigma$ on which $\ell = \logZ$ is a system of holomorphic coordinates.\endnote{We can choose $U$ to be small enough that $U\to \Hom(\Lambda,\bC) \to (\bC^*)^n$ given by $Z\mapsto (Z(P_1),\ldots, Z(P_n))$ is a holomorphic coordinate system. For any $\tau \in U$, $\logZ_\tau(P_i) = \ln \lvert Z_\tau(P_i)\rvert + i\pi \phi_\tau(P_i)$ and this in particular defines a logarithm of $Z_\tau(P_i)$ for all $1\le i \le n$ and all $\tau \in U$. Consider the universal covering $q:\bC^n\to (\bC^*)^n$ given by $(z_1,\ldots, z_n)\mapsto (e^{z_1},\ldots, e^{z_n})$. We obtain a commutative diagram 
    \[
    \begin{tikzcd}[ampersand replacement = \&]
        \&\bC^n\arrow[d,"q"]\\
        U\arrow[r,"Z"]\arrow[ur,dashed,"\logZ"]\&(\bC^*)^n
    \end{tikzcd}
    \] 
    In particular, $\logZ$ defines a system of holomorphic coordinates on $U$. Since the $\bC$-action respects stability, we can assume that $U$ is $\bC$-invariant up to replacing it by $\bC\cdot U$. The $\bC$-action on $U$ in these coordinates is given by $z\cdot(\logZ(P_1),\ldots, \logZ(P_n)) = (\logZ(P_1)+z,\ldots, \logZ(P_n)+z)$. Consequently, on $\pi(U) \subset \Stab(\cC)/\bC$ we get a coordinate system given by $[\logZ(P_1),\ldots, \logZ(P_n)]\in \bC^n/\bC$.} Choosing $\epsilon>0$ sufficiently small that $\overline{\Delta}_\epsilon(\ell(\sigma)) \subset \ell(V)$, the result follows. 
\end{proof}

When $\{P_1,\ldots, P_n\}$ is understood, denote the open neighborhood of $\sigma$ in \Cref{L:Brdeformation} by $\overline{\Delta}_\epsilon(\sigma)$. If $\pi:\Stab(\cC)\to \Stab(\cC)/\bC$ denotes the projection, we let $\Delta_\epsilon(\sigma):=\pi^{-1}(\overline{\Delta}_\epsilon(\sigma))$.

\begin{const}[Gluing]
\label{Const:gluing}
    Consider a generic $\sigma = \langle \cC_\bullet|\sigma_\bullet\rangle_\Sigma \in \Astab^{\rm{adm}}(\cC)$ and a coarsen\-ing of multi-scale lines $\Sigma \rightsquigarrow \Sigma'$ with $f:\Gamma(\Sigma)\twoheadrightarrow\Gamma(\Sigma')$. By \Cref{L:admissiblecontractions}, there exists a coarsening $\langle \cC_\bullet\rangle_{\Sigma}\rightsquigarrow\langle \cB_\bullet\rangle_{\Sigma'}$. In this case, for $w\in \Gamma(\Sigma')_{\rm{term}}$ we have a semiorthogonal decomposition 
    \[
        \cB_{\le w} = \langle \cC_{\le v}|v\in V(\Sigma)_{\rm{term}}\cap f^{-1}(w)\rangle.
    \]
    Consider $\vec{\tau} = (\tau_{v_i})_{i=1}^k \in \prod_i \Delta_\epsilon(\sigma_{v_i})$ such that for each $w\in V(\Sigma')_{\rm{term}}$, the collection of stability conditions $\{\tau_{v_i}|f(v_i) = w\}$ is gluable to a stability condition on $\cB_{\le v}$ in the sense of \cite{CP10}*{\S 2}. We define the \emph{gluing} of $\vec{\tau}$ along $f$ as $\langle \cB_\bullet|\gl_f(\vec{\tau})_\bullet \rangle_{\Sigma'}$ where for each $w\in \Gamma(\Sigma')_{\rm{term}}$, $\gl_f(\vec{\tau})_w = \gl(\{\sigma_{v_i}|f(v_i) = w\})$. $\langle \cB_\bullet|\gl_f(\vec{\tau})_\bullet \rangle_{\Sigma'}$ is a generic element of $\Astab(\cC)$ which coarsens $\sigma$. To emphasize the condition $\vec{\tau} \in \prod_{i} \Delta_\epsilon(\sigma_{v_i})$ we say that $\langle \cB_\bullet|\gl_f(\vec{\tau})_\bullet \rangle_{\Sigma'}$ is $\epsilon$\emph{-glued} from $\sigma$.\endnote{We verify that $\langle \cB_\bullet|\gl_f(\vec{\tau})_\bullet \rangle_{\Sigma'}$ is a generic element of $\Astab^{\rm{adm}}(\cC)$ which coarsens $\sigma$. First of all, we check that $\langle \cB_\bullet \rangle_{\Sigma'}$ is a multi-scale decomposition. For each $w\in V(\Sigma')_{\rm{term}}$, $\cB_{\le w}$ is a thick triangulated subcategory of $\cC$ since it is a semiorthogonal component of $\cC$.

    First, note that if $w\le_{i,0} w'$, then $f(w)\le_{i,0} f(w')$ by \Cref{D:multi-scaleSODcoarsening}(a). Consequently, if $\cB_{\le v} = \langle \cC_{\le w}:f(w) = v\rangle$ and $\cB_{\le v'} = \langle \cC_{\le w'}:f(w') = v'\rangle$, then $v\le_{i,0} v'$ if and only if $w\le_{i,0} w'$ for all $f(w) =v$ and $f(w') = v'$. This implies \Cref{D:multi-scale_decomposition}(1) and (2). (3) is because all of the categories $\cC_{<v}$ and $\cC_{\le v}$ for $v\in V(\Sigma')$ are admissible (simply being coarsenings of the semiorthogonal decomposition indexed by the terminal components) and hence thick. Finally, for (4) note that $\gr_v(\cB_\bullet) = \cB_{\le v}\supset \cC_{\le w}\ne 0$ for $f(w) = v$.}
\end{const}

\begin{lem}
\label{L:gluedexist}
    There exist $\epsilon >0$ and an open neighborhood $V_\epsilon \subset \rmscbar_n$ of $\ell(\sigma)$ such that for any $x\in V_\epsilon$, there is a unique $\tau \in \Astab(\cC)$ that is $\epsilon$-glued from $\sigma$ such that $\ell(\tau) = x$.
\end{lem}

\begin{proof}
    We use the notation of \Cref{Const:gluing}. Let $x$ have underlying multi-scale line $\Sigma'$, which coarsens $\Sigma$ via $f:\Gamma(\Sigma)\twoheadrightarrow \Gamma(\Sigma')$, and has marked points $q_1,\ldots, q_n$. Let $q_v = \{q_i|P_i\in \cC_{\le v}\}$ for each $v\in V(\Sigma)_{\rm{term}}$. By the topology on $\cA_n^{\bR}$, for any $L\gg0$ and $\epsilon>0$ small, there exists an open neighborhood $V_{\epsilon,L}$ of $\ell(\sigma)$ such that for all $x\in V_{\epsilon,L}:$
    \begin{itemize}
        \item for all $w\in \Gamma(\Sigma')_{\rm{term}}$, letting $\{v_1 \le_{i,0}\cdots\le_{i,0} v_p\} = \Gamma(\Sigma)_{\rm{term}} \cap f^{-1}(w)$ we have $\Im(c_j) - \Im(c_{j-1}) > L$ for all $j$, where $c_j$ is the centroid of $q_{v_j}$; and 
        \item $q_{v}\approx_\epsilon \ell_\sigma(P_{v})$ for all $v\in \Gamma(\Sigma)_{\rm{term}}$.
    \end{itemize}
    Now, fix $w$ as above and choose a representative configuration of $(q_{v_1},\ldots,q_{v_p})$ in $\bC$ which is unique up to overall translation. For $\epsilon>0$ sufficiently small, for each $j=1,\ldots, p$, \Cref{L:Brdeformation} gives a unique $\tau_{v_j} \in \Delta_\epsilon(\sigma_{v_j})$ such that $\ell_{\tau_{v_j}}(P_{v_j}) = q_{v_j}$. Next, for all $L\gg0$ as above, $(\tau_{v_i})_{i=1}^p \in \prod_{i=1}^p \Stab(\cC_{\le v_i})$ is gluable to $\tau_w \in \Stab(\cB_{\le w})$ with $\ell_{\tau_w}(P_{v_i}) = q_{v_i}$ for all $i=1,\ldots, p$. The class of $\tau_w$ in $\Stab(\cB_{\le w})/\bC$ is independent of the choice of representative configuration in $\bC$. Repeating this for each $w\in V(\Sigma')_{\rm{term}}$ constructs a unique $\tau = \langle \cB_\bullet|\tau_\bullet\rangle_{\Sigma'}$ which is $\epsilon$-glued from $\sigma$ and such that $\ell(\tau) = x$, for any $x\in V_{\epsilon,L}$.
\end{proof}

For any $\epsilon>0$, $U_\epsilon = \{\tau \in \Astab(\cC)| \sigma \rightsquigarrow\tau \text{ and }\vec{d}(\sigma,\tau)<\epsilon\}$ is an open neighborhood of $\sigma$. By \Cref{T:topology}\eqref{I:stability}, there exists an open neighborhood $\Omega$ of $\sigma$ such that $\forall\tau \in \Omega$ each $P_v$ for $v\in V(\Sigma)_{\rm{term}}$ is equivalent modulo $\cC_{<v}$ to a $\tau$-stable object. We define an open set $W_\epsilon:= U_\epsilon \cap \Omega \cap \ell^{-1}(V_\epsilon)$, where $V_\epsilon$ is as in \Cref{L:gluedexist}.

\begin{lem}
\label{L:allglued}
    There exists an $\epsilon > 0$ such that all elements of $W_\epsilon$ are $\epsilon$-glued from $\sigma$. 
\end{lem}

\begin{proof}
    Consider $\eta = \langle \cB_\bullet|\tau_\bullet\rangle_{\Sigma'} \in W_\epsilon(\sigma)$. For $\epsilon>0$ small, by \Cref{L:gluedexist} there is a $\tau$ that is $\epsilon$-glued from $\sigma$ such that $\ell(\eta) = \ell(\tau)$. We claim $\eta = \tau$. It suffices to consider the case where $\tau$ and $\eta$ are in $\Stab(\cC)/\bC$. Choose an $n$-marked multi-scale line $(\bP^1,dz,\infty,p_1,\ldots,p_n)$ representing $\ell(\tau)$. 
    
    Up to shrinking $V_\epsilon = V_{\epsilon,L}$ by taking $L$ large (see the proof of \Cref{L:gluedexist}), we may assume that $\tau$ is \emph{strongly} glued from $(\tau_{i})_{i=1}^k \in \prod_{i=1}^k \overline{\Delta}_\epsilon(\sigma_{v_i})$. The key property of strong gluing is that for all $\phi \in \bR$, we have $\cP(\phi) = \bigoplus_{i=1}^k \cP_{\tau_i}(\phi)$ \cite{quasiconvergence}*{Lem. 3.5}. There is a unique lift of $\tau_i$ to $\Stab(\cC_{\le v_i})$ with $\ell(P_{v_i}) = p_{v_i}$ for all $i=1,\ldots, k$. By abuse of notation, we denote each lift by $\tau_i$ and write $\gl(\tau_i) = \tau$. There is also a unique lift of $\eta$ to $\Stab(\cC)$ with $\ell_\eta(P_i) = p_i$ for each $i$, also denoted $\eta$.
    
    Because $\eta \in W_\epsilon(\sigma)$, each $P_i \in \cC_{\le v}$ is equivalent to an $\eta$-stable object mod $\cC_{<v} = 0$. Since $\sigma$ is generic, so is $\eta$, and so $P_i \in \cC_{\le v} \subset \cB_{\le f(v)}$ is also $\eta_{f(v)}$-stable. So, $\ell_\eta(P_i) = \logZ_{\eta}(P_i)$ for each $i$ and $\ell(\eta) = \ell(\tau)$ implies that $Z_\eta = Z_{\tau}$. By \cite{Br07}*{Lem. 6.4}, it now suffices to show $d(\cP_\eta,\cP_{\tau})<1$. Because $\tau$ is strongly glued, we may assume $E\in \cP_{\tau_i}(\phi)$ for some $i$. Since $\tau_i \in \Delta_\epsilon(\sigma_i)$, we have $\lvert \phi_{\tau}(E/P_i) - \phi^{\pm}_{\sigma_i}(E/P_i)\rvert<\epsilon$ and because $\eta \in W_\epsilon(\sigma)$, by \eqref{E:simplifiedddistance}, $\lvert \phi_{\sigma_i}^{\pm}(E/P_i) - \phi_\eta^{\pm}(E/P_i)\rvert < \epsilon$. Therefore, $\lvert \phi_{\sigma'}(E/P_i) - \phi_\tau^{\pm}(E/P_i)\rvert \le 2\epsilon$. 
\end{proof}

\begin{proof}[Proof of \Cref{T:genericmanifoldwithcorners}]
    Let $\epsilon>0$ and $W_\epsilon$ be given such that \Cref{L:gluedexist} and \Cref{L:allglued} hold. Then $\ell:W_\epsilon \to \rmscbar_n$ is bijective onto its image and continuous by \Cref{D:weak_topology}.\endnote{If $x\in \ell(W_\epsilon),$ then in particular $x\in V_\epsilon$ and by \Cref{L:gluedexist} there is a unique $\tau \in \Astab(\cC)$ that is $\epsilon$-glued from $\sigma$ such that $\ell(\tau) = x$. On the other hand, if $\tau \in W_\epsilon$ has $\ell(\tau) = x$, then by \Cref{L:allglued} it is necessarily $\epsilon$-glued from $\sigma$ and hence unique. This gives the necessary injectivity.} Next, we claim that for $\delta\le \epsilon$ sufficiently small, $W_\epsilon$ maps onto $V_\delta$. Since $\delta \le \epsilon$, for all $x\in V_\delta$ there exists a unique $\tau \in \Astab(\cC)$ that is $\epsilon$-glued from $\sigma$ with $\ell(\tau) = x$. As discussed in the proof of \Cref{L:gluedexist}, we can take $\tau$ to be strongly glued from $\sigma$. Using this, one can verify that there exists $\delta\le \epsilon$ such that $q_{v}\approx_\delta \ell(P_{v})$ for all $v\in V(\Sigma)_{\rm{term}}$ implies $\tau \in W_\epsilon$.\endnote{The condition that $\tau \in U_\epsilon$ means that $\vec{d}(\sigma,\tau)<\epsilon$. Since $\sigma$ is generic, by \eqref{E:simplifiedddistance} this is equivalent to 
    \begin{equation}
    \label{E:vectordistance}
        \max\left\{\lvert \phi_\sigma^{\pm}(E/P_{v}) - \phi_\tau^{\pm}(E/P_{v})\rvert,\lvert \log m_\sigma(E/P_{v}) -\log m_\tau(E/P_{v})\rvert\right\}<\epsilon
    \end{equation}
    for all $v \in V(\Sigma)_{\rm{term}}$ and all $E\in \cC_{\le v}$ (here we abuse notation a bit by letting $P_v$ refer to some/any choice of $P_i\in P_v$). Let $\tau = \langle \cB_\bullet|\tau_\bullet\rangle_{\Sigma'}$ and consider $w\in V(\Sigma')_{\rm{term}}$. Let 
    \[ 
    f^{-1}(w) \cap V(\Sigma)_{\rm{term}} = \{v_1\le_{i,0}\cdots\le_{i,0} v_p\}
    \]
    Then $\tau_w$ is strongly glued from stability conditions $\tau_{v_1},\ldots, \tau_{v_p}$ with $\tau_{v_i} \in \overline{\Delta}_\epsilon(\sigma_{v_i})$ for each $i=1,\ldots, p$. Any $E\in \cC_{\le v_i}$ has a $\tau_{v_i}$-HN-filtration with subquotients $H_* \in \cP_{\tau_{v_i}}(\phi_*)$. However, since $\tau_w$ is strongly glued from $\tau_{v_1},\ldots, \tau_{v_p}$, $\cP_{\tau_{v_i}}(\phi)\subset \cP_{\tau_w}(\phi)$ for each $\phi \in \bR$. Thus, the $\tau_{v_i}$-HN filtration of $E$ is also the $\tau_w$-HN filtration of $E$. Also, $m_{\tau_w}(E/P_{v_i}) = m_{\tau_{v_i}}(E/P_{v_i})$. In particular, \eqref{E:vectordistance} becomes 
    \[
    \max\left\{\lvert \phi_{\sigma_{v}}^{\pm}(E/P_{v}) - \phi_{\tau_v}^{\pm}(E/P_{v})\rvert,\lvert \log m_{\sigma_v}(E/P_{v}) -\log m_{\tau_v}(E/P_{v})\rvert\right\}<\epsilon
    \]
    for all $v \in V(\Sigma)_{\rm{term}}$ and $E\in \cC_{\le v}$. This in turn is equivalent to $d_{P_v}(\sigma_v,\tau_v)<\epsilon$ for all $v \in V(\Sigma)_{\rm{term}}$, where $d_{P_v}$ is the local metric of \Cref{R:localmetric}. By \Cref{L:Brdeformation}, we can find a $\delta$ sufficiently small that $\tau_v\in \overline{\Delta}_\delta(\sigma_v)$ implies $P_v$ is $\tau_v$-stable and $d_{P_v}(\sigma_v,\tau_v)<\epsilon$. At the level of multi-scale lines, this translates into the condition that $q_v\approx_\delta \ell(P_v)$ for all $v\in V(\Sigma)_{\rm{term}}$.
    } So, $\ell:W_\epsilon \to \cA_n^{\bR}$ maps surjectively onto $V_\delta$ and we take $U = \ell^{-1}(V_\delta)\cap W_\epsilon$, noting that $\ell(U) = V_\delta$ is open. 
    
    Since $\rmscbar_n$ is first countable and locally path connected, to show continuity of $\ell^{-1}$ it suffices to show that $\ell^{-1}$ respects limits of convergent paths $\gamma:[0,\infty)\to \ell(U)$.\endnote{Since $\rmscbar_n$ is a manifold with corners, it is first countable and in particular sequential. That is, continuity of maps out of $\rmscbar_n$ can be verified using limits of sequences. Next, suppose given $x_n \to x$ in $\rmscbar_n$. Since $\rmscbar_n$ is locally path connected, choose a path connected open neighborhood $V$ of $x$ and $n_0 \in \bN$ such that $n\ge n_0 \Rightarrow x_n \in V$. Next, define a (continuous) path $\gamma:[n_0,\infty)\to V$ such that $\gamma(n) = x_n$ for all $n\ge n_0$. $x_{(-)}:\{n\in \bN:n\ge n_0\}\to V$ is a cofinal subnet of $\gamma:[0,\infty)\to V$, so if given a map $f:\rmscbar_n \to X$, for $X$ a topological space, $\lim_{t\to\infty} (f\circ \gamma)(s) = f(x)$ implies $\lim_{n\to\infty} f(x_n) = f(x)$.} We will consider the case where $\gamma(t)$ is irreducible for all $t$, the general case being similar. Let $\sigma_t = \ell^{-1}\gamma(t)$ for all $t$ and let $\sigma_\infty := \ell^{-1}(\lim_{t\to\infty} \gamma(t))$. By \Cref{L:gluedexist}, for all $t\gg 0$, there exists $(\tau_{i,t})_{i=1}^k\in \prod_{i=1}^k \Delta_\epsilon(\sigma_{v_i})$, unique up to $\bC$ action, with $\rm{gl}(\tau_{i,t})_{i=1}^k = \sigma_t$ in $\Stab(\cC)/\bC$ and such that $\lim_{t\to\infty} \tau_{i,t} = \sigma_{v_i}\in \Stab(\cC_{\le v_i})/\bC$ for all $1\le i \le k$.\endnote{This is because $\{\logZ_t(P_v)\}\in \bC^{r_v}/\bC$ gives local coordinates on $\overline{\Delta}_\epsilon(\sigma_v)$. By definition of convergence in $\cA_n^{\bR}$, it also follows that $\lim_{t\to\infty} \{\logZ_t(P_v)\} = \{\logZ_{\sigma_v}(P_v)\}$. Consequently, in $\overline{\Delta}_\epsilon(\sigma_{v_i})$ it follows that $\sigma_{i,t}\to \sigma_{v_i}$.} 
    
    We show that $\sigma_t$ is strongly quasi-convergent (\Cref{D:stronglyqconv}). Consider $(c_{1,t},\ldots, c_{k,t}) \in \bC^k/\bC$, where $c_{i,t}$ is the centroid of $\ell(P_{v_i})$ for $i=1,\ldots, k$. Since $\gamma(t)$ is converging to a generic boundary point $\Sigma$, $\lim_{t\to\infty} \Im(c_{j,t} - c_{i,t}) = \infty$ for all $i<j$. So, for all $t\gg0$, $\sigma_t$ is strongly glued from $(\sigma_{i,t})_{i=1}^k$. Since $\lim_{t\to\infty} \tau_{i,t} =  \sigma_{v_i}$ for all $i$, the limit semistable objects for $\sigma_t$ are $\bigcup_{i=1}^k \cP_{\sigma_i}$, where $\cP_{\sigma_i} := \bigcup_{\phi \in \bR} \cP_{\sigma_i}(\phi)$ is embedded using $\cC_{\le v_i}\hookrightarrow \cC$. $\sigma_t$ is quasi-convergent \cite{quasiconvergence}*{Defn. 2.8} because:
    \begin{enumerate}[label=(\roman*)]
        \item Every $E\in \cC$ has a canonical filtration with associated graded pieces $\gr_i(E) \in \cC_{\le v_i}$ for $1\le i \le k$. Next, $\gr_i(E)$ admits an HN-filtration with respect to $\sigma_{v_i}$. Concatenating these filtrations, we obtain the limit HN-filtration for $\sigma_t$.\endnote{This uses the fact that $\Im(c_{j,t} - c_{i,t})\to \infty$ for all $i<j$, which allows us to concatenate the filtrations in a way that respects ordering of phases.}
        \item Given a pair of limit semistable objects $E,F$ in $\cC$,  
        \[
        \lim_{t\to\infty} \frac{\ell_{\sigma_t}(E/F)}{1+\lvert \ell_{\sigma_t}(E/F)\rvert} = 
        \begin{cases}
            \frac{\ell_{v_i}(E/F)}{1+\lvert \ell_{v_i}(E/F)\rvert}& E,F\in \cC_{\le v_i}\\
            \mathfrak{p}(v_j,v_i)& E\in \cC_{\le v_i},F\in \cC_{\le v_j}\text{ and }i\ne j.
        \end{cases}
        \]
    \end{enumerate}
    Next, $\{P_1,\ldots, P_n\}$ contains an asymptotic indexing set obtained by choosing one rep\-resentative from each $P_v$, denoted $P_v$ by abuse of notation.\endnote{Indeed, $P_j\sim^i P_k$ if and only if $P_j \sim P_k$ if and only if $P_j,P_k\in \cC_{\le v}$ for some $v$.} We recover $\cC = \langle \cC_{v_1},\ldots, \cC_{v_k}\rangle$ by noting that $\cC_{\le v} = \cC^{P_v}$ for each $v \in V(\Sigma)_{\rm{term}}$. The stability conditions obtained by applying \cite{quasiconvergence}*{Thm. 2.30} are $(\sigma_{v_i})_{i=1}^k$. 
    
    \Cref{D:stronglyqconv}(1) holds by the hypothesis that $\ell(\sigma_t) = \gamma(t)$ converges in $\rmscbar_n$. Next, if $E$ is $\sigma_{v_i}$-stable, then since $\tau_{i,t}\to \sigma_{v_i}$ there exists $t_E\in \bR$ such that $t\ge t_E \Rightarrow$ $E$ is $\tau_{i,t}$-stable. Because $\sigma_t$ is strongly glued from $(\tau_{i,t})_{i=1}^k$ for all $t\gg0$, we have $\cP_{\sigma_t}(\phi) = \bigoplus_{i=1}^k \cP_{\tau_{i,t}}(\phi)$ so that $E$ is $\sigma_t$-stable. (3) follows from the fact that $\tau_{i,t}\to \sigma_{v_i}$ in $\Stab(\cC_{\le v_i})/\bC$ and \Cref{R:localmetric}. Therefore, by \Cref{P:quasiconvconv}, $\lim_{t\to\infty} \sigma_t = \sigma_\infty$. 
\end{proof}


\subsection{First examples}
\label{S:firstexamples}

We work out some first examples coming from $\DCoh(X)$ for $X$ a variety. We let $\Astab(X) = \Astab(\DCoh(X))$.

\subsubsection{Disjoint union of points}
\label{S:disjointpoints}
We write $\pt = \Spec k$, for $k$ an arbitrary field. The map $\sigma \mapsto \logZ_\sigma(E)$, where $E$ is a one-dimensional $k$-vector space, gives a biholomorphism $\Stab(\DCoh(\pt)) \cong \bC$.\endnote{A bounded $t$-structure on $\DCoh(\pt)$ is determined by which shift of $E$ lies in the heart. The central charge is uniquely determined by $Z(E)$. Both of these data are contained in $\logZ_\sigma(E)$, which is holomorphic by the definition of the complex structure on $\Stab(\cC)$.} Let $X$ be a disjoint union of copies of $\pt$.

\begin{lem}
\label{L:orthogonalproduct}
    Suppose $\cC_1,\ldots, \cC_n$ are triangulated categories, each with a surjective homo\-morphism $v_i :  \rm{K}_0(\cC_i)\twoheadrightarrow \Lambda_i$ for $\Lambda_i$ a finite rank lattice. Let $\cC = \bigoplus_{i=1}^n \cC_i$ and $\Lambda = \bigoplus_{i=1}^n \Lambda_i$.\endnote{Given triangulated categories $\cC_1$ and $\cC_2$, we let $\cC_1\oplus \cC_2$ denote the triangulated category where $\Ob(\cC_1\oplus \cC_2) = \{A_1\oplus A_2:A_i\in \Ob(\cC_i)\text{ for } i=1,2\}$, and $\Hom_{\cC_1\oplus \cC_2}(A_1\oplus A_2,B_1\oplus B_2) = \Hom_{\cC_1}(A_1,B_1)\oplus \Hom_{\cC_2}(A_2,B_2)$.} There is a biholo\-morphism
    \[
        \Stab_\Lambda(\cC) \xrightarrow{\sim} \prod_{i=1}^n \Stab_{\Lambda_i}(\cC_i)
    \]
    given by $(Z,\cP) \mapsto (Z_i,\cP_i)_{i=1}^n$, where $Z_i = Z|_{\Lambda_i}$ and $\cP_i = \cP \cap \cC_i$.
\end{lem}

\begin{proof}
    It is an exercise to verify that $(Z_i,\cP_i)$ defines a stability condition on $\cC_i$ for each $i$. The gluing map $\prod_{i=1}^n \Stab_{\Lambda_i}(\cal{C}) \to \Stab_\Lambda(\cal{C})$ defined using \cite{quasiconvergence}*{\S3.2} is a biholo\-morphic inverse to $(Z,\cP)\mapsto(Z_i,\cP_i)_{i=1}^n$.
\end{proof}

By \Cref{L:orthogonalproduct}, $\Stab(X) \cong \bC^n$ and $\Stab(X)/\bC \cong \bC^n/\bC$. All thick triangulated sub\-categories of $\DCoh(X)$ are of the form $\bigoplus_{i\in S} \DCoh(\pt_i)$ for $S\subset \{1,\ldots, n\}$ since thick triang\-ulated subcategories of $\DCoh(X)$ correspond to Serre subcategories of $\Coh(X)$.\endnote{First of all, the stated correspondence between thick triangulated subcategories and Serre subcategories of $\Coh(X)$ can be checked using the fact that all objects of $\DCoh(X)$ are isomorphic to the sums of their cohomology objects in this case. Since Serre subcategories of $\Coh(X)$ are in particular closed under summands, they are all of the form $\Coh_S = \bigoplus_{i\in S} \Coh(\pt_i)$ for $S\subset \{1,\ldots, n\}$. The corresponding thick triangulated subcategory of $\DCoh(X)$ is then $\bigoplus_{i\in \bZ}\Coh_S[i]$.} So, all thick triangulated subcategories of $\DCoh(X)$ are admissible.

\begin{lem}
\label{L:asconpoints}
    Consider a real oriented isomorphism class of multi-scale line $\Sigma$ such that $\Gamma(\Sigma)_{\rm{term}} = \{v_i\}_{i=1}^k$, a partition $\{1,\ldots, n\} = \bigsqcup_{j=1}^k S_j$, and $c_j\in \bC^{S_j}/\bC$ for each $1\le j\le k$. There is a unique element of $\Astab(X)$ with underlying multi-scale line $\Sigma$, 
    \[
    \cC_{<v_j} = \Span\left\{\bigcup_{\mathfrak{p}(v_i,v_j) = 1}\{E_s:s\in S_i\}\right\}, \quad \gr_{v_j}(\cC_\bullet) = \Span\{E_s:s\in S_j\},
    \]
    \[
    \cC_{\le v_j} = \cC_{<v_j}\oplus \gr_{v_j}(\cC_\bullet),
    \] 
    and $\sigma_j \in \Stab(\gr_{v_j}(\cC_\bullet))/\bC$ such that $\logZ(E_s:s\in S_j) = c_j$. Conversely, every element of $\Astab(X)$ is of this type.
\end{lem}

\begin{proof}
    \Cref{L:orthogonalproduct} implies that $\sigma_j$ is determined by the condition $\logZ(E_s:s\in S_j) = c_j$. We show that $\cC_\bullet = \{\cC_{\le v_j}\}$ and $\Sigma$ satisfy \Cref{D:multi-scale_decomposition}. First, $\cC_{\le v_i}\cap \cC_{\le v_j} = \Span\{\cC_{\le v_k}| v_k\le_{1,\infty} v_i,v_j\}$ by inspection. Next, if $v_i\le_{i,0}v_j$, then $\Hom(\cC_{\le v_j},\cC_{\le v_i}) = 0$ in $\cC/\cC_{\le v_i}\cap \cC_{\le v_j}$, since nonzero morphisms between objects in $\cC_{\le v_j}$ and $\cC_{\le v_i}$ correspond to objects $E_l\in \cC_{\le v_i}\cap \cC_{\le v_j}$, but such morphisms are zero in $\cC/\cC_{\le v_i}\cap \cC_{\le v_j}$. For any $v\in V(\Sigma)$, $\cC_{< v} = \bigoplus_{i\in S'} \DCoh(\pt_i)$ for $S' = \bigcup_{v_i<_{1,\infty} v} S_i$, and $\cC_{\le v} = \bigoplus_{i\in S}\DCoh(\pt_i)$ for $S = S' \cup \bigcup_{v_i\subset v} S_i$, giving \Cref{D:multi-scale_decomposition}(3) and (4).

    Conversely, the thick triangulated subcategories of $\DCoh(X)$ are of the form $\bigoplus_{i\in S} \DCoh(\pt_i)$ for $S\subset \{1,\ldots, n\}$. For each $i$, let $S_{\le i} = \{k:E_k\in \cC_{\le v_i}\}$ and define $S_{<i}$ analogously. Let $S_i = S_{\le i}\setminus S_{<i}$ so that $\gr_{v_i}(\cC_\bullet) = \Span\{E_s:s\in S_i\}$ for each $i$. It follows from \Cref{D:multi-scale_decomposition} that $\{1,\ldots,n\} = \bigsqcup_{i=1}^k S_i$. The claimed description follows.
\end{proof}

The map $\ell:=\ell_{(-)}(E_1,\ldots, E_n):\Astab(X) \to \cA_n^{\bR}$ is globally defined by \Cref{L:asconpoints}. It follows from \Cref{P:npoints} that $\DCoh(X)$ is an  example where \Cref{conj:manifoldwithcorners} holds globally in arbitrary dimension.

\begin{prop}
\label{P:npoints}
    $\ell:\Astab(X)\to \rmscbar_n$ is a homeo\-morphism which restricts to a biholo\-morphism on $\Stab(X)/\bC$.
\end{prop}

\begin{proof}
    $\ell$ is continuous by \Cref{D:weak_topology} and bijective by \Cref{L:asconpoints}. We verify that $\ell^{-1}$ is continuous. Consider a net $(\Sigma_\alpha)_{\alpha \in I}\to \Sigma$ in $\cA_n^{\bR}$; under $\ell^{-1}$ this gives rise to a net $(\sigma_\alpha)_{\alpha \in I}$, which we claim converges to $\sigma = \ell^{-1}(\Sigma)$. The conditions of \Cref{D:convergence_conditions} follow from the observation that $\cC_{\le v_i}^{\rm{ss}}$ consists of objects of the form $E = \bigoplus_{j} E_j^{\oplus s_j}[n_j]$ such that $j\in S_i$, $n_j\in \bZ$, and $s_j \ge 0$ such that $j\in S_i$ and such that $\phi(E/P) = \phi(E_j[n_j]/P)$ for all $j$; in particular, $\phi_\alpha(E/P)\to \phi(E/P)$.\endnote{We check the conditions of \Cref{D:convergence_conditions} more carefully: (1) is by our hypothesis that $(\Sigma_\alpha)\to \Sigma$ in $\cA_n^{\bR}$. For (2)\eqref{I:cosh_bound}, choose $P = E_j$ for $j\in S_i$. Then, we know that $\phi_\alpha(E/P)\to \phi_\alpha(E/P)$ which implies that $c_\alpha^t(E) \to 1$ for all $t$, since $\phi_\alpha^+(E) - \phi_\alpha^-(E) \to 0$. This also implies convergence of the imaginary part of \eqref{I:log_central_charge}. For the real part, by hypothesis, for any $j\in S_i$ one has $\lim_\alpha \log m_\alpha(E_j/P) = \log m_\sigma(E_j/P)$. Consequently, $\lim_\alpha m_\alpha(E_j/P)$ exists for each $j$ and by the classification of stability conditions on $\DCoh(\gr_{v_i}(\cC_\bullet))$ from \Cref{L:orthogonalproduct}, we see that $m_\alpha(E) = \sum_j s_j m_\alpha(E_j)$, so that $\lim_\alpha m_\alpha(E/P)$ converges.} 

    It remains to verify \Cref{T:topology}\eqref{I:phasewidthunif} and \eqref{I:stability}. Up to passing to a cofinal subnet, we may assume that the underlying curve of $\Sigma_\alpha$, the signature of $\Sigma_\alpha$, and the categories $\cC_{\le v}^\alpha$ of $\sigma_\alpha$ indexed by $V(\Sigma_\alpha)_{\rm{term}}$ are constant. By \Cref{L:asconpoints}, $\sigma_\alpha$ corresponds to a partition $\{1,\ldots, n\} = \bigsqcup S_i'$ coarsening that of $\sigma$. For any $E\in \gr_{v_i}(\cC_\bullet),$ choosing $E$ as its own lift $E'$ in $\cC_{\le v_i}$ as in \Cref{D:directeddistance} gives $\lim_\alpha \vec{d}_{E,E'}(\sigma_\alpha,\sigma) = 0$. The uniformity condition \eqref{I:phasewidthunif} follows from the fact for each $i$, all $\sigma$-stable in $\cC_{\le v_i}$ are shifts of the finite set $\{E_j:j\in S_i\}$. Finally, \eqref{I:stability} follows from the fact that the image of any stable object under $\cC_{\le v_i}\to \gr_{v_i}(\cC_\bullet)$ is again $\sigma_\alpha$-stable.\endnote{We verify \eqref{I:phasewidthunif} in more detail: the question is about uniformity of convergence of $\phi_\alpha^{\pm}(E/P)-\phi^{\pm}(E/P)$ for $E\in \cC_{\le v_i}$ with some $P\in \{E_j:j\in S_i\}$ fixed. Write $E = \bigoplus_j F_j$ where $F_j \in \langle E_j\rangle$ for each $j\in S_i$. One has $m_\alpha(E/P) = \sum_j s_j\cdot m_\alpha(E_j/P)$ where $s_j$ is the number of simple objects in $\langle E_j\rangle$ used to build $F_j$. This implies that $\log m_\alpha(E/P)\to \log m(E/P)$ uniformly in $E$, since 
    \[
    \min_j\left\{\frac{m_\alpha(E_j)}{m(E_j)}\right\} \le \frac{m_\alpha(E)}{m(E)} = \frac{\sum_j s_j m_\alpha(E_j)}{\sum_j s_j m(E_j)} \le \max_j \left\{\frac{m_\alpha(E_j)}{m(E_j)}\right\}.
    \]
    Let us consider the average phase, using the notation from above: without loss of generality we may set $\phi_\alpha(P) = 0$ for all $\alpha$ and $\phi(P) = 0$. Then, the relevant quantity is 
    \[
        \frac{1}{m_\alpha(E)}\sum_{j,k}m_{j,k}\cdot  \theta_{j,k}(\alpha)\cdot m_\alpha(E_j) - \frac{1}{m(E)} \sum_{j,k} m_{j,k}\cdot  \theta_{j,k} \cdot m(E_j)
    \]
    where $F_j$ has factors $E_j[n_k]^{\oplus m_{j,k}}$ for $n_k\in \bZ$ and $m_{j,k} \ge 0$ and $\theta_{j,k}(\alpha) = \phi_\alpha(E_j[n_k]) = \phi_\alpha(E_j)+n_k$. We use similar notation for $E$ without the subscript $\alpha$. The above quantity can be rewritten as
    \[
        \frac{m(E) - m_\alpha(E)}{m(E)m_\alpha(E)} \left(\sum_{j,k}m_{j,k}\cdot\big( \theta_{j,k}(\alpha)\cdot m_\alpha(E_j) - \theta_{j,k}\cdot m(E_j)\big)\right) 
    \]
    which has absolute value bounded above by 
    \begin{equation}
    \label{E:unifconvbound}
    \left(\frac{1}{m_\alpha(E)}+ \frac{1}{m(E)}\right)\left(\sum_{j,k} m_{j,k} \cdot\big(\theta_{j,k}(\alpha)\cdot m_\alpha(E_j) - \theta_{j,k}\cdot m(E_j)\big)\right). 
    \end{equation}
    Now, for any $\epsilon>0$, we have $m_\alpha(E) \ge \min\{m(E_j):j\in S_i\}-\epsilon$ for all $\alpha$ sufficiently large. Uniform convergence in $E$ of \eqref{E:unifconvbound} to $0$ now follows from 
    \begin{align*}
    \frac{m_{j,k}\cdot (\theta_{j,k}(\alpha)\cdot m_\alpha(E_j) - \theta_{j,k}\cdot m(E_j))}{m(E)} & \le \frac{m_{j,k}\cdot (\theta_{j,k}(\alpha)\cdot m_\alpha(E_j) - \theta_{j,k}\cdot m(E_j))}{m_{j,k}\cdot m(E_j)}\\
     & = \frac{\theta_{j,k}(\alpha)\cdot m_\alpha(E_j) - \theta_{j,k}\cdot m(E_j)}{m(E_j)}
    \end{align*}
    and the fact that $\theta_{j,k}(\alpha) - \theta_{j,k}\to 0$ uniformly in $j,k$.}
\end{proof}

\subsubsection{The projective line}
The stability manifold of $\DCoh(\bP^1)$ was first studied in \cite{Ok06}, where it was shown to be biholomorphic to $\bC^2$. Since then, more explicit identifications of $\Stab(\bP^1)$ with $\bC^2$ have been found \cites{HKK,NMMP}. In this section, we will completely describe $\Astab(\bP^1)$. Let $\Gamma$ denote the unique level tree with two terminal vertices $v_1,v_2$. Up to reindexing, we assume $\mathfrak{p}(v_1,v_2)\in \bH \cup \bR_{>0}$. 

\begin{lem}
\label{L:admissibleclassificationP^1}
    Admissible points $\sigma$ of $\Astab(\bP^1)$ are elements of $\Stab(\bP^1)/\bC$ or have level tree $\Gamma$. In the latter case, 
    \begin{enumerate} 
        \item if $\sigma$ is generic, i.e. $\mathfrak{p}(v_1,v_2)\ne 1$, then $\cC_{\le v_1} = \langle\cO(k)\rangle$ and $\cC_{\le v_2} = \langle\cO(k+1)\rangle$ for some $k\in \bZ$; and \vspace{-2mm}
        \item if $\mathfrak{p}(v_1,v_2) = 1$, then $\cC_{\le v_1} = \langle\cO(k)\rangle$ for some $k\in \bZ$ and $\cC_{\le v_2} = \DCoh(\bP^1)$.
    \end{enumerate}
\end{lem}

\begin{proof}
    Suppose $\sigma = \langle \cC_\bullet|\sigma_\bullet\rangle_{\Sigma} \in \Astab(\bP^1)$. By \Cref{D:generalized_stability_condition}, $\rank  \rm{K}_0(\bP^1) = 2$ bounds the number of terminal vertices of $\Gamma(\Sigma)$. When $\sigma$ is generic, it corresponds to a semiorthogonal decomposition $\langle \cC_{\le v_1},\cC_{\le v_2}\rangle$, all of which are classified by \cite{GKR03} and are as in (1). Otherwise, if $v_1\le_{1,\infty} v_2$ then $\cC_{\le v_1}\subset \DCoh(\bP^1)$ is an admissible subcategory and hence equals $\langle \cO(k)\rangle$ for some $k\in \bZ$. 
\end{proof}

Let $X_k$ denote the locus in $\Stab(\bP^1)/\bC$ where $\cO(k)$ and $\cO(k-1)$ are simultaneously stable. Okada \cite{Ok06} proves that $\varphi_k: X_k \to \bC$ given by $\sigma \mapsto \logZ_\sigma(\cO(k)) - \logZ_\sigma(\cO(k-1))$ is a biholomorphism onto $\bH$. Moreover, $X_k$ contains a fundamental domain $U_k$ for the $\bZ$-action on $\Stab(\bP^1)/\bC$ generated by $-\otimes \cO(1)$. $U_k$ corresponds under $\varphi_k$ to the region in $\bH$ bounded below by $\{(x,y): e^{\lvert x\rvert}\cos y = 1\}$. In \cite{NMMP}, an explicit biholomorphism $\mathscr{B}:\bC \xrightarrow{\sim} \Stab(\bP^1)/\bC$ is given which identifies each horizontal strip $\bR + i\pi[k-1,k)$ with $U_k$. The restriction of $\mathscr{B}$ to $\bR+ i\pi[0,1)\to U_1$ is denoted $\mathscr{B}_1$. By \cite{NMMP}*{Prop. 25}, 
\begin{equation}
\label{E:bessel}
    (\varphi_1\circ \mathscr{B}_1)(\tau) = \log\left(\frac{ \rm{K}_0(e^\tau)+ i\pi I_0(e^\tau)}{ \rm{K}_0(e^\tau)}\right)
\end{equation}
where $I_0(u)$ and $ \rm{K}_0(u)$ are modified Bessel functions of the first and second kind, respectively. For any $k\in \bZ$, let $\Omega_k$ denote the union of $U_k$ with $\partial_k:= \{\sigma \in \Astab(\bP^1): \cC_{\le v_1} = \cO(k) \text{ and } \mathfrak{p}(v_1,v_2) = e^{i\pi\theta}, \theta \in [0,1)\}$. By \Cref{L:admissibleclassificationP^1}, $\cO(n)\cdot \Omega_0 = \Omega_n$ for all $n \in \bZ$ and therefore $\Astab^{\rm{adm}}(\bP^1) = \bigsqcup_{n\in \bZ} \Omega_n$. 

Next, let $\bC_+ := \{x+iy:x\in (-\infty,\infty],y\in \bR\}$, equipped with the product topology. $\bZ$ acts on $\bC_+$ by translation: $n\cdot (x+iy) = x+i(y+n\pi)$. A strict fundamental domain for this action is given by the strip $S_1 = (-\infty,\infty] + i\pi[0,1)$. Write $S_1^\circ = (-\infty,\infty]+i\pi(0,1)$. Define $F_1:S_1 \to \Omega_1$ such that $F_1|_{\bR + i\pi[0,1)} = \mathscr{B}_1$ and for $\theta \in [0,1)$, $F_1(\infty+i\pi\theta)$ is the element of $\partial_1$ with $\mathfrak{p}(v_1,v_2) = e^{i\pi \theta}$.

\begin{thm}
\label{T:P1}
    There is a unique map $F:\bC_+ \to \Astab^{\rm{adm}}(\bP^1)$ which is $\bZ$-equivariant and such that $F|_{S_1} = F_1$. $F$ is a homeomorphism and restricts to the biholomorphism $\mathscr{B}:\bC\to \Stab(\bP^1)/\bC$.
\end{thm}

\begin{proof}
    $F_1$ maps a strict fundamental domain for the $\bZ$-action on $\bC_+$ to a strict fundamental domain for the $\bZ$-action on $\Astab^{\rm{adm}}(\bP^1)$ bijectively. Therefore, $F$ is uniquely defined on $S_{n+1} = n\cdot S_1$ by $F(n\cdot z) = n\cdot F_1(z)$. and is bijective. We next show that $F$ is a homeomorphism. This is clear on $\bC \subset \bC_{+}$, since $F|_{\bC} = \mathscr{B}$. Consider $z = \infty + i\pi y$ for $y\in (0,1)$. By definition, $\varphi_1\circ\mathscr{B}_1$ is $\tau \mapsto \logZ_{\mathscr{B}(\tau)}(\cO(1)) - \logZ_{\mathscr{B}(\tau)}(\cO)$, given in coordinates by \eqref{E:bessel}. We claim that $F$ extends $\mathscr{B}$ continuously to the boundary of $(-\infty,\infty] +i\pi[0,1]$. Consider a sequence $\tau_n = x_n +iy_n\pi$ in $(-\infty,\infty] + i\pi[0,1]$ with $x_n \to \infty$. Write $u = e^\tau$ and let $u_n = e^{\tau_n}$. As $\lvert u \rvert \to \infty$, we have asymptotic estimates 
    \[ 
         \rm{K}_0(u) = \sqrt{\frac{\pi}{2u}}e^{-u}\epsilon_1(u),\quad 
        I_0(u) = \frac{e^u}{\sqrt{2\pi u}}\epsilon_2(u) + i\frac{e^{-u}}{\sqrt{2\pi u}} \epsilon_1(u)
    \]
    where $\epsilon_1(u),\epsilon_2(u) \to 1$ as $\lvert u\rvert\to \infty$ by \cite{NMMP}*{p. 26}. So, as $\lvert u \rvert\to \infty$ we have 
    \begin{align*}
        \log\left(\frac{ \rm{K}_0(u) + i\pi I_0(u)}{ \rm{K}_0(u)}\right) & = i\frac{\pi}{2} + 2u + \log\left(\frac{\epsilon_2(u)}{\epsilon_1(u)}\right). 
    \end{align*}
    In particular,\endnote{As $n\to \infty$,
    \begin{align*}
            \log\left(\frac{ \rm{K}_0(u_n) + i\pi I_0(u_n)}{ \rm{K}_0(u_n)}\right)  & \approx \log \left(\frac{\sqrt{\frac{\pi}{2u_n}} e^{-u_n} + i\pi(\frac{e^{u_n}}{\sqrt{2\pi u_n}} + i\frac{e^{-u_n}}{\sqrt{2\pi u_n}})}{\sqrt{\frac{\pi}{2u_n}}e^{-u_n}}\right)\\
            & = \log \left(1+ i\pi \left(\frac{e^{2u_n} + i}{\pi}\right)\right)\\
            & = \log(i) +2 u_n = i\frac{\pi}{2}+2u_n.
    \end{align*}
    Finally, we have 
    \[
        \lim_{n\to \infty} \frac{u_n}{\lvert u_n\rvert} = \lim_{n\to\infty} \frac{e^{x_n}e^{iy_n\pi}}{e^{x_n}} = \lim_{n\to\infty} e^{iy_n\pi} = e^{iy\pi}.
    \]}
    \begin{equation}
    \label{E:limitangle}
        \lim_{n\to\infty} \frac{(\varphi_1\circ \mathscr{B}_1)(u_n)}{1+\lvert (\varphi_1\circ \mathscr{B}_1)(u_n)\rvert } = \lim_{n\to\infty} \frac{u_n}{\lvert u_n\rvert} = \lim_{n\to\infty} e^{i\pi y_n}.
    \end{equation}
    By \Cref{T:genericmanifoldwithcorners}, to prove that $f$ is a local homeomorphism at a boundary point $\infty + i\pi y$ for $y\in (0,1)$, we may show that a net $\{\tau_\alpha\}$ converges to $\infty + i\pi y$ if and only if $\logZ_{f(\tau_\alpha)}(\cO(1)) - \logZ_{f(\tau_\alpha)}(\cO)$ converges to the point in $\rmscbar_2$ with $\mathfrak{p}(v_1,v_2) = e^{i\pi y}$. Given $\tau_\alpha \to\infty+i\pi y$, \eqref{E:limitangle} implies that $\lim_\alpha \logZ_{f(\tau_\alpha)}(\cO(1)) - \logZ_{f(\tau_\alpha)}(\cO) = \Sigma$ with $\mathfrak{p}_\Sigma(v_1,v_2) = e^{iy\pi}$. Conversely, suppose given a net $\{\sigma_\alpha\}$ in $\Astab(\bP^1)$ converging to $\sigma \in \partial_1$ with $\mathfrak{p}(v_1,v_2) = e^{i\pi y}$ and $y\in (0,1)$. Denote by $\tau_\alpha = x_\alpha + i\pi y_\alpha$ in $(-\infty,\infty]+ i\pi(0,1)$ the net $f^{-1}(\sigma_\alpha)$. The asymptotic estimates of $I_0(u)$ and $ \rm{K}_0(u)$ as $u\to 0$ \cite{NMMP}*{p. 25} imply that $x_\alpha \to \infty$.\endnote{$u$ is a parameter on $\bH$ given by $u = e^\tau$. The cited asymptotic estimates read $I_0(u) = 1 + O(\lvert u\rvert^2)$ and $ \rm{K}_0(u) = - \log(\frac{u}{2}) -C_{\rm{eu}} + O(\lvert u\rvert^2\cdot \lvert \ln(u)\rvert)$ as $u\to 0$. In particular, as $\tau$ goes to the left in the strip (i.e. $x\to -\infty$) we have 
    \[
        \log\left(\frac{ \rm{K}_0(u) + i\pi I_0(u)}{ \rm{K}_0(u)}\right) \to 0
    \]
    and the convergence depends only on $x$. Consequently, for 
    \[ 
        f(\tau_\alpha) = \log\left(\frac{ \rm{K}_0(u_\alpha) + i\pi I_0(u_\alpha)}{ \rm{K}_0(u_\alpha)}\right)
    \]
    to diverge to infinity, we must have $x_\alpha \to \infty$.}. On the other hand, \eqref{E:limitangle} implies that $y_\alpha \to y$. 

    Next, we check continuity of $f$ at $\{\infty + i\pi n : n \in \bZ\}$. Consider a net $\{z_\alpha\}\to \infty = \infty+ i\pi \cdot 0$. Up to passing to a cofinal subnet, we may separately consider the cases where $\Im(z_\alpha) \in [0,1)\pi$ and $\Im(z_\alpha) \in (-1,0]\pi$ for all $\alpha$. Let $f(z_\alpha) = \sigma_\alpha$ and consider the former case. We verify that $\{\sigma_\alpha\}$ satisfies the conditions of \Cref{T:topology}. Set $P_1= \cO$ and $P_2 = \cO(1)$. \Cref{D:convergence_conditions}\eqref{I:marked_line_convergence_1} is by \eqref{E:limitangle}. $\cC_{\le v_1}^{\rm{ss}} = \{\cO^{\oplus p}[t]:p\ge 1,t\in \bZ\}$, so \Cref{D:convergence_conditions}\eqref{I:cosh_bound} and \eqref{I:log_central_charge} are immediate for $v_1$.\endnote{Indeed, $\ell_\alpha(\cO^{\oplus p}[t])  - \ell_\alpha(\cO) = \log \frac{p\cdot m_\alpha(\cO)}{m_\alpha(\cO)} + i\pi t$, and is in particular constant. Similarly, since $\cO$ is $\sigma_\alpha$-stable for all $\sigma_\alpha$ in this subnet, it follows that $c_t^\alpha(E) = 1$ for all $t > 0$ and $\alpha$.} 
    
    Next, $\cC_{\le v_2}^{\rm{ss}}$ consists of objects $E\in \DCoh(\bP^1)$ which fit into a triangle $\cO(1)^{\oplus p}[n] \to E\to  \bigoplus_i \cO^{\oplus q_i}[s_i]$. By \eqref{E:limitangle}, $\lim_\alpha m_\alpha(\cO(1))/m_\alpha(\cO) = \infty$ and it follows that $\phi_\alpha(E/\cO(1)) \to n$.\endnote{We have 
    \[
    \phi_\alpha(E/\cO(1)) = \frac{p\cdot (\phi_\alpha(\cO(1)) + n - \phi_\alpha(\cO(1)))m_\alpha(\cO(1)) +\sum_i q_i\cdot (\phi_\alpha (\cO) -\phi_\alpha(\cO(1)) +s_i)m_\alpha(\cO)}{p\cdot m_\alpha(\cO(1)) + (\sum_i q_i)\cdot m_\alpha(\cO)}.
    \]
    and so using the fact that $\lim_\alpha m_\alpha(\cO(1))/m_\alpha(\cO) = \infty$ and that $\lvert \phi_\alpha - \phi_\alpha(\cO(1))\rvert$ is bounded in $\alpha$, we get that 
    \[
        \lim_\alpha \phi_\alpha(E/\cO(1)) = \lim_\alpha \frac{p\cdot n\cdot m_\alpha(\cO(1))}{p\cdot m_\alpha(\cO(1))}
    \]
    which gives the claimed result.}
    A calculation using the fact that $\ell_\alpha(E/P_2) = \logZ_\alpha(E) - \logZ_\alpha(\cO(1))$ implies that 
    \[  
        \lim_\alpha \log(m_\alpha(E)/m_\alpha(\cO(1))) = p
    \]
    and so \Cref{D:convergence_conditions}\eqref{I:log_central_charge} holds. Finally, a direct calculation shows that \Cref{D:convergence_conditions}\eqref{I:cosh_bound} holds for any $E\in \cC_{\le v_2}^{\rm{ss}}$.\endnote{An object $E\in \cC_{\le v_2}^{\rm{ss}}$ decomposes as $\cO(1)^{\oplus p}[r] \to E\to \bigoplus_i \cO^{\oplus q_i}[s_i]$. \Cref{D:convergence_conditions}\eqref{I:log_central_charge} follows from this description, the fact that $\lim_\alpha m_\alpha(\cO(1))/m_\alpha(\cO)=0$, and the fact that 
    \[ 
    \lim_\alpha \frac{\logZ_\alpha(\cO(1)) - \logZ_\alpha(\cO)}{1+\lvert \logZ_\alpha(\cO(1)) - \logZ_\alpha(\cO)\rvert } = 1.
    \]
    It follows, in particular, that for $E\in \cC_{\le v_2}^{\rm{ss}}$ as above, we have $\lim_\alpha \phi_\alpha(E/\cO(1)) = r$. For \Cref{D:convergence_conditions}\eqref{I:cosh_bound}, we need to show that 
    \[
        c_\alpha^t(E) = \frac{1}{m_\alpha(E)}\int_{\bR} \cosh(t(\theta-\phi(E))) \:dm_E(\theta)\to 1.
    \]
    There are two types of terms to consider: one is 
    \[
        \frac{\cosh (t(\phi_\alpha(\cO(1)) + r - \phi_\alpha(E)))\cdot m_\alpha(\cO(1))}{m_\alpha(\cO(1))} = \cosh(t(\phi_\alpha(\cO(1))  + r- \phi_\alpha(E))).
    \]
    By the previous discussion, the argument of $\cosh$ tends to zero, so this expression tends to 1. The other type of summand is 
    \[
    \frac{\cosh(t(\phi_\alpha(\cO) + s_i-\phi_\alpha(E)))\cdot m_\alpha(\cO)q_i}{m_\alpha(\cO(1))}.
    \]
    Here, we know that $m_\alpha(\cO)/m_\alpha(\cO(1))\to 0$, but $\phi_\alpha(\cO)+s_i - \phi_\alpha(E) \approx \phi_\alpha(\cO) + s_i - \phi_\alpha(\cO(1))$ could grow arbitrarily large. However, using the asymptotic estimates for $\lvert u \rvert \to \infty$, as $z_\alpha \to \infty+i\pi 0$, it follows that $\Im(z_\alpha) \to 0$ and consequently that 
    \[
    \log\left(\frac{\rm{K}_0(u_\alpha) + i\pi I_0(u_\alpha)}{\rm{K}_0(u_\alpha)}\right) = i\frac{\pi}{2} + 2u_\alpha + \log\left(\frac{\epsilon_2(u_\alpha)}{\epsilon_1(u_\alpha)}\right)
    \]
    where $\Im(u_\alpha) \to 0$. Next, $\cosh t(\phi_\alpha(\cO) + s_i - \phi_\alpha(E)) \le \cosh t(\phi_\alpha(\cO) + s_i - \phi_\alpha(\cO(1)))$ and so since $e^{\pm t (\phi_\alpha(\cO) + s_i - \phi_\alpha(\cO(1))} = e^{\pm t\cdot s_i(\pi/2)}\cdot e^{f_\alpha}$ where $f_\alpha \to 0$, it follows that we have 
    \[
    \lim_\alpha 
    \frac{\cosh(t(\phi_\alpha(\cO) + s_i-\phi_\alpha(E)))\cdot m_\alpha(\cO)q_i}{m_\alpha(\cO(1))}
    = 0
    \]
    which gives \Cref{D:convergence_conditions}\eqref{I:cosh_bound}.} We leave as an exercise the verification that \Cref{T:topology}\eqref{I:phasewidthunif} holds.\endnote{We have to show that $\vec{d}(\sigma_\alpha,\sigma)\to 0$. Consider $E\in \DCoh(\bP^1)\setminus \langle \cO\rangle$. $E$ has projection to $\langle \cO(1)\rangle$ given by $\Pi(E) \cong \bigoplus_j \cO(1)^{\oplus m_j}[n_j]$. Now, $\log m_\alpha(\Pi(E)/\cO(1)) =  \log m_\alpha(\Pi(E)) - \log m_\alpha(\cO(1)) = \sum_j m_j$. In particular, $\log m_\alpha(\Pi(E)/\cO(1)) = \log m_\sigma(\Pi(E)/\cO(1))$ for all $\alpha$. Let $n_{\rm{min}} = \min \{n_j\}$ and define $n_{\rm{max}}$ analogously. $\phi_\sigma^-(\Pi(E)/\cO(1)) = \phi_\sigma(\cO(1)[n_{\rm{min}}]/\cO(1)) = n_{\rm{min}}$ and likewise for $\phi_\sigma^+$. It follows that $\vec{d}(\sigma_\alpha,\sigma)\to 0$.}

    Finally, we check \Cref{T:topology}\eqref{I:stability}. The $\sigma$-stable objects are $\{\cO[n]:n\in \bZ\}$ and the set of objects in $\DCoh(\bP^1)$ that project to $\cO(1)$. $\cO[n]$ is $\sigma_\alpha$-stable for each $\alpha$, so there is nothing to check in this case. On the other hand, if $E \equiv \cO(1) \mod \langle \cO\rangle$ then we are done because $\cO(1)$ is $\sigma_\alpha$-stable for all $\alpha$. Consequently, $\sigma_\alpha \to \sigma$ with respect to the topology of \Cref{T:topology}. The case where $\Im(z_\alpha) \in (-1,0]\pi$ for all $\alpha$ follows from analogous reasoning. 
    
    It remains to show that if $\sigma_\alpha \to\sigma \in \partial_1$ with $\mathfrak{p}(v_1,v_2) = 1$ then $f^{-1}(\sigma_\alpha)\to f^{-1}(\sigma)$. By \Cref{T:topology}\eqref{I:stability}, there exists an $\alpha_0$ such that $\alpha \ge \alpha_0$ implies that $\cO$ is $\sigma_\alpha$-stable. In particular, $\sigma_\alpha \in X_0 \cup X_1$. We can consider separately the cases where the net lies entirely in $X_0$ or $X_1$. In the latter case, the fact that 
    \[
        \lim_\alpha \frac{\logZ_{\alpha}(\cO(1)) - \logZ_\alpha (\cO)}{1+\lvert\logZ_{\alpha}(\cO(1)) - \logZ_\alpha (\cO) \rvert} =  1
    \] 
    implies by \eqref{E:limitangle} that $f^{-1}(\sigma_\alpha)\to \infty + i\pi 0$. The other case is proven analogously.
\end{proof}

Next, we classify the non-admissible points of $\Astab(\bP^1)$. 

\begin{lem}
\label{L:nonadmissibleP1}
    The only non-admissible point $\sigma_0$ in $\Astab(\bP^1)$ has underlying curve of type $\Gamma$ with $\mathfrak{p}(v_1,v_2) = 1$ and $\cC_{\le v_1}$ the category of torsion complexes on $\bP^1$. Furthermore, $\sigma_0$ lies in the closure of each of the regions $U_k$ for $k\in \bZ$.
\end{lem}

\begin{proof}
    Suppose $\cT$ is a nonzero thick subcategory of $\DCoh(\bP^1)$ giving rise to an element of $\Astab(\bP^1)$. By \Cref{D:generalized_stability_condition}, $\rm{K}_0(\cT) \subset \rm{K}_0(\bP^1)$ is rank $1$. If $\cT$ contains no complexes with zero dimensional support then it must contain one of the sheaves $\cO(k)$. It cannot contain any other $\cO(\ell)$ or any complex with finite support since in this case $v(\cT)$ is of rank $2$. In particular, $\cT = \langle \cO(k)\rangle$ and we are in the case of \Cref{L:admissibleclassificationP^1}. 
    
    Next, suppose that $\cT$ contains a torsion complex supported at a point $p$. It then contains $\cO_p$ and thus $v(\cT)$ is spanned by $v(\cO_p)$.\endnote{A torsion sheaf supported at a point corresponds to $M = k[x]/(x^n)$ in a suitable choice of affine chart $\bA^1 = \Spec k[x]$. In particular, we can consider the map $\cdot \: x: M\to M$ which has kernel $k\cdot x^{n-1}$ and cokernel $k[x]/(x^{n-1})$. The cohomology $\Cone(M\to M)$ contains $\ker$ and $\coker$ as summands. In particular, $k\cdot x^{n-1}\cong k$ corresponds to $\cO_p$ in this chart.} So, any stability condition on $\DCoh(\bP^1)/\cT$ factoring through $\rm{K}_0(\bP^1)/v(\cT)$ must not have any torsion sheaves as stable objects. In particular, stable objects are the images under $\DCoh(\bP^1)\to \DCoh(\bP^1)/\cT$ of shifts of locally free sheaves, all of which are equivalent in the quotient to $\cO^{\oplus n}$ for some $n\ge 1$. If there is some $\cO_q\not\in \cT$, then the triangle $\cO(-1)\to \cO \to \cO_q$ gives a triangle $\cO\to \cO \to \cO_q$ in the quotient and in particular a nontrivial endomorphism of $\cO$. Consequently, $\cO$ cannot be stable with respect to any stability condition on $\DCoh(\bP^1)/\cT$. So, $\cT$ contains all zero dimensional sheaves. 
    
    Since $\cT$ is not admissible, the corresponding augmented stability condition has curve $\Sigma$ with dual tree $\Gamma$ with $\mathfrak{p}(v_1,v_2) = 1$. To see that $0\subsetneq \cT\subsetneq \DCoh(\bP^1)$ has stability conditions on the associated graded categories, it suffices to note that $\cT$ has a stability condition given by restricting slope stability and that $\DCoh(\bP^1)/\cT$ is the derived category of finite dimensional vector spaces over the function field of $\bP^1$ by \cite{Meinhardtquotient}*{Prop. 3.13}.

    Finally, consider $U_k$ for any $k\in \bZ$. Under the map $\varphi_k:U_k \to \bH$, this corresponds to the region bounded below by the curve $\{(x,y):e^{\lvert x\rvert}\cos y = 1\}$. Consider the path given by $\varphi_k(t) = i/t$ for $t\in [1,\infty)$. By similar reasoning to \cite{quasiconvergence}*{\S 4.2}, $\varphi_k(t)$ defines a quasi-convergent path which converges in $\Astab(\bP^1)$ to $\sigma_0$. 
\end{proof}

\Cref{L:admissibleclassificationP^1}, \Cref{T:P1}, and \Cref{L:nonadmissibleP1} give a global picture of $\Astab(\bP^1)$ com\-patible with \cite{NMMP}*{Fig. 1, p. 29}. The admissible points with graph $\Gamma$ give rise to a boundary diffeomorphic to $\bR$ at the right-hand side of the figure in \emph{loc. cit.} On the other hand, any path tending toward the left limits to the non-admissible point $\sigma_0$. 

By \Cref{T:topology}, there is a continuous $\Aut(\bP^1)$-action on $\Astab(\bP^1)$ which factors through the $\bZ$-action generated by $-\otimes \cO(1)$. On $\bC \cong \Stab(\bP^1)/\bC\subset \Astab(\bP^1)$, this action corresponds to the translation action $(n,z)\mapsto z + in$, thus identifying the double quotient $\Aut(\bP^1)\backslash\Stab(\bP^1)/\bC$ with an open cylinder $S^1\times \bR$. The homeomorphism $F:\bC_+ \to \Astab^{\rm{adm}}(\bP^1)$ of \Cref{T:P1} realizes the quotient of $\Astab^{\rm{adm}}(\bP^1)$ by $\Aut(\bP^1)$ as $S^1\times (-\infty,\infty]$. 

Finally, observe that $\sigma_0$ is the sole fixed point of the $\bZ$-action because $-\otimes \cO(1)$ acts as the identity on the category of torsion complexes $\cT$. Combining these observations realizes $\Astab(\bP^1)/\Aut(\bP^1)$ as a cone compactifying the open cylinder $S^1\times \bR$ with boundaries $S^1\times \{\infty\}$ corresponding to the admissible points and $*\times \{-\infty\}$ corresponding to $\sigma_0$ (see Figure \ref{F:P1compactification}). Quotienting by the admissible boundary $S^1\times \{\infty\}$ produces a sphere compactifying $\Aut(\bP^1)\backslash\Stab(\bP^1)/\bC$ with a pair of special points: one ``singular'' point $\sigma_0$ with stabilizer group $\bZ$ and one point $\sigma_\infty$ which corresponds to the admissible boundary.

\begin{figure}[h]
\begin{center}
    \begin{tikzpicture}
        \draw[thick] (0,0) -- (3,1);
        \draw[thick] (0,0) -- (3,-1);
        \filldraw[blue] (0,0) circle (0.05);
        \draw[red,thick] (3,1) arc (90:-90:0.5cm and 1cm);
        \draw[dashed, red] (3,1) arc (90:270:0.5cm and 1cm);
        \fill[gray!50, opacity = 0.20] (0,0) -- (3,1) -- (3,-1);
        \fill[gray!50, opacity = 0.20] (3,1) arc (90:-90:0.5cm and 1cm);
        \draw (1.5,.5) arc (90:-90:.25 cm and .5 cm);
        \draw[dashed] (1.5,.5) arc (90:270:.25cm and .5 cm);

        \draw (7,1) arc (90:-90:0.5cm and 1cm);
		\draw[dashed] (7,1) arc (90:270:0.5cm and 1cm);
		\draw (7,0) circle (1cm);
		\filldraw[blue] (6,0) circle  (0.05); 
		\filldraw[red] (8,0) circle (0.05);
		\shade[ball color=blue!10!white,opacity=0.20] (7,0) circle (1cm);

        \draw node at (-.25,-.25) {$\textcolor{blue}{\sigma_0}$};
        \draw node at (3.7,0) {$\textcolor{red}{\partial}$};
        \draw node at (4.75,0) {$\longrightarrow$};
        \draw node at (5.75,-.25) {$\textcolor{blue}{\sigma_0}$};
        \draw node at (8.3,0) {$\textcolor{red}{\partial}$};
    \end{tikzpicture}
    
\end{center}
\caption{The cone with boundary on the left represents $\Astab(\bP^1)/\Aut(\bP^1)$. The boundary circle on the right side of the surface labeled ``$\partial$'' depicts the admissible boundary. The horizontal arrow is the quotient map which contracts the admissible boundary to a point. The gray interior of both pictures corresponds to $\Aut(\bP^1)\backslash \Stab(\bP^1)/\bC$, which is biholomorphic to an open cylinder.}
\label{F:P1compactification}
\end{figure}

\subsubsection{Curves of higher genus}
In this section, $X$ denotes a smooth and proper curve of genus at least one. By \cite{OkawaSODs}, there are no proper admissible subcategories of $\DCoh(X)$, so $\Astab^{\rm{adm}}(X) = \Stab(X)/\bC$. We will construct non-admissible augmented stability conditions on $\DCoh(X)$. 

Consider the filtration $0\subsetneq \cT\subsetneq \DCoh(X)$, where $\cT$ is the thick triangulated subcategory of torsion complexes. As we have seen in \cite{quasiconvergence}*{Lem. 4.2} this filtration is associated to certain quasi-convergent paths in $\Stab(X)$. $\cT$ admits a stability condition, unique up to $\bC$-action, whose central charge $Z$ factors through $H^2(X;\bZ)$ and is given by $Z(T) = \deg(T)$. By \cite{Meinhardtquotient}*{Prop. 3.13}, $\DCoh(X)/\cT \simeq \DCoh(\rm{mod}\:K(X))$, which also has a unique stability condition up to $\bC$-action with central charge given by the rank function. 

Let $\Sigma$ denote a multi-scale line with one root component and two terminal components corresp\-onding to $v_1,v_2\in V(\Sigma)_{\rm{term}}$ such that $\mathfrak{p}(v_1,v_2) = 1$. Set $\cC_{\le v_1} = \cT$ and $\cC_{\le v_2} = \DCoh(X)$ and let $\sigma_i \in \Stab(\gr_{v_i}\DCoh(X))/\bC$ denote the unique elements for $i=1,2$. By \cites{Br07,Macricurves}, $\Stab(X)/\bC\cong \bH$ via $\tau \mapsto Z_\tau(\cO_p)/Z_\tau(\cO_X)$, where $p\in X$ is any closed point. In \cites{NMMP,quasiconvergence} paths in $\Stab(X)/\bC$ are studied, which in these coordinates are given by 
\[
\tau(s) = \frac{2\pi i}{e^{i\theta}s+2(g(X)-1)C_{\rm{eu}}}
\]
where $C_{\rm{eu}}$ is the Euler-Mascheroni constant and $\theta \in (-\pi/2,\pi/2)$. 

\begin{lem}
\label{L:nonadmissiblehighergenus}
    In the above notation, $\sigma = \langle \cC_\bullet|\sigma_\bullet\rangle_{\Sigma}$ defines a non-admissible augmented stability condition and $\tau(s)$ converges to $\sigma$ in the topology of \Cref{T:topology}.
\end{lem}

\begin{proof}
    The relevant abelian group in \Cref{D:generalized_stability_condition} is $\Lambda = H^0(X;\bZ) \oplus H^2(X;\bZ)$ and we have $M_{v_1} = H^2(X;\bQ)$ and $M_{v_2} = H^0(X;\bQ)$. The conditions of \Cref{D:convergence_conditions} are brief to verify.\endnote{First, consider $\cC_{v_2}^{\rm{ss}}$, i.e. the objects that project to a semistable object in $\db(\rm{mod}\:K(X))$. This is in fact all objects. We may restrict ourselves to objects $E\in \Coh(X)$. These have Chern character valued in $H^0(X;\bZ)\oplus H^2(X;\bZ)$ given by $\ch(E) = d[\cO_p] + r[\cO_X]$ for $d,r\ge 0$. In particular, 
    \[
        \frac{Z_s(E)}{Z_s(\cO_X)} = r + \frac{Z_s(\cO_p)}{Z_s(\cO_X)} \to r
    \]
    as $s\to\infty$. Consequently, for any $E$, all of its Harder--Narasimhan factors have (normalized) phase converging to $1$. In particular, for any $t\ge 0$ $c^t_s(E) \to 1$. Also, $m_s(E)/m_s(\cO_X) \to r$. This gives \Cref{D:convergence_conditions} for $v_2$. The case of $v_1$ is immediate since all torsion sheaves are $\tau_s$-semistable for all $s$.} \Cref{T:topology}\eqref{I:stability} is immediate, since the $\sigma$-stable objects are all of the form $\cO_p$ for $p\in X$ and line bundles $L$ up to shift and $\tau(s)$ is equivalent to slope stability for all $s$. Finally, \Cref{T:topology}\eqref{I:phasewidthunif} follows from a short calculation, using the fact that the lattices for the stability conditions on $\cT$ and $\DCoh(X)/\cT$ are one-dimensional.\endnote{The essential observation is that the equivalent $\DCoh(X)/\cT\to \db(\rm{mod}\:K(X))$ is given by sending a sheaf to its stalk at the generic point of $X$. Consequently, any object of $\DCoh(X)$ is (up to shift) equivalent to $\cO^{\oplus r}$ modulo $\cT$, where $r = \rk(E)$.}
\end{proof}

\printendnotes

\bibliography{refs}{}
\bibliographystyle{plain}

\end{document}